\documentclass[a4paper,12pt,leqno]{article}
\usepackage{amsmath,amssymb,amsthm}
\theoremstyle{definition}
\newtheorem{Theorem}{Theorem}[section]
\newtheorem{Proposition}[Theorem]{Proposition}
\newtheorem{Lemma}[Theorem]{Lemma}

\newtheorem{Definition}[Theorem]{Definition}
\newtheorem{Corollary}[Theorem]{Corollary}

\newtheorem{Remark}[Theorem]{Remark}

\numberwithin{equation}{section}

\title{Basic techniques\\ in two-dimensional critical Ising percolation\\ with investigation of scaling relations}
\author{Yasunari Higuchi, Masato Takei, and Yu Zhang\\
{\footnotesize \it Kobe University, Osaka Electro-Communication University, and University of Colorado}}
\begin{document}
%\frontmatter
\maketitle

\begin{abstract}
We consider the percolation problem in the high-temperature Ising model on the two-dimensional square lattice at/near critical external fields. We show that all scaling relations,
except a single hyperscaling relation, hold under the power law assumptions
for the one-arm path and four-arm paths. %In addition, we show that
%the power laws hold for two- and three-arm paths in the half space.
%and five-arm paths in the whole space.
\end{abstract}

\tableofcontents
\newpage
\section{Introduction}

Consider the $\mathbf{Z}^2$ lattice and the sample space $\Omega = \{ -1,+1\}^{\mathbf{Z}^2}$ of spin configurations on $\mathbf{Z}^2$. Given a sample $\omega \in \Omega$ and $x \in \mathbf{Z}^2$, $\omega(x)$ denotes the spin value at $x$ in the configuration $\omega$. For any set $V \subset \mathbf{Z}^2$, denote by $\mathcal{F}_V$ the $\sigma$-algebra generated by $\{ \omega(x) : x \in V \}$, and we simply write $\mathcal{F}$ for $\mathcal{F}_{\mathbf{Z}^2}$.  Let $|x|$ and $|x|_{\infty}$ denote the $\ell^1$-norm and the $\ell^{\infty}$-norm of $x \in \mathbf{Z}^2$, respectively:
\[ |x| := |x^1| + |x^2| \mbox{ and } |x|_{\infty} := \max \{|x^1|,|x^2|\} \mbox{ for } x= (x^1,x^2) \in \mathbf{Z}^2. \]
% $d(x,y)$ be the Euclidean distance for $x,y \in  \mathbf{Z}^2$. 
For any finite $V$, we define the Hamiltonian for a configuration $\sigma \in \Omega_V =  \{ -1,+1\}^{V}$ by
\[ H_{V,h}^{\omega} (\sigma) = -\frac{1}{2} \sum_{x,y \in V,\,|x-y|=1} \sigma(x)\sigma(y) - \sum_{x \in V} \left(h + \sum_{y \notin V,\,|x-y|=1} \omega(y) \right) \sigma (x), \]
where $h$ is a real number called the external field. We then define the finite Gibbs measure on $\Omega_V$ by
\[ q_{V,T,h}^{\omega} (\sigma) = \left[ \sum_{\sigma' \in \Omega_V} \exp \{ -(\mathfrak{K}T)^{-1} H_{V,h}^{\omega} (\sigma') \} \right]^{-1} \exp \{ -(\mathfrak{K}T)^{-1} H_{V,h}^{\omega} (\sigma) \}. \]
Here $T$ is a positive number called the temperature, and $\mathfrak{K}$ is the Boltzmann constant. For each $T > 0$ and $h \in \mathbf{R}^1$, a Gibbs measure is a probability measure $\mu_{T,h} $ on $\Omega$ in the sense of the following DLR equation:
\[ \mu_{T,h} (\;\cdot\,|\,\mathcal{F}_{V^c})(\omega) = q_{V,T,h}^{\omega} (\;\cdot\;) \quad \mbox{$\mu_{T,h}$-almost every $\omega$}, \]
where $V^c = \mathbf{Z}^2 \setminus V$. Let $T_c$ be the critical value such that if $T > T_c$ or $h \neq 0$, the Gibbs measure is unique for $(T, h)$.

For two vertices $x$ and $y$ of $\mathbf{Z}^2$, 
we say they are adjacent if 
\[ |x-y| = 1.\]
Moreover, we say they are $(*)$-adjacent
if 
\[ |x-y|_{\infty} =1. \]
In words, for each vertex, its adjacent vertices are its four vertical and horizontal neighbors, while its $(*)$-adjacent vertices are its four vertical and horizontal neighbors together with the other four diagonal neighbors. 
A path [$(*)$-path]
\[\gamma=\{ x_1,x_2,\ldots, x_s\}\]
is a sequence of vertices such that $x_{i-1}$ and $x_{i}$
are adjacent [$(*)$-adjacent]. A path is called a $(+)$-path in if the spin value is $+$ for every point of this path. Similarly, an $(*)$-path is called a $(-*)$-path in if the spin value is $-$ for every point of this path. A  $(+)$-cluster [$(-*)$-cluster] is the set of vertices
connected by $(+)$-paths [$(-*)$-paths].
Let ${\bf C}_u^+$ be
the $(+)$-cluster that contains $u$.
In particular, let Let ${\bf C}_0^+$ be
the $(+)$-cluster that contains the origin.
For $T >0$, we define $h_c(T)$ by
\[ h_c(T)= \inf \{ h: \mu_{T, h}( \#{\bf C}_0^+=\infty)>0\}. \]
It follows from our definition that there exists an infinite
$(+)$-cluster with probability one when $h > h_c(T)$. In this case, percolation
occurs.
It has been proved that (see \cite{Hig93b}) that if $T \leq T_c$,
then 
\[ h_c(T)= 0.\]
On the other hand, if $T > T_c$,
then 
\begin{equation} \label{Z1.1}
h_c(T)> 0. %\eqno(1.1)
\end{equation}
In this paper, we would like to focus on the high-temperature case.
%(\ref{Z1.1}). 

The most interesting problem in statistical physics
is to understand the behaviors of various quantities
in percolation when $h$ is near $h_c(T)$.
Indeed, it is widely believed (see e.g. Grimmett \cite{Gri99}) that the critical exponents of various quantities in percolation behave like power laws of
$|h-h_c(T)|$ as $h$ approaches $h_c(T)$. To express these
conjectures precisely, we would like to define a few quantities.
Let 
\[ B(r,R)= [-R, R]^2\setminus (-r,r)^2.\]
We first define the probability of $k$-arm paths in $B(r, R)$.
For $k_1 \geq k_2$ with $k_1+k_2=k$,
let $\mathcal{B}(k_1, k_2, r, R)$ be the event that there exist $k_1$ disjoint $(+)$-paths   
from $\partial [-r, r]^2$ to $\partial [-R,R]^2$,
and 
there exist $k_2$ disjoint
$(-*)$-paths   
from $\partial [-r, r]^2$ to $\partial [-R,R]^2$,
and all $(-*)$-paths are separated by $(+)$-paths.

In addition to these $k$-arm paths in $B(r, R)$, we may also consider 
$k$-arm paths in the half space. 
Let 
\[ B^+(r,R)= [0, R]\times [-R, R]\setminus (0, r)\times (-r, r).\]
For $k_1 \geq k_2$ with $k_1+k_2=k$,
let $\mathcal{B}^+(k_1, k_2, r, R)$ be the event that there exist $k_1$ disjoint
$(+)$-paths   
from $\partial [-r, r]^2$ to $\partial [-R,R]^2$,
and 
there exist $k_2$ disjoint
$(-*)$-paths   
from $\partial [-r, r]^2$ to $\partial [-R,R]^2$,
all paths stay in the half
space $(0, \infty)\times (-\infty, \infty)$,
and all $(-*)$-paths are separated by $(+)$-paths.

With these definitions, it is believed (see e.g. \cite{ADA99}) that
\begin{align} 
\mu_{T, h_c(T)} (\mathcal{B}(1, 0, r, R)) &\asymp \left( {\dfrac{r}{R}}\right)^{5/48}, \notag \\
\mu_{T, h_c(T)} (\mathcal{B}(k_1, k_2, r, R)) &\asymp \left( {\dfrac{r}{R}}\right)^{(k^2-1)/12}\mbox{ for } k_1+k_2=k \geq 2,\label{Z1.2} %\eqno{(1.2)}
\end{align}
where $f(n) \asymp g(n)$ means that $C_1 g(n) \leq f(n) \leq C_2 g(n)$.
For $k$-arm paths in the half space, it is also believed that
\begin{equation} \label{Z1.3}
\mu_{T, h_c(T)} (\mathcal{B}^+(k_1, k_2,r, R)) \asymp \left( {\dfrac{r}{R}}\right)^{(k(k+1))/6}\mbox{ for } k_1+k_2=k.
%\eqno{(1.3)}
\end{equation}
The conjectures (\ref{Z1.2}) and (\ref{Z1.3}) have been proved for the
independent percolation model, $T=\infty$, in the triangular lattice 
(see Smirnov-Werner \cite{SW01}). In addition, for $k_1=3$ and $k_2=2$, the conjecture
of (\ref{Z1.2}) was proved to be true (see Kesten-Sidoravicius-Zhang \cite{KSZ98}) for all two-dimensional periodic lattice. For $k_1=1$ and $k_2=1$, or $k_1=2$ and $k_2=1$ the conjecture of (\ref{Z1.3}) is also true (Zhang \cite{Zpower}) for all two-dimensional periodic lattice. In fact, 
the conjecture of (1.2) for $k_1=1$ and $k_2=0$ (one-arm path)
and $k_1=2$ and $k_2=2$ (four-arm paths)
play the most important roles. Kesten \cite{K87scaling} showed that
if one-arm and four-arm conjectures in (\ref{Z1.2}) hold, then
almost all critical exponents exist and satisfies the scaling relations (see the definitions below) for the
independent percolation model. 
More precisely, it is believed that
\[\mu_{T, h_c(T)} (\mathcal{B}(1, 0,1, R)) \asymp \left( \dfrac{1}{R}\right)^{1/\delta_{\text r}}, \]
and
\[\mu_{T, h_c(T)} (\mathcal{B}(2,2,1, R)) \asymp \left( \dfrac{1}{R}\right)^{2-1/\nu} \]
for some constants $\delta_{\text r}$ and $\nu$. The simulations 
indicate that
\[ \delta_{\text r} =48/5 \mbox{ and } \nu=4/3.\]

%In this paper, we show that some special cases of (\ref{Z1.2}) and (\ref{Z1.3}) hold.

%\begin{Theorem} For $T > T_c$,  
%\begin{align*}
%\mu_{T, h_c(T)} (\mathcal{B}^+(1, 1,r, R)) &\asymp {\dfrac{r}{R}}, \\
%\mu_{T, h_c(T)} (\mathcal{B}^+(2, 1,r, R)) &\asymp \left( {\dfrac{r}{R}}\right)^{2}, \\
%\mu_{T, h_c(T)} (\mathcal{B}(3, 2, r, R)) &\asymp \left( {\dfrac{r}{R}}\right)^{2}.
%\end{align*}
%\end{Theorem}

We next want to introduce more critical exponents for a given
$T > T_c$.
\begin{itemize}
\item Percolation probability:
\[ \theta(h)=\mu_{T, h}(\# {\bf C}_0^+=\infty).\]
\item Average number of clusters per site:
\[ \kappa(h)= \sum_{n=1}^\infty  \dfrac{1}{n} \mu_{T, h}( \# {\bf C}_0^+=n) =\mu_{T, h} [(\# {\bf C}_0^+)^{-1}:\# {\bf C}_0^+>0].\]
\item Mean cluster size:
\[\chi(h)=\mu_{T, h} [\# {\bf C}_0^+ : \# {\bf C}_0^+ <\infty].\]
\item Correlation length:
\[L(h)=
\begin{cases}
\min\{ n: \mu_{T, h}(A^+(n,n)) \geq 1-\varepsilon\} &\mbox{if } h > h_c(T),\\
\min\{ n: \mu_{T, h}(A^+(n,n)) \leq \varepsilon \} &\mbox{if } h < h_c(T),\\
\end{cases}
\]
where $A^+(n,n) := \{ \mbox{there exists a horizontal $(+)$-crossing of $[-n,n]^2$} \}$.
\end{itemize}

With these definitions, it is widely believed (see e.g. Grimmett \cite{Gri99}) 
that the following power laws hold with the exponents:
\begin{align*}
L(h) &\asymp |h -h_c(T)|^{-\nu}, \\
|\kappa'''(h)| &\asymp |h -h_c(T)|^{-1-\alpha }, \\
\theta(h) &\asymp (h -h_c(T))^\beta \quad \mbox{for $h > h_c(T)$,} \\
\chi (h) &\asymp |h -h_c(T)|^{-\gamma},\\
\dfrac{\mu_{T, h} [(\# {\bf C}_0^+)^{k} : \# {\bf C}_0^+ <\infty]}
{\mu_{T, h} [(\# {\bf C}_0^+)^{k-1} : \# {\bf C}_0^+ <\infty]} &\asymp 
|h-h_c(T)|^{-\Delta_k} \quad \mbox{for $k \geq 1$.}
\end{align*}
In addition to the power laws, it is also believed that the 
exponents satisfy the following scaling relations for low enough $d$.
\begin{align*}
\alpha&=2-d\nu, \\
\beta&=\dfrac{d\nu}{\delta+1}, \\
\gamma&=d\nu \dfrac{\delta-1}{\delta +1}, \\
\Delta_k&=d\nu \dfrac{\delta}{\delta+1} \quad \mbox{for $k \geq 1$.}\\
\eta&=2-d\dfrac{\delta-1}{\delta+1}. 
\end{align*}
\newpage

\section{Preliminaries}

The aim of this paper is to establish scaling relations for the Ising
percolation in two dimensions,
parallel to the ones 
obtained in  \cite{K87scaling} for independent percolation.
We know that the critical value of the external field $h_c(T)$ is positive
if $T>T_c$, and hereafter we fix such a $T$.
Since we are looking at the behaviors of quantities like percolation 
probability as the external field $h$ approaches to $h_c(T)$, 
we may assume that $h$ is near $h_c(T)$.
Further, we need several $h$-independent estimates later to
obtain desired relations, but often such estimates do not hold
for all $h$'s.
Therefore we have to restrict the range of $h$'s at the very beginning,
and throughout this paper we assume that the value $h$ of the
external field of our model satisfies
\begin{equation}\label{bound for the external field}
0 \leq h \leq 2h_c(T).
\end{equation}
In this section, we summarize known results.
\subsection{Mixing property}

For $V_1,\,V_2 \subset \mathbf{Z}^2$, $d(V_1,V_2)$ denotes the $\ell^1$-distance between $V_1$ and $V_2$; that is,
\[ d(V_1,V_2) = \inf\{ |x-y| : x \in V_1,\,y \in V_2 \}. \]
We also use the $\ell^\infty $-distance
\[ d_\infty (V_1,V_2)= \inf \{ |x-y|_\infty : x\in V_1, y\in V_2 \} . \]

The following is a refinement of Theorem 2 (ii) of \cite{Hig93a}, which can be obtained without changing the proof given in \cite{Hig93a}.
\begin{Theorem}\label{th:mixing} 
Let $V\subset \Lambda $ be finite subsets of ${\mathbf Z}^2$,
$T > T_c$ and $h\geq 0$. Let $\ell $ be a positive integer. Assume that $A \in {\mathcal F}_V$ and $\omega_1, \omega_2 \in \Omega $
satisfy
\[ \omega_1(x) = \omega _2(x)=+1 \]
 for every $x\in \partial\Lambda $ with $d(x,V)<\ell $.
Then there exist constants $C>0$ and $\alpha >0$ such that
\[ | q_{\Lambda ,T ,h}^{\omega_1}(A)- 
       q_{\Lambda ,T ,h}^{\omega_2}(A)|
       \leq C \sum_{x\in V}
           \sum_{y\in \partial \Lambda, d(y,V)\geq \ell }
             e^{-\alpha |x-y|}.
\]
\end{Theorem}
In particular, %there exist
%%constants $C>0$ and $\alpha > 0$ such that 
for every pair of finite subsets $V$ and $W$ of $\mathbf{Z}^2$ with $V \subset W$, 
\begin{align} \label{HZ(1.1)}
&\sup_{\omega \in \Omega,\,A \in \mathcal{F}_V} \left| \mu_{T,h}(A) - \mu_{T,h}(A\,|\,\mathcal{F}_{W^c} )(\omega)\right| \\
&\leq C|V| d(V,W^c)\exp\{-\alpha d(V,W^c)\}. \notag
\end{align}

Sometimes we use the mixing property (\ref{HZ(1.1)}) in the following form: 
if $V_1,\,V_2 \subset \mathbf{Z}^2$ are finite sets, $V_1 \cap V_2 = \emptyset$ and $A,B$ are cylinder sets such that $A \in \mathcal{F}_{V_1}$ and $B \in \mathcal{F}_{V_2}$, then
\begin{align} \label{HZ(1.3)}
&\left|\mu_{T,h} (A \cap B) - \mu_{T,h} (A) \mu_{T,h} (B) \right| \\
&\leq C|V_1| d(V_1,V_2)\exp\{-\alpha d(V_1,V_2)\} \mu_{T,h} (B). \notag
\end{align}

\subsection{Gibbs measures with periodic boundary condition}
As in \cite{K87scaling}, we will have to take derivative in $h$ of
$\mu_{T,h}$-probability of some event in a finite box.
However, this will cause a new problem, since $\mu_{T,h}$ itself is a limit
of finite Gibbs measures.
Here we use the fact that $\mu_{T,h}$ is a limit of Gibbs measures 
$\mu_{T,h}^N$ on $S(N)$ with the periodic boundary condition, i.e.,
we identify the left side of $S(N)$ with the right side of it, and the
top side with the bottom side.

If $2n<N$, then for any $A\in {\mathcal F}_{S(n)}$, by (\ref{HZ(1.1)}), we have
\begin{align}\label{eq:periodic-1}
| \mu_{T,h}^N(A)-\mu_{T,h}(A)|&\leq 
  \sup_{\omega \in \Omega }|q^\omega_{S(2n),T,h}(A)-\mu_{T,h}(A)|\\
 &\leq 4Cn^3e^{-\alpha n}\nonumber
\end{align}
Also, fixing an $h\in (0,2h_c(T))$, we introduce $h(t)$ as
\begin{equation}\label{eq:peridic-2}
h(t)=\begin{cases}
     h_c(T)+t(h-h_c(T)), & \mbox{ if } h>h_c(T),\\
     h+t(h_c(T)-h), &\mbox{ if }  h<h_c(T),
    \end{cases}
\end{equation}                   
and $\mu_t^N$ as
\begin{equation}\label{eq:peridic-3}
\mu_t^N=\mu_{T,h(t)}^N
\end{equation}
for $t\in [0,1]$.
We will have to investigate $\displaystyle \frac{d}{dt}\mu_t^N(A)$ for
events in $S(n)$ with $2n<N$.

\subsection{Correlation length and crossing probabilities}
The definition of the correlation length $L(h)$ depends on the choice of 
threshold $\varepsilon $.
So, to emphasize the dependence on $\varepsilon$, we write
$L(h,\varepsilon )$ for the correlation length throughout the paper.
A path [resp. A $(*)$-path] is a sequence of points $x_1, x_2, \ldots , x_s$ 
in ${\mathbf Z}^2$
such that $|x_i-x_{i-1}|=1 $ [resp. $|x_i-x_{i-1}|_\infty =1$] for every $i\leq s$.
If every spin value on a given path $\gamma $ is $+1$, 
then we call this $\gamma $
a $(+)$-path. $(-)$-paths, $(+*)$-paths, $(-*)$-paths are defined similarly.
A circuit is a sequence of points 
$\gamma = \{ x_1,x_2,\ldots,x_s \}$ such that 
$\gamma$ is a path except that $x_s$ and $x_1$ satisfy $|x_s-x_1|=1$. 
Similarly, we define a $(*)$-circuit  by replacing 
``path''  with ``$(*)$-path'', and $|\cdot |$ with $|\cdot |_\infty $.
A circuit $\gamma$ is called a $(+)$-circuit in the configuration 
$\omega$ if $\omega(x)=+1$ for every $x \in \gamma$. 
A $(-*)$-circuit is defined in the 
same way.
Let
\[ A^+(n,m) := \{ \mbox{there exists a horizontal $(+)$-crossing of $[-n,n] \times [-m,m]$} \}, \]
where a horizontal $(+)$-crossing of a rectangle is a $(+)$-path connecting the left and the right sides of that rectangle. We define $A^{-*}(n,m)$ in the same way. 

Further, for $n\geq 1$, let $S(n)$ denote the square $[-n,n]^2$.
We use the notation $S(x,n)$ for the shifted square $S(n)+x$. 
We are also interested in the event that there is a $(+)$- or $(-*)$-path
in $S(n)$ which connects the origin with the inner boundary
$$
\partial_{in}S(n)= \{ x \in S(n) : |x|_\infty =n \} ,
$$
Let $\mathcal{A},\mathcal{B}, V $ be subsets of ${\mathbf Z}^2$,
such that $\mathcal{A}\cap \mathcal{B}=\emptyset,
\mathcal{A}\cup \mathcal{B}\subset V $.
Then let
\begin{align*}
&\{ \mathcal{A} \stackrel{+}{\leftrightarrow }\mathcal{B}
   \mbox{ in } V \} \\
&:=\left\{ \omega \in \Omega :
\begin{array}{@{\,}c@{\,}}
\mbox{there exists a $(+)$-path in $V$ connecting}\\
\mbox{some point $x\in \mathcal{A}$ with some point $y\in \mathcal{B}$}
\end{array}
\right\} .
\end{align*}
If $V ={\mathbf Z}^2$, then we simply write it 
$\displaystyle \{ \mathcal{A} \stackrel{+}{\leftrightarrow }\mathcal{B}\} $.
We also use the notation 
$\displaystyle \{ {\mathbf O}\stackrel{+}{\leftrightarrow }
      \partial_{in}S(n) \} $
for
$\displaystyle \{ {\mathbf O}\stackrel{+}{\leftrightarrow }
      \partial_{in}S(n) \mbox{ in } V \} $
if $S(n)\subset V $.
Events $\displaystyle
\{ \mathcal{A}\stackrel{-*}{\leftrightarrow}\mathcal{B} 
   \mbox{ in } V \} $
and $ \displaystyle 
\{ \mathbf{O}\stackrel{-*}{\leftrightarrow} \partial_{in}S(n)\} $ 
are defined similarly.
\begin{Lemma}[an ACCFR-type rescaling lemma]\label{lem:accfr}
Let $\lambda =1/64$, and $0<\theta < 1$.
We put $n_0$ to be the integer such that
\begin{equation}\label{eq:1-st bound for n}
\max\{ 1,C\} (3n)^3 e^{-\alpha n} < 
 \lambda \theta \quad  \mbox{ for every } n\geq n_0,
\end{equation}
where $C$ and $\alpha $ are constants appearing in the mixing properties (\ref{HZ(1.1)}) and (\ref{HZ(1.3)}).
Then for any $L \geq n_0$,
if we have
$$
\mu_{T,h} \bigl( A^+(3L,L) \bigr) \geq 1-\lambda \theta ,
$$
then we have for every $k\geq 1$,
$$
\mu_{T,h} \bigl( A^+(3^{k+1} L,3^k L) \bigr) \geq 1 - \lambda \theta^{2^k}. 
$$
The same statement holds for $(-*)$-connection, too.
\end{Lemma}
\begin{proof} We use a rescaling argument in \cite{ACCFR}.  
%%  It is easy to see that the mixing property in Theorem \ref{th:mixing} 
%%  is sufficient to prove the lemma.
%%
%% \begin{quote}
%% ---------------------------------------------
%%  \noindent \framebox{\bf INTERMISSION}
Assume that the inequality holds for $k$;
\[ \mu_{T,h} \bigl( A^+ (3^{k+1} L,3^{k} L) \bigr)
 \geq 1 - \lambda \theta^{2^k}. \]
By the mixing property and the translation-invariance, we have
\begin{align*} 
&\mu_{T,h} \bigl( A^+(3^{k+2}L, 3^{k+1}L) \bigr) \\
&\geq \mu_{T,h} \left( 
\begin{array}{@{\,}c@{\,}} 
\mbox{there exists a horizontal $(+)$-crossing of}\\
\mbox{$[-3^{k+2}L,3^{k+2}L] 
\times [-3^{k+1}L,-3^{k}L]$, or there exists} \\
\mbox{a horizontal $(+)$-crossing of $[-3^{k+2}L,3^{k+2}L] 
\times [3^{k}L,3^{k+1}L]$} \\
\end{array} 
\right) \\
&\geq 1 - \bigl\{1 - \mu_{T,h}\bigl(A^+(3^{k+2} L,3^{k} L)\bigr) 
+ C (4\cdot 3^{k+2} L \cdot 3^k L)\times 2\cdot 3^k L 
   e^{-\alpha 2\cdot 3^k L} \bigr\}^2 \\
&= 1-\bigl\{1 - \mu_{T,h}\bigl(A^+ (3^{k+2} L,3^{k} L)\bigr) 
    + 8C 3^{3k+2} L^3 e^{-\alpha 2\cdot 3^k L} \bigr\}^2.
\end{align*}
By the FKG inequality it is easy to see that
\begin{align*}
\mu_{T,h}\bigl(A^+ (3^{k+2} L,3^{k} L)\bigr) 
   &\geq \mu_{T,h}\bigl(A^+ (3^{k+1} L,3^{k} L)\bigr)^7 \\
&\geq (1-\lambda \theta^{2^k})^7 \geq 1-7\lambda \theta^{2^k}.
\end{align*}
Thus we have
\begin{align*} 
\mu_{T,h} \bigl(A^+ (3^{k+2}L,3^{k+1}L)\bigr) 
\geq 1-( 7\lambda \theta^{2^k} + 8C 3^{3k+2} L^3 e^{-2\alpha 3^k L} )^2.
\end{align*}

Now we go back to the case where $k=0$. 
By the above calculation, and then from (\ref{eq:1-st bound for n}),
\begin{align*} 
\mu_{T,h}\bigl(A^+(3^2 L,3L)\bigr) 
&\geq 1-( 7\lambda \theta + 8C 3^2 L^3 e^{-2\alpha  L} )^2 \\
&\geq 1- (8\lambda \theta )^2\\
&= 1- \lambda \theta^2,
\end{align*}
since $\lambda = 1/64$.

In general, if
\[ \max\{ 1, C\}8\cdot 3^{3k+2} L^3 e^{-2\alpha 3^k L} \leq \lambda 
\theta^{2^k}, \]
then
\begin{align*}
\max\{ 1, C\} 8\cdot  3^{3(k+1)+2} L^3 e^{-2\alpha 3^{k+1} L} 
\leq \lambda \theta^{2^k} \cdot 3^3 e^{-\alpha 4 \cdot 3^k L} 
\leq  \lambda \theta^{2^k} \cdot \theta^{2^k} =  \lambda \theta^{2^{k+1}},
\end{align*}
and
\begin{align*} 
\mu_{T,h} \bigl(A^+(3^{k+2}L,3^{k+1}L)\bigr) 
\geq 1-( 8\lambda \theta^{2^k})^2 \geq 1-64\lambda^2 \theta^{2^{k+1}} \geq 1-\lambda \theta^{2^{k+1}}.
\end{align*}

%%  ---------------------------------------------------
%%  \end{quote}
\end{proof}

The RSW-type theorem (Lemmas 2.4--2.6 in \cite{Hig93b}) ensures 
that the following statements are true.
\begin{enumerate}
\item For every $\varepsilon >0$, there exists an integer ${\tilde n}_1=
{\tilde n}_1(T, \varepsilon )$ such that if for some $n\geq {\tilde n}_1$,
\begin{equation}\label{rsw-c1}
  \mu_{T,h} \bigl(A^+(m,m) \bigr)  \geq \varepsilon
\end{equation} 
for every $1\leq m\leq n$, then for every integer $k$ there exists a
constant $\delta_k=\delta_k(\varepsilon)$ such that
\begin{equation}\label{rsw-1}
\mu_{T,h} \bigl(A^+(kn,n)\bigr) \geq \delta_k
\end{equation}
for the same $n$.
The same statement holds true for $(-*)$-connection, too.
\item If there exists a constant $\eta >0$ and
an integer $\tilde{n}_2\geq 1$
such that for every $n\geq \tilde{n}_2$ 
\begin{equation}\label{rsw-c2}
 \mu_{T,h} \bigl(A^+(n,n)\bigr)\geq \eta  \quad \mbox{ or } \quad
    \mu_{T,h}\bigl(A^{-*}(n,n)\bigr) \geq \eta
\end{equation}
holds, then for every $\varepsilon >0$,
there exists an integer 
$\tilde{n}_3=\tilde{n}_3(T,\varepsilon , \eta )\geq \tilde{n}_2$
such that if
\begin{equation}\label{rsw-c3}
 \mu_{T,h}\bigl(A^+(n,n)\bigr)\geq 1-\varepsilon ,
\end{equation}
for some $n\geq \tilde{n}_3$, then
\begin{equation}\label{rsw-2}
 \mu_{T,h}\bigl(A^+(3n,n)\bigr)\geq (1-\varepsilon )^3(1-{\root 8\of \varepsilon})^{36}.
\end{equation}
If (\ref{rsw-c3}) holds for $(-*)$-connection for some $n\geq \tilde{n}_3$,
then so does (\ref{rsw-2})
for $(-*)$-connection for the same $n$, too.
\end{enumerate}

%%%%We take an $\varepsilon_0>0$ sufficiently small so that the right hand side
%%%%of (\ref{rsw-2}) is larger than $1-\lambda \theta $, where $\lambda , \theta $
%%%%are given in  Lemma \ref{lem:accfr}, and hereafter fix it.
%%%%We then choose $\eta = \varepsilon_0$, $\tilde{n}_2=\tilde{n}_1$, and we write $n_1$ for the resulting $\tilde{n}_3(T,\varepsilon_0 )$.

%\begin{quote}
%---------------------------------------------
%
%\noindent \framebox{\bf INTERMISSION}

The following lemma states the above result and we present 
a proof here for the sake of completeness.
%%% in this note, but in the final version, we may skip this part.

%For our purpose, we formulate the above result in the following way.
\begin{Lemma}\label{RSW-gen} 
\noindent
(i) Let $\mu = \mu_{T,h}$ or $\mu_t^N$ for $0\leq t \leq 1$.
For every $\varepsilon >0$, there is an $n_1^*=n_1^*(\varepsilon )\geq 1, $
and also there are positive numbers $\{ \delta_k(\varepsilon ) : k \in {\mathbf N} \} $
such that if
\begin{equation}\label{wrsw-cond}
\textstyle
\mu \bigl(A^{+} (n,n)\bigr)\geq \varepsilon , \ \mbox{ and } \  
\mu \bigl(A^{+} (\frac{2n}{5},\frac{2n}{5}) \bigr)\geq \varepsilon ,
\end{equation}
for some $n\geq n_1^*$, then we have
\begin{equation}\label{wrsw-result}
\mu \bigl(A^{+ }(kn,n)\bigr)\geq \delta_k(\varepsilon )
\end{equation}
for the same $n$, and for every $N\geq kn$ if $\mu =\mu_t^N$.
Moreover for $k\geq 2$, $\delta_{k+1}(\varepsilon )
 \geq \delta_k(\varepsilon )\cdot (\varepsilon \delta_2(\varepsilon ))$,
therefore we have $\delta_{k+1}(\varepsilon )
  \geq \delta_{2}(\varepsilon)\cdot (\varepsilon\delta_2(\varepsilon ))^{k-1}$.

\noindent
(ii) There exists an $\varepsilon_0>0$ and an 
$n_1= n_1(\varepsilon_0)\geq \max\{ n_0, n_1^*(\varepsilon_0)\} $ 
such that the following statement holds.
Let $M$ be an integer satisfying
\begin{equation}\label{wrsw-cond2}
M \geq \frac{\log \varepsilon_0}{\log (1-\delta_8(\varepsilon_0)^4/2)},
\end{equation}
and take a large $m\geq n_1$ such that $m\geq 2\cdot 4^{M+1}(\log m)^2$.
Assume that the inequality (\ref{wrsw-cond}) holds for 
$\mu = \mu_{T,h}, \varepsilon = \varepsilon_0$ and
for every $n$ with $n_1^*(\varepsilon_0)\leq n\leq m$, and that
\begin{equation}\label{wrsw-cond3}
\mu_{T,h} \bigl( A^{+ } (m,m)\bigr) \geq 1-\varepsilon_0.
\end{equation}
Then we have
$$
\mu_{T,h} \bigl(A^{+ }(3m,m)\bigr)\geq 1-\lambda \theta ,
$$
where $\lambda $ and $\theta $ are the constants given in 
Lemma \ref{lem:accfr}.
The same statements hold for $(-*)$-connection, too.
\end{Lemma}
\begin{proof}
We only prove the above statements for $(+)$-connection.
The proof goes parallel for $(-*)$-connection, too.

\noindent
$1^\circ )$ About the first statement:
Let $\ell $ be the segment connecting
$(-n,\lfloor\frac{4n}{5}\rfloor )$ with $(n,\lfloor \frac{4n}{5}\rfloor )$,
where $\lfloor u\rfloor $ is the largest integer not more than $u$. 
We divide $\ell $ into three segments by 
$a_L=(-\lfloor \frac{2n}{5}\rfloor, \lfloor \frac{4n}{5}\rfloor )$
and $a_R=(\lfloor \frac{2n}{5}\rfloor ,\lfloor \frac{4n}{5}\rfloor )$, 
and call by $\ell_1 $ the segment
connecting $(-n,\lfloor \frac{4n}{5}\rfloor )$ with $a_L$, $\ell_2 $ 
the segment connecting
$a_L$ with $a_R$, and by $\ell_3$ the segment connecting $a_R$ with
$(n,\lfloor \frac{4n}{5}\rfloor )$.
Let $R(\omega )$ denote the lowest horizontal $(+)$-crossing of $S(n)$,
and let the event $E$ be the subset of $A^+(n,n)$ given by
$$
E=\left\{ 
\begin{array}{@{\,}c@{\,}}
 \mbox{there is a point $x\in \ell_2 $ such that $x$ is} \\
 \mbox{below the lowest $(+)$-crossing $R(\omega )$ of $S(n)$}
\end{array}
\right\} .
$$
Either $\mu (E)\geq \varepsilon /2$ or 
$\mu \bigl(A^+(n,n)\setminus E\bigr)\geq \varepsilon /2$.

Assume first the latter case;
$\mu \bigl(A^+(n,n)\setminus E\bigr)\geq \varepsilon /2$.
Define also the following events:
\begin{align*}
F_0&=\left\{
\begin{array}{@{\,}c@{\,}}
 \mbox{there exists a horizontal $(+)$-crossing of $S(n)$}\\
 \mbox{below the segment  $\ell $}
 \end{array}
 \right\} , \\
F_1&=\left\{
 \begin{array}{@{\,}c@{\,}}
 \mbox{there exists a $(+)$-path in $S(n)$ below $\ell $}\\
 \mbox{which connects $\ell_1$ with the segment} \\
  \{ (n, j) : -n\leq \ j\leq \frac{4n}{5}\}
 \end{array}
 \right\} , \\
F_2&=\left\{ 
\begin{array}{@{\,}c@{\,}}
 \mbox{there exists a $(+)$-path in $S(n)$ below $\ell $}\\
 \mbox{connecting $\ell_1$ with $\ell_3$}
\end{array}
\right\} , \\
F_3&=\left\{
 \begin{array}{@{\,}c@{\,}}
  \mbox{there exists a $(+)$-path in $S(n)$ below $\ell $}\\
  \mbox{which connects $\ell_3$ with the segment} \\
  \{ (-n, j) : -n\leq \ j\leq \frac{4n}{5} \}
 \end{array}
 \right\}.
\end{align*}
Then, it is easy to see that 
$ \bigl( A^+(n,n)\setminus E \bigr) \subset \cup_{0\leq i\leq 3}F_i $.  
Since $F_i, 0\leq i\leq 3$ are increasing, by the FKG inequality,
we have
\begin{equation}\label{eq:wrsw-1}
\mu (F_i)\geq 1-\root 4 \of {1-\varepsilon /2}
\end{equation}
for some $0\leq i\leq 3$.
If (\ref{eq:wrsw-1}) holds for $i=0$, we have already
$$
\textstyle
\mu \bigl(A^+ (n,\frac{9n}{10} ) \bigr)\geq  1-\root 4 \of {1-\varepsilon /2},
$$
and by the FKG inequality we have
\begin{equation}\label{eq:wrsw-2}
\textstyle
\mu \bigl(A^+ (\frac{11n}{10},n )\bigr) \geq 
{\bigl(  1-\root 4 \of {1-\varepsilon /2} \bigr)}^3.
\end{equation}
If (\ref{eq:wrsw-1}) is true for $i=1$ or $3$, then by the
symmetry of $\mu $ with respect to the line 
$\{ x^1= -\lfloor \frac{2n}{5}\rfloor \}$
or the line $\{ x^1=\lfloor \frac{2n}{5}\rfloor \} $, 
and by the FKG inequality we have
\begin{equation}\label{eq:wrws-3}
\textstyle
\mu \bigl(A^+ (\frac{14n}{10},n )\bigr)\geq 
{\bigl(  1-\root 4 \of {1-\varepsilon /2} \bigr)}^2.
\end{equation}
Also, if (\ref{eq:wrsw-1}) is true for $i=2$, then by the symmetry of 
$\mu $ with respect to the line $\{ x^1= -\lfloor \frac{2n}{5}\rfloor \}$
and the line $\{ x^1=\lfloor \frac{2n}{5}\rfloor \} $, and by the FKG 
inequality we have
\begin{equation}\label{eq:wrsw-4}
\textstyle
\mu \bigl(A^+ (\frac{12n}{10},n )\bigr)\geq 
 {\bigl(  1-\root 4 \of {1-\varepsilon /2} \bigr)}^3.
\end{equation}
Combining (\ref{eq:wrsw-1})--(\ref{eq:wrsw-4}), we have
\begin{equation}\label{eq:wrsw-5}
\textstyle
\mu \bigl(A^+ (\frac{11n}{10},n ) \bigr)\geq 
{\bigl(  1-\root 4 \of {1-\varepsilon /2} \bigr)}^3.
\end{equation}

Finally assume that $\mu (E)\geq \varepsilon /2$.
Let $R$ be a horizontal crossing of $S(n)$ such that there exists a 
point $x\in \ell_2$ below $R$.
Take arbitrarily a point $x^*\in\ell_2$ which lies below $R$.
Let $S(x^*,\lfloor \frac{2n}{5}\rfloor )$ be the square centered at
$x^*$ with radius
$\lfloor \frac{2n}{5}\rfloor $, i.e., its side length is
$2\lfloor \frac{2n}{5}\rfloor $.
Then consider the conditional probability
\begin{equation}\label{eq:wrsw-6}
\mu \left(
\begin{array}{@{\,}c@{\,}}
\mbox{there exists a $(+)$-path in $S(x^*,\lfloor \frac{2n}{5}\rfloor )$}\\
\mbox{connecting its top side with $R$}
 \end{array}
\, \bigg\vert \,
 R(\omega ) = R 
       \right) .
\end{equation}
For simplicity let us write $G$ for the event in the above
conditional probability.
Since $G$ depends on the configurations in the region
on and above $R$, by the Markov property of $\mu $
and by the FKG inequality, the above conditional probability
is not less than
\begin{equation}\label{eq:wrsw-7}
q^{\omega^-}_{V ,T,h}
\left( G \mid \omega (x)=+1, \  x\in R  \right) ,
\end{equation}
where $V $ is given by
$$
\textstyle
V = [-\frac{7n}{5},n]\times \left[-n,\frac{7n}{5} \right],
$$
and $\omega^-(x)= -1$ for every $x\in \mathbf{Z}^2$.
Then again by the FKG inequality the above conditional probability 
is not less than
$$
\mu (G \mid \omega (x) = -1, \ x \in \partial V ).
$$
Now, $G$ is the event in $S(x^*,\lfloor \frac{2n}{5}\rfloor )$ whose distance
from $\partial V $ is not less than $\frac{n}{5}$.
Therefore by the mixing property, we have
\begin{equation}\label{eq:wrsw-mix}
\mu (G \mid \omega (x) = -1, \ x \in \partial V )
\geq \mu (G) - Cn^3e^{-\alpha \frac{n}{5}}.
\end{equation}
We also note that by the symmetry of $\mu $ with respect to the line
$\{ x^2=\lfloor \frac{4n}{5}\rfloor \} $ and by the FKG inequality
$$
\mu \left( 
 \begin{array}{@{\,}c@{\,}}
\mbox{there is a horizontal $(+)$-crossing of} \\
\mbox{$S(x^*,\frac{2n}{5})$ passing below $x^*$}
\end{array}
 \right) \geq 1- \sqrt{1-\varepsilon }.
$$
Therefore by the FKG inequality and the rotation symmetry we have
$$
\mu (G) \geq {\bigl( 1-\sqrt{1-\varepsilon }\bigr) }^4,
$$
since $R$ passes above $x^*$.
Therefore if we take $n^*_1=n_1^*(\varepsilon )$ sufficiently large so that
\begin{equation}\label{eq:2nd bound for n}
Cn^3e^{-\alpha \frac{n}{5}} \leq \sqrt{1-\varepsilon }
 {\bigl( 1-\sqrt{1-\varepsilon }\bigr) }^4, \quad
             \mbox{ for } n\geq n_1^*,
\end{equation}
from (\ref{eq:wrsw-mix}) we have
$$
\mu (G \mid \omega (x) = -1, \ x \in \partial V )
\geq   {\bigl( 1-\sqrt{1-\varepsilon }\bigr) }^5.
$$       
This implies that the conditional probability in (\ref{eq:wrsw-6})
is not less than $(1-\sqrt{1-\varepsilon })^5$.
First multiplying this inequality by $\mu (R(\omega )=R)$
and then summing up the resulting inequality over $R$'s, we have 
$$
\mu \left( 
  \begin{array}{@{\,}c@{\,}}
      \mbox{there is a horizontal $(+)$-crossing of $S(n)$}\\
      \mbox{which is $(+)$-connected to the line}\\
      \{ x^2=\lfloor \frac{6n}{5}\rfloor \} \mbox{ in } [-n,n]\times 
                                    [-n,\lfloor \frac{6n}{5}\rfloor ]
   \end{array}
  \right) \geq \frac{\varepsilon }{2}
              {\bigl( 1-\sqrt{1-\varepsilon }\bigr) }^5.
$$
Finally, by the FKG inequality and the rotation invariance of
$\mu $, this implies that
\begin{equation}\label{eq:wrsw-8}
\textstyle
\mu \bigl(A^+ (\frac{11n}{10},n )\bigr)\geq 
   \frac{\varepsilon^2}{2}
              {\bigl( 1-\sqrt{1-\varepsilon }\bigr) }^5.
\end{equation}

Combining (\ref{eq:wrsw-5}) with (\ref{eq:wrsw-8}), we obtain
$$
\textstyle
\mu \bigl(A^+ (\frac{11n}{10},n )\bigr)
\geq \min \left\{ \frac{\varepsilon^2}{2}
              {\bigl( 1-\sqrt{1-\varepsilon }\bigr) }^5,
              {\bigl( 1-\root 4 \of {1-\varepsilon /2} \bigr) }^3
\right\} .
$$
The rest of the proof of the first statement of the lemma is obvious.

\medskip
\noindent
$2^\circ $) About the second statement:
By (\ref{wrsw-cond3}) and the FKG inequality we have
$$
\mu_{T,h} \left(
 \begin{array}{@{\,}c@{\,}}
 \mbox{there is a horizontal $(+)$-crossing of $S(m)$}\\
 \mbox{passing below the origin}
  \end{array}
       \right) \geq 1-\sqrt{\varepsilon_0}.
$$
But we have to take care of the mixing property.
So, let $V'=[-m,m-(\log m)^2]\times [-m,m]$ and let $R'(\omega )$
denotes the lowest horizontal $(+)$-crossing of $V'$.
As above, we have
\begin{equation}\label{eq:srsw-1}
\mu_{T,h} \left(
 \begin{array}{@{\,}c@{\,}}
 \mbox{there is a horizontal $(+)$-crossing of $V'$}\\
 \mbox{passing below the origin}
  \end{array}
       \right) \geq 1-\sqrt{\varepsilon_0}.
\end{equation}
Let $R'$ be a horizontal crossing of $V'$ passing below the origin,
and let $R''$ be its reflection with respect to the line $\{ x^1=m\} $.
If we write the right endpoint of $R'$ by $b_L$, then 
$b_R=b_L+(2(\log m)^2,0)$ is the left endpoint of $R''$.
Let $\xi $ be the line segment connecting $b_L$ with $b_R$,
and we write $R^*$ for the path $R'\cup \xi \cup R''$ which connects
the left and the right sides of $[-m,3m]\times [-m,m]$.
Then we look at the region $\Theta $ in $S((m,0),m)$ above $R^*$
and outside the square $S(b_L+((\log m)^2,0), 2(\log m)^2)$.
Let $H$ be the event defined by
$$
H=\left\{
 \begin{array}{@{\,}c@{\,}}
 \mbox{there is a $(+)$-path in $\Theta $ connecting}\\
 \mbox{the top side of  $S((m,0),m)$ with
          $\partial^+ \Theta $}
  \end{array}
       \right\} ,
$$
where $\partial^+ \Theta $ denotes the left half of the lower
boundary of $\Theta $.
Namely, it starts from the last intersection with the left side 
of $S((m,0),m)$ and $R'$, and then it goes along $\partial \Theta $
until it reaches the middle point $b_L+((\log m)^2,2(\log m)^2)$ 
in the lower side of $ \Theta $.
Let $1\leq j \leq M$ and let $r(j)=2\cdot 4^j(\log m)^2$.
Then consider the annulus 
$$
A(j) = b_L +((\log m)^2, 0)
  + [-r(j),r(j)]^2\setminus [-3r(j-1),3r(j-1)]^2.
$$
Then $ \Theta \cap A(j)$ is divided into some connected components.
Among them there is a component connecting $R'$ with $R''$, which 
we call by the main component of  $ \Theta \cap A(j)$.
The boundary of the main component consists of two paths belonging 
to opposite boundaries of $A(j)$, a path belonging to $R'$, and  
a path belonging to $R''$.
Let $K_j$ be the event defined by
$$
K_j=\left\{
\begin{array}{@{\,}c@{\,}}
\mbox{in the main component of $\Theta \cap A(j)$, there is a $(+)$-path}\\
\mbox{in this component such that it connects $R'$ with $R''$}\\
   \end{array}
\right\}.
$$
Then $K_j$ is an event occurring in $\Theta $, and 
by the FKG inequality and the Markov property of $\mu_{T,h}$,
we have
\begin{align}\label{eq:srsw-2}
&\mu_{T,h} \left( \left. H\cap \bigcup_{1\leq j\leq M}K_j \,\right|\, R'(\omega )=R' \right)\\
&\geq
\mu_{T,h} \left( \left. H \cap \bigcup_{1\leq j\leq M}K_j \,\right|\, 
 \begin{array}{@{\,}c@{\,}}
    \omega (x) = +1,\, x\in R'\\
    \omega (x) = -1,\, x\in L_1\cup L_2 \cup L_3\cup L_4
 \end{array}
   \right) , \nonumber
\end{align}
where $L_i, i=1,2,3,4$ are line segments defined below.
\begin{align*}
L_1&= \mbox{the segment connecting the left endpoint of $R'$ with
            $(-m,-m-(\log m)^2)$},\\
L_2&= \mbox{the segment connecting $(-m,-m-(\log m)^2)$ with 
$(m,-m-(\log m)^2 )$},\\
L_3&= \mbox{the segment connecting $(m,-m-(\log m)^2 )$ 
            with $b_L+((\log m)^2,0)$},\\
L_4&= \mbox{the segment connecting $b_L$ with $b_L+((\log m)^2,0)$}.
\end{align*}
Again, by the FKG inequality, the above conditional probability 
is not less than
\begin{equation}\label{eq:srsw-3}
\mu_{T,h}  \left( \left. H \cap \bigcup_{1\leq j\leq M}K_j
\,\right|\,  \omega (x) = -1 ,\, x \in L_1\cup L_2\cup L_3\cup L_4 \right).
\end{equation}
Since the distance between $\Theta $ and $L_1\cup L_2\cup L_3\cup L_4$ is
not less than $(\log m)^2$, by the mixing property (\ref{HZ(1.3)}) we have
\begin{align}\label{eq:srsw-4}
&\mu_{T,h} \left( \left.  H \cap \bigcup_{1\leq j\leq M}K_j
\,\right|\, \omega (x) = -1 ,\, x \in L_1\cup L_2\cup L_3\cup L_4 \right)\\
&\geq  \mu_{T,h} \left( H \cap \bigcup_{1\leq j\leq M}K_j \right) 
- 4Cm^2(\log m)^2 e^{-\alpha (\log m)^2}. \nonumber
\end{align} 
Also, since $A(j)$ surrounds the box
$S(b_L+((\log m)^2,0),2(\log m)^2)$, if $H\cap K_j$ occurs 
for some $j$, then there exists a $(+)$-path in $S((m,0),m)$
connecting the top side of $S((m,0),m)$ with $R'$.
By the FKG inequality
\begin{equation}\label{eq:srsw-5}
\mu_{T,h} \left( 
H \cap \bigcup_{1\leq j \leq M}K_j \right) 
\geq \mu_{T,h} (H)\mu_{T,h} \left( 
\bigcup_{1\leq j \leq M}K_j \right).
\end{equation}
By the symmetry of $\mu_{T,h} $ with respect to the reflection and
invariance under rotations by right angles, from (\ref{wrsw-cond3})
we have
$$
\mu_{T,h} (H)\geq 1-\sqrt{\varepsilon_0}.
$$
Also, by the FKG inequality it is easy to see that
$$
\mu_{T,h} ( K_j) \geq \delta_8(\varepsilon_0)^4.
$$
By the mixing property, for any event $E$ occurring on
$\cup_{1\leq i\leq j-1}A(i)$, 
$$
\mu_{T,h} (K_j \mid E )\geq \mu_{T,h} (K_j)-
  8Cm^2(\log m )^2e^{-2\alpha (\log m)^2}.
$$ 
By (\ref{eq:1-st bound for n}),   
the right hand side is not less than
$$
\mu_{T,h}(K_j)-4Cm^2(\log m)^2e^{-\alpha (\log m)^2}
$$
provided that
\begin{equation}\label{eq:new bound for n-1}
(\log m)^2 \geq n_0.
\end{equation}
Therefore if we take $n_2^*=n_2^*( \varepsilon_0)\geq 1 $ sufficiently large so that
(\ref{eq:new bound for n-1}) and 
$$
4Cm^2(\log m)^2e^{-\alpha (\log m)^2}<\frac{1}{2}\delta_8(\varepsilon_0)^4
$$
hold for $m\geq n_2^*$, 
then for such an $m$ with $m \geq 2\cdot 4^{M+1}(\log m)^2$, we have
\begin{equation}\label{eq:srsw-6}
\mu_{T,h} \left(\bigcup_{1\leq j\leq M}K_j \right) \geq
    1-(1-\delta_8(\varepsilon_0)^4/2)^M
\geq 1- \varepsilon_0
\end{equation} 
by our assumption on $M$.
Also, from (i) of the lemma,
$\delta_8(\varepsilon_0)$ is much smaller than $\varepsilon_0$,
and we have
\begin{equation}\label{4th bound for n}
4Cn^2(\log n)^2 e^{-\alpha (\log n)^2} \leq \varepsilon_0. 
\end{equation}
Thus, from (\ref{eq:srsw-2})--(\ref{4th bound for n}), we have
\begin{align}\label{eq:srsw-7}
&\mu_{T,h} \left( \left.
\begin{array}{@{\,}c@{\,}}
\mbox{there is a $(+)$-path in $S((m,0),m)$ which}\\
\mbox{connects the top side of $S((m,0),m)$ with $R'$}
\end{array}
\,\right|\,R'(\omega )=R'
 \right) \\
&\geq  (1-\sqrt{\varepsilon_0})(1-\varepsilon_0)-\varepsilon_0.
\nonumber
\end{align}
We multiply both sides of the above inequality by 
$\mu_{T,h}(R'(\omega )=R')$ and then summing up over 
horizontal crossings $R'$ of $V'$ such that $R'$ passes below
the origin, from (\ref{eq:srsw-1}) and (\ref{eq:srsw-7}) we obtain
\begin{align}\label{srsw-8}
&\mu_{T,h}\left(
 \begin{array}{@{\,}c@{\,}}
 \mbox{there exists a $(+)$-path in $[-m,2m]\times [-m,m]$}\\
 \mbox{such that it connects the left side of this box}\\
 \mbox{with $\{ x^2=m\} $, and it separates the line segment }\\
 \ell = \{ (0,t) : 0\leq t \leq m \} \mbox{ from the bottom side }\\
 \mbox{ of } [-m, 2m ]\times [-m,m ] 
 \end{array}
 \right) \\
&\geq (1-\sqrt{\varepsilon_0})[(1-\sqrt{\varepsilon_0})(1-\varepsilon_0)
                               -\varepsilon_0].\nonumber
\end{align}
Then forcing this event to have a horizontal $(+)$-crossing in
$S((m,0),m)$ starting from the left side of it above the origin,
we have by the FKG inequality that
$$
\mu_{T,h} \left( \begin{array}{@{\,}c@{\,}}
\mbox{there is a horizontal $(+)$-crossing in}\\
\mbox{the rectangle } [-m,2m]\times [-m,m]
\end{array}
\right)\geq g(\varepsilon_0),
$$
where we put
$$
g(x)= (1-\sqrt{x})^2\left[ (1-\sqrt{x})(1-x)-x\right] .
$$

From this it is the same argument as the original RSW theorem to obtain
$$
\mu_{T,h} \bigl( A^+(3m,m) \bigr) \geq g(\varepsilon_0 )^4(1-\varepsilon_0 )^3.
$$
Thus, if we take $\varepsilon_0>0$ such that
\begin{equation}\label{cond-for-epsilon}
g(\varepsilon_0 )^4(1-\varepsilon_0 )^3\geq 1-\lambda \theta ,
\end{equation}
then we have desired estimate for every 
$m\geq \max\{n_0, n_1^*(\varepsilon_0),n_2^*(\varepsilon_0)\} $
with $m\geq 2\cdot 4^{M+1}(\log m)^2$.
We can take 
$n_1(\varepsilon_0)=\max\{n_0, n_1^*(\varepsilon_0),n_2^*(\varepsilon_0)\} $.

\end{proof}

%---------------------------------------------
%\end{quote}
For the Gibbs measure $\mu_{T,h_c(T)}$ at the percolation threshold,
we abbreviate it to $\mu_{\text{cr}}$.
\begin{Lemma}[cf. \cite{K87scaling} (2.15)] \label{RSW-cr}
For any integer $k > 0$, there exists a constant $C(k)$ such that for all $n$,
\[ \mu_{\text{cr}} \bigl( A^+(kn,n) \bigr) \geq C(k)
\mbox{ and }
\mu_{\text{cr}} \bigl( A^{-*} (kn,n) \bigr)  \geq C(k). \]
\end{Lemma}

\begin{proof}
It is shown in \cite{Hig93b} that 
$$
0<\liminf_{n\rightarrow \infty }\mu_{\text{cr}} \bigl( A^+(kn,\ell n) \bigr)
\leq \limsup_{n\rightarrow \infty }\mu_{\text{cr}} \bigl( A^+(kn,\ell n) \bigr)
<1
$$
for any fixed pair of integers $k,\ell \geq 1$.
This proves the lemma for the $(+)$-connection, and since
$${\bigl( A^+(kn, \ell n)\bigr) }^c = A^{-*}(\ell n,kn),$$
the assertion of the lemma is proved.
\end{proof}
For later use, we choose $0<\varepsilon_0 <\frac{C(1)}{2}$, and then from 
Lemma \ref{RSW-cr}
we can regard $L(h_c(T),\varepsilon_0)$ as $\infty$.

By Lemma \ref{RSW-gen}, we can see that if $h<h_c(T)$, then for every $n$ with 
$n_1(\varepsilon_0) \leq n<L(h,\varepsilon_0)$ we have $\mu_{T,h}\bigl(A^+(n,n)\bigr) >\varepsilon_0$ and
$\mu_{T,h}\bigl(A^+(kn,n)\bigr)\geq \delta_k(\varepsilon_0)$.
If $N>2n$, then by (\ref{eq:periodic-1}) 
\begin{equation}\label{eq:periodic-2}
\mu_t^N\bigl( A^+(n,n)\bigr) \geq \varepsilon_0-4Cn^3e^{-\alpha n}
\end{equation}
for $t\in [0,1]$.
The right side of the above inequality is not less than $\varepsilon_0/2$
if $n$ is sufficiently large.
Let
\begin{equation}\label{eq:bound n_2}
 n_2 = n_2(\varepsilon_0):= \begin{cases}
  \max\{ n\geq n_1 : 4Cn^3e^{-\alpha n}
               \geq \frac{\varepsilon_0}{2} \} +1, \\
  \hspace{10mm}n_1, \mbox{ if the above set is empty. } \\
 \end{cases}
\end{equation}
 Now if $n_2\leq n<\min\{ L(h,\varepsilon_0), \frac{N}{2}\}$, then
 \begin{equation}\label{eq:periodic-3}
 \mu_t^N\bigl( A^+(n,n)\bigr) \geq \frac{\varepsilon_0}{2}.
 \end{equation}
 Therefore by Lemma \ref{RSW-gen} (i), we have both for $\mu =\mu_{T,h}$ and
 $\mu = \mu_t^N$ 
 \begin{equation}\label{wrsw-result-2}
 \mu \bigl( A^+(kn,n)\bigr) \geq \delta_k(\varepsilon_0/2)
\end{equation}
provided that $N>kn$ when $\mu = \mu_t^N$.
On the other hand, if $h>h_c(T)$, then by Lemma \ref{RSW-cr} and the FKG inequality
we have 
$$
\mu_{T,h} \bigl( A^+(kn,n)\bigr) \geq \mu_{\text{cr}}\bigl( A^+(kn,n)\bigr)
  \geq C(k).
$$
By Lemma \ref{RSW-gen}, we can take $C(k)$ as $\delta_k(C(1))$ if $n\geq n_1(\varepsilon_0)$.
If $N>2n$, then as above we have 
$$
\mu_t^N\bigl( A^+(n,n)\bigr) \geq C(1)-4Cn^2e^{-\alpha n}.
$$
Thus, if $n_2\leq n<\frac{N}{2}$, then we have 
$$
\mu_t^N\bigl( A^+(n,n)\bigr) \geq C(1)-\frac{\varepsilon_0}{2}>\frac{3}{4}C(1)
\geq \frac{\varepsilon_0}{2}.
$$
Therefore anyway we have the following; the estimate for $(-*)$ connection can be obtained by the same reason.
\begin{Lemma} \label{RSWbound}
Let $\mu = \mu_{T,h}$ or $\mu_t^N$. If $kn<N$ and $n_2\leq n< L(h,\varepsilon_0)$, then
\begin{align}
\label{eq:crossing}
\textstyle
\mu\bigl( A^+(kn,n)\bigr)\geq  
    \delta_k(\frac{\varepsilon_0}{2}), \\
\label{eq:*crossing}
\textstyle
    \mu\bigl( A^{-*}(kn,n)\bigr)\geq 
    \delta_k(\frac{\varepsilon_0}{2}).
\end{align}
\end{Lemma}
We take $\delta_k$ as $\delta_k(\frac{\varepsilon_0}{2})$ hereafter
so that both (\ref{eq:crossing}) and (\ref{eq:*crossing}) hold
for $\mu\in \{ \mu_{T,h}\}\cup\{ \mu_t^N : t\in [0,1]\} $
if $kn<N$ and $n_2\leq n< L(h,\varepsilon_0)$. 

\begin{Lemma}[cf. \cite{K87scaling} (2.19), (2.20)] \label{lem:circuits}
Let $\mu =\mu_{T,h} $ or $\mu_t^N$. Then for any  $k > 2$, and any $n$ with
$4n_2\leq n <  L(h,\varepsilon_0) $, 
the following hold:
\begin{align}
&\mu \left( 
\begin{array}{@{\,}c@{\,}} 
\mbox{there exists a $(+)$-circuit surrounding 
} \\
\mbox{$S((k-2)n-1)$ in $S(kn) \setminus S((k-2)n-1)$} \\
\end{array} 
\right) \geq \delta_k^4, \label{K(2.19)}\\
&\mu \bigl( \mathbf{O}  \stackrel{+}{\leftrightarrow} \partial_{in} S(2n) \bigr)
\geq \delta_4^4 \delta_2 \mu \bigl(
  \mathbf{O}  \stackrel{+}{\leftrightarrow} \partial_{in} S(n) \bigr).
\label{K(2.20)}
\end{align}
If $\mu =\mu_t^N$, then we also assume that $N>kn$ in (\ref{K(2.19)}),
and $N>2n$ in (\ref{K(2.20)}).
%Here, $K_4$ depends only on $T,n_1$, and $\varepsilon_0$.
The same statement holds true when $(+)$-connection is replaced with 
$(-*)$-connection.
\end{Lemma}
\begin{proof}
(\ref{K(2.19)}) follows from the (\ref{eq:crossing})
together with the FKG inequality.
%%(\ref{K(2.20)}) then follows from (\ref{K(2.19)}) and the FKG
%% inequality. 
%%
%%\begin{quote}
%% ---------------------------------------------
%%
%% \noindent \framebox{\bf INTERMISSION}
%%
%% We use a similar argument as in the proof of (7) in \cite{K86i}:
As for (\ref{K(2.20)}), we use a similar argument as in the proof of (7) in \cite{K86i}: 
\begin{align*}
&\mu \bigl( \mathbf{O}  \stackrel{+}{\leftrightarrow} \partial_{in} S(2n) \bigr) \\
&\geq \mu
\left( 
\begin{array}{@{\,}c@{\,}} 
\mathbf{O}   \stackrel{+}{\leftrightarrow} \partial_{in} S(n), \mbox{ there exists a $(+)$-circuit} \\
\mbox{in $S(n) \setminus S(n/2)$ surrounding $S(n/2)$, and } \\
\mbox{there exists a horizontal $(+)$-crossing} \\
\mbox{of $[n/2,2n] \times [-n/2,n/2]$} \\
\end{array} 
\right) \\
 &\geq \mu \bigl( \mathbf{O}  \stackrel{+}{\leftrightarrow} \partial_{in} S(n) \bigr) \cdot \delta_4^4 \cdot \delta_2.
\end{align*}
%% ---------------------------------------------
%% \end{quote}
\end{proof}

\begin{Lemma}[cf. \cite{K87scaling} p.121] \label{Kp121}
Let $\lambda =1/64$, $\theta \in (0,1)$ and $n_0$ be the same as in
Lemma \ref{lem:accfr}. If there is an 
\begin{equation}\label{eq:low-bd-L}
L\geq \max\{ n_0, 8C(8e^{-1}\alpha^{-1})^4 \}
\end{equation} 
such that
\begin{equation}\label{K(2.23)}
\mu_{T,h}\bigl( A^+(L,3L)\bigr)< \lambda \theta ,
\end{equation}
then we can find constants $K_6, K_7>0$ depending only on $\theta $,
such that
\begin{equation} \label{K(2.24)}
\mu_{T,h} \bigl( S(L)  \stackrel{+}{\leftrightarrow} \partial_{in} S(kL) \bigr) \leq
K_6 \exp (-K_7 k).
\end{equation}
The same statement holds for $(-*)$-connection, too.
\end{Lemma}
%%%%%%%%%%%%%%
%%%% Hereafter we fix the integers $n_0, n_1=n_1(\varepsilon_0)$ and $n_2=n_2(T, \varepsilon_0)$ as above.
%%%%%%%%%%%%%
\begin{proof}
$1^\circ)$ We write $g(k)$ for 
$\displaystyle \mu_{T,h}\bigl( 
  S(L) \stackrel{+}{\leftrightarrow } \partial_{in} S(kL) \bigr)$.
It suffices to show that there exists an $m$ such that
\begin{equation}\label{K(2.24)-1}
g(k)\leq \frac{1}{2}g(k-m)
\end{equation}
if $k>m$.
Let $m=2\cdot 3^j\leq k$ for some $j\geq 1$. By the mixing property we have 
\begin{equation}\label{K(2.24)-2}
g(k) \leq \left[ g(m/2) + C \frac{(mL)^3}{2}e^{-\alpha mL/2} \right]
        \times \mu_{T,h}\bigl( S(mL)\stackrel{+}{\leftrightarrow }
                  \partial_{in} S(kL)\bigr).
\end{equation}
Now, we break $\partial_{in} S(mL)$ into pieces each of which has length
$2L$. 
Each of these pieces belongs to the boundary of $S({\mathbf n}L,L)$
for some ${\mathbf n}= (n_1,n_2)\in {\mathcal D}_m$, where
$$
{\mathcal D}_m := \left\{ (\pm (m-1),2i-1), (2i-1,\pm (m-1)) : -\frac{m}{2}+1
        \leq i \leq \frac{m}{2} \right\}.
$$
Since
\begin{align*}
\{ S(mL)\stackrel{+}{\leftrightarrow } \partial_{in} S(kL) \}
& \subset \bigcup_{{\mathbf n}\in {\mathcal D}_m}
  \{ S({\mathbf n}L,L)\stackrel{+}{\leftrightarrow }
      \partial_{in} S(kL)\} \\
& \subset \bigcup_{{\mathbf n}\in {\mathcal D}_m}
  \{  S({\mathbf n}L,L)\stackrel{+}{\leftrightarrow }
      \partial_{in} S({\mathbf n}L,(k-m)L)\} ,
\end{align*}
we have by (\ref{K(2.24)-2}) and by the translation invariance
\begin{align*}
g(k) &\leq  \left[ g(m/2)+ C\frac{(mL)^3}{2}e^{-\alpha mL/2} \right]
    \times (\# {\mathcal D}_m) g(k-m) \\
   &\leq 4m \left[ g(m/2)+ C\frac{(mL)^3}{2}e^{-\alpha mL/2} \right] g(k-m).
\end{align*}
Therefore to show (\ref{K(2.24)-1}),
we have only to show the existence of $m$ such that
\begin{equation}\label{K(2.24)-3}
4mg(m/2) < \frac{1}{4} \quad \mbox{and} \quad 
 2Cm(mL)^3e^{-\alpha mL/2}<\frac{1}{4}.
\end{equation}
Note that the latter inequality is satisfied when $L\geq 8C(8e^{-1}\alpha^{-1})^4$. 

\noindent $2^\circ)$ Consider an annulus
$$
\textstyle
A_m = %\left[ -\frac{m}{2}L, \frac{m}{2}L \right]^2 \setminus 
     % \left[-\frac{m}{6}L+1, \frac{m}{6}L-1 \right]^2.
     S(\frac{mL}{2})\setminus S(\frac{mL}{6}-1).
$$
Then by the FKG inequality we have
$$
\mu_{T,h}\left(
\begin{array}{@{\,}c@{\,}} 
 \mbox{there exists a $(-*)$-circuit} \\
 \mbox{surrounding the origin in $A_m$}
\end{array}
      \right)
\geq \mu_{T,h}\bigl(A^{-*}(mL/2,mL/6)\bigr)^4.
$$
Then by Lemma \ref{lem:accfr}, we have for $mL/6\geq n_0$,
$$g(m/2)\leq 1-(1-\lambda \theta^{2^{j-1}})^4 \leq 4\lambda \theta^{2^{j-1}}, $$
since $m=2\cdot 3^j$.
This implies the existence of $m$ satisfying (\ref{K(2.24)-3}).
Note that this $m$ depends only on $\theta $.
\end{proof}
From the above lemma, we can easily obtain the following corollary.

\begin{Corollary}[cf. \cite{K87scaling} p.121]
Let $h$ be taken close to $h_c(T)$ so that 
$L(h,\varepsilon_0)>\max\{ n_1, 8C(8e^{-1}\alpha^{-1})^4\} $.
If $h<h_c(T)$, then (\ref{K(2.24)}) for $L=L(h,\varepsilon_0)$ holds. If
$h>h_c(T)$, then (\ref{K(2.24)})
with $(+)$ being replaced by $(-*)$ for $L=L(h,\varepsilon_0)$ holds.
\end{Corollary}

\begin{Lemma}[cf. \cite{Hig93a} Lemma 5.3] \label{Hig93aL5.3} 
Let $\mu = \mu_{T,h}$ or $\mu_t^N$ for $t\in [0,1]$.
Then for every $\varepsilon > 0$, there exists a positive integer 
$M_0=M_0(\varepsilon)$ 
such that if $n_2 \leq n < 2^Mn < L(h,\varepsilon_0)$ for some $M \geq M_0$, we have
\begin{align} 
&\mu \left(  \begin{array}{@{\,}c@{\,}} 
\mbox{there exists a $(+)$-circuit surrounding the origin} \\
\mbox{in $S(2^Mn) \setminus S(n)$} \\
\end{array} \right) \geq 1 - \varepsilon,  \label{Hig93a5.8+}
\intertext{and}
&\mu \left(  \begin{array}{@{\,}c@{\,}} 
\mbox{there exists a $(-*)$-circuit surrounding the origin} \\
\mbox{in $S(2^Mn) \setminus S(n)$} \\
\end{array} \right) \geq 1 - \varepsilon. \label{Hig93a5.8-}
\end{align}
\end{Lemma}
This follows from Lemma \ref{lem:circuits} and the mixing property.
The proof is quite the same as the one to derive (\ref{eq:srsw-6}).

\newpage
\section{Connection lemma}
For $n \geq 1$, we define
\[
\begin{cases}
\ell(n) := \lfloor \log_4 n \rfloor = \min \{ j \in \mathbf{N} : j \geq \log_4 n \}, \\
\ell_2(n) := \ell(\ell(n)), \\
\ell_3(n) := \ell(\ell_2(n)).
\end{cases}
\]
Throughout this section we assume that $\ell_3(n)> n_2$,
and $n<\min\{ L(h,\varepsilon_0), \frac{N}{2} \} $.
Therefore we assume that $h$ is very close to $h_c(T)$.
Also, let $\mu $ be $\mu_{T,h}$ or $\mu_t^N$ with $0\leq t\leq 1$.
Let $k>1$ be given. 
For our purpose it is sufficient to assume that $k=p/4$ with an integer
$p\geq 5$.
Let 
$$
V(n)=[-n,n]\times [-kn,kn].
$$
For a horizontal $(*)$-crossing $\gamma $ of $V(n)$, let $L_n(\gamma )$
be the region in $V(n)$ below $\gamma $, and $U_n(\gamma )$ be the region
of $V(n)$ above $\gamma $.
Also, for every $\ell \geq 1$, let $U_n(\gamma ,\ell )$ denote the
connected component of the set
$$
\{ x \in U_n(\gamma ) : d(x,\gamma ) \geq \ell \}
$$
which contains the top side of $V(n)$.
If such a component does not exist, then we put 
$U_n(\gamma ,\ell )=\emptyset $.
In the same way we define $L_n(\gamma , \ell )$ for $L_n(\gamma )$.
Let 
$\varDelta U_n(\gamma ,\ell )= U_n(\gamma ,\ell )\setminus U_n(\gamma ,\ell -1)$,
which consists of points in $U_n(\gamma )$ such that the distance from 
$\gamma $ is exactly equal to $\ell $.
Similarly, we put
$$\varDelta L_n(\gamma , \ell ) =
 \{ x \in L_n(\gamma ) : d(x,\gamma ) = \ell \} .
$$
\begin{Lemma}[Connection lemma] \label{connection lemma}
Let $\gamma_1,\gamma_2$ be horizontal $(*)$-crossings of $V(n)$ such that
\[ \textstyle
\gamma_1 \subset [-n,n]\times [-kn,-(k-\frac{1}{2})n] , \ \mbox{ and }\ 
\gamma_2 \subset [-n,n]\times [(k-\frac{1}{2})n, kn].
\]
Let also $\mu = \mu_{T,h}$ or $\mu_t^N$ for $t\in [0,1]$.
There exists an $n_3=n_3(k,\varepsilon_0)\geq n_2$, such that for every
$E\in {\mathcal F}_{V(n)^c}$ and every 
$F\in {\mathcal F}_{L_n(\gamma_1)\cup U_n(\gamma_2)\cup \{ \gamma_1 \}
\cup \{ \gamma_2\} }$,
\begin{equation}\label{eq:connection lemma} 
\mu \left( \left.
 \begin{array}{@{\,}c@{\,}}
\mbox{there exists a $(+)$-path in $U_n(\gamma_1)\cap L_n(\gamma_2)$,}\\
\mbox{connecting $\varDelta U_n(\gamma_1,1) $ with $\varDelta L_n(\gamma_2,1) $}
 \end{array}
\, \right| \, E\cap F \right) \geq \frac{\delta_{8k}}{4}
\end{equation}
for $n_3\leq n<L(h,\varepsilon_0)$ and $N>kn$ when $\mu = \mu_t^N$.
The same estimate holds for $(-*)$-path, too.
\end{Lemma}
\begin{proof}
We will prove (\ref{eq:connection lemma}). The argument is quite parallel
for the $(-*)$-path.

\noindent  $1^{\circ})$ Set 
$T(n) := [-\frac{n}{8}, \frac{n}{8}] \times [-kn, kn]$
and 
\[ G_n^+ := \left\{ \begin{array}{@{\,}c@{\,}} 
\mbox{there exists a $(+)$-path from $\varDelta U_n(\gamma_1,\ell(n)^2)$ to 
$\varDelta L_n(\gamma_2,\ell(n)^2)$} \\
\mbox{in $U_n(\gamma_1,\ell(n)^2) \cap L_n(\gamma_2,\ell(n)^2) \cap T(n)$}
\end{array} \right\}. \]
By (\ref{eq:crossing})
we have $\mu( G_n^+)\geq \delta_{8k}$ if $8n_2\leq n<L(h,\varepsilon_0)$.
Using the mixing property (\ref{HZ(1.1)})
$$
\mu (G_n^+ \mid E\cap F) \geq 
      \mu (G_n^+) -4Ckn^2\ell (n)^2e^{-\alpha \ell (n)^2}.
$$
Thus, we can find an $N_1=N_1(k, \varepsilon_0)\geq 8n_2$ such that
$$
4Ckn^2\ell (n)^2e^{-\alpha \ell (n)^2}<\frac{1}{2}\delta_{8k}
$$
and hence
\begin{equation}\label{eq:3.1}
\mu (G_n^+ \mid E\cap F) \geq \frac{1}{2}\delta_{8k}
\end{equation}
for $N_1\leq n<\min\{ L(h,\varepsilon_0), \frac{N}{k}\} $.

\noindent $2^{\circ})$ Let $n \geq N_1$. For $\omega \in G_n^+$, let $\gamma(\omega)$ be the leftmost (+)-path in 
\[ \mbox{$U_n(\gamma_1,\ell(n)^2) \cap L_n(\gamma_2,\ell(n)^2) \cap T(n)$} \]
connecting $\varDelta U_n(\gamma_1,\ell(n)^2)$ and $\varDelta L_n(\gamma_2,\ell(n)^2)$.
Let $\gamma $ be a realization of $\gamma (\omega )$ for some
$\omega \in G_n^+$, and let $v_1(\gamma ), v_2(\gamma )$ be the intersection
points of $\gamma $ with $\varDelta U_n(\gamma_1,\ell(n)^2)$ and
$\varDelta L_n(\gamma_2,\ell(n)^2)$, respectively.
Consider the following annuli around $v_1(\gamma )$:
\begin{equation}\label{eq:3.2}
A_{1,j}= v_1(\gamma )+ S(4^{j+1})\setminus S(3\cdot 4^j-1)
\end{equation}
for 
\begin{equation}\label{Hig93a5.21}
\sqrt{n} \leq 2 \cdot 4^j < 4^{j+1} \leq \frac{n}{2}. 
\end{equation}
Then $A_{1,j}$'s are disjoint from each other, and
are subsets of $V(n)$.
Let $\psi_i, i=1,2$ be paths of length $\ell (n)^2$, such that 
$\psi_i $ connects $v_i(\gamma )$ with $\gamma_i$, respectively
for $i=1,2$.
Then any point $x\in \psi_i$ satisfies that
$$
d(x, \gamma_i)\leq \ell (n)^2.
$$
Therefore 
$$(\psi_1 \cup \psi_2) \cap \bigl( U_n(\gamma_1, \ell (n)^2) \cap
                              L_n(\gamma_2, \ell(n)^2 ) \bigr)
 = \emptyset .
$$
Further, (\ref{Hig93a5.21}) implies that
$$
d(A_{1,j},\psi_1\cup \psi_2)\geq \sqrt{n}-\ell (n)^2
     \geq \ell (n)^2
$$
if 
\begin{equation}\label{eq:3.4}
\sqrt{n}\geq 2\ell (n)^2.
\end{equation}
The path $\gamma $ divides 
 $U_n(\gamma_1, \ell (n)^2)\cap L_n(\gamma_2, \ell (n)^2)$ 
into two regions.
One is to the right of $\gamma $, and the other is
to the left of $\gamma $.
Let $\Theta_+(\gamma )$ denote the region to the right of $\gamma $.
Then we extend $\gamma $ to ${\tilde \gamma }=\gamma \cup \psi_1\cup\psi_2$
which separates $U_n(\gamma_1)\cap L_n(\gamma_2)$ into two parts.
Let ${\tilde \Theta }_+({\tilde \gamma })$ be the connected component
of $\{ U_n(\gamma_1)\cap L_n(\gamma_2)\} \setminus \{ {\tilde \gamma }\} $,
containing $\Theta_+(\gamma )$. 

Let $F_1$ be an event occurring in $S(v_1(\gamma ), 2\cdot 4^j)$, and
let
$$
H_{1,j}:=\left\{  \begin{array}{@{\,}c@{\,}} 
\mbox{there exists a $(+)$-path in $A_{1,j}\cap \Theta_+(\gamma )$,}\\
\mbox{connecting $\gamma $ with $\varDelta U_n(\gamma_1,\ell (n)^2)$}
\end{array}
\right\} .
$$
Then by the Markov property and the FKG inequality
\begin{align}
&\mu \left( H_{1,j} \,\bigg\vert\, \{ \gamma (\omega )=\gamma \}\cap E\cap F\cap F_1 
   \right) \label{eq:3.5} \\
&\geq  \mu \left( H_{1,j} \,\bigg\vert\, \begin{array}{@{\,}c@{\,}} 
 \omega (x) = 1 \mbox{ for every } x \in \gamma \\
 \omega (x) =-1 \mbox{ for every } x \in \partial 
   \Xi_j(\gamma )\setminus \gamma 
 \end{array} \right) , \nonumber
\end{align}
where we put 
$$
\Xi_j(\gamma )=  {\tilde \Theta }_+({\tilde \gamma })\setminus
                 S(v_1(\gamma ),2\cdot 4^j ).
$$
Again by the FKG inequality the right hand side of (\ref{eq:3.5}) is
not less than
$$
\mu \left( H_{1,j} \,\bigg\vert\, \omega (x) = -1
   \mbox{ for every } x \in \partial \Xi_j(\gamma )\setminus \gamma
 \right) .
$$
Note that 
$ d( \Theta_+(\gamma ), \partial   \Xi_j(\gamma )\setminus \gamma ) $
is not less than $\ell (n)^2$, under the conditions (\ref{Hig93a5.21}) and 
(\ref{eq:3.4}).
Therefore by the mixing property and the FKG inequality we have
\begin{align*}
\mu \left( H_{1,j} \,\bigg\vert\, \{ \gamma (\omega )=\gamma \}\cap E\cap F\cap F_1
\right) &\geq  \mu ( H_{1,j}) -
     4Ckn^2\ell (n)^2e^{-\alpha \ell (n)^2 } \\
&\geq  \delta_8^4 -  4Ckn^2\ell (n)^2e^{-\alpha \ell (n)^2 }
\end{align*}
if $8n_2\leq n <\min \{ L(h,\varepsilon_0), \frac{N}{k}\} $.
So we can find $N_2=N_2(k,\varepsilon_0)\geq N_1$ such that
(\ref{eq:3.4}) holds for $n\geq N_2$, and
$$
4Ckn^2\ell (n)^2e^{-\alpha \ell (n)^2 } <\frac{\delta_8^4}{2},
$$
and hence
\begin{equation}\label{eq:3.6}
\mu \left( H_{1,j} \,\bigg\vert\, \{ \gamma (\omega )=\gamma \}\cap E\cap F\cap F_1
\right) \geq \frac{1}{2}\delta_8^4
\end{equation}
for every $N_2\leq n < \min \{ L(h,\varepsilon_0), \frac{N}{k}\} $.

If $\omega \in H_{1,j}$, then there exists a $(+)$-path which connects 
$\gamma $ with $\varDelta U_n(\gamma_1,\ell (n)^2)$ in 
$\Theta_+(\gamma )\cap A_{1,j}$.
Among such $(+)$-paths, let $\gamma_{1,j}(\omega )$ be the ``minimal'' path
in the following sense.
For every self-avoiding path $\xi $ in $\Theta_+(\gamma )\cap A_{1,j}$,
connecting $\gamma $ with $\varDelta U_n(\gamma_1,\ell (n)^2)$, 
$\Theta_+(\gamma_1,\ell (n)^2)\setminus \{ \xi \} $ separates into 
connected components.
Let $C_\xi $ denote the component which contains a $(*)$-nearest neighbor
of $v_1(\gamma )$.
$\gamma_{1,j}(\omega )$ has the minimal region $C_{\gamma_{1,j}(\omega )}$.
Namely, for every $(+)$-path $\zeta $ in  $\Theta_+(\gamma )\cap A_{1,j}$,
connecting $\gamma $ with $\varDelta U_n(\gamma_1,\ell (n)^2)$,
$C_{\gamma_{1,j}(\omega )}\subset C_\zeta $.
Note that $\{ \gamma_{1,j}(\omega )=\zeta \} $ depends on configurations in 
$\zeta \cup ( C_\zeta \cap A_{1,j})$.
Let $\gamma_{1,j}$ be a realization of $\gamma_{1,j}(\omega )$ for some
$\omega \in H_{1,j}$, and let $v_1(\gamma_{1,j})$ be its intersection 
with $\varDelta U_n(\gamma_1,\ell (n)^2)$.
Let also $\gamma \circ \gamma_{1,j}$ be the path which starts at
$v_2(\gamma )$, goes along $\gamma $ until it meets $\gamma_{1,j}$, and
then changes to go along $\gamma_{1,j}$ ending at $v_1(\gamma_{1,j})$.

Take $M_0=M_0(\frac{1}{4})$ in Lemma \ref{Hig93aL5.3},
and let $A'(\gamma_{1,j})$ be the annulus given by
$$
A'(\gamma_{1,j})=v_1(\gamma_{1,j})+
     S(2\cdot4^{M_0}\ell (n)^2)\setminus S(2\cdot \ell (n)^2).
$$
Consider the following event.
$$
H'(\gamma \circ\gamma_{1,j})=\left\{ \begin{array}{@{\,}c@{\,}} 
\mbox{there exists a $(+)$-path in $A'(\gamma_{1,j})$ which}\\
\mbox{connects $\gamma\circ\gamma_{1,j}$ with $\varDelta U_n(\gamma_1,\ell (L_1)^2)$}
\end{array}
\right\},
$$
where we put $L_1 = 2\cdot 4^{M_0}\ell (n)^2$.
Then as before we have
\begin{align}
&\mu \left( H'(\gamma \circ\gamma_{1,j}) \,\bigg\vert\,
  \{ \gamma (\omega ) = \gamma \} \cap 
    \{ \gamma_{1,j}(\omega ) = \gamma_{1,j} \} \cap
  E\cap F \cap F_1 
\right) \label{eq:3.7} \\
&\geq \mu\bigl(  H'(\gamma \circ\gamma_{1,j}) \bigr) -
  4CL_1^2\ell(L_1)^2e^{-\alpha \ell (L_1)^2}\nonumber  \\
& \geq \frac{3}{4} - 4CL_1^2\ell(L_1)^2e^{-\alpha \ell (L_1)^2},\nonumber
\end{align}
by Lemma \ref{Hig93aL5.3}
if 
\begin{equation}\label{eq:3.8}
\ell (L_1)^2\geq n_2.
\end{equation}
So we can take $N_3=N_3(k,\varepsilon_0)\geq N_2$ such that (\ref{eq:3.8})
and 
\begin{equation}\label{eq:3.9}
4CL_1^2\ell(L_1)^2e^{-\alpha \ell (L_1)^2}\leq \frac{1}{4}
\end{equation}
for $n\geq N_3$, so that
\begin{equation}\label{eq:3.10}
\mu\left( H'(\gamma \circ\gamma_{1,j}) \,\bigg\vert\,
\{ \gamma (\omega ) = \gamma \} \cap 
    \{ \gamma_{1,j}(\omega ) = \gamma_{1,j} \} \cap
  E\cap F \cap F_1 
\right) \geq \frac{1}{2}.
\end{equation}

For $\omega \in H'(\gamma \circ \gamma_{1,j})$, let $\varphi (\omega )$ be
the ``minimal'' $(+)$-path connecting $\gamma \circ \gamma_{1,j}$ with
$ \varDelta U_n(\gamma_1,\ell (L_1)^2)$.
Fixing a realization $\varphi $ of $\varphi (\omega )$ for some 
$\omega \in H'(\gamma \circ \gamma_{1,j})$, we do the same thing again.
Namely, let $v'_1(\varphi )$ be the endpoint of $\varphi $ in 
$ \varDelta U_n(\gamma_1,\ell (L_1)^2)$ and let
$$
H''((\gamma \circ \gamma_{1,j})\circ \varphi ) =\left\{
 \begin{array}{@{\,}c@{\,}} 
\mbox{there exists a $(+)$-path in the annulus}\\
 v'_1(\varphi ) + S(2\cdot 4^{M_0}\ell (L_1)^2)\setminus S(2\ell (L_1)^2)\\
 \mbox{such that it connects $(\gamma \circ \gamma_{1,j})\circ \varphi $}\\
 \mbox{with $\varDelta U_n(\gamma_1,\ell (L_2)^2)$}
\end{array}
\right\}
$$
where $L_2= 2\cdot 4^{M_0}\ell(L_1)^2$.
Then
\begin{equation}\label{eq:3.11}
\mu \left( H''((\gamma \circ \gamma_{1,j})\circ \varphi ) \,\bigg\vert\,
 \begin{array}{@{\,}c@{\,}} 
\{ \gamma (\omega )=\gamma \} \cap \{ \gamma_{1,j}(\omega )=\gamma_{1,j} \} 
 \\ 
 \cap \{\varphi (\omega )=\varphi \} \cap E\cap F\cap F_1
\end{array}
\right)   \geq \frac{1}{2},
\end{equation}
provided that 
\begin{equation}\label{eq:3.12}
\ell (L_2)^2\geq n_2, \quad \mbox{and }\quad 
 4CL_2^2\ell (L_2)^2 e^{-\alpha \ell (L_2)^2 }\leq \frac{1}{4},
\end{equation}
which is possible if $N_4=N_4(k,\varepsilon_0)$ is sufficiently large
and $N_4\leq n <\min\{ L(h, \varepsilon_0), \frac{N}{k}\} $.

Let ${\overline A}_{1,j}= v_1(\gamma ) + 
 S(4^{j+1}+L_1+L_2)\setminus S(2\cdot 4^j)$
with $\sqrt{N}<2\cdot 4^j < 4^{j+1}< \frac{n}{2}$.
If $L_1+L_2<\frac{\sqrt{n}}{2}$, then 
$$4^{j+1}+L_1+L_2<4^{j+1}+4^j=
\frac{5}{4}4^{j+1}<\min \left\{ \frac{5}{8}n, 2\cdot 4^{j+1} \right\} .
$$
So, $H_{1,j}(\gamma ), H'_{1,j}(\gamma_{1,j})$ and
$H''((\gamma \circ \gamma_{1,j})\circ \varphi )$ are occurring in 
$S(v_1(\gamma ), 2\cdot 4^{j+1})\cap V(n)\cap \{ x^2<0\} $.
Thus, for every $j$ with $\sqrt{n}<2\cdot 4^j <4^{j+1}<\frac{n}{2}$,
we have
\begin{equation}\label{eq:3.13}
\mu\left( \begin{array}{@{\,}c@{\,}}
\mbox{there exists a $(+)$-path in ${\overline A}_{1,j}$}\\
\mbox{connecting $\gamma $ with $\varDelta U_n(\gamma_1, \ell (L_2)^2)$}
\end{array}
\,\bigg\vert\,
  \begin{array}{@{\,}c@{\,}}
 \{ \gamma (\omega )=\gamma \} \\
 \cap E\cap F\cap F_1 
  \end{array}
\right)
\geq \frac{1}{8}\delta_8^4.
\end{equation}

\medskip
\noindent
$3^\circ $) Let
$$
{\tilde H}_{1,j}(\gamma ) =\left\{\begin{array}{@{\,}c@{\,}}
\mbox{ there exists a $(+)$-path in ${\overline A}_{1,j}$}\\
\mbox{which connects $\gamma $ with $\varDelta U_n(\gamma_1, 1)$}
\end{array}
\right\} .
$$
Then by (\ref{eq:3.13}) and by the finite energy property
\begin{equation}\label{eq:3.14}
\mu \left( {\tilde H}_{1,j}(\gamma ) \,\bigg\vert\, \{ \gamma (\omega ) = \gamma \}
  \cap E \cap F\cap F_1 \right) \geq \frac{\delta_8^4}{8}c(T)^{\ell (L_2)^2},
\end{equation}
where $c(T)= [1+ e^{-8/(\mathfrak{K}T)}]^{-1}$.
Let $j_*$ and $j^*$ be minimum and maximum of $j$'s such that
(\ref{Hig93a5.21}) holds, respectively.
Then for every $j$ with $j_*\leq j\leq j^*$, we have by (\ref{eq:3.14})
\begin{align*}
&\mu\left( \bigcap_{p=j_*}^j {\tilde H}_{1,p}^c \,\bigg\vert \,
   \{ \gamma (\omega ) = \gamma \} \cap E\cap F \right) \\
&= \mu\left( \bigcap_{p=j_*}^{j-1} {\tilde H}_{1,p}^c \,\bigg\vert \,
  \{ \gamma (\omega ) = \gamma \} \cap E\cap F \right) \\
&\quad \times 
\mu\left( 
{\tilde H}_{1,j}^c \,\bigg\vert\,  \{ \gamma (\omega ) = \gamma \} 
  \cap E\cap F \cap  \bigcap_{p=j_*}^{j-1} {\tilde H}_{1,p}^c \right) \\
&\leq  \mu\left( \bigcap_{p=j_*}^{j-1} {\tilde H}_{1,p}^c \,\bigg\vert \,
  \{ \gamma (\omega ) = \gamma \} \cap E\cap F \right) 
  \left( 1 - \frac{\delta_8^4}{8}c(T)^{\ell (L_2)^2}\right) \\
&\leq  {\left( 1 - \frac{\delta_8^4}{8}c(T)^{\ell (L_2)^2}\right) }^{j-j_*+1}.
\end{align*}

Now, $j^*-j_* = \frac{\ell(n)}{2} + O(1)$ and $\ell (L_2)= \ell_3(n) +O(1)$.
Therefore we know that
$$ {\left( 1 - \frac{\delta_8^4}{8}c(T)^{\ell (L_2)^2}\right) }^{j-j_*+1}
  \rightarrow 0 \ \mbox{ as } \ n\rightarrow \infty .
$$
Thus, we can take $N_5= N_5(k,\varepsilon_0)\geq N_4$, such that 
\begin{equation}\label{eq:3.15}
\mu \left( \left.
\begin{array}{@{\,}c@{\,}}
\mbox{there exists a $(+)$-path }\\
\mbox{in $S(v_1(\gamma ), \frac{n}{2})$, connecting}\\
\mbox{$\gamma $ with $\varDelta U_n(\gamma_1,1)$}
\end{array}
\,\right|\, \{\gamma (\omega )=\gamma \}\cap E\cap F \right) \geq \frac{3}{4}
\end{equation}
for $N_5\leq n<\min\{ L(h,\varepsilon_0), \frac{N}{k}\} $.

We can apply the same argument for $\gamma_2$ and $v_2(\gamma )$,
and obtain
\begin{equation}\label{eq:3.16}
\mu \left(\left.
\begin{array}{@{\,}c@{\,}}
\mbox{there exists a $(+)$-path }\\
\mbox{in $S(v_2(\gamma ), \frac{n}{2})$, connecting}\\
\mbox{$\gamma $ with $\varDelta L_n(\gamma_2,1)$}
\end{array}
\,\right|\, \{\gamma (\omega )=\gamma \}\cap E\cap F \right) \geq \frac{3}{4}
\end{equation}
for $N_5\leq n<\min\{ L(h,\varepsilon_0), \frac{N}{k}\} $.
Combining (\ref{eq:connection lemma}), (\ref{eq:3.15}) and (\ref{eq:3.16}),
we obtain
$$
\mu \left(  \begin{array}{@{\,}c@{\,}}
\mbox{there exists a $(+)$-path in $U_n(\gamma_1)\cap L_n(\gamma_2)$}\\
\mbox{connecting $\varDelta U_n(\gamma_1,1)$ with $\varDelta L_n(\gamma_2,1)$}
\end{array} \,\bigg\vert\, E\cap F \right) \geq \frac{1}{4}\delta_{8k}.
$$
It suffices to take $n_3=N_5$ to see that Lemma \ref{connection lemma} is proved.

\end{proof}

\begin{Remark}\label{box-size-bound}
As we saw in the proof, since we used rectangles of width not exceeding 
$n/4$, we can apply Lemma \ref{RSWbound} as long as $n<L(h,\varepsilon_0)$.
Therefore Lemma \ref{connection lemma} holds true for every 
$n_3\leq n<8L(h,\varepsilon_0)$.
\end{Remark}
\newpage

\section{Fence argument}

In this section, we give an Ising version of the basic result in 
Kesten \cite{K87scaling}, concerning the notion of fences.
Although the argument is applicable to $S(n)$ for any large $n$, but 
as in \cite{K87scaling}, we restrict ourselves to the cases where
$n=2^k$.

Let $r$ be a path from $\mathbf{O}$ to $\partial_{in} S(2^k)$ in $S(2^k)$,
for some $k \geq 2$; assume that all spins on $r$ other than $\mathbf{O}$ 
are $+$. 
For the sake of argument, assume that the endpoint of $r$ lies on 
$\{ -2^k \} \times [-2^k,2^k]$, the left side of $S(2^k)$. 
Let $r'$ be the piece of $r$ from the last intersection with the line 
$\{ x^1= -2^k+2^{k-2} \} $ to the left side of $S(2^k)$; 
$r'$ is a horizontal crossing of the rectangle
$$
\mathcal{B}_1:=[-2^k, -2^k+2^{k-2}] \times [-2^k,2^k].
$$ 
Let $\mathcal{C}=\mathcal{C}(r,k)$ be the $(+)$-cluster in 
$\mathcal{B}_1$ which contains $r'$.
We call $\mathcal{C}$ the crossing $(+)$-cluster in $\mathcal{B}_1$
containing $r'$. Namely a crossing $(+)$-cluster in $\mathcal{B}_1$
is a $(+)$-cluster in $\mathcal{B}_1$ such that it contains a horizontal
$(+)$-crossing of $\mathcal{B}_1$.
The lowest point of $\mathcal{C}$ on the left side of $S(2^k)$ is denoted 
by $a=a(\mathcal{C})$. 
Let $\mathcal{B}_1^+=\mathcal{B}_1^+(r)$ denote the region in $\mathcal{B}_1$ above $r$, 
and $\mathcal{B}_1^-=\mathcal{B}_1^-(r)$ denote the region below $r$. 

We say that $r$ (or $\mathcal{C}$) has an $(\eta,k)$-fence if all three
of the following conditions hold:
\begin{align}
&\mbox{If $t$ is any path from $\mathbf{O}$ to $\partial_{in} S(2^k)$ 
which lies in ${S}(2^k-1)$, } \label{K(2.26)} \\
&\mbox{except for its endpoint, and on which all spins except for 
$\mathbf{O}$ are $+$,} \notag \\
&\mbox{and its corresponding component $\mathcal{C}(t,k)$ satisfies that} 
\notag \\
&\mbox{$\mathcal{C}(t,k) \cap \mathcal{C} = \emptyset$, then $| a(\mathcal{C}(t,k)) - a(\mathcal{C})| > 2\sqrt{\eta}2^k$.} \notag \\
&\mbox{If $r^*$ is any $(*)$-path from $\mathbf{O}$ to 
$\partial_{in} S(2^k)$ which lies in ${S}(2^k-1)$, } 
\label{K(2.27)} \\
&\mbox{except for its endpoint, and on which all spins except for $\mathbf{O}$ are $-$, } \notag \\
&\mbox{and its corresponding $(*)$-component is $\mathcal{C}^*(r^*,k)$, then} \notag \\
&\mbox{$| a^* (\mathcal{C}^*(r^*,k)) - a(\mathcal{C})| > 2\sqrt{\eta}2^k$.} \notag \\
&\mbox{There exists a vertical $(+)$-crossing  of the rectangle} 
\label{K(2.28)} \\
&\mbox{$[a^1-\sqrt{\eta}2^k,a^1-1] \times [a^2-\eta2^k,a^2+\eta2^k]$,} \notag \\
&\mbox{which is $(+)$-connected to $\mathcal{C}$ in $S(a,\sqrt{\eta}2^k)$. }\notag \\
&\mbox{(Here $a=(a^1,a^2)$.)} \notag
\end{align}

We can also define an $(\eta,k)$-fence for a $(-*)$-cluster $\mathcal{C}^*$ of 
$\mathcal{B}_1$ by interchanging $(+)$ and $(-*)$ everywhere in the 
above.
Similarly, let $\mathcal{B}_2,\, \mathcal{B}_3$ and 
$\mathcal{B}_4$ be rectangles such that $\mathcal{B}_{i+1}$ is 
clock-wise rotation by a right angle of $\mathcal{B}_i$ for each $i=1,2,3$.
Note that we have to consider vertical $(+)$-crossings or $(-*)$-crossings
in ${\mathcal B}_2$ and ${\mathcal B}_4$.
\begin{Lemma}[cf. \cite{K87scaling} Lemma 2] \label{lem:fence}
Let $\mu=\mu_{T,h}$ or $\mu_t^N$. For each $\delta > 0$, there
exists an $\eta=\eta(\varepsilon_0, \delta)>0$ and $n_4=
n_4(\eta ,\varepsilon_0)$
such that if
$n_4<2^k<\min\{ 4L(h,\varepsilon_0), \frac{N}{2}\} $, then
\[ \mu \left(  \begin{array}{@{\,}c@{\,}} 
\mbox{there exists a horizontal $(+)$-crossing of $\mathcal{B}_1$ } \\
\mbox{whose $(+)$-cluster $\mathcal{C}$ in $\mathcal{B}_1$ 
does not have an $(\eta,k)$-fence} \\
\end{array} \right) \leq \delta.  \]
The same inequality holds for each $i$ and $(-*)$-connection,
with obvious modifications; the word ``horizontal'' is replaced with
``vertical'' when $i=2,4$, and ``$(+)$'' is replaced with ``$(-*)$'',
respectively.
\end{Lemma}

\begin{proof} This can be proved along the same line as Lemma 2 of Kesten \cite{K87scaling}. We need to avoid using the BK inequality; this can be done by conditioning step by step.

%% \begin{quote}
%% ---------------------------------------------
%%
%% \noindent \framebox{\bf INTERMISSION}

Assume that there exists a horizontal $(+)$-crossing of $\mathcal{B}_1$. 
Let $\mathcal{R}_1$ be the lowest of such crossings and $\mathcal{C}_1$ its 
$(+)$-cluster in $\mathcal{B}_1$; $a=(a^1,a^2)$ denotes the left
endpoint of  $\mathcal{R}_1$.
As was done in \cite{K87scaling}, in order that the conditions 
(\ref{K(2.26)})--(\ref{K(2.28)}) are satisfied for ${\mathcal R}_1$,
and $\mathcal{C}_1$, it is sufficient to have a $(+)$-path $\xi_1$
in $S(a,\sqrt{\eta }2^k)\setminus S(a,\eta 2^k)$ and a $(+)$-path
$\xi_2$ in $S(a,{\root 4 \of \eta }2^k)\setminus S(a,2\sqrt{\eta }2^k)$ 
both connecting ${\mathcal R}_1$ with the half line 
$\{ (x^1,a^2-\eta 2^k); x^1< a^1\} $ in the anti-clockwise direction.
The existence of $\xi_1$ assures the condition (\ref{K(2.28)}), and 
the existence of $\xi_2$ assures the conditions (\ref{K(2.26)}) and
(\ref{K(2.27)}).

We first take a small $\varepsilon > 0$, and consider the annuli
\begin{align*}
A_1(\mathcal{R}_1) &= S(a,\sqrt{\eta}2^k) \setminus S(a,\eta 2^k), \\ 
A_2(\mathcal{R}_1) &= S(a,\sqrt[4]{\eta}2^k) \setminus S(a,2\sqrt{\eta} 2^k).
\end{align*}
By Lemma \ref{Hig93aL5.3}, we take $M_0 = M_0 (\varepsilon)$ so that
as in the derivation of (\ref{eq:3.10}), we have
\begin{align*}
&\mu \left( \left.
\begin{array}{@{\,}c@{\,}} 
\mbox{there exists a $(+)$-path $\xi_1$ in $A_1(\mathcal{R}_1)$,} \\
\mbox{connecting $r_1$ with $\{(x^1,a^2-\eta 2^k) ; x^1 < a^1 \}$} \\
\mbox{in the anti-clockwise direction}
\end{array}
\,\right|\, \mathcal{R}_1(\omega)=r_1 \right) \\
&\geq 1- \varepsilon - C\eta^2 2^{3k} e^{-\alpha \eta 2^k}
\end{align*}
if $n_2 \leq \eta 2^k < 2^{M_0} \eta 2^k < \sqrt{\eta} 2^k < L(h,\varepsilon_0)$.
Let $n^*$ be the number such that
\[ C\eta^{-1}n^3 e^{-\alpha n} < \varepsilon \quad \mbox{for all $n \geq n^*$}, \]
and we assume that $\eta 2^k \geq n^*$, so that
\begin{align*}
&\mu \left( \left.
\begin{array}{@{\,}c@{\,}} 
\mbox{there exists a $(+)$-path $\xi_1$ in $A_1(\mathcal{C}_1)$,} \\
\mbox{connecting $r_1$ with $\{(x^1,a^2-\eta 2^k) ; x^1 < a^1 \}$} \\
\mbox{in the anti-clockwise direction}
\end{array}
\,\right|\, \mathcal{R}_1(\omega)=r_1 \right) \\
&\geq 1- 2\varepsilon.
\end{align*}
In the same way, for every $E \in \mathcal{F}_{S(a,2\sqrt{\eta} 2^k)}$,
\begin{align*}
&\mu \left( \left.
\begin{array}{@{\,}c@{\,}} 
\mbox{there exists a $(+)$-path $\xi_2$ in $A_2(\mathcal{C}_1)$,} \\
\mbox{connecting $r_1$ with $\{(x^1,a^2-\eta 2^k) ; x^1 < a^1 \}$} \\
\mbox{in the anti-clockwise direction}
\end{array}
\,\right|\, \{\mathcal{R}_1(\omega)=r_1\} \cap E \right) \\
&\geq 1- 2\varepsilon
\end{align*}
provided that
\begin{equation}\label{condition for eta}
C\sqrt{\eta}2^{2k} \cdot \sqrt{\eta} 2^k e^{-\alpha\sqrt{\eta} 2^k}
  <\varepsilon, \ \mbox{ and } \ 
2^{M_0}2\sqrt{\eta }2^k < \root 4 \of {\eta }2^k.
\end{equation}
The first inequality is satisfied if $\eta2^k>n^*$, since the left hand
side is equal to 
$\eta^{-1/2} C(\sqrt{\eta} 2^k)^3 e^{-\alpha \sqrt{\eta} 2^k}$.
The second inequality reduces to 
\begin{equation}\label{eq:bound for eta}
2^{M_0+1}\root 4\of \eta <1
\end{equation}
Therefore we have
\begin{align}\label{K(2.30)}
&\mu \left(  \begin{array}{@{\,}c@{\,}} 
\mbox{there exists the lowest horizontal $(+)$-crossing $\mathcal{R}_1$
of $\mathcal{B}_1$,} \\
\mbox{but one of the conditions (\ref{K(2.26)})--(\ref{K(2.28)}) fails for 
its cluster $\mathcal{C}_1$} \\
\end{array} \right) \\
&\leq 4\varepsilon \notag
\end{align}
under the conditions (\ref{eq:bound for eta}), 
$\eta 2^k \geq \max\{ n^*, n_2\} $, and 
$ \root 4\of \eta 2^k<L(h,\varepsilon_0)$.

Now assume that $\{ \mathcal{R}_i,\,\mathcal{C}_i \}_{1 \leq i \leq \sigma}$ 
are given horizontal $(+)$-crossings of $\mathcal{B}_1$, 
with their corresponding $(+)$-clusters in $\mathcal{B}_1$. 
Assume further the following:
\begin{itemize}
\item[(i)] $\mathcal{R}_i$ is disjoint from $\mathcal{C}_j$ for $i \neq j$. 
\item[(ii)] the $\mathcal{R}_i$ are ordered such that 
$\mathcal{C}_i \subset \mathcal{B}_1^-(\mathcal{R}_j)$ for $i < j$. 
($\mathcal{R}_{\sigma}$ is the highest crossing among 
$\{ \mathcal{R}_i \}_{1 \leq i \leq \sigma}$.)
\end{itemize}
If there exists still another $(+)$-crossing of $\mathcal{B}_1$  
above $\bigcup_{1 \leq i \leq \sigma} \mathcal{C}_i$, then let 
$\mathcal{R}_{\sigma+1}$ be the lowest such crossing. 
Denote its endpoint on the left side by $a_{\sigma+1}$, and its $(+)$-cluster
in $\mathcal{B}_1$ by $\mathcal{C}_{\sigma + 1}$. 
(In this case, 
$\mathcal{C}_i \subset {\mathcal{B}_1^-(\mathcal{R}_{\sigma+1})}$ 
for $1 \leq i \leq \sigma$.) 
We can repeat the above argument and obtain the following:
\begin{align}\label{K(2.31)}
&\mu \left( \left. \begin{array}{@{\,}c@{\,}} 
\mbox{$\mathcal{R}_{\sigma+1}$ exists in $\mathcal{B}_1$, above
$\cup_{1\leq i\leq \sigma}{\mathcal C}_i$,} \\
\mbox{but one of the conditions (\ref{K(2.26)})--(\ref{K(2.28)}) fails}\\
\mbox{for its cluster $\mathcal{C}_{\sigma+1}$} \\
\end{array} 
\,\right|\, \{ \mathcal{R}_i,\,\mathcal{C}_i \}_{1 \leq i \leq \sigma} \right) \\
&\leq 4\varepsilon. \notag
\end{align}

Consequently, we have for every integer $\rho >0$,
\begin{align} \label{K(2.32)}
&\mu \left(  \begin{array}{@{\,}c@{\,}} 
\mbox{there exists a horizontal $(+)$-crossing $\mathcal{R}$ of 
$\mathcal{B}_1$, such that} \\
\mbox{one of the conditions (\ref{K(2.26)})--(\ref{K(2.28)}) fails for 
its cluster $\mathcal{C}$} \\
\end{array} \right) \\
&\leq \mu \bigl(  
\mbox{there exist more than $\rho$ disjoint crossing $(+)$-clusters of
$\mathcal{B}_1$} \bigr) \notag \\
&\quad + 4\rho \varepsilon. \notag
\end{align}
Using the Connection lemma, we can see that
\begin{align*} 
&\mu \bigl(  
\mbox{there exist more than $\rho$ disjoint crossing $(+)$-clusters of
$\mathcal{B}_1$} \bigr) \\
&\leq \prod_{k=1}^{\rho} \mu \left( \left. \begin{array}{@{\,}c@{\,}} 
\mbox{there exists a horizontal $(+)$-crossing of $\mathcal{B}_1$} \\
\mbox{above $\mathcal{C}_k$} \\
\end{array} 
\,\right|\, \{ \mathcal{R}_i,\,\mathcal{C}_i \}_{1 \leq i \leq k} \right) \\
&\quad \times \mu \bigl( \mbox{there exists a horizontal 
$(+)$-crossing of $\mathcal{B}_1$} \bigr) \\
& \leq (1-\delta_{64}/4)^{\rho +1},
\end{align*}
if in addition  $n_3(8,\varepsilon_0)\leq 2^{k-2}<L(h,\varepsilon_0)$,
where $n_3$ is given in section 3.
Combining this with (\ref{K(2.32)}), and taking first $\rho$ large, then 
$\varepsilon$ small so that 
\begin{equation}\label{condition for rho and epsilon}
(1-\delta_{64}/4)^{\rho}+4\rho \varepsilon \leq \delta ,
\end{equation}
we can obtain
\begin{align} \label{K(2.33)}
&\mu \left(  \begin{array}{@{\,}c@{\,}} 
\mbox{there exists any horizontal $(+)$-crossing $\mathcal{R}$ of 
$\mathcal{B}_1$, such that} \\
\mbox{one of the conditions (\ref{K(2.26)})--(\ref{K(2.28)}) fails for 
its cluster $\mathcal{C}$} \\
\end{array} \right) \\
&\leq \delta . \notag
\end{align}

To this $\varepsilon>0$, we choose $M_0(\varepsilon)$ by Lemma \ref{Hig93aL5.3}, and choose $\eta$ satisfying (\ref{eq:bound for eta}). 
After that we take $k$ so that
\[ 2^k \geq \max \bigl\{ \eta^{-1} n_2,\,\eta^{-1} n^*,\, 4 n_3(8,\varepsilon_0) \bigr\}. \]
Then $h$ should satisfy that $L(h,\varepsilon_0)>2^{k-2}$.
%%  \end{quote}
\end{proof}
\begin{Remark}\label{box-size-bound2}
1) In the proof, we used the Connection lemma for rectangles with widths
not exceeding $2^{k-1}$, therefore from Remark \ref{box-size-bound}, 
Lemma \ref{lem:fence} holds still true when $n_4< 2^k< 2^{5}L(h, \varepsilon_0)$ provided that $M_0\geq 4$ (see (\ref{eq:bound for eta})).

2) We will require a stronger condition than (\ref{condition for rho and epsilon})
for $\rho $ and $\eta $  in the later 
discussion, but we will not have to change 
the statement of Lemma \ref{lem:fence}.
See the discussion in the next section 5.3. 
\end{Remark}

\newpage

\section{Extension argument}

Here we present Ising version of Lemmas 4 and 5 of \cite{K87scaling}.

\subsection{Blocks}
The main idea in the subsequent sections is to divide $S(2^k)$ into
suitable blocks. 
The sizes of blocks differ according to problems and also the relative 
location of these blocks in $S(2^k)$.
Let us begin with definition of such blocks.

\begin{Definition}
Let $1<j<k$ be integers and for every  $v\in S(2^k)$, let $Q_j(v)$
denote the square 
$$
( \ell_12^j,(\ell_1+1)2^j]\times (\ell_22^j,(\ell_2+1)2^j],
$$
which contains $v$, unless $\ell_12^j=-2^k$ or $\ell_22^j=-2^k$.
If $\ell_12^j=-2^k$ but $\ell_22^j\not=-2^k$, then we put
$$
Q_j(v) = [-2^k,-2^k+2^j]\times (\ell_22^j,(\ell_2+1)2^j].
$$
If  $\ell_22^j=-2^k$ but $\ell_12^j\not= -2^k$, then we put
$$
Q_j(v) = (\ell_12^j,(\ell_1+1)2^j]\times [-2^k,-2^k+2^j].
$$ 
Finally, if $\ell_12^j=\ell_22^j=-2^k$, then we put
$$
Q_j(v) = [-2^k,-2^k+2^j]\times [-2^k,-2^k+2^j].
$$
\end{Definition}
Thus, the totality of distinct $Q_j(v)$'s form a partition of $S(2^k)$.
Hereafter in this section we fix $v\in S(2^k)$ with $v=(v^1,v^2)$ such that
$0\leq v^2\leq v^1$.
The following argument can be easily modified to $v\in S(2^k)$ in other
cases.

Let $x_j(v)$ be the lower left corner of $Q_j(v)$, i.e.,
\begin{equation}\label{lower left corner of Q_j(v)}
x_j(v) = (x_j^1(v),x_j^2(v)) = (\ell_12^j, \ell_22^j ) .
\end{equation}
Then for $m\geq 0$, let
\begin{equation}\label{mth box }
S_j^m(v)= x_j(v) + S(2^{j+m}).
\end{equation}
If $v$ is near the boundary of $S(2^k)$, then 
$S_j^m(v)$ may not be inside of $S(2^k)$.
In this case, we  consider the following box $T_j^m(v)$ instead
of $S_j^m(v)$,
\begin{equation}\label{boundary mth box}
T_j^m(v) = [2^k-2^{j+m+1},2^k ]\times [x_{j}^2(v)-2^{j+m}, x_j^2(v)+2^{j+m}]
\end{equation}
if $x_j^2(v)+2^{j+m}\leq 2^k$, and 
\begin{equation}\label{corner mth box}
T_j^m(v)={[2^k-2^{j+m+1},2^k ]}^2,
\end{equation}
if $x_j^2(v)+2^{j+m}>2^k$.
\begin{Lemma}\label{S_j and T_j}
\begin{enumerate}
\item $T_j^{m+1}(v) \supset T_j^m(v)$.
\item If $S_j^m(v) \subset S(2^k)$ and $S_j^{m+1}(v)\not\subset S(2^k)$, then
$S_j^m(v) \subset T_j^{m+1}(v)$.
\end{enumerate}
\end{Lemma}
The proof is straightforward, so we omit it.
Let 
\begin{equation}\label{m_1^*}
m_1^*=m_1^*(v) = \max\{ m\geq 0 : S_{j_1}^m(v) \subset S(2^k) \} .
\end{equation}

%Ising version of Kesten-Lemma4
%\renewcommand{\labelenumi}{\theenumi)}

\subsection{Block events: inwards}
We put 
\begin{equation}\label{eq:the constant C1}
C_1:= {\left( \frac{\delta_{80}\delta_{16}}{16}
              \right) }^{-4}.
\end{equation}
Let $\delta =\delta (\varepsilon_0)>0$ be given by 
\begin{equation}\label{eq:delta-1}
\delta = \frac{1}{18}C_1^{-1}
\end{equation}
By Lemma \ref{lem:fence},
we can choose $\eta = \eta (\varepsilon_0, \delta )>0$ and
$n_4=n_4(\eta ,\varepsilon_0)$ such that
\[ \mu \left(  \begin{array}{@{\,}c@{\,}} 
\mbox{there exists a crossing $(+)$-cluster in $\mathcal{B}_1$} \\
\mbox{which does not have an $(\eta ,k)$-fence} 
\end{array} \right) \geq 1- \delta.  \]
for every $k$ with $n_4\leq 2^k< \min\{ L(h,\varepsilon_0), \frac{N}{2}\} $,
where $\mu = \mu_{T,h}$ or $\mu_t^N$ for $0\leq t\leq 1$.
The above inequality is valid for crossing $(-*)$-clusters, too.
Let $j_1$ be sufficiently large such that 
\begin{equation}\label{eq:bound for j_1}
\max \bigl\{ \eta^{-1} n_3 (\lceil 8\eta^{-1} \rceil, \varepsilon_0) ,\, n_4(\eta ,\varepsilon_0) \bigr\} \leq 2^{j_1},
\end{equation}
and we assume that $2^{j_1}<2^k< \min\{ L(h,\varepsilon_0), \frac{N}{2}\} $.
Here, for a real value $x$, $\lceil x\rceil $ denotes the smallest integer
not less than $x$.

Let $v=(v^1,v^2)\in S(2^k)$. 
As before we assume that $0\leq v^2\leq v^1$ for the sake of argument.  
We will first define three events on $S_{j_1}^m(v)$, namely
$\Gamma (v, S_{j_1}^m(v)), \Lambda (v, S_{j_1}^m(v))$ and 
$\Delta (v, S_{j_1}^m(v))$  for $1\leq m$, 
in the same way as in \cite{K87scaling}.
These events are prototypes of events we introduce later.
After that we have to modify them as
$\Gamma (v, T_{j_1}^m(v)), \Lambda (v, T_{j_1}^m(v))$ 
and ${\Delta }(v, T_{j_1}^m(v))$ for $m>m_1^*$.

Let $\Gamma (v, S_{j_1}^m(v))$ be the event such that all the following occur.
\begin{enumerate}
\item There exist two $(+)$-paths $r_1,r_3$ in $S_{j_1}^m(v)$, connecting $v$
with the inner boundary $\partial_{in}S_{j_1}^m(v) $ of $S_{j_1}^m(v)$
such that $r_1\setminus \{ v\} $ and $r_3\setminus \{ v\} $ are disjoint. 
\item There exist two disjoint $(-*)$-paths $r_2^*, r_4^*$ in $S_{j_1}^m(v)$ 
connecting 
$(*)$-neighbor of $v$ with $\partial_{in}S_{j_1}^m(v)$.
\item $r_1\cup r_3 $ separates $r_2^*$ and $r_4^*$ in $S_{j_1}^m(v)$.
\end{enumerate}
As in the previous section, for each $S(2^j)$ let ${\mathcal B}_i$, 
for $i=1,2,3,4$ be given by
\begin{align*}
{\mathcal B}_1&= [-2^j, -2^j +2^{j-2}]\times [-2^j, 2^j], \\
{\mathcal B}_2&= [-2^j,2^j]\times [2^j-2^{j-2}, 2^j], \\
{\mathcal B}_3&= [2^j-2^{j-2}, 2^j]\times [-2^j,2^j], \\
{\mathcal B}_4&= [-2^j,2^j]\times [-2^j, -2^j+2^{j-2}].
\end{align*}
Note that 
$$
\bigcup_{i=1}^4 {\mathcal B}_i = S(2^j)\setminus S(2^{j-1}+2^{j-2}-1).
$$
${\mathcal B}_i$'s are defined for each $S(2^j)$ by their relative location
to the square in consideration. 
So, we will define ${\mathcal B}_i$, $i=1,2,3,4$ for $S_{j_1}^m(v)$ as shifts
of ${\mathcal B}_i$'s which are originally defined for $S(2^{j_1+m})$, by
$x_{j_1}(v)$.

On the event $\Gamma (v, S_{j_1}^m(v))$, $r_1$ crosses one of ${\mathcal B}_i$'s of $S_{j_1}^m(v)$.
Assume that the endpoint of $r_1$ is in ${\mathcal B}_1$. 
Then $r_1$ surely crosses ${\mathcal B}_1$.
Let ${\mathcal C}_1$ be the $(+)$-connected component in ${\mathcal B}_1$ containing 
the endpoint of $r_1$.
Similarly, we define $(+)$-connected component ${\mathcal C}_3$ of some 
${\mathcal B}_i$ which contains the endpoint of $r_3$,
and $(-*)$-components ${\mathcal C}_2^*, {\mathcal C}_4^*$ of some
of ${\mathcal B}_i$'s containing endpoints of $r_2^*$, and $r_4^*$,
respectively.
Then we define 
\begin{equation}\label{lambdaevent}
\Lambda (v, S_{j_1}^m(v), \eta ) =  \left\{
\omega \in \Gamma (v, S_{j_1}^m(v))  :
 \begin{array}{@{\,}l@{\,}}
 \mbox{any of ${\mathcal C}_1,{\mathcal C}_2^*,{\mathcal C}_3, {\mathcal C}_4^*$}\\
 \mbox{has an $(\eta , j_1+m)$-fence} \\
 \end{array} \right\} . 
\end{equation}

To define $\Delta (v,S_{j_1}^m(v))$, we introduce other rectangles in $S(2^j)$.  
Let ${\mathcal A}_i, i=1,2,3,4$ be given by
\begin{equation}\label{eq:Ai-1}
\begin{cases}
{\mathcal A}_1 =  [-2^j, -2^j+2^{j-2}]\times [-2^{j-2},2^{j-2}], \\
{\mathcal A}_2 =  [-2^{j-2},2^{j-2}]\times [2^j-2^{j-2}, 2^j],  \\
{\mathcal A}_3 =  [2^j-2^{j-2},2^j]\times [-2^{j-2},2^{j-2}], \\
{\mathcal A}_4 = [-2^{j-2},2^{j-2}]\times [-2^j, -2^j+2^{j-2}].
\end{cases}
\end{equation}
We also define ${\mathcal A}_i$ by their relative locations to $S(2^j)$,
and we can define them for $S_{j_1}^m(v)$ as shifts of ${\mathcal A}_i$'s, 
which are 
originally defined for $S(2^{j_1+m})$, by $x_{j_1}(v)$.

Now, define the event $\Delta (v, S_{j_1}^m(v))$ as the subset of 
$\Gamma (v, S_{j_1}^m(v))$ such that all the following occur.
\begin{enumerate}
\item $r_1$ and $r_3$ connect $v$ with the left and the right
     sides of $S_{j_1}^m(v)$, respectively; $r_2^*$ and $r_4^*$
     connect $(*)$-neighbors of $v$ with the top and the bottom
     sides of $S_{j_1}^m(v)$, respectively.
\item $r_i \cap (\cup_{1\leq i\leq 4}\mathcal{B}_i) \subset {\mathcal A}_i$ 
and $r_{i+1}^* \cap (\cup_{1\leq i\leq 4}\mathcal{B}_i) \subset {\mathcal A}_{i+1}$,
     for $i=1,3$.
\item There are vertical $(+)$-crossings in ${\mathcal A}_1$ and ${\mathcal A}_3$,
     and horizontal $(-*)$-crossings in ${\mathcal A}_2$ and ${\mathcal A}_4$.
\end{enumerate}
These events are not occurring inside $S(2^k)$ if $m_1^*<m$, and in this case, we
have to modify the definition of them. 

Let 
$\Gamma (v,T_{j_1}^m(v))$ be the event such that all of the following occur.
\begin{enumerate}
\item There exists a $(+)$-path $r_1$ in $T_{j_1}^m(v)$, connecting $v$ with
     $\partial_{in} T_{j_1}^m(v)\setminus \partial_{in}S(2^k)$,
     and a $(+)$-path $r_3$ connecting $v$ with
     $\partial_{in} T_{j_1}^m(v)$, such that $r_1\setminus \{ v\} $ and 
     $r_3\setminus \{ v\} $ are disjoint. 
\item There exists a $(-*)$-path $r_2^*$ in $T_{j_1}^m(v)$,
     connecting $(*)$-neighbor of $v$ with 
     $\partial_{in} T_{j_1}^m(v)\setminus \{ x^1=2^k\} $,
     and  a $(-*)$-path $r_4^*$ connecting $(*)$-neighbor of
     $v$ with $\partial_{in}T_{j_1}^m(v)\setminus \partial_{in}S(2^k)$ 
     in $T_{j_1}^m(v)$.
\item $r_1\cup r_3$ separates $r_2^*$ and $r_4^*$ in $T_{j_1}^m(v)$.
\end{enumerate}
Then we have to define ${\mathcal B}_i$'s for $T_{j_1}^m(v)$.
We take the same relative location to $T_{j_1}^m(v)$ as before.
For example,
$$ 
{\mathcal B}_1= [2^k-2^{j_1+m+1}, 2^k-2^{j_1+m+1}+2^{j_1+m-2}]\times 
 [t^2-2^{j_1+m}, t^2+2^{j_1+m}],
$$
where $t=(t^1,t^2)$ denotes the center of $T_{j_1}^m(v)$.
Note that $t^1= 2^k-2^{j_1+m}$, and $t^2=x_{j_1}^2(v)$ if 
$x_{j_1}^2(v)+2^{j_1+m}\leq 2^k$, $t^2=2^k-2^{j_1+m}$ if 
$x_{j_1}^2(v)+2^{j_1+m}> 2^k$.
The left side of ${\mathcal B}_1$ is the same as the left side of
$T_{j_1}^m(v)$.
In the same way, we can define ${\mathcal B}_2, {\mathcal B}_3$ and ${\mathcal B}_4$.
Then we also define connected components ${\mathcal C_i}, {\mathcal C}_{i+1}^* $ 
corresponding to $r_i$ and $r_{i+1}^*$ for $i=1,3$ as before.
Let
\begin{equation}\label{lambda*event}
\Lambda (v, T_{j_1}^m(v)) =\left\{
\begin{array}{@{\,}c@{\,}} 
\mbox{each of ${\mathcal C}_1, {\mathcal C}_{2}^*,{\mathcal C}_3,
{\mathcal C}_{4}^* $ has an $(\eta ,j_1+m)$-fence}\\
\mbox{if its endpoint is  not in $\partial_{in}S(2^k)$}\\
\end{array}
\right\}
\end{equation} 
Let ${\mathcal A}_i$ be defined for $T_{j_1}^m(v)$ so that
their relative locations for $T_{j_1}^m(v)$ are the same as those for $S(2^{j_1+m})$.

Now we define $\Delta (v,T_{j_1}^m(v))$.
First, in the case where $T_{j_1}^m(v)\not\ni (2^k,2^k)$,we define
$\Delta (v,T_{j_1}^m(v))$ as a subset of  
$\Gamma (v,T_{j_1}^m(v))$ such that all the following occur.
\begin{enumerate}
\item $r_1$ and $r_3$ connect $v$ with the left side and the right side of
    $T_{j_1}^m(v)$, respectively.
\item $r_2^*$ and $r_4^*$ connect $(*)$-neighbor of $v$ with the top side 
    and the bottom side of $T_{j_1}^m(v)$, respectively.  
\item $r_1\cap (\cup_{i=1,2,4}\mathcal{B}_i) \subset {\mathcal A}_1$, and
    $r_{i+1}^*\cap (\cup_{i=1,2,4}\mathcal{B}_i) \subset {\mathcal A}_{i+1}$,
    for $i=1,3$.
\item There exist a vertical $(+)$-crossing in ${\mathcal A}_1$, and
    horizontal $(-*)$-crossings in both ${\mathcal A}_2$ and ${\mathcal A}_4$.
\item $r_3 \cap (\cup_{i=1,2,4}\mathcal{B}_i)=\emptyset$.
\end{enumerate}

Next, when $T_{j_1}^m(v)\ni (2^k,2^k)$, we define
$\Delta (v,T_{j_1}^m(v))$ as the subset of 
$\Gamma (v,T_{j_1}^m(v))$
such that all the following occur.
\begin{enumerate}
\item $r_1$ and $r_3$ connect $v$ with the left side and the right side of
    $T_{j_1}^m(v)$, respectively.
\item $r_2^*$ and $r_4^*$ connect $(*)$-neighbor of $v$ with the top side 
     and the bottom side of $T_{j_1}^m(v)$, respectively.  
\item $r_1\cap (\mathcal{B}_1\cup\mathcal{B}_4) \subset {\mathcal A}_1$, and
    $r_4^* \cap (\mathcal{B}_1\cup\mathcal{B}_4)\subset {\mathcal A}_4$.
\item There exist a vertical $(+)$-crossing in ${\mathcal A}_1$, and 
    a horizontal $(-*)$-crossing in ${\mathcal A}_4$.
\item $(r_2^*\cup r_3) \cap (\mathcal{B}_1\cup\mathcal{B}_4)=\emptyset $.
\end{enumerate}
For later use let us introduce the notation $R_{j_1}^m(v)$ to denote
$S_{j_1}^m(v)$ if $m\leq m_1^*$, and $T_{j_1}^m(v)$ if $m>m_1^*$.
\begin{Lemma}[cf. (2.43) in \cite{K87scaling}] \label{relation of deltas}
Let $2^{j_1}<2^k<\min\{ L(h,\varepsilon_0), \frac{N}{2}\} $.
Then for $t\in [0,1]$ and every $m\geq 1$, we have
\begin{equation}\label{eq:relation of deltas}
\mu_t^N\bigl( \Delta (v,R_{j_1}^m(v))\bigr) 
\leq C_1 \mu_t^N\bigl( \Delta (v,R_{j_1}^{m+1}(v))\bigr) ,
\end{equation}
where $C_1$ is the constant given by (\ref{eq:the constant C1}).

As a result we can find some constant $C_2>0$ which depends only on
$\varepsilon_0$ and $j_1$, such that
$$
\mu_t^N\bigl( \Delta (v,R_{j_1}^m(v))\bigr) \geq C_2C_1^{-m} 
$$
for every $ m\geq 1$.
\end{Lemma}

\begin{proof}
This lemma can be proved essentially in the same way as in \cite{K87scaling}
by using the connection lemma in place of independence.

%------------------------------------------------------------------------------
%%%%%%%%%%%%%%%%%%%%%% The following part can be erased after we have checked
%%%%%%%%%%%%%%%%%%%%%% all the details
\noindent
First, we consider the case where $R_{j_1}^m(v)=S_{j_1}^m(v)$ and
$R_{j_1}^{m+1}(v)=S_{j_1}^{m+1}(v)$.
Let ${\mathcal A}_i$ and ${\mathcal A}'_i$ denote ${\mathcal A}_i$'s
corresponding to $S_{j_1}^m(v)$ and $S_{j_1}^{m+1}(v)$, respectively.
Then take a new rectangle ${\mathcal D}_1$ which intersects both 
${\mathcal A}_1$
and ${\mathcal A}'_1$. To be precise, we define ${\mathcal D}_1$ as
the rectangle
$$
x_{j_1}(v)+ [ -2^{j_1+m+1}, -2^{j_1+m}+2^{j_1+m-2}]
               \times [-2^{j_1+m-2}, 2^{j_1+m-2}].
$$
Also we write ${\mathcal D}_2, {\mathcal D}_3, {\mathcal D}_4$ for rectangles 
obtained by rotating ${\mathcal D}_1$ successively by right angles in
the clockwise direction around $x_{j_1}(v)$,
so that ${\mathcal D}_i$ intersects ${\mathcal A}_i$ and ${\mathcal A}'_i$
for $i=1,2,3,4$.
Let $D_i$ be the event given by
$$
D_i=\left\{
\begin{array}{@{\,}c@{\,}} 
\mbox{there are a vertical $(+)$-crossing in ${\mathcal A}'_i$ and}\\
\mbox{a horizontal $(+)$-crossing in ${\mathcal D}_i$}
\end{array}
  \right\}
$$
for $i=1,3$ and
$$
D_i=\left\{
\begin{array}{@{\,}c@{\,}} 
  \mbox{there are a horizontal $(-*)$-crossing in ${\mathcal A}'_i$ and} \\
  \mbox{a vertical $(-*)$-crossing in ${\mathcal D}_i$}
\end{array}
 \right\} 
$$
for $i=2,4$.
Then we have
$$
\Delta (v, S_{j_1}^m(v)) \cap \bigcap_{i=1}^4 D_i  \subset \Delta (v,S_{j_1}^{m+1}(v)).
$$
Therefore we have to estimate $\mu_t^N$-probability of  
$\Delta (v, S_{j_1}^m(v)) \cap D_1 \cap \cdots \cap D_4$.
To do this, for $\omega \in \Delta (v,S_{j_1}^m(v))$ let 
$\tau_1(\omega ), \tau_2^*(\omega ), \tau_3(\omega ), \tau_4^*(\omega )$
be given as follows.
\begin{itemize}
\item $\tau_1(\omega ) $ is the rightmost vertical $(+)$-crossing in ${\mathcal A}_1$,
\item $\tau_2^* (\omega )$ is the lowest horizontal $(-*)$-crossing in ${\mathcal A}_2$,
\item $\tau_3 (\omega )$ is the leftmost vertical $(+)$-crossing in ${\mathcal A}_3$, and
\item $\tau_4^*(\omega ) $ is the highest horizontal $(-*)$-crossing in ${\mathcal A}_4$.
\end{itemize}
So we divide $\Delta (v,S_{j_1}^m(v))$ into disjoint subsets according to the
shape of $\tau_i, \tau_{i+1}^*$'s. 
We write it simply by
$$
\Delta (v,S_{j_1}^m(v) )= \bigcup_{\tau_1,\tau_2^*,\tau_3,\tau_4^*}
          \Delta (v,S_{j_1}^m(v); \tau_1,\tau_2^*,\tau_3,\tau_4^*).
$$
$\tau_i$ ($\tau_{i+1}^*$) divides ${\mathcal A}_i$ (${\mathcal A}_{i+1}$) into
two parts. 
We denote by $\Theta (\tau_i)$ ($\Theta (\tau_{i+1}^*)$) the 
part of ${\mathcal A}_i$ (${\mathcal A}_{i+1}$) adjacent
to $S(x_{j_1}(v), 2^{j_1+m}-2^{j_1+m-2}-1)$.
Then we define for each realization $\tau_1,\tau_2^*,\tau_3,\tau_4^*$,
${\tilde \Delta }(v,S_{j_1}^m(v); \tau_1,\tau_2^*,\tau_3,\tau_4^*)$
as the event occurring in the region 
$$
\cup_{i=1,3} \Theta(\tau_i)\cup 
\Theta (\tau_{i+1}^*) \cup S(x_{j_1}(v), 2^{j_1+m}-2^{j_1+m-2}-1)
$$
such that all the following occur.
\begin{enumerate}
\item $\tau_i(\omega ) =\tau_i, \tau^*_{i+1}(\omega )=\tau^*_{i+1}, i=1,3$.
\item There exists a $(+)$-path ${\tilde r}_i$ in 
$\Theta (\tau_i)\cup S(x_{j_1}(v), 2^{j_1+m}-2^{j_1+m-2}-1)$ connecting
$ \tau_i$ with $v$, for $i=1,3$. ${\tilde r}_1\setminus \{ v\} $ and 
${\tilde r}_3\setminus \{ v\} $ are disjoint.
\item There exists a $(-*)$-path ${\tilde r}_{i+1}^*$ in 
$\Theta (\tau_{i+1}^*)\cup S(x_{j_1}(v), 2^{j_1+m}-2^{j_1+m-2}-1)$ connecting
$\tau_{i+1}^*$ with a $(*)$-neighbor point of $v$. ${\tilde r}_2^*$ and
${\tilde r}_4^*$ are disjoint.
\item ${\tilde r}_1\cup {\tilde r}_3$ separates ${\tilde r}_2^*$ and ${\tilde r}_4^*$
in the region 
$$
\bigcup_{i=1,3}\left[ \Theta(\tau_i)\cup 
\Theta (\tau_{i+1}^*)\right] \cup S(x_{j_1}(v), 2^{j_1+m}-2^{j_1+m-2}-1).
$$
\end{enumerate}
Further, let
$$
{\tilde D}_i(\tau_i) =\left\{
\begin{array}{@{\,}c@{\,}} 
\mbox{there exists a vertical $(+)$-crossing in ${\mathcal A}'_i$, and} \\
\mbox{$\tau_i$ is connected by a $(+)$-path with }\partial_{in}S_{j_1}^{m+1}(v) 
 \mbox{ in } {\mathcal D}_i
\end{array}
\right\}
$$
and 
$$
{\tilde D}_{i+1}(\tau_{i+1}^*) =\left\{
\begin{array}{@{\,}c@{\,}}
\mbox{there exists a horizontal $(-*)$-crossing in } {\mathcal A}_{i+1}',\\ 
\mbox{and  $\tau_{i+1}^*$ is connected by a $(-*)$-path with}\\
\mbox{$ \partial_{in}S_{j_1}^{m+1}(v)$ in ${\mathcal D_{i+1}}$}
\end{array}
\right\}
$$
for $i=1,3$.
${\tilde D}_i(\tau_i)$ is an event occurring in ${\mathcal D}_i\setminus \Theta (\tau_i )$,
and ${\tilde D}_{i+1}(\tau_{i+1}^*)$ is an event occurring in 
${\mathcal D}_{i+1} \setminus \Theta (\tau_{i+1}^*)$.
By this  notation, we have
\begin{align*}
&\Delta (v, S_{j_1}^m(v) ) \cap \bigcap_{i=1}^4 D_i \\
&= \bigcup_{\tau_1,\tau_2^*,\tau_3,\tau_4^*}
   {\tilde \Delta }(v, S_{j_1}^m(v); \tau_1,\tau_2^*,\tau_3,\tau_4^*)
    \cap \bigcap_{i=1,3}{\tilde D}_i(\tau_i)\cap {\tilde D}_{i+1}(\tau_{i+1}^*)
\end{align*}
Note that the union in the right hand side is disjoint.
By the Connection lemma and the FKG inequality, we have for $i=1,3$,
$$
\mu_t^N\bigl( {\tilde D}_i(\tau_i) \mid 
   {\mathcal F}_{[ {\mathcal D}_i\setminus \Theta (\tau_i) ]^c} \bigr)
\geq \frac{ \delta_{16}\delta_{20}}{16}.
$$      
The same estimate is valid for ${\tilde D}_{i+1}(\tau_{i+1}^*)$ for $i=1,3$, too.
Therefore we have
\begin{align*}
&\mu_t^N\bigl( \Delta (v, S_{j_1}^m(v))\cap \bigcap_{i=1}^4 D_i \bigr) \\
&\geq { \biggl( \frac{\delta_{16}\delta_{20}}{16} \biggr) }^4
\sum_{\tau_1,\tau_2^*,\tau_3,\tau_4^*}
 \mu_t^N\bigl( 
     {\tilde \Delta }(v, S_{j_1}^m(v); \tau_1,\tau_2^*,\tau_3,\tau_4^*) 
        \bigr) \\
&\geq { \biggl( \frac{\delta_{16}\delta_{20}}{16} \biggr) }^4
   \mu_t^N\bigl( \Delta (v,S_{j_1}^m(v)) \bigr) .
\end{align*}

Thus, we obtain 
$$
\mu_t^N\bigl( \Delta (v,S_{j_1}^m(v)) \bigr)
  \leq {\biggl( \frac{\delta_{16}\delta_{20}}{16}\biggr) }^{-4}
\mu_t^N\bigl( \Delta (v, S_{j_1}^{m+1}(v)) \bigr) .
$$

Next, we consider the case where $m=m_1^*$. So, $R_{j_1}^m(v)=
S_{j_1}^{m_1^*}(v)$ and $R_{j_1}^{m+1}(v)=T_{j_1}^{m_1^*+1}(v)$.
In this case we do not use ${\mathcal A}'_3$.
There are two possible cases.

\noindent
\underline{Case 1} $T_{j_1}^{m_1^*+1}(v)$ does not contain the upper right 
corner  $(2^k,2^k)$ of $S(2^k)$.
In this case, we take  rectangles ${\mathcal D}_1, {\mathcal D}_3$ 
and corridors $U_2, U_4$ 
to connect ${\mathcal A}_i$ with ${\mathcal A}'_i$ for each $i$.
Namely,
$U_2$ is a corridor of side length $2^{j_1+m_1^*-1}$ and it connects
${\mathcal A}_2$ with ${\mathcal A}'_2$ in the following way:
\begin{align*}
U_2&={\mathcal D}_{2,1}\cup {\mathcal D}_{2,2}\cup {\mathcal D}_{2,3}\\
{\mathcal D}_{2,1} &=
 [x^1_{j_1}(v)-2^{j_1+m_1^*-2}, x^1_{j_1}(v) +2^{j_1+m_1^*-2}]\\
&\quad \times [x^2_{j_1}(v)+2^{j_1+m_1^*}-2^{j_1+m_1^*-2},
           x^2_{j_1}(v)+2^{j_1+m_1^*} +2^{j_1+m_1^*-1}], \\
{\mathcal D}_{2,2} &=
 [ 2^k-2^{j_1+m_1^*+1}, x^1_{j_1}(v)+2^{j_1+m_1^*-2}]\\
&\quad \times [x^2_{j_1}(v)+2^{j_1+m_1^*}, 
     x^2_{j_1}(v)+2^{j_1+m_1^*}+2^{j_1+m_1^*-1}], \\
{\mathcal D}_{2,3} &=
 [2^k-2^{j_1+m_1^*+1},2^k-2^{j_1+m_1^*+1}+2^{j_1+m_1^*-1} ]\\
&\quad\times [x^2_{j_1}(v)+2^{j_1+m_1^*}, x^2_{j_1}(v)+2^{j_1+m_1^*+1}],
\end{align*}
As the length of $U_2$, we take the sum of lengths of $\{ {\mathcal D}_{2,j}\}
{}_{j=1,2,3}$, so it is at most $(\frac{3}{2}+\frac{5}{2}+2)2^{j_1+m_1^*-1}
=6\cdot 2^{j_1+m_1^*-1}$.
We take $U_4$ as a symmetric image of $U_2$ with respect to the line 
$\{ x^2= x^2_{j_1}(v)\} $.
The rectangles ${\mathcal D}_1, {\mathcal D}_3$ are given by
\begin{align*}
{\mathcal D}_1&=[ 2^k-2^{j_1+m^*_1+2}, x_{j_1}^1(v)-2^{j_1+m_1^*}+2^{j_1+m_1^*-2}] \\
&\quad \times [ x^2_{j_1}(v)- 2^{j_1+m_1^*-2}, x^2_{j_1}(v)+ 2^{j_1+m_1^*-2}] \\
{\mathcal D}_3&= [ x_{j_1}^1(v)+2^{j_1+m_1^*}-2^{j_1+m_1^*-2}, 2^k]\\
&\quad \times [ x^2_{j_1}(v)- 2^{j_1+m_1^*-2}, x^2_{j_1}(v)+ 2^{j_1+m_1^*-2}].
\end{align*}
Their lengths are at most $(4+\frac{1}{2})2^{j_1+m_1^*-1}$ and 
$(2+\frac{1}{2})2^{j_1+m_1^*-1}$.

Now we do the same thing as before. 
We use the Connection lemma to connect $\tau_i$ with 
$\partial_{in}T_{j_1}^{m_1^*+1}(v)$ in ${\mathcal D}_i$,
and $\tau^*_{i+1}$ with $\partial_{in}T_{j_1}^{m_1^*+1}(v)$ in $U_{i+1}$,
for $i=1,3$.
Also, we require the existence of a vertical $(+)$-crossing in ${\mathcal A}_1$,
and a horizontal $(-*)$-crossing in each of ${\mathcal A}_2$ and ${\mathcal A}_4$.

Then by the connection lemma, we have
$$
\mu_t^N\bigl( \Delta(v,S_{j_1}^{m_1^*}(v)) \bigr)
\leq {\biggl( \frac{\delta_{48}\delta_{16}}{16}\biggr) }^{-2}
    {\biggl( \frac{\delta_{36}\delta_{16}}{16} \biggr) }^{-1}
    {\biggl( \frac{\delta_{20}}{4}\biggr) }^{-1}
     \mu_t^N\bigl( \Delta (v,T_{j_1}^{m_1^*+1}(v))\bigr) .
$$

\noindent
\underline{Case 2} $T_{j_1}^{m_1^*+1}(v)$ contains the upper right corner
$(2^k,2^k)$ of $S(2^k)$.
In this case, we do not use ${\mathcal A}'_2$, either.
We use ${\mathcal D}_2$ to connect straightly ${\mathcal A}_2$ with the
top side of $T_{j_1}^{m_1^*+1}(v)$ and ${\mathcal D}_3$ as above.
We need corridors $U_1$ and $U_4$ to connect ${\mathcal A}_i$ with 
${\mathcal A}'_i$ for $i=1,4$, whose lengths do not exceed $(\frac{3}{2}+
\frac{5}{2}+4)2^{j_1+m_1^*-1}$.
Then arguing as above, we have
$$
\mu_t^N\bigl( \Delta (v,S_{j_1}^{m_1^*}(v)) \bigr)
\leq {\biggl( \frac{\delta_{64}\delta_{16}}{16}
              \frac{\delta_{20}}{4}
      \biggr) }^{-2}\mu_t^N\bigl(
         \Delta (v,T_{j_1}^{m_1^*+1}(v)) \bigr) .
$$

Finally, we consider the case where $m>m_1^*$.
In this case, also we have to consider whether $T_j^{m+1}$ contains
the upper right corner of $S(2^k)$ or not.
But the essential changes are:
\begin{itemize}
\item We do not need ${\mathcal D}_3$ or ${\mathcal A}_3$.
For, on the event $\Delta (v,T_{j_1}^m(v))$ we have already a $(+)$-path
connecting $v$ with the right boundary of $T_{j_1}^{m+1}(v)$ in
the smaller box $T_{j_1}^m(v)$.
\end{itemize}
As a result, we have
$$
\mu_t^N\bigl( \Delta (v,T_{j_1}^m(v)) \bigr)
\leq \biggl(\frac{\delta_{48}\delta_{16}}{16}\biggr)^{-2}
 {\biggl(\frac{\delta_{36}\delta_{16}}{16}\biggr) }^{-1}
 \mu_t^N\bigl( \Delta (v,T_{j_1}^{m+1}(v))\bigr) 
$$
when $T_{j_1}^{m+1}(v)$ does not contain $(2^k,2^k)$, and
$$
\mu_t^N\bigl(\Delta (v,T_{j_1}^m(v)) \bigr)
\leq {\biggl(\frac{\delta_{64}\delta_{16}}{16}\biggr) }^{-2}
{\biggl(\frac{\delta_{20}}{4}\biggr) }^{-1}
 \mu_t^N\bigl( \Delta (v,T_{j_1}^{m+1}(v))\bigr) 
$$
if $T_{j_1}^{m+1}(v)$ contains $(2^k,2^k)$,
since we use ${\mathcal D}_2$ to connect ${\mathcal A}_2$
straightly with the top side of $S(2^k)$.

Finally, for $m\geq 1$, (\ref{eq:relation of deltas}) implies that
$$
\mu_t^N\bigl( \Delta(v,R_{j_1}^m(v)) \bigr) \geq 
 C_1^{-(m-1)}\mu_t^N\bigl( \Delta (v,R_{j_1}^1(v)) \bigr)
 \geq C_1^{-m}C_2,
$$
where we can take
$$
C_2=C_1\cdot {\bigl( 1 + e^{(4h_c + 8)/(\mathfrak{K}T)} \bigr) }^{-\# S_{j_1}^1(v)}.
$$
\end{proof}

\begin{Remark}\label{box-size-bound3}
Since we use Connection lemma in the proof,
and by Remarks \ref{box-size-bound} and \ref{box-size-bound2},
the restriction that $2^k<L(h,\varepsilon_0)$ can be relaxed to 
$2^k<2^6L(h,\varepsilon_0)$.
\end{Remark}

\begin{Lemma}[cf. (2.38) in \cite{K87scaling}] \label{lambda-to-delta}
Let $2^{j_1}<2^k<\min\{ 2^5L(h,\varepsilon_0),\frac{N}{2}\} $, and $v\in S(2^k)$.
Then there exists a constant $C_3(\eta )>0$ depending only on $\varepsilon_0$, and $\eta $,
such that the following statements
hold.

\noindent
(i) For $1\leq m\leq k-j_1$, we have
$$ 
\mu_t^N\bigl( \Lambda (v,S_{j_1}^m(v))\bigr) \leq C_3(\eta )
                    \mu_t^N\bigl( \Delta (v,S_{j_1}^{m+1}(v))\bigr) .
$$
(ii) For $m_1^*+1\leq m \leq k$, we have
$$
\mu_t^N\bigl( \Lambda (v,T_{j_1}^m(v))\bigr)
     \leq C_3(\eta )\mu_t^N\bigl( \Delta (v,T_{j_1}^{m+1}(v))\bigr) .
$$
\end{Lemma}

\begin{proof}
The proof of this lemma is the same as in \cite{K87scaling}. We use corridors
to connect each of $r_1,r_2^*,r_3$ and $r_4^*$ to the boundary of 
$S_{j_1}^{m+1}(v)$ or $T_{j_1}^{m+1}(v)$ by using the connection lemma.

%------------------------------------------------------------------------
%%%%%%%%%%%%%%%%%%%%%% The following part can be erased after we have checked
%%%%%%%%%%%%%%%%%%%%%% all the details

\noindent (i) Assume that $\Lambda (v,S_{j_1}^m(v))$ occurs. Then let $r_1,r_3$
be  $(+)$-paths connecting $v$ with $\partial_{in} S_{j_1}^m(v)$, and let 
$r_2^*,r_4^*$ be $(-*)$-paths connecting $*$-neighbor of $v$ with 
$\partial_{in} S_{j_1}^m(v)$, such that $r_1\setminus \{ v\} $ and 
$r_3\setminus \{ v\} $ are disjoint, and $r_1\cup r_3$ separates
$r_2^*$ from $r_4^*$.
$r_1$ crosses one of ${\mathcal B}_i$'s, say ${\mathcal B}_1$ for
simplicity. Then let ${\mathcal C_1}$ be the $(+)$-connected component
in ${\mathcal B}_1$, containing the endpoint of $r_1$ in 
$\partial_{in}S_{j_1}^m(v)$. 
We assume that $r_1$ contains the left endpoint $a(1)=(a^1(1),a^2(1))$ 
of the lowest crossing
of ${\mathcal B}_1$ in ${\mathcal C}_1$.
Corresponding to this $r_1$, we prepare a corridor of width 
$2\eta 2^{j_1+m}$ connecting $\partial S_{j_1}^m(v)$ with
the left side of $S_{j_1}^{m+1}(v)$ such that it contains the rectangle
$$ 
[a^1(1)-\sqrt{\eta }2^{j_1+m},a^1(1)-1]\times 
 [a^2(1)-\eta 2^{j_1+m}, a^2(1)+\eta 2^{j_1+m}].
$$  
The length of the corridor can be made less than 
$8\cdot 2^{j_1+m+1} =2^{j_1+m+4}$
Further, this corridor crosses ${\mathcal A}_1'$ horizontally.
Here, as before $\mathcal{A}_i$ corresponds to $S_{j_1}^m(v)$ and
$\mathcal{A}'_i$ corresponds to $S_{j_1}^{m+1}(v)$.
In the same way, the corridor corresponding to $r_2^*$ contains the
rectangle of longer side length $\sqrt{\eta }2^{j_1+m}$ and the shorter 
side length $2\eta 2^{j_1+m}$, one of the shorter side of which 
is neighbor to $\partial_{in}S_{j_1}^m(v)$, and this rectangle is located
outside of $S_{j_1}^m(v)$. Further this corridor connects the endpoint of
$r_2^*$ with the top side of $\partial_{in}S_{j_1}^{m+1}(v)$ crossing
${\mathcal A}_2'$ vertically.
In this way we can prepare corridors corresponding to $r_1,r_2^*,r_3,r_4^*$.
Since $\Lambda (v,S_{j_1}^m(v))$ occurs, these corridors can be chosen
to be pairwise disjoint, if $\eta $ is sufficiently small.
By (\ref{eq:bound for eta}),
and $M_0\geq 1$, this is possible.

Since ${\mathcal C}_1$ has an $(\eta , j_1+m)$-fence, ${\mathcal C}_1$
is connected by a $(+)$-path with the lower side of the rectangle
$$ 
[a^1(1)-\sqrt{\eta }2^{j_1+m},a^1(1)-1]\times 
 [a^2(1)-\eta 2^{j_1+m}, a^2(1)+\eta 2^{j_1+m}],
$$
going above $a(1)$ and in the region 
$S(a(1), \sqrt{\eta }2^{j_1+m})\setminus S(a(1),\eta 2^{j_1+m})$.
Conditioning on the lowest of such $(+)$-paths, and using the
Connection lemma, we see that there is a constant $C_3^*>0$ depending on 
$ \varepsilon_0$ and $\eta $, such that
the conditional probability of the event that there is a $(+)$-path in this
corridor connecting the above lowest $(+)$-path with the left side of
$S_{j_1}^{m+1}(v)$ is not less than $C_3^*$.
We do the same thing for each corridors, except that we have to consider
$(-*)$-paths in the corridors  corresponding to $r_2^*$ and $r_4^*$.
Note that we can choose a common constant for the above constant $C_3^*$ 
for each of corridors.
Thus, we have 
$$
\mu_t^N\left( \Lambda (v, S_{j_1}^m(v))\right)
 \leq (C_3^*)^{-4} {\biggl( \frac{\delta_{16}}{4}\biggr) }^{-4}\mu_t^N\left(
       \Delta (v,S_{j_1}^{m+1}(v))\right) .
$$
Note that $C_3^*$ can be taken to satisfy
\begin{equation}\label{eq:the constant C3(eta)}
C_3^*\geq \frac{\delta_{8p}\delta_{16}}{16}
\end{equation}
with an integer $p \geq 2^3\eta^{-1}$.

\noindent 
(ii) The argument is the same as above, but we have to take care of 
two cases.
In $\Lambda (v, T_{j_1}^m(v))$, each of the paths $r_1, r_2^*, r_4^*$ ends
one of three sides of $T_{j_1}^m(v)$, namely the left, top or
bottom side. 
Only $r_3$ can end at the right side.

If  $r_3$ does not end at the right side
of $T_{j_1}^m(v)$, then by definition of $\Lambda (v,T_{j_1}^m(v))$
each of the $(+)$-clusters ${\mathcal C}_i$'s ($(-*)$-clusters
${\mathcal C}_{i+1}^*$'s)
containing the endpoints 
$a(i)$ of $r_i$ ($a^*(i+1)$ of $r_{i+1}^*$) for $i=1,3$, has an 
$(\eta , j_1+m)$-fence.
Then we choose corridors corresponding to $r_1,r_2^*$ and $r_4^*$
as before.
We choose the corridor corresponding to $r_3$ in 
$$
T_{j_1}^{m+1}(v)\setminus ({\mathcal B}_1\cup {\mathcal B}_2
\cup {\mathcal B}_4) 
$$
so that it connects $S(a(3),\sqrt{\eta }2^{j_1+m})$ with the
right side of $T_{j_1}^{m+1}(v)$.

If $r_3$ ends at the right side of $T_{j_1}^m(v)$, then
we only take care of corridors corresponding to $r_1,r_2^*$ and $r_4^*$.

Thus, we can choose $C_3(\eta )$ in the statement of the lemma
as the above $(C_3^*)^{-4}(\delta_{16}/4)^{-4}$.

Finally, the bound that $2^k<2^5L(h,\varepsilon_0)$ comes from 
the fence argument, Lemma \ref{lem:fence}.
\end{proof}

\begin{Lemma}[cf. Lemma 4, (2.37) in \cite{K87scaling}] \label{comparison of gamma and delta}
Besides the condition (\ref{eq:bound for j_1}), we assume further that
\begin{equation}\label{eq:cond-for-j1-2}
C(4n)^3e^{-\alpha n}\leq \delta \quad \mbox {for every } n\geq 2^{j_1}. 
\end{equation}
Then there exists a constant $K>0$ depending only on $\varepsilon_0, \eta $,
and $j_1$, 
such that for $2^{j_1}<2^k<\min\{2^5 L(h,\varepsilon_0), \frac{N}{2}\} ,$
the following statements hold for $t\in [0,1]$.

\noindent
(i) For $3\leq m \leq k-j_1 $, we have
$$
\mu_t^N\bigl( \Gamma (v,S_{j_1}^m(v))\bigr) 
  \leq K \mu_t^N\bigl( \Delta (v,S_{j_1}^m(v))\bigr) .
$$

\noindent (ii)
$\displaystyle
\mu_t^N\bigl( \Gamma (v,T_{j_1}^{m_1^*+1}(v))\bigr)
  \leq K \mu_t^N\bigl( \Delta (v,T_{j_1}^{m^*_1+1}(v))\bigr) ,
$
if $m^*+1\geq 3$.

\noindent
(iii) For $\max\{ 3, m^*_1+1\} \leq m \leq k-j_1$, we have
$$
\mu_t^N\bigl( \Gamma (v,T_{j_1}^m(v))\bigr) \leq 
   K \mu_t^N\bigl( \Delta (v,T_{j_1}^m(v))\bigr) .
$$
\end{Lemma}

Hereafter we always assume that $j_1$ satisfies (\ref{eq:bound for j_1}) and
(\ref{eq:cond-for-j1-2}).
\begin{proof}
(i)  This part is just the same as the proof of Lemma 4 of \cite{K87scaling}.
We start with the following inequality. 
\begin{align*}
\mu_t^N\bigl( \Gamma (v,S_{j_1}^m(v))\bigr)
&\leq \mu_t^N\bigl( \Gamma (v,S_{j_1}^{m-1}(v))\bigr)\\  
&\leq  \mu_t^N\bigl( \Lambda (v,S_{j_1}^{m-1}(v))\bigr) +
 \mu_t^N\bigl( \Gamma (v,S_{j_1}^{m-1}(v)) \setminus  
   \Lambda (v,S_{j_1}^{m-1}(v))\bigr).
\end{align*}
If $\Gamma (v,S_{j_1}^{m-1}(v))$ occurs but $ \Lambda (v,S_{j_1}^{m-1}(v))$
does not occur, then we can see that 
\begin{itemize}
\item $\Gamma (v,S_{j_1}^{m-2}(v)) $ occurs, and
\item for at least one of rectangles ${\mathcal B}_i, i=1,2,3,4$
     which correspond to $S_{j_1}^{m-1}(v)$,
    there is a $(+)$-crossing cluster or $(-*)$-crossing cluster,
    connecting longer sides of the rectangle, such that this
    $(+)$-cluster (or $(-*)$-cluster) does not have an $(\eta, j+m-1)$-fence.
\end{itemize}
Since $S_{j_1}^{m-2}(v)$ and ${\mathcal B}_1\cup \cdots \cup {\mathcal B}_4$
are of $\ell^\infty$-distance $2^{j_1+m-3}$, we have
\begin{align*}
&\mu_t^N \bigl(  \Gamma (v,S_{j_1}^{m-1}(v)) \setminus  
   \Lambda (v,S_{j_1}^{m-1}(v))\bigr) \\
& \leq  ( 8\delta + C2^{3(m+j_1-1)}e^{-\alpha 2^{m+j_1-3}})
     \mu_t^N \bigl( \Gamma (v, S_{j_1}^{m-2}(v))\bigr) .
\end{align*}
By (\ref{eq:cond-for-j1-2}) we have
\begin{align*}
&\mu_t^N\bigl( \Gamma (v,S_{j_1}^{m-1}(v))\bigr)  \\
&\leq  \mu_t^N\bigl( \Lambda (v,S_{j_1}^{m-1}(v))\bigr) +
 9\delta \mu_t^N\bigl( \Gamma (v,S_{j_1}^{m-2}(v))\bigr) .
\end{align*}
Iterating this until we get to $S_{j_1}^1(v)$, we have
\begin{align*}
& \mu_t^N\bigl( \Gamma (v,S_{j_1}^{m-1}(v))\bigr) \\
&\leq  \sum_{\ell =0}^{m-3}(9\delta )^\ell \mu_t^N\bigl( 
            \Lambda (v,S_{j_1}^{m-1-\ell }(v)) \bigr) 
            +(9\delta )^{m-2}\mu_t^N\bigl( 
                    \Gamma (v,S^1_{j_1}(v)) \bigr).
\end{align*}
By Lemma \ref{lambda-to-delta}, 
$$
\mu_t^N\bigl( \Lambda (v,S_{j_1}^{m-1-\ell }(v))\bigr)
\leq C_3(\eta )\mu_t^N\bigl( \Delta (v,S_{j_1}^{m-\ell }(v)) \bigr),
$$
and since by Lemma \ref{relation of deltas}, 
$$
\mu_t^N\bigl(  \Gamma (v,S^1_{j_1}(v)) \bigr)
\leq 1\leq C_2^{-1}C_1^m \mu_t^N\bigl( \Delta (v,S_{j_1}^m(v)) \bigr) ,
$$
and 
$$
\mu_t^N\big( \Delta (v,S_{j_1}^{m-\ell}(v))\bigr) \leq
C_1^\ell \mu_t^N\bigl( \Delta (v,S_{j_1}^m(v)) \bigr) ,
$$
we obtain 
$$
\mu_t^N\bigl( \Gamma (v,S_{j_1}^m(v))\bigr) 
  \leq K_1 \mu_t^N\bigl( \Delta (v,S_{j_1}^m(v))\bigr)
$$
for some positive constant $K_1$ which depends only on $\varepsilon_0, j_1$ and $\eta $.

\noindent
(ii) By the above argument we have
$$
\mu_t^N\bigl( \Gamma (v,S_{j_1}^{m_1^*}(v))\bigr)
  \leq K_1 \mu_t^N\bigl( \Delta (v,S_{j_1}^{m_1^*}(v)) \bigr).
$$
and by Lemma \ref{relation of deltas}, we have
\begin{align*}
\mu_t^N\bigl( \Gamma (v, T_{j_1}^{m_1^*+1}(v)) \bigr) 
&\leq  \mu_t^N\bigl( \Gamma (v, S_{j_1}^{m_1^*}(v)) \bigr) \\
&\leq  K_1 \mu_t^N\bigl( \Delta (v,S_{j_1}^{m_1^*}(v)) \bigr) \\
&\leq  K_1 C_1 \mu_t^N\bigl( \Delta \bigl( v,T_{j_1}^{m_1^*+1}(v))\bigr) .
\end{align*}

\noindent
(iii) Since we proved the inequality for $m=m_1^*+1$, we can assume 
that $m>m_1^*+1$.
The same argument as in the proof of (i) shows that
\begin{align*}
\mu_t^N\bigl( \Gamma (v,T_{j_1}^{m}(v))\bigr)& \leq 
\sum_{\ell =0}^{m-m_1^*-2}(9\delta )^\ell 
         \mu_t^N\bigl( \Lambda (v,T_{j_1}^{m-\ell }(v))\bigr) \\
  & \quad + (9\delta )^{m-m_1^*-1}\mu_t^N\bigl( 
              \Gamma (v,T_{j_1}^{m_1^*+1}(v)) \bigr)  \\
  &\leq  C_3(\eta ) \sum_{\ell =0}^{m-m_1^*-2}(9\delta C_1 )^\ell
        \mu_t^N\bigl( \Delta (v,T_{j_1}^m(v)) \bigr) \\
  & \quad +(9\delta )^{m-m_1^*-1}K_1C_1\mu_t^N\bigl( \Delta (v,T_{j_1}^{m_1^*+1}(v)) 
                     \bigr).
\end{align*} 
Here, we used the result in (ii).     
By this and Lemma \ref{relation of deltas}, we have desired inequality.
\end{proof}

\subsection{Block events: outwards}
We start with the simplest case where $0\leq m<m_1^*$. 
In this case, $S_{j_1}^{m+1}(v)$ is a subset of $S(2^k)$.

Let ${\tilde \Gamma }(S_{j_1}^m(v),S(2^k))$ be the event such that all the following occur.
\begin{enumerate}
\item There exist $(+)$-paths $r_1, r_3$ such that $r_1$ connects 
$\partial S_{j_1}^m(v)$ with the left side of $S(2^k)$, and 
$r_3$ connects 
$\partial S_{j_1}^m(v)$ with the right side of $S(2^k)$.
\item There exist $(-*)$-paths $r_2^*, r_4^*$ such that $r_2^*$ connects 
$\partial S_{j_1}^m(v)$ with the top side of $S(2^k)$, and 
$r_4^*$ connects 
$\partial S_{j_1}^m(v)$ with the bottom side of $S(2^k)$.
\end{enumerate}
Next, in order to define fences, we introduce ${\tilde {\mathcal B}}_i,
1\leq i\leq 4$. 
For $a=(a^1,a^2)$ and $r>0$, let
$$
S_1(a,r)=S(a,r)\cap \{ x^1\leq a^1\} ,
$$
and let $S_{i+1}(a,r)$ be the clockwise rotation of $S_i(a,r)$ by a right angle
around $a$ for $i=1,2,3$. 
Then we put
\begin{equation}\label{eq:bi-tilde}
{\tilde {\mathcal B}}_i:= S_i(x_{j_1}(v), 2^{j_1+m}+2^{j_1+m-2})
                    \setminus S_{j_1}^m(v)
\end{equation}
for $i=1,2,3,4$.
Note that ${\tilde {\mathcal B}}_i\subset S(2^k)$ if $m<m_1^*$.

\medskip

\noindent
{\bf Crossing clusters:}
Let 
\begin{equation}\label{fence,phi-i-1}
\varphi_i:= \partial_{in}S_i(x_{j_1}(v),2^{j_1+m}+2^{j_1+m-2})
              \setminus S_{j_1}^m(v),
\end{equation}
and $\xi_1, \xi_2,\xi_3,\xi_4$ be the left, top, right, bottom sides 
of $\partial S_{j_1}^m(v)$, respectively.
A crossing $(+)$-cluster [$(-*)$-cluster] in ${\tilde {\mathcal B}}_i$ is a
$(+)$-cluster [$(-*)$-cluster] in  ${\tilde {\mathcal B}}_i$ such that
it connects $\varphi_i$ with $\xi_i$.
Then we can define crossing $(+)$-clusters and crossing $(-*)$-clusters
in terms of $\varphi_i$ and $\xi_i$.
Let ${\mathcal C}$ be a crossing $(+)$-cluster of ${\tilde {\mathcal B}}_i$.
We define its endpoint $a({\mathcal C})$ by the lowest point of 
${\mathcal C}\cap \xi_i $ if $i=1,3$ and  the leftmost point of
${\mathcal C}\cap \xi_i$ if $i=2,4$.
(This definition corresponds to our assumption that $0\leq v^2\leq v^1$.)
The endpoint $a({\mathcal C}^*)$ of a $(-*)$-cluster ${\mathcal C}^*$ in 
${\tilde {\mathcal B}}_i$
is defined in the same way.
\medskip

\noindent
{\bf Fences:}
Let ${\mathcal C}_1$ be a crossing $(+)$-cluster in ${\tilde {\mathcal B}}_1$.
We say that ${\mathcal C}_1$ has an $(\eta ,j_1+m)$-fence if all the following
events occur:
\begin{enumerate}
\item $|a({\mathcal C}_1)-a({\mathcal C})|>2\sqrt{\eta }2^{j_1+m}$\quad 
  for every crossing $(+)$-cluster ${\mathcal C}$ of   
 ${\tilde {\mathcal B}}_1$, such that 
 ${\mathcal C}_1\cap {\mathcal C}=\emptyset $.
\item $ |a({\mathcal C}_1)-a({\mathcal C^*})|>2\sqrt{\eta }2^{j_1+m}$\quad 
 for every crossing $(-*)$-cluster ${\mathcal C^*}$ of 
  ${\tilde {\mathcal B}}_1$,
\item Put $a=a({\mathcal C}_1)=(a^1,a^2)$, and \\ 
  $$T_a:=[a^1+1,a^1+\sqrt{\eta }2^{j_1+m}]\times 
   [a^2-\eta2^{j_1+m},a^2+2^{j_1+m}].$$
  Then there exists a vertical $(+)$-crossing of $T_a\cap S_{j_1}^m(v)$ 
  which is connected by a $(+)$-path with ${\mathcal C}_1$ in 
  $ S(a,\sqrt{\eta }2^{j_1+m})\setminus S(a, \eta 2^{j_1+m})$.
\end{enumerate}
The definition of fences for a crossing $(+)$-cluster ${\mathcal C}_i$
[$(-*)$-cluster ${\mathcal C}^*_{i}$] for $1\leq i \leq 4$ will be obvious.
Now we define ${\tilde \Lambda } (S_{j_1}^m(v), S(2^k))$ as the subset of
${\tilde \Gamma }(S_{j_1}^m(v), S(2^k))$ such that any of crossing
$(+)$-clusters and $(-*)$ clusters in ${\tilde {\mathcal B}}_i$
of $S_{j_1}^m(v)$ has an $(\eta ,j_1+m)$-fence, for every $1\leq i \leq 4$.
Then let
$$
{\tilde {\mathcal A}}_1:= x_{j_1}(v)+ [-2^{j_1+m}-2^{j_1+m-2},-2^{j_1+m}]
 \times [-2^{j_1+m-2},2^{j_1+m-2}],
$$
and let ${\tilde {\mathcal A}}_{i+1}$ be the clock-wise rotation of 
${\tilde {\mathcal A}}_{i}$ by a right angle around $x_{j_1}(v)$ for $i=1,2,3$.
Finally, we define ${\tilde \Delta }( S_{j_1}^m(v), S(2^k))$ as a subset
of ${\tilde \Gamma }(S_{j_1}^m(v), S(2^k))$ such that
\begin{enumerate}
\item 
$\displaystyle r_i\cap (\cup_{1\leq j\leq 4} {\tilde {\mathcal B}}_j)
  \subset {\tilde {\mathcal A}}_i$
and 
$\displaystyle r^*_{i+1}\cap(\cup_{1\leq j\leq 4} {\tilde {\mathcal B}}_{j}) 
  \subset {\tilde {\mathcal A}}_{i+1}$
for $i=1,3$, and
\item for $i=1,3$ there is a vertical $(+)$-crossing in 
${\tilde {\mathcal A}}_i$, and a horizontal $(-*)$-crossing 
in ${\tilde {\mathcal A}}_{i+1}$.
\end{enumerate}

Next, we consider the case where $m=m_1^*$.
In this case, ${\tilde {\mathcal B}}_2$ and ${\tilde {\mathcal B}}_3$
may not be inside $S(2^k)$.

If $S_{j_1}^{m^*_1}(v) \subset 
   \bigl\{ \max \{ x^1, x^2\}\leq 2^k-2^{j_1+m_1^*-2 }
 \bigr\} $, then we can use the events 
${\tilde \Gamma }(S_{j_1}^{m_1^*}(v), S(2^k))$,
${\tilde \Lambda} (S_{j_1}^{m_1^*}(v), S(2^k))$ and
${\tilde \Delta }(S_{j_1}^{m_1^*}(v), S(2^k))$ as above. 

If $d(S_{j_1}^{m^*_1}(v),S(2^k)^c)<2^{j_1+m_1^*-2}$ and 
$S_{j_1}^{m_1^*}(v)\subset \{ x^2\leq 2^k-2^{j_1+m_1^*-2}\} $, then we put
$$
{\tilde S}_{j_1}^{m_1^*}(v)=[x_{j_1}^1(v)-2^{j_1+m_1^*},2^k]
   \times [ x_{j_1}^2(v)-2^{j_1+m_1^*}, x_{j_1}^2(v)+2^{j_1+m_1^*}],
$$
and let
$
{\tilde \Gamma }({\tilde S}_{j_1}^{m_1^*}(v), S(2^k))$ be defined in the 
same way as 
${\tilde \Gamma }(S_{j_1}^{m_1^*}(v), S(2^k))$, 
except that we do not require the  existence of
$r_3$.
Correspondingly we do not use ${\tilde {\mathcal B}}_3$ or
${\tilde {\mathcal A}}_3$ for ${\tilde S}_{j_1}^{m_1^*}(v)$, either.
Then we define 
\begin{equation}\label{eq:bi-tilde2}
{\tilde {\mathcal B}}_i:=
S(2^k)\cap \bigl[ 
S_i(x_{j_1}(v), 2^{j_1+m_1^*}+2^{j_1+m_1^*-2})
                    \setminus {\tilde S}_{j_1}^{m_1^*}(v)
         \bigr] ,
\end{equation}
and 
\begin{equation}\label{eq:phii-tilde2}
\varphi_i:= S(2^k)\cap \bigl[
  \partial_{in}S_i(x_{j_1}(v),2^{j_1+m_1^*}+2^{j_1+m_1^*-2})
              \setminus {\tilde S}_{j_1}^{m_1^*}(v) \bigr]
\end{equation}
for $i=1,2,4$.
Let $\xi_i $ denote the left, top, and bottom sides of 
$\partial {\tilde S}_{j_1}^{m_1^*}$,for $i=1,2,4$, respectively.
We define ${\tilde {\mathcal A}}_i$ as before for $i=1,2,4$.
We have to add a condition to the definition of fences in ${\mathcal B}_i$.
Namely, a crossing $(+)$-cluster ${\mathcal C}$ in ${\mathcal B}_i$
has an $(\eta ,j_1+m_1^*)$-fence if all the following occurs,
\begin{enumerate}
\item $|a({\mathcal C})-a({\mathcal C}_1)| >2\sqrt{\eta }2^{j_1+m_1^*}$
for every crossing $(+)$-cluster ${\mathcal C}_1$ in ${\mathcal B}_i$,
such that ${\mathcal C}\cap {\mathcal C}_1=\emptyset$.
\item $|a({\mathcal C})-a({\mathcal C}^*)| >2\sqrt{\eta }2^{j_1+m_1^*}$
for every crossing $(-*)$-cluster ${\mathcal C}^*$ in ${\mathcal B}_i$.
\item There exists a $(+)$-path which crosses shorter direction of 
$T_{a({\mathcal C})}$, and is connected by a $(+)$-path with 
${\mathcal C}$ in 
$$
S(2^k)\cap \bigl[ S(a({\mathcal C}),\sqrt{\eta }2^{j_1+m_1^*})
  \setminus S(a({\mathcal C}),\eta 2^{j_1+m_1^*}) \bigr] .
$$
\end{enumerate}
Fences for crossing $(-*)$-clusters in  ${\mathcal B}_i$
are defined in the same way.

Then ${\tilde \Lambda }({\tilde S}_{j_1}^{m_1^*}(v), S(2^k))$ is defined as a 
subset of
${\tilde \Gamma }({\tilde S}_{j_1}^{m_1^*}(v), S(2^k))$ such that every
crossing $(+)$-cluster and crossing $(-*)$-cluster of 
${\tilde {\mathcal B}}_i$, $(i=1,2,4)$ has an $(\eta ,j_1+m_1^*)$-fence.
The set ${\tilde \Delta }({\tilde S}_{j_1}^{m_1^*}(v), S(2^k))$ is then 
defined as a
subset of  ${\tilde \Gamma }({\tilde S}_{j_1}^m(v), S(2^k))$ such that
\begin{enumerate}
\item 
$r_1\cap (\cup_{j=1,2,4}{\tilde {\mathcal B}}_j)\subset {\tilde {\mathcal A}}_1$,
\item $r^*_{i+1}\cap(\cup_{j=1,2,4}{\tilde {\mathcal B}}_j)\subset {\tilde {\mathcal A}}_{i+1}$
for $i=1,3$, and 
\item  there exists a vertical $(+)$-crossing in ${\tilde {\mathcal A}}_1$,
and there exists a horizontal $(-*)$-crossing in ${\tilde {\mathcal A}}_{i+1}$
for $i=1,3$.
\end{enumerate}

If $d(S_{j_1}^{m_1^*}(v),S(2^k)^c)<2^{j_1+m_1^*-2}$ and 
$S_{j_1}^{m_1^*}(v)\not\subset \{ x^2\leq 2^k-2^{j_1+m_1^*-2}\} $, then we put
$$
{\tilde S}_{j_1}^{m_1^*}(v)=[x_{j_1}^1(v)-2^{j_1+m_1^*},2^k]
   \times [ x_{j_1}^2(v)-2^{j_1+m_1^*}, 2^k].
$$
and we define 
$
{\tilde \Gamma }({\tilde S}_{j_1}^{m_1^*}(v), S(2^k))$ as the event such that
\begin{enumerate}
\item there exists a $(+)$-path $r_1$ connecting 
$\partial {\tilde S}_{j_1}^{m_1^*}(v)$ with the left side of $S(2^k)$, and
\item there exists a $(-*)$-path $r^*_4$ connecting
$\partial {\tilde S}_{j_1}^{m_1^*}(v)$ with the bottom side of
$S(2^k)$.
\end{enumerate}
${\tilde {\mathcal B}}_i$ and $\varphi_i$ are given by 
(\ref{eq:bi-tilde2}) and (\ref{eq:phii-tilde2}), and $\xi_1,\xi_4$
are given as the left and the bottom sides of ${\tilde S}_{j_1}^{m_1^*}(v)$,
respectively.
${\mathcal A}_i$ are the same as before for $i=1,4$.

Then we define ${\tilde \Lambda }( {\tilde S}_{j_1}^{m_1^*}(v), S(2^k))$
and ${\tilde \Delta}( {\tilde S}_{j_1}^m(v), S(2^k))$
in terms of $r_1,r^*_4$ and  
${\tilde {\mathcal B}}_i, {\tilde {\mathcal A}}_i $
for $i=1,4$, as before.

Finally, consider the case where $m>m_1^*$.
If $T_{j_1}^m(v) \subset \{ x^2\leq 2^k-2^{j_1+m-2}\} $, then 
let ${\tilde \Gamma }(T_{j_1}^{m}(v), S(2^k))$ be the event such that
\begin{enumerate}
\item there exists a $(+)$-path $r_1$ in $S(2^k)$ connecting 
$\partial T_{j_1}^m(v)$ with the left side of $S(2^k)$, and
\item there exist $(-*)$-paths $r_2^*$ and $r_4^*$ in $S(2^k)$ 
such that $r_2^*$
connects $\partial T_{j_1}^m(v)$ with the top side of $S(2^k)$, and 
$r_4^*$ connects $\partial T_{j_1}^m(v)$ with the bottom side of $S(2^k)$.
\end{enumerate}
To define ${\tilde {\mathcal B}}_i$, recall that $t=(t^1,t^2)$ denotes the center
of $T_{j_1}^m(v)$, i.e.,
$$
t^1= 2^k-2^{j_1+m}, \quad \mbox{ and } \quad t^2=x_{j_1}^2(v),  
$$
in this case.
Then $T_{j_1}^m(v)=S(t,2^{j_1+m})$, and we put
\begin{equation}\label{eq:bi-tilde3}
{\tilde {\mathcal B}}_i:=
S(2^k)\cap \bigl[ 
S_i(t, 2^{j_1+m}+2^{j_1+m-2})
                    \setminus T_{j_1}^{m}(v)
         \bigr] ,
\end{equation}
and
\begin{equation}\label{eq:phii-tilde3}
\varphi_i:= S(2^k)\cap \bigl[
  \partial_{in}S_i(t ,2^{j_1+m}+2^{j_1+m-2})
              \setminus T_{j_1}^{m}(v)
  \bigr]
\end{equation}
for $i=1,2,4$.
We call $\xi_1$ the left side of $T_{j_1}^m$,
$\xi_2$ the top side of  $T_{j_1}^m$, and
$\xi_4$ the bottom side of  $T_{j_1}^m$.
We define ${\tilde {\mathcal A}}_i$ for $T_{j_1}^m(v)$ as the $t$-shift of 
${\tilde {\mathcal A}}_i$ for $S(2^{j_1+m})$ for $i=1,2,4$.
Then, correspondingly we can define ${\tilde \Lambda }(T_{j_1}^m(v),S(2^k))$
and ${\tilde \Delta }(T_{j_1}^m(v),S(2^k))$.

If $T_{j_1}^m(v)\not\ni (2^k,2^k)$ and 
$T_{j_1}^m(v)\not\subset \{ x^2 \leq 2^k-2^{j_1+m-2}\} $, then we
define
$${\tilde T}_{j_1}^m(v)= [2^k-2^{j_1+m+1},2^k]\times 
          [t^2 -2^{j_1+m},2^k]
$$
and define ${\tilde \Gamma }({\tilde T}_{j_1}^{m}(v), S(2^k))$ as the
event such that 
\begin{enumerate}
\item  there exists a $(+)$-path $r_1$ in $S(2^k)$ connecting 
$\partial {\tilde T}_{j_1}^m(v)$ with the left side of $S(2^k)$, and
\item  there exists a $(-*)$-path $r_4^*$ in $S(2^k)$ connecting 
$\partial {\tilde T}_{j_1}^m(v)$ with the bottom side of $S(2^k)$.
\end{enumerate}
In this case and the next case, we use ${\tilde {\mathcal B}}_i$, $\varphi_i$,
$\xi_i$ and ${\tilde {\mathcal A}}_i$ for $i=1,4$.
Then definitions of ${\tilde {\mathcal B}}_1$, $\varphi_1$, and $\xi_1$
are modified as follows:
\begin{align*}
{\tilde {\mathcal B}}_1&= \bigl[ 
 S_1(t, 2^{j_1+m}+2^{j_1+m-2})\setminus {\tilde T}_{j_1}^m(v)
  \bigr] \cap S(2^k),\\
\varphi_1&=\bigl[\partial_{in}S_1(t, 2^{j_1+m}+2^{j_1+m-2})\setminus {\tilde T}_{j_1}^m(v)
   \bigr] \cap S(2^k),
\end{align*}
where $t$ is the center of $T_{j_1}^m(v)$, and $\xi_1$ denotes the
left side of ${\tilde T}_{j_1}^m(v)$.
Definitions of ${\tilde {\mathcal B}}_4$, $\varphi_4$,
$\xi_4$ and ${\tilde {\mathcal A}}_i, i=1,4 $
are the same as in the previous case.
Then, correspondingly we can define 
${\tilde \Lambda }({\tilde T}_{j_1}^m(v),S(2^k))$
and ${\tilde \Delta }({\tilde T}_{j_1}^m(v),S(2^k))$.

If $T_{j_1}^m(v)\ni (2^k,2^k)$, then we define 
${\tilde \Gamma }( T_{j_1}^{m}(v), S(2^k))$, 
${\tilde \Lambda }( T_{j_1}^{m}(v), S(2^k))$ and
${\tilde \Delta }( T_{j_1}^{m}(v), S(2^k))$,
by using only $r_1$ and $r_4^*$
connecting $\partial T_{j_1}^m(v)$ with the left and bottom sides
of $S(2^k)$, respectively.
${\tilde {\mathcal B}}_i$ are defined by (\ref{eq:bi-tilde3}),
and $\varphi_i$ are defined by (\ref{eq:phii-tilde3}), both for $i=1,4$.
The definitions of ${\tilde {\mathcal A}}_i$ and $\xi_i $ are the same as
in the case where $(T_{j_1}^m(v)\subset  \{ x^2 \leq 2^k-2^{j_1+m-2} \}$. 

With these modifications, as in section 4 we obtain
$$
\mu \left( \begin{array}{@{\,}c@{\,}}
\mbox{there exists a crossing $(+)$-cluster in }{\tilde {\mathcal B}}_1\\
  \mbox{which does not have an $(\eta ,j_1+m)$-fence}
 \end{array} \right) \leq \delta ,
$$  
for $\mu = \mu_{T,h}$ or $\mu_t^N$, if $n_4<2^{j_1+m}\leq 2^k<\frac{N}{2}$.
But we have to choose $\rho, \varepsilon $ to satisfy
\begin{equation}\label{condition for rho and epsilon-2}
{\left( 1-\frac{\delta_{176}}{4}\right) }^\rho +4\rho\varepsilon \leq \delta
\end{equation}
instead of (\ref{condition for rho and epsilon}), 
and $\eta $ to
satisfy (\ref{condition for eta}) and (\ref{eq:bound for eta}).
Anyway, the statement of Lemma \ref{lem:fence} 
is correct in this case, too.

Let ${\tilde R}_{j_1}^m(v)$ denote one of $S_{j_1}^m(v)$, 
${\tilde S}_{j_1}^m(v)$,
$T_{j_1}^m(v)$ and ${\tilde T}_{j_1}^m(v)$ corresponding to each cases
discussed above. 

\begin{Lemma}\label{lem: relation of deltas-tilde}
Let $2^{j_1}<2^k<\min\{ 2^6L(h,\varepsilon_0), \frac{N}{2}\} $. 
Then for $t\in [0,1]$ and every $1\leq m\leq k-j_1$, we have
\begin{align}
&\mu_t^N\bigl( {\tilde \Delta }({\tilde R}_{j_1}^m(v),S(2^k))\bigr) \leq
 C_1 \mu_t^N\bigl(  {\tilde \Delta }({\tilde R}_{j_1}^{m-1}(v),S(2^k))\bigr) ,
\ \mbox{ and }\label{eq:relation of tilde-deltas}\\
& \mu_t^N\bigl( {\tilde \Delta }({\tilde R}_{j_1}^m(v),S(2^k))\bigr)
  \geq C_1^{-(k-j_1-m)}, \label{eq:bound for delta-tilde}
\end{align}
where $C_1$ is the constant given by (\ref{eq:the constant C1}).
\end{Lemma}
The proof of this lemma goes parallel to that of Lemma \ref{relation of deltas}.
This time the bound $(\delta_{68}\delta_{16})/16$ appears
when $T_{j_1}^m(v)\ni (2^k,2^k)$.
As for (\ref{eq:bound for delta-tilde}), note that 
${\tilde R}_{j_1}^{k-j_1}(v)=S(2^k)$, and we understand that
${\tilde \Delta }(S(2^k),S(2^k))=\Omega $.

\begin{Lemma}\label{lem:lambda-to-delta-tilde}
Let $2^{j_1}<2^k<\min\{ 2^5L(h,\varepsilon_0), \frac{N}{2}\} $.
There exists a constant $C_4(\eta )>0$ depending only on $\varepsilon_0$
and $\eta $, such that for $t\in [0,1]$ and $0\leq m\leq k-j_1-1$,
\begin{equation}\label{eq:lambda-to-delta-tilde}
\mu_t^N\bigl( {\tilde \Lambda }({\tilde R}_{j_1}^{m+1}(v), S(2^k))\bigr)
\leq C_4(\eta )\mu_t^N\bigl(
     {\tilde \Delta }({\tilde R}_{j_1}^{m}(v), S(2^k))\bigr).
\end{equation}
\end{Lemma} 
The proof of this lemma goes parallel to that of 
Lemma \ref{lambda-to-delta}, but we have to remark two points.

\noindent $1^\circ $)
By definition of $T_a$ in defining fences, half of $T_a$ may not be 
inside ${\tilde R}_{j_1}^{m+1}(v)$. 
Therefore the width of corridors connecting $T_a$ and one of
${\tilde {\mathcal A}}_i$'s may be equal to $\eta 2^{j_1+m+1}$.
The length of the corridors can be made less than or equal to
$8\cdot 2^{j_1+m+1}$.
Thus, we can use the same constant $C_3^*$ in (\ref{eq:the constant C3(eta)})
in this case, too.

\noindent
$2^\circ $) 
We have to consider the case where ${\tilde S}_{j_1}^m(v)$ or
${\tilde T}_{j_1}^m(v)$ appears as ${\tilde R}_{j_1}^m(v)$, or 
the case where ${\tilde R}_{j_1}^m(v)=S_{j_1}^m(v)$ and 
${\tilde R}_{j_1}^{m+1}(v)=T_{j_1}^{m+1}(v)$.
But the proof of the lemma in these cases is a combination 
of the use of corridors and the proof of
Lemma \ref{lem: relation of deltas-tilde}. 

\begin{Lemma}\label{lem:comparison of gamma and delta-tilde}
Assume the conditions (\ref{eq:bound for j_1}) and (\ref{eq:cond-for-j1-2}).
Let $2^{j_1}<2^k<\min\{ 2^5L(h,\varepsilon_0), \frac{N}{2}\} $.
Then there exists a constant ${\tilde K}>0$ depending only on 
$j_1, \eta $ and $\varepsilon_0$, such that for 
$t\in[0,1]$ and $0\leq m\leq k-j_1-2$,
\begin{equation}\label{eq:gamma-to-delta}
\mu_t^N\bigl( {\tilde \Gamma }({\tilde R}_{j_1}^m(v), S(2^k))\bigr)
 \leq {\tilde K}
\mu_t^N\bigl( {\tilde \Delta }({\tilde R}_{j_1}^m(v), S(2^k))\bigr) .
\end{equation} 
\end{Lemma}
The proof goes parallel to that of Lemma \ref{comparison of gamma and delta}.

\newpage
\section{Extension argument II} 

In this section, we give an analogy to the argument in the previous section 
related to the one-arm event 
$$ 
\{ {\mathbf O} \stackrel{+}{\leftrightarrow } \partial_{in }S(2^k) \} . 
$$ 
The argument here is similar to that of subsection 5.3,
so we are going to use essentially the same notation as in 5.3.
Also, as in the previous section, we assume that $v\in S(2^k)$ and 
$0\leq v^2\leq v^1$ for the sake of argument. 
Throughout this section, we assume that $j_1$ satisfies 
the conditions 
(\ref{eq:bound for j_1}) and (\ref{eq:cond-for-j1-2}).

First, we  define the event 
${\tilde \Gamma }_0({\tilde R}_{j_1}^m(v), S(2^k))$ in two cases.

\noindent
(i) If ${\tilde R}_{j_1}^m(v)=S_{j_1}^m(v)$, then let 
${\tilde \Gamma }_0({\tilde R}_{j_1}^m(v), S(2^k))$ be the event such that 
all the following occur.
\begin{enumerate}
\item There exists a $(+)$-path $r_1$ in $S(2^k)$ connecting
${\mathbf O}$ with $\partial {\tilde R}_{j_1}^m(v)$.
\item There exists another $(+)$-path $r_3$ in $S(2^k)$ connecting
$\partial {\tilde R}_{j_1}^m(v)$ with $\partial_{in}S(2^k)$.
\item There exists a $(-*)$-path $t^*$  in 
$S(2^k)\setminus {\tilde R}_{j_1}^m(v)$
starting and ending at
$\partial {\tilde R}_{j_1}^m(v)$ such that it separates $r_1$ and
$r_3$ in $S(2^k)\setminus {\tilde R}_{j_1}^m(v)$.
\end{enumerate}

\noindent
(ii) If ${\tilde R}_{j_1}^m(v)$ is one of ${\tilde S}_{j_1}^m(v)$,
${\tilde T}_{j_1}^m(v)$ and $T_{j_1}^m(v)$, then we define
${\tilde \Gamma }_0({\tilde R}_{j_1}^m(v), S(2^k))$ as the event 
such that above 1 and the following 3' occur.
\begin{enumerate}
\item[3'.] There exists a $(-*)$-path $t^*$ 
in 
$S(2^k)\setminus {\tilde R}_{j_1}^m(v)$
starting and ending at
$\partial {\tilde R}_{j_1}^m(v)$ such that it separates $r_1$ and
$\partial S(2^k)$ in $S(2^k)\setminus {\tilde R}_{j_1}^m(v)$.
\end{enumerate}
We do not require the existence of $r_3$ in this case.

\medskip
Let ${\tilde \Lambda }_0({\tilde R}_{j_1}^m(v), S(2^k))$ be the subset
of ${\tilde \Gamma }_0({\tilde R}_{j_1}^m(v), S(2^k))$ such that for every $i$,
any crossing $(+)$-cluster and any crossing $(-*)$-cluster in 
${\tilde {\mathcal B}}_i$
has an $(\eta ,j_1+m)$-fence.
Here, we use ${\{ {\tilde {\mathcal B}}_i\} }_{1\leq i\leq 4} $ 
if ${\tilde R}_{j_1}^m(v)=S_{j_1}^m(v)$,
$ {\{ {\tilde {\mathcal B}}_i\} }_{ i=1,2,4}$ if 
${\tilde R}_{j_1}^m(v)\not= { S}_{j_1}^m(v)$
such that ${\tilde R}_{j_1}^m(v)\not\ni (2^k,2^k)$,
and $\{ {\tilde {\mathcal B}}_1, {\tilde {\mathcal B}}_4 \} $ if
${\tilde R}_{j_1}^m(v)\ni (2^k,2^k)$.

\medskip
In order to define $ {\tilde \Delta }_0({\tilde R}_{j_1}^m(v), S(2^k))$,
there are cases we have to modify locations of ${\tilde {\mathcal A}}_1$
and  ${\tilde {\mathcal A}}_2$.
We define  ${\tilde {\mathcal A}}_i$ as in the previous section
unless $d({\tilde R}_{j_1}^m(v), \{ x^2=2^k \} )< 3\cdot 2^{j_1+m-2}$.
If this occurs, then we put 
${\tilde {\mathcal A}}_1$
and  ${\tilde {\mathcal A}}_2$ on the left side of ${\tilde R}_{j_1}^m(v)$.
Namely, we put
\begin{align*}
{\tilde {\mathcal A}}_1&:= [y^1-2^{j_1+m-2}, y^1]
\times [y^2, y^2+2^{j_1+m-1}], \\
{\tilde {\mathcal A}}_2&:= [y^1-2^{j_1+m-2}, y^1]
\times [y^2+2^{j_1+m+1}-2^{j_1+m-1},y^2+2^{j_1+m+1}],  
\end{align*}
where $y=(y^1,y^2)$ is the lower left corner of ${\tilde R}_{j_1}^m(v)$.
${\tilde {\mathcal A}}_3, {\tilde {\mathcal A}}_4$ are
defined as in the previous section.
Then we define ${\tilde \Delta }_0({\tilde R}_{j_1}^m(v), S(2^k))$
as a subset of ${\tilde \Gamma }_0({\tilde R}_{j_1}^m(v), S(2^k))$
such that
\begin{enumerate}
\item 
$r_1 \cap \cup_i {\tilde {\mathcal B}}_i \subset {\tilde {\mathcal A}}_1$,
\item $r_3 \cap \cup_i {\tilde {\mathcal B}}_i \subset {\tilde {\mathcal A}}_3$
when ${\tilde R}_{j_1}^m(v)=S_{j_1}^m(v)$, and
\item $t^*$ starts in ${\tilde {\mathcal A}}_2$ and ends in
${\tilde {\mathcal A}}_4$ such that
$$t^*\cap \cup_i {\tilde {\mathcal B}}_i \subset {\tilde {\mathcal A}}_2
\cup {\tilde {\mathcal A}}_4.$$ 
\end{enumerate}

Let 
\begin{equation}\label{eq:m(v)}
m(v) := \min \{ m \geq 0 : S_{j_1}^m(v) \supset S(v,2^{-3} |v|_{\infty}) 
\}. 
\end{equation}
We will consider only those $m$'s which satisfy $0 \leq m \leq m(v)$, 
and therefore  ${\tilde R}_{j_1}^m(v)$  
is in distance of 
the same order as $|v|_\infty $ from the origin. 

By these definitions, as remarked in 5.3, if we choose 
$\rho , \varepsilon $ to satisfy (\ref{condition for rho and epsilon-2}),
and $\eta $ to satisfy (\ref{condition for eta}) and (\ref{eq:bound for eta}),
the statement of Lemma \ref{lem:fence}
is correct.
Further, we have analogous lemmas as in the previous section. 

\begin{Lemma}\label{lambda0-to-delta0-tilde} 
Let
$2^{j_1}<2^k<\min\{ 2^5L(h,\varepsilon_0),\frac{N}{2}\} $.
Then we have for $0\leq t \leq 1$ and $0\leq m\leq m(v)-1$,
\begin{equation}\label{eq:lambda0-to-delta0-tilde}
\mu_t^N\bigl( {\tilde \Lambda }_0({\tilde R}_{j_1}^{m+1}(v), S(2^k))\bigr) 
\leq C_4(\eta ) 
 \mu_t^N\bigl( {\tilde \Delta }_0( {\tilde R}_{j_1}^m(v), S(2^k))\bigr) , 
\end{equation}
where $C_4(\eta )$ is the same as in Lemma \ref{lem:lambda-to-delta-tilde}.
\end{Lemma} 
\begin{proof} 
The proof is quite similar to the proof of Lemma \ref{lem:lambda-to-delta-tilde}. 
We use $t^*$ instead of $r_2^*$ and $r_4^*$.
Let ${\tilde {\mathcal B}}'_i$ denote ${\tilde {\mathcal B}}_i$
for ${\tilde R}^{m+1}_{j_1}(v)$, and ${\tilde {\mathcal A}}_i$
denote ${\tilde {\mathcal A}}_i$ for ${\tilde R}_{j_1}^m(v)$.
Since any crossing $(+)$-cluster and any $(-*)$-cluster in
any of ${\tilde {\mathcal B}}'_i$ has an $(\eta ,j_1+m+1)$-fence,
on the event ${\tilde \Lambda }_0({\tilde R}_{j_1}^{m+1}(v), S(2^k))$,
we use corridors to connect these endpoints to the corresponding 
${\tilde {\mathcal A}}_i$ of ${\tilde R}_{j_1}^m(v)$.
Namely, they connect endpoint of $r_1$ with ${\tilde {\mathcal A}}_1$,
endpoint of $r_3$ with ${\tilde {\mathcal A}}_3$ if it exists,
endopoints of $t^*$ with ${\tilde {\mathcal A}}_2$ and
${\tilde {\mathcal A}}_4$, respectively so that these corridors do
not intersect.
The length of these corridors can be made less than $8\cdot 2^{j_1+m+1}$,
which ensures that the proof of Lemma \ref{lem:lambda-to-delta-tilde}
is valid in this case, too. 
\end{proof} 

To compare $\mu_t^N$-probabilities of 
${\tilde \Delta}_0({\tilde R}_{j_1}^{m+1}(v),S(2^k))$ and
${\tilde \Delta}_0({\tilde R}_{j_1}^{m}(v),S(2^k))$, we need a little more
care.

\begin{Lemma}\label{lem:comparison of delta0s-tilde}
Let $2^{j_1}<2^k<\min\{ 2^6L(h,\varepsilon_0), \frac{N}{2}\} $.
Then we have for $0\leq t\leq 1$ and for $0\leq m \leq m(v)-1$,
\begin{equation}\label{eq:comparison of delta0s-tilde}
\mu_t^N\bigl( {\tilde \Delta }_0({\tilde R}_{j_1}^{m+1}(v),S(2^k))\bigr)
\leq C_1 \mu_t^N\bigl( 
       {\tilde \Delta }_0({\tilde R}_{j_1}^{m}(v),S(2^k))\bigr) .
\end{equation}
\end{Lemma}
\begin{proof}
The proof is quite similar to those of Lemmas \ref{relation of deltas} and \ref{lem: relation of deltas-tilde}.
As before let us write ${\tilde {\mathcal A}}'_i$ for ${\tilde {\mathcal A}}_i$
corresponding to ${\tilde R}_{j_1}^{m+1}(v)$, and we use 
${\tilde {\mathcal A}}_i$ for those corresponding to ${\tilde R}_{j_1}^m(v)$.
Let us consider for example the case where ${\tilde R}_{j_1}^m(v)$ is equal to
$T_{j_1}^m(v)$ and it contains $(2^k,2^k)$.
Then of course ${\tilde R}_{j_1}^{m+1}(v)=T_{j_1}^{m+1}(v)$ also
contains $(2^k,2^k)$.
In this case in order to connect ${\tilde {\mathcal A}}'_1$ with 
${\tilde {\mathcal A}}_1$ in $T_{j_1}^{m+1}(v)$, we need a corridor of width
$2^{j_1+m-1}$ and its length is equal to $ 5\cdot 2^{j_1+m}$.
It goes straightly to the left boundary of $T_{j_1}^{m+1}(v)$ and turns down
until it hits the line $\{ x^2= y^2+2^{j_1+m-1} \} $, and again
turns to the left to cross ${\tilde {\mathcal A}}'_1$.
In this case $\delta_{80}\delta_{16}/16$ appears.
\end{proof}

\begin{Lemma} \label{lem:comparison of gamma0 and delta0-tilde} 
Let $2^{j_1+5}<2^k<\min\{ 4L(h,\varepsilon_0),\frac{N}{2}\} $, and 
assume that $|v|_\infty \geq 2^{j_1+5}$. 
Then there exists a constant $\tilde{K}_0>0$ depending only on 
$\varepsilon_0, j_1$ and $ \eta $, 
such that for $t\in [0,1]$ and for $1\leq m\leq m(v)$, 
\begin{equation}\label{eq:comparison of gamma0 and delta0-tilde}
\mu_t^N\bigl(\tilde{\Gamma}_0 ({\tilde R}_{j_1}^m (v), S(2^k)  ) \bigr) 
  \leq \tilde{K}_0 \mu_t^N\bigl( \tilde{\Delta}_0 ( R_{j_1}^m (v),S(2^k) 
)\bigr).
\end{equation} 
\end{Lemma} 
\begin{proof} 
We first prove the case where $m=m(v)$ in two steps. 

\noindent 
$1^\circ$) Assume that $|v|_\infty \leq 2^{k-2}$. 
Then $m(v)<m_1^*$ and ${\tilde R}_{j_1}^m(v)=S_{j_1}^m(v)$ for every $m\leq m(v)$. 

%% \begin{quote} 
%%
%% ----------------------------------------------------------------------- 
%%
%% \noindent 
%% \framebox{\bf NOTE ONLY FOR HTZ} 
%%
%% The definition of $m(v)$ and $S_{j_1}^m(v)$ implies that 
%% $$ 
%% 2^{-3}|v|_\infty \leq 2^{j_1+m(v)}\leq 2^{-2}|v|_\infty +2^{j_1+1}. 
%% $$ 
%% Therefore for every $x\in S_{j_1}^{m(v)+1}(v)$, 
%% \begin{align*} 
%% |x|_\infty &\leq |v|_\infty +2^{j_1}+ 2^{j_1+m(v)+1}\\ 
%%            &\leq |v|_\infty (1+2^{-1})+2^{j_1+2}(1+2^{-2}). 
%% \end{align*} 
%% So, since $|v|_\infty \leq 2^{k-2}$, and  $j_1+5<k$, 
%%  we obtain 
%% \begin{align*} 
%% |x|_\infty &\leq 2^{k-2}(1+2^{-1})+ 2^{k-4}(1+2^{-2})\\ 
%%            &\leq 2^{k-2}\bigl( (1+2^{-1})+2^{-2}(1+2^{-2})\bigr) <2^{k-1}. 
%% \end{align*} 
%% 
%% 
%% ----------------------------------------------------------------------- 
%% \end{quote} 

Let $\tilde{G}_0 (S_{j_1}^{m(v)} (v),S(2^k))$ be the event such that all the 
following conditions 
are satisfied. 
\begin{enumerate} 
\item There exists a $(+)$-path $r_1$ connecting $\mathbf{O}$ with 
$\partial S_{j_1}^{m(v)} (v)$ in $S(2^k)$. 
\item There exists another $(+)$-path $r_3$ connecting 
$\partial S_{j_1}^{m(v)} (v)$ with $\partial_{in} S (2^k)$ in $S(2^k)$. 
$r_1, r_3$ are disjoint. 
\end{enumerate} 
Note that $\tilde{G}_0 (S_{j_1}^{m(v)}(v), S(2^k)) \supset \tilde{\Gamma}_0 
(S_{j_1}^{m(v)}(v), S(2^k))$. We define the following annuli: 
\begin{align*} 
\mathcal{H}_v &:= S(4|v|_{\infty}) \setminus S(2|v|_{\infty}), \\ 
\mathcal{H}_0 &:= S(2^{-2}|v|_{\infty}) \setminus S(2^{-3}|v|_{\infty}). 
\end{align*} 
Note also that ${\mathcal H}_v$ and ${\mathcal H}_0$ are subsets of $S(2^k)$, 
and $d_\infty ({\mathcal H}_v, {\mathcal H}_0)\geq |v|_\infty $. 
Next, let $U_0$ and $U_v$ be corridors in $S(2^k)$ with the following 
properties. 
\begin{enumerate} 
\item The width of $U_0$ and $U_v$ is $2^{j_1+m(v)-2}$. 
\item $U_v$ starts from the right side of $\partial S_{j_1}^{m(v)}(v)$ in 
${\tilde {\mathcal A}}_3$ of $S_{j_1}^{m(v)}$ and it goes straightly to the 
right side of $S(4|v|_\infty)$ crossing the annulus ${\mathcal H}_v$. 
\item $U_0$ connects ${\tilde {\mathcal A}}_1$ of $S_{j_1}^{m(v)}(v)$ with 
$\partial S(2^{-3}|v|_\infty )$. 
Essentially, $U_0$ goes up from the top side of $S(2^{-3}|v|_\infty )$ 
until the height of ${\tilde{\mathcal A}}_1$ of $S_{j_1}^{m(v)}(v)$, 
and then it goes to the right until it crosses the ${\tilde{\mathcal A}}_1$. 
When the above route is impossible, then we take $U_0$ to go to the left first 
from the left side of $S(2^{-3}|v|_\infty )$ with length 
$2^{-2}|v|_\infty $, 
and then to go up until the height of the ${\tilde{\mathcal A}}_1$, and turn to 
the 
right as above. 
\end{enumerate} 

Let $E_0$ be the event that there exists a $(+)$-path in $U_0$ connecting 
shorter sides of $U_0$, and there exists a $(+)$-circuit in ${\mathcal H}_0$ 
surrounding the origin. 
Also, let $E_v$ be the event that there exists a $(+)$-path in 
$U_v$ connecting shorter sides of $U_v$, and there exists a $(+)$-circuit 
in ${\mathcal H}_v$ surrounding the origin, respectively. 
Then in the event $E_0\cap E_v\cap {\tilde G}_0(S_{j_1}^{m(v)}(v),S(2^k))$, 
all the following events occur. 
\begin{enumerate} 
\item There exists a $(+)$-path $r_1$ in $U_0\cup S(2^{-2}|v|_\infty )$ 
which connects the origin with $\partial S_{j_1}^{m(v)}(v)$, and 
$r_1\cap S(x_{j_1}(v),2^{j_1+m(v)}+2^{j_1+m(v)-2})\subset {\tilde {\mathcal 
A}}_1$.
\item There exists a $(+)$-path $r_3$ in 
$U_v\cup {\mathcal H}_v \cup S(4|v|_\infty )^c$ such that $r_3$ connects the 
right 
side of $\partial S_{j_1}^{m(v)}(v)$ with the right side of 
$S(2^k)$, and 
$$r_3\cap S(x_{j_1}(v),2^{j_1+m(v)}+2^{j_1+m(v)-2})\subset {\tilde {\mathcal 
A}}_3.
$$ 
\end{enumerate} 
Here, $ {\tilde {\mathcal A}}_i$'s correspond to $S_{j_1}^{m(v)}(v)$
Further, let ${\tilde G}_0^\#(S_{j_1}^{m(v)}(v),S(2^k))$ be the 
%subset of the 
event 
%$E_0\cap E_v\cap {\tilde G}_0(S_{j_1}^{m(v)}(v),S(2^k))$ 
such that the above 1 and 2 occur, and there are 
vertical 
$(+)$-crossings in ${\tilde{\mathcal A}}_1$ and  ${\tilde{\mathcal A}}_3$, 
respectively.  
Then by the FKG inequality, 
\begin{equation}\label{eq:6-2} 
\mu_t^N\bigl( {\tilde G}_0^\# (S_{j_1}^{m(v)}(v),S(2^k))\bigr) 
 \geq C_1^\# \mu_t^N\bigl( {\tilde G}_0(S_{j_1}^{m(v)}(v),S(2^k))\bigr) 
\end{equation} 
for some constant $C_1^\# $ which depends only on $j_1$ and $\varepsilon_0$. 

Then we prepare another corridor $U^*$ in 
$(S_{j_1}^{m(v)}(v)\cup S(2^{-2}|v|_\infty )^c$ with width 
$2^{j_1+m(v)-2}$. 
which connects the top side and the bottom 
side of $S_{j_1}^{m(v)}(v)$, and 
$$ 
U^*\cap S(x_{j_1}(v),2^{j_1+m(v)}+2^{j_1+m(v)-2})
 \subset {\tilde{\mathcal A}}_2 \cup {\tilde{\mathcal A}}_4. 
$$ 
We choose $U^*$ not to intersect $U_0$ or $U_v$, and its length less than 
$8|v|_\infty $. 
Let 
$$ 
E^*=\left\{ 
  \begin{array}{@{\,}l@{\,}}
  \mbox{there exists a $(-*)$-path $t^*$ in $U^*$ connecting the top}\\ 
  \mbox{and the bottom sides of $\partial S_{j_1}^{m(v)}(v)$, and there are} 
\\ 
  \mbox{horizontal $(-*)$-crossings in } 
   {\tilde{\mathcal A}}_2 \mbox{ and }  {\tilde{\mathcal A}}_4 
 \end{array} 
  \right\} . 
$$ 
Then by our assumptions (\ref{eq:bound for eta}) and (\ref{eq:bound for j_1}),
we can apply the Connection lemma to obtain 
\begin{equation}\label{eq:6-3} 
\mu_t^N\bigl( E^* \mid {\tilde G}_0^\# (S_{j_1}^{m(v)}(v),S(2^k)) \bigr) 
\geq C_2^\# 
\end{equation} 
for some constant $C_2^\# $ depending only on $j_1$ and $\varepsilon_0$. 
Apparently, we have 
$$ 
E^* \cap {\tilde G}_0^\# (S_{j_1}^{m(v)}(v),S(2^k)) 
 \subset {\tilde \Delta }_0(S_{j_1}^{m(v)}(v),S(2^k)), 
$$ 
and by (\ref{eq:6-2}), (\ref{eq:6-3}), and from the fact that 
$$ 
{\tilde G}_0(S_{j_1}^{m(v)}(v),S(2^k)) \supset 
{\tilde \Gamma }_0(S_{j_1}^{m(v)}(v),S(2^k)), 
$$ 
we have desired inequality. 

\medskip 
\noindent 
$2^\circ $) When $|v|_\infty \geq 2^{k-2}$, then we do not use the annulus 
${\mathcal H}_v$. Instead, $U_v$ goes to the right until it reaches 
$\partial_{in}S(2^k)$. 
Its length is not larger than $4|v|_\infty $. 
Therefore by the same argument as above we obtain the desired inequality. 
This completes the proof of the second statement of the lemma. 

\medskip 
In the case where $1\leq m\leq m(v)-1$, the proof is similar to that 
of Lemmas \ref{comparison of gamma and delta}  and \ref{lem:comparison of gamma and delta-tilde}
in the previous section. 

Suppose that $1 \leq m \leq m(v)-1$. We start with the following inequality. 
\begin{align*} 
&\mu_t^N\bigl( \tilde{\Gamma}_0 ({\tilde R}_{j_1}^m (v), S(2^k) ) \bigr)\\  
&\leq  \mu_t^N\bigl( 
\tilde{\Lambda}_0 ( {\tilde R}_{j_1}^m (v), S(2^k) ) \bigr) + 
 \mu_t^N\bigl(  
 \tilde{\Gamma}_0 ( {\tilde R}_{j_1}^m (v), S(2^k) )  \setminus  
    \tilde{\Lambda}_0 ( {\tilde R}_{j_1}^m (v), S(2^k) )\bigr). 
\end{align*} 
If $\tilde{\Gamma}_0 ( {\tilde R}_{j_1}^m (v), S(2^k) ) $ occurs but 
$\tilde{\Lambda}_0 ( {\tilde R}_{j_1}^m (v), S(2^k) ) $ 
does not occur, then we can see that 
\begin{itemize} 
\item $\tilde{\Gamma}_0 ( {\tilde R}_{j_1}^{m+1}(v), S(2^k) )$  occurs, and 
\item for at least one of $\tilde{\mathcal B}_i$'s of ${\tilde R}_{j_1}^m(v)$, 
    there is a crossing $(+)$-cluster or crossing $(-*)$-cluster, 
    connecting $\varphi_i$ with $\xi_i$ for some $i$, such that this 
    $(+)$-cluster (or $(-*)$-cluster) does not have an $(\eta, j_1+m)$-fence. 
\end{itemize} 
If $m\geq m_1^*$, we do not use ${\tilde {\mathcal B}}_3$ in the second 
statement. 
Since $S(2^k)\setminus {\tilde R}_{j_1}^{m+1}(v)$ and $\tilde{\mathcal B}_i$'s 
for ${\tilde R}_{j_1}^m(v)$ 
are of $\ell^\infty$-distance $3\cdot 2^{j_1+m-2}$, 
by the mixing property we have 
\begin{align*} 
& \mu_t^N\bigl(  \tilde{\Gamma}_0 ({\tilde R}_{j_1}^m (v), S(2^k)) \setminus  
    \tilde{\Lambda}_0 ({\tilde R}_{j_1}^m (v), S(2^k) ) \bigr) \\ 
& \leq  ( 8\delta + C\cdot 6(2^{j_1+m})^3e^{-\alpha 3\cdot 2^{j_1+m-2}}) 
     \mu_t^N \bigl(  
  \tilde{\Gamma}_0 ({\tilde R}_{j_1}^{m+1}(v), S(2^k) ) \bigr). 
\end{align*} 
By (\ref{eq:cond-for-j1-2})
we have 
\begin{align*} 
&\mu_t^N\bigl( \tilde{\Gamma}_0 ({\tilde R}_{j_1}^m (v), S(2^k) ) \bigr)\\  
&\leq  \mu_t^N\bigl( 
    \tilde{\Lambda}_0 ({\tilde R}_{j_1}^m (v), S(2^k) ) \bigr) + 
9\delta \mu_t^N \bigl( 
   \tilde{\Gamma}_0 ({\tilde R}_{j_1}^{m+1}(v), S(2^k) ) \bigr). 
\end{align*} 
Iterating this until we get to ${\tilde R}_{j_1}^{m(v)}(v)$, we have 
\begin{align*} 
&\mu_t^N\bigl( \tilde{\Gamma}_0 ({\tilde R}_{j_1}^m (v), S(2^k) ) \bigr) \\ 
&\leq  \sum_{\ell =0}^{m(v)-m-1} (9\delta )^\ell 
\mu_t^N\bigl( \tilde{\Lambda}_0 ( {\tilde R}_{j_1}^{m+\ell}(v), S(2^k)) \bigr) 
            +(9\delta )^{m(v)-m} 
            \mu_t^N\bigl( 
    \tilde{\Gamma}_0({\tilde R}_{j_1}^{m(v)}(v), S(2^k)) \bigr). 
\end{align*} 
By Lemma \ref{lambda0-to-delta0-tilde}
$$ 
\mu_t^N\bigl( \tilde{\Lambda}_0 ( {\tilde R}_{j_1}^{m+\ell}(v), S(2^k)) \bigr) 
\leq \tilde{C}_4(\eta )\mu_t^N\bigl( 
   \tilde{\Delta}_0 ( {\tilde R}_{j_1}^{m+\ell-1}(v), S(2^k))\bigr) 
$$ 
for $0\leq \ell \leq m(v)-m-1$. 
Further, we already proved that
$$ 
\mu_t^N\bigl( 
 \tilde{\Gamma}_0 ({\tilde R}_{j_1}^{m(v)}(v), S(2^k)) \bigr) \leq  
C_2^\# \mu_t^N\bigl( \tilde{\Delta}_0 ({\tilde R}_{j_1}^{m(v)} (v), S(2^k)) 
\bigr). 
$$ 
By Lemma \ref{lem:comparison of delta0s-tilde}
$$ 
\mu_t^N\big( \tilde{\Delta }_0 ({\tilde R}_{j_1}^{m+\ell}(v), S(2^k)) \bigr) 
  \leq 
C_1^\ell \mu_t^N\bigl( \tilde{\Delta}_0 ({\tilde R}_{j_1}^m (v), S(2^k))\bigr) 
$$ 
for every $0\leq \ell \leq m(v)-m-1$. 
Hence we have
$$ 
\mu_t^N\bigl(\tilde{\Gamma}_0 ({\tilde R}_{j_1}^{m}(v), S(2^k) ) \bigr) 
  \leq \tilde{K}_0 \mu_t^N\bigl( 
 \tilde{\Delta}_0 ({\tilde R}_{j_1}^{m-1} (v), S(2^k))\bigr) 
$$ 
for some positive constant $\tilde{K}_0$. 
Since ${\tilde \Gamma }_0({\tilde R}_{j_1}^m(v), S(2^k))$ is increasing in 
$m\leq m(v)$, 
this implies the desired inequality. 
\end{proof} 

\newpage
\section{Power law estimates for $k$-arm paths}\label{s:powerlaw}

In this section, we provide some power law estimate for arm events.
The restriction that $n<L(h,\varepsilon_0)$ in each lemma and each theorem
is not serious. Namely, by using narrower rectangles properly in the 
following arguments, this restriction can be easily relaxed to 
$n<2L(h,\varepsilon_0)$ or even to $n<4L(h,\varepsilon_0)$.
So, we remark here that although all the statements in this section are
restricted to $n<L(h,\varepsilon_0)$, but we can obtain similar result
for $n$'s with $n<4L(h,\varepsilon_0)$.

\subsection{Power law estimate for one-arm path}

We define
\begin{align*}
\pi_{h}(n) &:= \mu_{T,h} \bigl( \mathbf{O} \stackrel{+}{\leftrightarrow} \partial_{in} S(n) \bigr), \\
\pi_{h}^*(n) &:= \mu_{T,h} \bigl( \mathbf{O} \stackrel{-*}{\leftrightarrow} \partial_{in} S(n) \bigr).
\end{align*}
We abbreviate $\pi_{h_c(T)}(n)$ (resp. $\pi_{h_c(T)}^*(n)$) to $\pi_{\text{cr}}(n)$ (resp. $\pi_{\text{cr}}^*(n)$).

\begin{Theorem} 
There exist positive constants $C_5,\,C_6$ and 
$0<\alpha < 1$ such that for $2^{j_1}<R<n<L(h,\varepsilon_0)$,
\[ C_5 \dfrac{R}{n} \leq \mu_{T,h}  \bigl( {\cal B}(1, 0,R, n) \bigr)
\leq C_6 \left( \dfrac{R}{n}\right)^{\alpha}. \]
In particular,
\begin{equation}
C_5 n^{-1} \leq \pi_{\text{cr}} (n) \leq C_6 n^{-\alpha}. \label{K(5.1)}
\end{equation}
for every $n\geq 2^{j_1}$.
\end{Theorem}

The strategy to obtain the upper bound is well-known (see e.g. (11.90) of
\cite{Gri99}). 
A proof for the lower bound is found in Lemma 5 of \cite{Z95}.
For a better lower bound for $\pi_{\text{cr}} (n)$, see the comment after Lemma \ref{inc-7}.

%% \begin{quote} 
%% ----------------------------------------------------------------------- 
%% 
%% \noindent 
%% \framebox{\bf INTERMISSION} 

\begin{proof}We begin with the upper bound. For $j \geq 1$, let 
$$A_j := S(4^{j+1}R) \setminus S(2 \cdot 4^{j}R).$$
For each $j$, we define a random variable $X_j$ by
\[ X_j :=
\begin{cases}
1 &\mbox{if there exists a $(-*)$-circuit in $A_j$ surrounding the origin,} \\
0 &\mbox{otherwise}. \\
\end{cases}
\]
Using the mixing property, we can show the following:  There exist an integer $j^*$ and a number $\delta>0$ such that 
\[ \mu_{T,h}  (X_j = 1 \,| \, X_1,\ldots,X_{j-1}) \geq \delta \]
for every $j \geq j^*$. Now we have
\[ \mu_{T,h}  \bigl( {\cal B}(1, 0,R, n) \bigr)\leq \mu_{T,h}  \left( \bigcap_{j=j^*}^{\lfloor \log_4 (R/n) \rfloor } \{ X_j = 0 \} \right) \leq (1-\delta)^{\lfloor \log_4 (R/n) \rfloor -j+1}. \]

Now we turn to the lower bound. By the RSW-type lemma, %page 884 of Chayes-Nolin \cite{CN09}.
$\mu_{T,h}  \bigl( A^+(n,n) \bigr) \geq \delta_1$.
On $A^+(n,n)$, there exists the lowest $(+)$-crossing ${\mathcal R}$ of $S(n)$. We define
\[ H_R({\mathcal R}) := 
\max \bigl\{ y \in [-n/R,n/R] \cap \mathbf{Z}  : \bigl((0,Ry)+S(R) \bigr) \cap 
{\mathcal R} \neq \emptyset \bigr\}. \]
Then we have
\begin{align*}
&\mu_{T,h} \bigl( A^+(n,n) \bigr) \\
&= \sum_{y \in [-n/R,-n/R] \cap \mathbf{Z}} \mu_{T,h}  \bigl(
  H_R({\mathcal R}) = y \bigr) \\
&\leq \sum_{y \in [-n/R,-n/R] \cap \mathbf{Z}} \mu_{T,h}  \left( \bigl((0,Ry)+S(R)\bigr) \stackrel{+}{\leftrightarrow} \partial\bigl((0,Ry)+S(n) \bigr) \right) \\ 
&=C'\dfrac{n}{R} \mu_{T,h}  \bigl( {\cal B}(1, 0,R, n) \bigr).
\end{align*}
\end{proof}

%% ----------------------------------------------------------------------- 
%% \end{quote}

%Here we give another lower bound; unfortunately it is still weaker than the one obtained for the independent case, where the BK inequality is applicable. We remark that for the triangular lattice case, $C_1 n^{-1/2} \leq \pi_{\text{cr}} (n)$ follows from the self-duality.
%
%\begin{Proposition} For all $T > T_c$, there exists a positive constant $C_1$ such that
%\begin{equation}
%C_1 n^{-1/2} \leq \max\left\{ \pi_{\text{cr}} (n) ,  \pi_{\text{cr}}^* (n) \right\}. 
%\end{equation}
%\end{Proposition}
%
%\begin{proof} \framebox{\bf To be written later.}
%\end{proof}

\subsection{Power law estimate for two-arm paths in half space}
Let $\mathcal{E}_2(n)$ be the event such that 
\begin{enumerate}
\item there is a $(-*)$-path $r^*$ in $S(n)$ connecting $(0,n)$ with
the boundary $\partial_{in}S(n)\setminus \{ x^2=n\} $, and
\item there is a $(+)$-path $r$ in $S(n)$ connecting $(-1,n)$ with
$\partial_{in}S(n)\setminus \{ x^2=n\} $.
\end{enumerate}
Interchanging the roles of $(+)$- and $(-*)$-paths in the above 
definition, we can define
another event ${\mathcal E}_2^*(n)$. 
Also, recall that ${\mathcal B}^+(1,1,R,n)$ is the event that
there are a $(+)$-path ${\tilde r}$ and a $(-*)$-path ${\tilde r}^*$
in
$$\bigl( S(n)\setminus S((0,n-R),R) \bigr) \cap \{ x^2 \leq 0 \} ,$$
both connecting 
$\partial S((0,n-R),R)$ with $\partial_{in}S(n)\setminus \{ x^2=n\} $.
Concerning these events we are going to prove the following 
theorem in this subsection.  

\begin{Theorem}\label{2 arm bound}
There exist positive constants $C_7,\ldots,C_{10}$ and an integer $j_2>j_1$, 
depending only on
$\varepsilon_0, j_1$ and $\eta $,
such that the following estimates hold 
for $0\leq t\leq 1$.
\begin{equation}\label{eq:2 arm bound} 
C_7n^{-1} \leq \mu_t^N \bigl( \mathcal{E}_2(n) \bigr) \leq C_8n^{-1},
\end{equation} 
for  $ 2^{j_1+4}<n<\min\{ L(h, \varepsilon_0), \frac{N}{2} \} $, and
\begin{equation}\label{eq:2 arm bound for boxes}
C_9\frac{R}{n} \leq \mu_t^N \bigl({\mathcal B}^+(1,1,R,n) \bigr)
    \leq C_{10}\frac{R}{n} 
\end{equation}
for $2^{j_2}\leq R<n<\min\{ L(h, \varepsilon_0), \frac{N}{2} \} $.
The estimate (\ref{eq:2 arm bound}) is valid for ${\mathcal E}^*_2(n)$,
too.
\end{Theorem}

The proof of this theorem is divided into two parts;
the proof of (\ref{eq:2 arm bound}) and the proof of
(\ref{eq:2 arm bound for boxes}).

First we prove (\ref{eq:2 arm bound}).
Let $L_n$ denote the leftmost vertical $(-*)$-crossing of $S(n)$.
Also, let $( Z(L_n),n)$ denote the starting point of $L_n$ in 
$\{ x^2=n\} $. 
Then by definition there exists a $(+)$-path in $S(n)$ from
$(Z(L_n)-1,n)$ or from $ (Z(L_n), n-1) $ to the left side of $S(n)$.
Put 
\[ E^*(n,x) := \{ Z(L_{n}) = x  \}. \]
\begin{Lemma}\label{lem:7-1}
Let $2^{j_1}<n<\min\{ L(h, \varepsilon_0), \frac{N}{2} \} $.
Then we have 
\[
\frac{1}{2}\delta_8^2\cdot \delta_1 \leq \sum_{|x|\leq n/2}
 \mu_t^N\bigl( E^*(n,x)\bigr) \leq 1
\]
for $0\leq t \leq 1$.
\end{Lemma}
\begin{proof}
Since $\{ E^*(n,x) : -n \leq x \leq n\} $ are disjoint, the right hand 
inequality is obvious.
Let us define events $A,B,C$ by
\begin{align*}
A&=\left\{ 
 \mbox{there is a vertical $(-*)$-crossing in 
 $ [ \frac{n}{4}, \frac{n}{2} ]\times [-n, n] $}
 \right\} , \\
B&= \left\{ 
 \mbox{there is a vertical $(+)$-crossing in
 $[ -\frac{n}{2}, -\frac{n}{4}]\times [ -n, n]$}
  \right\} , \\
C&=\left\{ 
 \mbox{there is a horizontal $(+)$-crossing in 
        $[-n, -\frac{n}{4}]\times [-n, n]$}
\right\}.
\end{align*}
Then it is clear that
$$
A\cap B\cap C \subset \left\{ -\frac{n}{2}\leq Z(L_n)\leq \frac{n}{2} \right\} .
$$
Hence we have
$$
\sum_{|x|\leq n/2} \mu_t^N( Z(L_n)=x ) \geq \mu_t^N( A\cap B\cap C).
$$ 
By the mixing property and the FKG inequality,
\begin{align*}
\mu_t^N( A\cap B\cap C) &\geq \left[ \mu_t^N( A ) 
              - C 2n\cdot \frac{n}{4}\frac{n}{2}e^{-\alpha n/2} \right]
              \mu_t^N( B\cap C ) \\
 &\geq  \left[\delta_8- C \dfrac{n^3}{4}e^{-\alpha n/2} \right]
          \delta_8\cdot \delta_1.
\end{align*}
By (\ref{eq:the constant C1}), (\ref{eq:delta-1}) and (\ref{eq:cond-for-j1-2}),
we know that 
$$
C\dfrac{n^3}{4}e^{-\alpha n/2} 
  \leq 2^{-3}\delta \leq \frac{\delta_8}{2},
$$ 
and hence $\mu_t^N( A\cap B\cap C) \geq \frac{1}{2}\delta_8^2\cdot \delta_1$.
\end{proof}

Next, let $\Hat{E}^*(n,0)$ be the event such that
\begin{enumerate}
\item there is a $(-*)$-path ${\hat r}^*$ in $[-\frac{n}{2}, \frac{n}{2}]
\times [-3n, n] $ connecting $(0,n)$ with the bottom side of $S((0,-n),2n)$, 
and
\item there is a $(+)$-path ${\hat r}$ in $[-2n, 2n]\times [-n, n]$
connecting $(-1,n)$ with the left side of $S((0,-n),2n)$.
\end{enumerate}
Then for $x\in {\mathbf Z}$, we define $\Hat{E}^*(n,x)$ as translation of 
$\Hat{E}^*(n,0)$ by $(x,0)$.
So, $\Hat{E}^*(n,0)$ is an event occurring in $S((x,-n),2n)$.
Note that  $\Hat{E}^*(n,x)$ is a subset of $E^*(n,x)$
for every $x$ with $|x|\leq n/2$.

By Lemma \ref{lem:7-1} and the translation invariance of $\mu_t^N$, we have
\[1 \geq  \sum_{x \in [-n/2,n/2]} \mu_t^N \bigl( \Hat{E}^*(n,x) \bigr) 
 =  n \mu_t^N\bigl( \Hat{E}^*(n,0) \bigr).\]

For the proof of (\ref{eq:2 arm bound}) and (\ref{eq:2 arm bound for boxes}),
we use the extension argument modified for the present setting.
Let us rewrite $\Gamma_2(S(2^j))= {\mathcal E}_2(2^j)$ to fit the extension 
argument. 
Then we put  $\Lambda_2(S(2^j))$  as the subset of $ \Gamma_2(S(2^j))$
such that any crossing $(+)$-cluster and any crossing $(-*)$-cluster
in ${\mathcal B}_i$ has an $(\eta , j)$-fence for $i=1,3,4$, where 
${\mathcal B}_i$'s correspond to $S(2^j)$.
Finally, let $\Delta_2(S(2^j))$ be the subset of $\Gamma_2(S(2^j))$ such that
the following events occur:
\begin{enumerate}
\item $r \cap ({\mathcal B}_1\cup {\mathcal B}_3\cup {\mathcal B}_4) \subset {\mathcal A}_1$, 
and $r^*\cap  ({\mathcal B}_1\cup {\mathcal B}_3\cup {\mathcal B}_4) \subset {\mathcal A}_1$.
\item There exists a vertical $(+)$-crossing in $  {\mathcal A}_1$,
and there exists a horizontal $(-*)$-crossing in $  {\mathcal A}_4$.
\end{enumerate}
Here, ${\mathcal A}_i$'s correspond to $S(2^j)$ as in section 5.
By the inwards extension argument in section 5 we have
\begin{align}
\label{eq:7-4}
\mu_t^N\bigl( \Delta_2(S(2^j)) \bigr) &\leq C_1 \mu_t^N\bigl(
\Delta_2(S(2^{j+1})) \bigr), \\
\label{eq:7-4-1}
\mu_t^N\bigl( \Delta_2(S(2^j)) \bigr) &\geq C_1^{-(j-j_1)}
        {\bigl(1+ e^{(4h_c + 8)/(\mathfrak{K}T)} \bigr) }^{-\# S(2^{j_1})}, \\
\intertext{and}
\label{eq:7-5}
\mu_t^N\bigl( \Lambda_2(S(2^j)) \bigr) &\leq C_3(\eta )\mu_t^N\bigl(
\Delta_2(S(2^{j+1})) \bigr). 
\end{align}
These are valid for $j_1\leq j<\lfloor \log_2n\rfloor $, and
\begin{equation}\label{eq:7-6}
\mu_t^N\bigl( \Gamma_2(S(2^j)) \bigr) \leq K\mu_t^N\bigl(
\Delta_2(S(2^{j})) \bigr)
\end{equation}
for $j_1+3\leq j<\lfloor \log_2n\rfloor $.
Here, the constants $C_1, C_3(\eta )$ and $K$ are the same as those
in Lemmas \ref{relation of deltas}, \ref{lambda-to-delta} and \ref{comparison of gamma and delta}.

\begin{Remark}
Proof of the above inequalities is essentially the same as those
of Lemmas \ref{relation of deltas}, \ref{lambda-to-delta} and \ref{comparison of gamma and delta}, but we list up where we modify them.
\begin{itemize}
\item  As for (\ref{eq:7-4}) and (\ref{eq:7-4-1}),
we introduce $\Delta_2(S((0,2^j),2^j))$ as
the $(0,2^j)$-translation of $\Delta_2(S(2^j))$.
Note that the top center point of $S((0,2^j),2^j)$ is $(0,2^{j+1})$
which is also the top center point of $S(2^{j+1})$.
Then we can compare $\displaystyle \mu_t^N\bigl( \Delta_2(S((0,2^j),2^j))
\bigr) $ and $\displaystyle \mu_t^N\bigl( \Delta_2(S(2^{j+1}))\bigr) $
by the same argument as in the proof of Lemma \ref{relation of deltas}.

\item
As for (\ref{eq:7-5}), we also introduce $\Lambda_2(S((0,2^j),2^j))$ as 
the $(0,2^j)$-shift of $\Lambda_2(S(2^j))$.

\item As for (\ref{eq:7-6}), we need for each $m$  the shifted events;\\ 
$\Gamma_2(S((0,2^{j}-2^{j-m}),2^{j-m}))$,
$\Lambda_2(S((0,2^{j}-2^{j-m}),2^{j-m}))$ and\\
$\Delta_2(S((0,2^{j}-2^{j-m}),2^{j-m}))$.
\end{itemize}
\end{Remark}
\begin{Lemma}\label{lem:7-2} Assume that $2^{j_1+4}\leq 2^{k-1}<n\leq 2^k<
\min\{ L(h,\varepsilon_0), \frac{N}{2}\} $. Then we have
$$
\mu_t^N\bigl( {\mathcal E}_2(n)\bigr) \leq K
 {\biggl( \frac{\delta_{128}}{16}\biggr) }^{-2}
  \   \mu_t^N\bigl( \Hat{E}^*(n,0)\bigr),
$$
which proves the upper bound in (\ref{eq:2 arm bound}),
where $K>0$ is given in (\ref{eq:7-6}).
\end{Lemma} 
\begin{proof}
Since 
$$
{\mathcal E}_2(n) \subset \Gamma_2(S((0,n-2^{k-1}),2^{k-1})),
$$
from (\ref{eq:7-6}), and the translation invariance of $\mu_t^N$, we have 
\begin{equation}\label{eq:7-7}
\mu_t^N\bigl( {\mathcal E}_2(n)\bigr) \leq 
 \mu_t^N\bigl( \Gamma_2(S(2^{k-1}))\bigr) \leq K\mu_t^N\bigl(
    \Delta_2(S(2^{k-1})) \bigr) .
\end{equation}
Let 
$$
{\mathcal D}'_1:=[-2n, -2^{k-1}+2^{k-3}]\times [n-2^{k-1}-2^{k-3},
                  n-2^{k-1}+2^{k-3}] ,
$$
which connects ${\mathcal A}_1$ of $S((0,n-2^{k-1}), 2^{k-1})$
with the left side of $S((0,-n),2n)$ in  
$[-2n,2n]\times [-n,n]$.
Also, let
$$
{\mathcal D}'_4:=[-2^{k-3}, 2^{k-3}]\times [-3n, n-2^k+2^{k-3}],
$$
which connects ${\mathcal A}_4$ of $S((0,n-2^{k-1}), 2^{k-1})$
with the bottom side of the box $S((0,-n),2n)$.
Then we define events occurring in ${\mathcal D}'_1$ and ${\mathcal D}'_4$
by
\begin{align*}
D_1'&=\{ 
\mbox{there is a horizontal $(+)$-crossing in ${\mathcal D}'_1$}
\}, \\
D_4'&=\{ 
\mbox{there is a vertical $(-*)$-crossing in ${\mathcal D}'_4$}
\}.
\end{align*}
Then 
$$
D_1'\cap D_4' \cap \Delta_2 ( S((0,n-2^{k-1}),2^{k-1}) ) 
\subset \Hat{E}^*(n,0),
$$
and the lengths of ${\mathcal D}'_1$ and ${\mathcal D}'_4$ are not longer than
$2^{k+2}$.
Therefore by the Connection lemma, we have
\begin{align*}
\mu_t^N \bigl( \Hat{E}^*(n,0) \bigr) &\geq  
\mu_t^N\bigl( D_1'\cap D_4'\cap \Delta_2( S((0,2n-2^{k-1}),2^{k-1}))
 \bigr) \\
&\geq { \biggl( \frac{\delta_{128}}{16}\biggr) }^2
 \ \mu_t^N\bigl( \Delta_2(S(2^{k-1}))\bigr) .
\end{align*}
Together with (\ref{eq:7-6}) and (\ref{eq:7-7}), this proves the lemma.
\end{proof}

For the lower bound in (\ref{eq:2 arm bound}), by the finite energy property,
there exists an absolute constant $C_1^\# >0$ such that for every $|x|\leq n/2$,
$$
\mu_t^N\bigl( E^*(n,x)\bigr) \leq C_1^\#
\mu_t^N\bigl( \Gamma_2(S((0,n-2^{k-2}), 2^{k-2}))\bigr) ,
$$
where $2^{k-1}<n\leq 2^k$.
Hence from Lemma \ref{lem:7-1}, and the translation invariance of $\mu_t^N$,
$$
\frac{1}{2}\delta_8^2\cdot \delta_1 \leq 
\sum_{|x|\leq n/2}\mu_t^N\bigl( E^*(n,x) \bigr)
 \leq C_1^\# n \mu_t^N\bigl( 
      \Gamma_2(S(2^{k-2}))\bigr) \leq 
C_1^\# n \mu_t^N\bigl( {\mathcal E}_2(n/2)\bigr) .
$$
This completes the proof of (\ref{eq:2 arm bound}).
It is clear that the above argument is applicable to ${\mathcal E}_2^*(n)$,
too.

Next, we prove (\ref{eq:2 arm bound for boxes}). 
Let $j$ be the integer such that $2^j\leq R<2^{j+1}$.
By the mixing property,
$$ %%\begin{align*}
\mu_t^N\bigl( {\mathcal E}_2(n) \bigr) %%&\leq 
%%   \mu_t^N\bigl( \Gamma_2(S(2^{j-1}))\cap {\mathcal B}^+(1,1,2^j,n)\bigr) \\
%% &
\leq  \mu_t^N\bigl( {\mathcal B}^+(1,1,2^j,n)\bigr)
  \left\{ \mu_t^N\bigl( \Gamma_2(S(2^{j-2}))\bigr) 
              + C2^{2(j-1)}\cdot 2^{j-1}e^{-\alpha 2^{j-1}} \right\} .
$$  %%\end{align*}
Since we set $\Gamma_2(S(2^{j-2}))= {\mathcal E}_2(2^{j-2})$,
by (\ref{eq:2 arm bound}) we have
$$
\mu_t^N\bigl( {\mathcal E}_2(2^{j-2}) \bigr) \geq C_7 2^{-j+2},
$$
if $j-2\geq j_1+4$.
So, if we assume that $j_2\geq j_1+6$ is sufficiently large such that 
\begin{equation}\label{eq:cond for j2}
C_72^{-j+2}\geq C2^{3(j-1)}e^{-\alpha 2^{j-1}} \quad \mbox{ for every $j\geq j_2$,}
\end{equation}
then we have 
$$
\mu_t^N\bigl( {\mathcal E}_2(n) \bigr) \leq 
2  \mu_t^N\bigl( {\mathcal B}^+(1,1,2^j,n)\bigr)
\mu_t^N\bigl( {\mathcal E}_2(2^{j-2}) \bigr) 
$$
for every $j\geq j_2$.
This, together with (\ref{eq:2 arm bound}), implies that
\begin{align*}
\mu_t^N\bigl( {\mathcal B}^+(1,1,R,n)\bigr) &\geq 
 \mu_t^N\bigl( {\mathcal B}^+(1,1,2^j ,n)\bigr) \\
&\geq  C_7C_8^{-1}\frac{2^{j-3}}{n}\geq \frac{1}{16}C_7C_8^{-1}\frac{R}{n}.
\end{align*}
This is valid if $ 2^{j_2}\leq R<n$.

To show the upper bound in (\ref{eq:2 arm bound for boxes}), we use an
outwards extension argument.
This time we consider $S((0,n-2^j),2^j)$ in $S(n)$.
First we put
$$
{\tilde {\mathcal B}}_i:= S(n)\cap \biggl[
  S_i((0,n-2^j),2^j+2^{j-2})\setminus S((0,n-2^j),2^j)\biggr] ,
$$
and 
$$
\varphi_i:= S(n)\cap \biggl[
 \partial_{in} S_i((0,n-2^j),2^j+2^{j-2})\setminus S((0,n-2^j),2^j)\biggr] ,
$$
for $i=1,3,4$, and let $\xi_1,\xi_3\xi_4$ be the left, the right and
the bottom sides of $S((0,n-2^j),2^j)$, respectively.% 3:the right,
Then we can define $(\eta ,j)$-fences. 
Note that everything should be defined in $S(n)$.

Let ${\tilde \Gamma}_2(S((0,n-2^j),2^j), S(n))$ be the event such that
\begin{enumerate}
\item there exist a $(+)$-path ${\tilde r}$ and a $(-*)$-path 
${\tilde r}^*$ in $S(n)\setminus S((0,n-2^j),2^j)$, both connecting  
$\partial S((0,n-2^j),2^j)$ with $\partial_{in}S(n)\setminus \{ x^2=n\} $,
and
\item ${\tilde r}$ is to the left of ${\tilde r}^*$.
\end{enumerate}
Also, let $ {\tilde \Gamma}^*_2(S((0,n-2^j),2^j), S(n))$ be the event 
such that the roles of ${\tilde r}$ and ${\tilde r}^*$ are interchanged
in the above definition.
Then it is clear that
$$
{\tilde \Gamma}_2(S((0,n/2-2^j),2^j), S(n/2))
\cup{\tilde \Gamma}^*_2(S((0,n/2-2^j),2^j), S(n/2))
\supset {\mathcal B}^+(1,1,2^j,n).
$$
Since the argument is the same, it is sufficient to prove the upper bound in 
(\ref{eq:2 arm bound for boxes}) for 
${\tilde \Gamma}_2(S((0,n-2^j),2^j), S(n))$ in place of 
${\tilde B}^+(1,1,R,n)$.

Define ${\tilde \Lambda }_2( S((0,n-2^j), 2^j), S(n)) $ as the 
subset of $ {\tilde \Gamma}_2(S((0,n-2^j),2^j), S(n))$ such that 
every crossing $(+)$-cluster and every crossing $(-*)$-cluster
in ${\tilde {\mathcal B}}_i$ has an $(\eta , j)$-fence, for $i=1,4$.
Then we put
\begin{align*}
{\tilde {\mathcal A}}_1&:= [-2^j-2^{j-2},-2^j]\times [n-2^j-2^{j-2},
                            n-2^j+2^{j-2}],\\
{\tilde {\mathcal A}}_4&:= [-2^{j-2},2^{j-2}]\times [n-2^{j+1}-2^{j-2},
                            n-2^{j+1}].
\end{align*}
These relative locations for $S((0,n-2^j), 2^j)$ are the same as
${\tilde A}_i$'s for $S(2^j)$.
Finally, let
${\tilde \Delta}_2(S((0,n-2^j),2^j), S(n))$ be the subset of 
$ {\tilde \Gamma}_2(S((0,n-2^j),2^j), S(n))$ such that 
\begin{enumerate}
\item ${\tilde r}^*\cap (\cup_i{\tilde {\mathcal B}}_i)\subset {\tilde {\mathcal A}}_4$,
and  ${\tilde r}\cap (\cup_i{\tilde {\mathcal B}}_i)\subset {\tilde {\mathcal A}}_1$,
and
\item there exists a vertical $(+)$-crossing in ${\tilde {\mathcal A}}_1$,
and there  exists a horizontal $(-*)$-crossing in ${\tilde {\mathcal A}}_4$.
\end{enumerate}
Then under the conditions (\ref{condition for eta}), 
(\ref{eq:bound for eta}) and (\ref{condition for rho and epsilon-2}),
the statement of Lemma \ref{lem:fence} is correct as in section 5.3.
Therefore we have as before,
\begin{align}\label{eq:delta2-tilde}
&\mu_t^N\bigl( {\tilde \Delta}_2(S((0,n-2^{j+1}),2^{j+1}), S(n)) \bigr) \\
&\leq C_1 
\mu_t^N\bigl( {\tilde \Delta}_2(S((0,n-2^{j}),2^{j}), S(n)) \bigr) \nonumber
\end{align}
for $2^{j_1}\leq 2^j< 2^{j+1}+2^{j-1}\leq n$.
Let $j^*$ be the maximum of $j$'s satisfying $2^j+2^{j-2}\leq n$.
Putting
\begin{align*}
{\mathcal D}_1&=[ -n, -2^{j^*}]\times [-2^{j^*-2},2^{j^*-2}], \\
{\mathcal D}_4&=[-2^{j^*-2},2^{j^*-2}]\times [-n, n-2^{j^*+1}],
\end{align*}
we can see that $ {\tilde \Delta}_2(S((0,n-2^{j^*}),2^{j^*}), S(n)) $
occurs if there exist 
\begin{itemize}
\item a horizontal $(+)$-crossing in ${\mathcal D}_1$ and a vertical 
$(+)$-crossing in ${\tilde {\mathcal A}}_1$, and
\item a vertical $(-*)$-crossing in ${\mathcal D}_4$ and a horizontal
$(-*)$-crossing in  ${\tilde {\mathcal A}}_4$.
\end{itemize}
Here,  ${\tilde {\mathcal A}}_1, {\tilde {\mathcal A}}_4$ correspond to
$S((0,n-2^{j^*}),2^{j^*})$.
Thus we have 
$$
\mu_t^N\bigl( {\tilde \Delta}_2(S((0,n-2^{j^*}),2^{j^*}), S(n))\bigr)
\geq \frac{\delta_{24}\delta_{16}}{16}\frac{\delta_{48}\delta_{16}}{16}=:C^\#.
$$
By this and (\ref{eq:delta2-tilde}), we have
\begin{equation}\label{eq:delta2-tilde-2}
\mu_t^N\bigl(  {\tilde \Delta}_2(S((0,n-2^{j}),2^{j}), S(n))\bigr)
\geq C_1^{-(j^*-j)}C^\#
\end{equation}
for $j_1\leq j\leq j^*$.
Also, we have
\begin{align}\label{eq:lambda2-to-delta2-tilde}
&\mu_t^N\bigl(  {\tilde \Lambda}_2(S((0,n-2^{j+1}),2^{j+1}), S(n))\bigr) \\
&\geq C_4(\eta ) 
\mu_t^N\bigl(  {\tilde \Delta}_2(S((0,n-2^{j}),2^{j}), S(n))\bigr)\nonumber
\end{align}
if $j_1\leq j<j+1\leq j_*$, where $C_4(\eta )$ is the same constant as in 
Lemma \ref{lem:lambda-to-delta-tilde}.
Finally, there exists a constant ${\tilde K}_2>0$, depending only on 
$\varepsilon_0, j_1$ and $\eta $, such that
\begin{align}\label{eq:gamma2-to-delta2-tilde}
&\mu_t^N\bigl(  {\tilde \Gamma}_2(S((0,n-2^{j}),2^{j}), S(n))\bigr) \\
& \leq {\tilde K}_2 
\mu_t^N\bigl(  {\tilde \Delta}_2(S((0,n-2^{j}),2^{j}), S(n))\bigr)\nonumber
\end{align}
if $j_1+3\leq j \leq j^*$.

Using these estimates, we will obtain the upper bound 
in (\ref{eq:2 arm bound for boxes}).
By the mixing property we have 
\begin{align}\label{eq:7-8}
&\mu_t^N\bigl( \Delta_2(S(0,n-2^j),2^j)\cap 
 {\tilde \Delta}_2(S((0,n-2^{j+1}),2^{j+1}),S(n)) \bigr) \\
&\geq \left\{ \mu_t^N\bigl( \Delta_2(S(2^j)) \bigr) 
             -C2^{2(j+1)}2^j e^{-\alpha 2^j} \right\} \nonumber \\
&\quad \times \mu_t^N\bigl( {\tilde \Delta}_2(S((0,n-2^{j+1}),2^{j+1}), S(n))\bigr) .
\nonumber
\end{align}
By (\ref{eq:gamma2-to-delta2-tilde}),
\begin{align}\label{eq:7-9}
&  \mu_t^N\bigl( {\tilde \Delta}_2(S((0,n-2^{j+1}),2^{j+1}), S(n)) \bigr) \\
& \geq {\tilde K}_2  \mu_t^N\bigl(
  {\tilde \Gamma}_2(S((0,n-2^{j+1}),2^{j+1}), S(n)) \bigr)  \nonumber \\
&\geq {\tilde K}_2\mu_t^N\bigl(
   {\tilde \Gamma}_2(S((0,n-R),R), S(n)) \bigr) .\nonumber
\end{align}
Further, by (\ref{eq:7-6}), we have
\begin{align*}
\mu_t^N\bigl( \Delta_2( S(2^j))\bigr) &\geq K^{-1}\mu_t^N\bigl( 
                          \Gamma_2( S(2^j)) \bigr) \\
         &\geq K^{-1}\mu_t^N \bigl( {\mathcal E}_2(2^j) \bigr)\\
         &\geq K^{-1}C_72^{-j}.
\end{align*}
Therefore if $j_2$ is sufficiently large such that
\begin{equation}\label{eq:bound for j2-2}
K^{-1}C_72^{-j-1}> C2^{3(j+1)}e^{-\alpha 2^j} \quad \mbox{for } j\geq j_2,
\end{equation}
then the right hand side of (\ref{eq:7-8}) is not less than 
$$
\frac{{\tilde K}_2}{2K}\mu_t^N \bigl({\mathcal E}_2(2^j) \bigr) \mu_t^N\bigl(
 {\tilde \Gamma}_2(S((0,n-R),R), S(n)) \bigr).
$$
On the other hand, by the Connection lemma, we have
\begin{align*}
&\mu_t^N\bigl( \Gamma_2(S(n))\bigr) \\
&\geq \mu_t^N\bigl(
\Delta_2(S(0,n-2^j),2^j) \cap 
{\tilde \Delta}_2(S(0,n-2^{j+1}),2^{j+1}), S(n) \bigr) \\
&\quad \times \biggl(\frac{\delta_{40}\delta_{16}}{16}\biggr)
 \biggl( \frac{\delta_{52}\delta_{16}}{16}\biggr) .
\end{align*}
Thus, we have
$$
C_8n^{-1} \geq C_1\frac{{\tilde K}_2}{2K}C_7 2^{-j}
\mu_t^N\bigl( {\tilde \Gamma}_2(S((0,n-R),R), S(n))\bigr).
$$
Since $2^j\leq R<2^{j+1}$, this implies the desired inequality.

\subsection{Power law estimate for three-arm paths in half space}
As we explained in the introduction, our estimate for 3 arm paths is
restricted to a special case:
We define the event ${\tilde {\mathcal B}}^+(2,1,R,n)$ to be the event 
such that two $(+)$-paths ${\tilde r_1}, {\tilde r_2}$ and one $(-*)$-path
${\tilde r}_3^*$ connect $\partial S(R)$ with $\partial_{in }S(n)$
in $( S(n) \setminus S(R) ) \cap \{ x^2\leq 0\} $, and that
these $(+)$-paths are separated by the $(-*)$-path ${\tilde r}_3^*$
in $( S(n) \setminus S(R) ) \cap \{ x^2\leq 0\} $.

Let  $\mathcal{E}_3(n)$ be the event that in the interior of 
$S(n)$ the following hold:
\begin{enumerate}
\item The spin value at the top center point $(0,n)$ is $+$.
\item There is a $(+)$-path $r_1$ from $(-1,n)$ to $\partial_{in} S(n)$.
\item There is a $(+)$-path $r_3$ from $(1,n)$ to $\partial_{in} S(n)$.
\item There is a $(-*)$-path $r_4^*$ from $(0,n-1)$ to 
$\partial_{in} S(n)$.
\end{enumerate}
Note that in the event  $\mathcal{E}_3(n)$, $r_1, r_3$ are disjoint and 
are separated by $r_4^*$.
Let also ${\mathcal E}_3^*(n)$ be the event such that the roles of
$(+)$- and $(-*)$-paths are interchanged in the above definition.
The aim of this subsection is to prove the following theorem.
\begin{Theorem}\label{3 arm bound}
There exist positive constants $C_{11}$--$C_{14}$ and an integer
$j_3> j_1$, depending only on $\varepsilon_0, j_1$ and $\eta $, such that
\begin{equation}\label{eq:3 arm bound}
C_{11} n^{-2}\leq \mu_t^N \bigl( \mathcal{E}_3(n) \bigr) \leq C_{12}n^{-2}
\end{equation}
for $2^{j_1+5}\leq n < \min\{ L(h,\varepsilon_0), \frac{N}{2}\} $, and
\begin{equation}\label{eq:3 arm bound for boxes}
C_{13}{\biggl( \frac{R}{n}\biggr) }^2 \leq 
 \mu_t^N\bigl( {\tilde {\mathcal B}}^+(2,1,R,n) \bigr) \leq 
C_{14}{\biggl( \frac{R}{n}\biggr) }^2
\end{equation}
for $2^{j_3}\leq R<n<2^k\leq \min\{ L(h,\varepsilon_0), \frac{N}{2} \} $.
The estimate (\ref{eq:3 arm bound}) is valid for ${\mathcal E}_3^*(n)$
in place of ${\mathcal E}_3(n)$, too.
The estimate (\ref{eq:3 arm bound for boxes}) is valid for 
${\tilde {\mathcal B}}^+(1,2,R,n)$, too.
\end{Theorem}

\begin{proof}
Proof of (\ref{eq:3 arm bound for boxes}) from (\ref{eq:3 arm bound})
is analogous to the proof of (\ref{eq:2 arm bound for boxes}) from 
(\ref{eq:2 arm bound}) by using an extension argument.
So we only prove (\ref{eq:2 arm bound}).

Let ${\mathcal R}={\mathcal R}(\omega )$ denote
the lowest horizontal $(+)$-crossing.
If ${\mathcal R}(\omega )$ exists, then we can find some point 
$v\in {\mathcal R}$ such that
\begin{enumerate}
\item $v$ is the highest point in ${\mathcal R}$, and
\item $v+(-1,0), v+(1,0)\in {\mathcal R}$.
\end{enumerate}
We call such $v$ as a central highest point of ${\mathcal R}$.
Note that if $v$ is the unique central highest point of ${\mathcal R}$, 
then since ${\mathcal R}$ is the lowest $(+)$-crossing, there is a 
$(-*)$-path from $v+(0,-1)$
to the bottom of $S(n)$.
Obviously,
\begin{equation}\label{eq:central highest point}
\mu_t^N\bigl(
\mbox{there is a unique central highest point of } {\mathcal R}
\bigr) \leq 1.
\end{equation}
For $x \in S(n/2)$, put
$$
H(x,n)=\left\{ 
\mbox{$v$ is the unique central highest point}
\right\} .
$$
Since $H(x,n)$ are disjoint, we have
$$
\sum_{x\in S(n/2)} \mu_t^N\bigl( H(x,n)\bigr) \leq 1.
$$
Let ${\hat H}(n)$ be the event such that all the following occur:
\begin{enumerate}
\item The spin value at $(0,n)$ is $+$.
\item There exist  $(+)$-paths $\gamma_1, \gamma_3$ in 
$[-2n,2n]\times [n/2, n]$ such that $\gamma_1$ connects 
$(-1,n)$ with the left side of $S(2n)$, and $\gamma_3$ 
connects $(1,n)$ with the right side of $S(2n)$.
\item $\gamma_1\setminus \{ (-1,n)\} $ does not intersect $\{ x^2=n\} $.
\item $\gamma_3\setminus \{ (1,n)\} $ does not intersect $\{ x^2=n\} $.
\item There exists a $(-*)$-path $\gamma_4^*$ in $[-n/2,n/2]\times [-n,n]$
connecting $(0,n-1)$ with the bottom side of $S(n)$.
\end{enumerate}
By definition ${\hat H}(n)$ is a subset of ${\mathcal E}_3(n)$.
Then for $x \in S(n/2)$ let ${\hat H}(x,n)$ be the $x$ translation of 
the event ${\hat H}(n)$ so that it is an event occurring in
$ [-2n,2n]\times [-n,n] +x $.
Note that for $x\in S(n/2)$, ${\hat H}(x,n)\subset H(x,n)$.
Therefore we have
\begin{equation}\label{eq:7-10}
\frac{n^2}{4}\mu_t^N \bigl( {\hat H}(n) \bigr)\leq 1.
\end{equation} 
In order to compare $\mu_t^N$-probabilities of ${\mathcal E}_3(n)$ and
${\hat H}(n)$, we use the extension argument.
Let us rewrite $\Gamma_3(S(2^j)):= {\mathcal E}_3(2^j)$ for $j\geq j_1$.
Then let $\Lambda_3(S(2^j))$ be the subset of $\Gamma_3(S(2^j))$ such that
every crossing $(+)$-cluster and every crossing $(-*)$-cluster
in  ${\mathcal B}_i$ has an $(\eta , j)$-fence for $i=1,3,4$,
where ${\mathcal B}_i$'s correspond to $S(2^j)$.
Also, let $\Delta_3(S(2^j))$ be the subset of $\Gamma_3(S(2^j))$ such that
all the following occur.
\begin{enumerate}
\item $r_1$ connects $(-1,n)$ with the left side of $S(2^j)$, and
$r_3$ connects $(1,n)$ with the right side of $S(2^j)$. 
\item $r_4^*$ connects $(0,n-1)$ with the bottom side of $S(2^j)$.
\item $r_i \cap (\cup_{j=1,3,4} {\mathcal B}_j) \subset {\mathcal A}_i$
for $i=1,3$.
\item $r_4^*\cap (\cup_{j=1,3,4} {\mathcal B}_j) \subset {\mathcal A}_4$.
\end{enumerate}
Here ${\mathcal A}_i$'s correspond to $S(2^j)$, too.
Then by the inwards extension argument we have
\begin{equation}\label{eq:7-11-a}
\mu_t^N \bigl( \Delta_3(S(2^j)) \bigr) \leq C_1 \mu_t^N( \Delta_3(S(2^{j+1})))
\end{equation}
for $2^{j_1}\leq 2^j<2^{j+1}<\min\{ L(h,\varepsilon_0), \frac{N}{2}\} $,
and as a result we have
$$
\mu_t^N \bigl( \Delta_3(S(2^j)) \bigr) \geq C_1^{-(j-j_1)}C_2,
$$
where we can take 
$$
C_2= (1+e^{8/\mathfrak{K}T})^{-\# S(2^{j_1})}.
$$
Also, we have
$$
\mu_t^N\bigl( \Lambda_3(S(2^j))\bigr)\leq C_3(\eta ) 
\mu_t^N\bigl( \Delta_3(S(2^{j+1}))\bigr)
$$
for $2^{j_1}\leq 2^j<2^{j+1}< \min\{ L(h,\varepsilon_0), \frac{N}{2}\}$,
and
\begin{equation}\label{eq:7-11-c}
\mu_t^N\bigl( \Gamma_3(S(2^j)) \bigr)\leq K
\mu_t^N\bigl( \Delta_3(S(2^j))\bigr)
\end{equation}
for $2^{j_1+3}\leq 2^j < \min\{ L(h,\varepsilon_0), \frac{N}{2}\}$.
The constants $C_1, C_3(\eta )$ and $K$ are the same constants as in
Lemmas \ref{relation of deltas}, \ref{lambda-to-delta} and \ref{comparison of gamma and delta}, respectively.
Let $n\geq 2^{j_1+5}$ and let $2^j\leq n<2^{j+1}$.
Let also $E$ be the event such that
\begin{enumerate}
\item there exists a horizontal $(+)$-crossing in the rectangle
$$ [-2n, -2^j +2^{j-2}]\times [n-2^j-2^{j-2},n-2^j+2^{j-2}],$$
and a horizontal $(+)$-crossing in 
$$
[ 2^j-2^{j-2}, 2n]\times [n-2^j-2^{j-2},n-2^j+2^{j-2} ],
$$ 
and
\item there exists a vertical $(-*)$-crossing in
$$
[ -2^{j-2}, 2^{j-2}]\times [n-2^{j+1}+2^{j-2}, -n].
$$
\end{enumerate}
Then by the Connection lemma we have
$$
\mu_t^N( E\cap \Delta_3(S(2^j) )\bigr) \geq 
{\left( \frac{\delta_{52}}{4}\right) }^2
\frac{\delta_{36}}{4} \mu_t^N\bigl( \Delta_3(S(2^j))\bigr) .
$$
Since $ E\cap \Delta_3(S(2^j))$ is a subset of ${\hat H}(n)$,
writing the constant in the right hand side in the above inequality by $C^\#$,
we obtain
\begin{equation}\label{eq:7-12}
\mu_t^N \bigl( {\hat H}(n) \bigr) \geq C^\# \mu_t^N \bigl( \Delta_3(2^j) \bigr)
\geq C^\#  K^{-1}\mu_t^N\bigl( \Gamma_3(S(2^j)) \bigr) .
\end{equation}
Since
$$
\Gamma_3(S(2^j))={\mathcal E}_3(2^j)\supset {\mathcal E}(n),
$$
from (\ref{eq:7-10}) and (\ref{eq:7-12}) we obtain the upper bound 
in (\ref{eq:3 arm bound}).

Now we turn to the proof of the lower bound in (\ref{eq:3 arm bound}).
First we show that there is a constant $C_2^\# >0$ depending only on
$\varepsilon_0$ such that
\begin{equation}\label{eq:7-13}
\mu_t^N\left(
\begin{array}{@{\,}c@{\,}}
\mbox{the highest point of the lowest $(+)$-crossing  ${\mathcal R}$}\\
\mbox{in $S(n)$ is only in $S(n/2)$}
\end{array}
\right) \geq C_2^\# .
\end{equation} 
To do this let us introduce the following rectangles:
$$
\begin{array}{ll}
T=[ -\frac{n}{8}, \frac{n}{8}]\times [-n, \frac{n}{8}], &
U_1=[-n, \frac{3n}{8}]\times [-\frac{n}{8},0], \\
U_2=[-\frac{n}{2}, -\frac{3n}{8}]\times [-\frac{n}{8},\frac{n}{2}], &
U_3=[-\frac{n}{2},\frac{n}{2}]\times [ \frac{3n}{8},\frac{n}{2}], \\
U_4=[\frac{3n}{8},\frac{n}{2}]\times  [-\frac{n}{8},\frac{n}{2}], &
U_5=[\frac{3n}{8},n]\times  [-\frac{n}{8},0].
\end{array}
$$
Then let $U$ be the corridor made up by the union of $U_1$--$U_5$.
By the Connection lemma, we have
$$
\mu_t^N \left(
\begin{array}{@{\,}c@{\,}}
\mbox{there exists a $(+)$-path in $U$ connecting}\\
\mbox{$\{ x^1=-n\} $ with $\{ x^1=n\} $, and}\\
\mbox{there exists a vertical $(-*)$-crossing in $T$}
\end{array}
\right) \geq \frac{\delta_{224}\delta_{36}}{16}=:C_2^\# .
$$
It is clear when the above event occurs, any of the highest points of ${\mathcal R}$
is in $S(n/2)$.
Let ${\tilde H}(x,n)$ be the event such that $x$ is a
central highest point
of the lowest $(+)$-crossing of $S(n)$. 
Then the above inequality says that
$$
\mu_t^N\left( \bigcup_{x \in S(n/2)}{\tilde H}(x,n) \right)\geq C_2^\# .
$$
By changing configurations in $(*)$-neighbors of $x$, we can obtain
the event $H(x+(0,1),n)$. 
By the finite energy property we have an absolute constant $C_3^\# >0$
such that
$$
\mu_t^N\bigl( {\tilde H}(x,n) \bigr) \leq C_3^\# \mu_t^N\bigl(
            H(x+(0,1)) \bigr) .
$$
Therefore we have
$$
\sum_{x\in S(n/2)}\mu_t^N\bigl( H(x+(0,1),n)\bigr)
\geq  (C_3^\# )^{-1} C_2^\# .
$$
Note that for $x\in S(n/2)$, 
$$
H(x+(0,1),n) \subset \Gamma_3(S(x+(0,1-\tfrac{n}{4}), \tfrac{n}{4})),
$$
where $\Gamma_3(S(a,r)) $ denotes the $a$-translation of $\Gamma_3(S(r))$.
Thus, by translation invariance we have
$$
\mu_t^N\bigl( H(x+(0,1),n) \bigr) \leq \mu_t^N\bigl(
   {\mathcal E}_3(n/4) \bigr) .
$$
This proves the lower bound.
\end{proof}

\newpage
\section{Ising counterpart of Russo's formula}

\subsection{Russo's formula}

In this section we present a version of Russo's formula for the Ising model.
But we have to restrict  events  for which our version of Russo's formula
can be applied, so we do not think that the present form is close to the final
form.
Still we can handle events with this version to obtain our main results.

Before stating the result, we introduce some notations.
We assume that $1<2^k<N$.
For $v\in S(2^k)$ and $\omega \in \Omega $, let $\omega^v \in \Omega$ denote the
configuration obtained from $\omega$ by flipping the spin at $x$:
$$
\omega^v(x)=\begin{cases}
       \omega (x), & x\not=v; \\
       -\omega (v), & x=v.
\end{cases}
$$

\begin{Definition}\label{teigi:pivotal sites}
Let $A\in{\mathcal F}_{S(2^k)}$ and $v\in S(2^k)$.
We say that $v$ is {\it pivotal} for $A$ in the configuration $\omega $
if 
$$1_A(\omega  )+1_A(\omega^v)=1,$$
i.e., $\omega $ and $\omega^v$ do not belong at the same time to either
$A$ or $A^c$.
\end{Definition}
Let
$$
\Delta_vA = \{\omega \in \Omega : v \mbox{ is pivotal for } A 
     \mbox{ in } \omega \} ,
$$
and
$$
\square_vA=\{ \omega \in A : v \mbox{ is not pivotal for } A 
     \mbox{ in } \omega \} .
$$
Then we extend this notation to squares $S$.
For $V\subset \mathbf{Z}^2$, and $\xi , \omega \in \Omega $ , let
$$
\xi_V\omega (x) = \begin{cases}
    \xi (x), & x\in V; \\
    \omega (x), & x\not\in V.
\end{cases}
$$ 
Then we put
$$
\Delta_S A =\left\{ \omega \in \Omega : 
 \begin{array}{@{\,}c@{\,}}
\mbox{there exist $\xi , \zeta \in\Omega $ such that}\\
\mbox{$\xi_S\omega \in A$ and $\zeta_S\omega \not\in A $}
\end{array}\right\},
$$
and
$$
\square_S A = \left\{ \omega \in A : \xi_S\omega \in A \mbox{ for any }
             \xi \in \Omega \right\} .
$$
It is clear that $A \setminus \Delta_SA = \square_SA$.
Similar to (6.1), %(\ref{eq:m(v)}),
for $1\leq j <k$, let
$$
m_j(v) :=\min\{ m\geq 1 : S_j^m(v)\supset S(v,2^{-3}|v|_\infty )
 \} ,
$$
where ${\tilde R}_{j}^m(v)$ is defined as in section 5, starting from 
$Q_j(v)$'s.
Finally, we introduce the support $ \text{Supp}(A)$ of the event $A$ by
$$
\text{Supp}(A) := \{ v \in {\mathbf Z}^2  :  \Delta_vA\not= \emptyset \} .
$$
\begin{Theorem}\label{russo formula-gen}
Let $2^k<\frac{N}{2}$ and $A \in {\mathcal F}_{S(2^k)}$.
Assume that $A$ satisfies the following condition (S).

\medskip
\noindent
Condition (S)

There exists $1\leq j_0<k-3, 0<a <\alpha /2$ and $b>0$ 
such that for every $v \in \text{Supp}(A)$,          % we can find a
  %%$v_0\in S(v,2^{j_0})\cap S(n)$ for which the following estimate holds:
\begin{equation}\label{eq:condition(S)}
\mu_t^N\bigl( \square_{{\tilde R}_{j_0}^m(v)}A\setminus
             \square_{{\tilde R}_{j_0}^{m+1}(v)}A \bigr)
\leq be^{a2^{j_0+m}}\mu_t^N\bigl( A \cap \Delta_{v}A \bigr) 
\end{equation}
for every $t\in [0,1]$ and every $2\leq m \leq m_{j_0}(v)-1$.

\medskip
\noindent
Then there exist positive constants $C_{15}$ and $C_{16}$, which depend on
$a, b$ and $j_0$ such that for every $N\geq 2^{k+1}$,
\begin{equation}\label{eq:russo formula-gen}
\left| \frac{d}{dt}\mu_t^N\bigl(A) \right| \leq \frac{|h-h_c|}{\mathfrak{K}T}
\left[ C_{15}\sum_{v\in \text{Supp}(A)}
\mu_t^N( A\cap \Delta_vA ) +C_{16}\mu_t^N(A) \right].
\end{equation}
\end{Theorem}
Before going into the proof of this theorem, we remark that
there is a primitive form of 
Ising version of Russo's formula which says that   
\begin{equation}\label{eq:primitive russo formula}
\frac{d}{dt} \mu_t^N (A) = \frac{|h-h_c|}{\mathfrak{K}T} \sum_{v \in S(N)} 
E_{\mu_t^N} [ \{ \omega(v) - E_{\mu_t^N} [\omega(v)] \} : A]
\end{equation}
for every $A\in {\mathcal F}_{S(N)}$.
This can be obtained by direct differentiation.

\begin{proof}[Proof of Theorem \ref{russo formula-gen}] 
For $v\in S(2^k)$, put
$$
z_v=E_{\mu_t^N} [ \{ \omega(v) - E_{\mu_t^N} [\omega(v)] \} : A].
$$

If $|v|_\infty \leq 2^{j_0+5}$, then we simply use the fact that 
$|\omega (v)|\leq 1$ to obtain
\begin{equation}\label{eq:8-0 crude}
|z_v| \leq 2\mu_t^N( A).
\end{equation}
If $v \in \text{Supp}(A)\setminus S(2^{j_0+5})$, then $m_0:=m_{j_0}(v)\geq 2$. 
Indeed, if $m_0=1$, then $S_{j_0}^1(v)\supset S(v,2^{-3}|v|_\infty ) $,
and this implies that $2^{j_0+1}\geq 2^{-3}|v|_\infty $, i.e.
$$ |v|_\infty \leq 2^{j_0+4},$$
which is impossible since by the  choice of $v$.%Condition (S),
$|v-v|_\infty \leq 2^{j_0}$.
Thus, we have 
\begin{align}\label{eq:8-1 division}
z_v= &E_{\mu_t^N} [ \{ \omega(v) - E_{\mu_t^N} [\omega(v)] \} : A\cap 
                               \Delta_vA ] \\
   &+E_{\mu_t^N} [ \{ \omega(v) - E_{\mu_t^N} [\omega(v)] \} :
      \square_vA \setminus \square_{{\tilde R}_{j_0}^2(v)}A ] \nonumber\\
   &+\sum_{m=2}^{m_0-1}
     E_{\mu_t^N} [ \{ \omega(v) - E_{\mu_t^N} [\omega(v)] \} : H_m(v)] 
            \nonumber \\
   &+E_{\mu_t^N} [ \{ \omega(v) - E_{\mu_t^N} [\omega(v)] \} :
      \square_{{\tilde R}_{j_0}^{m_0}(v)}A ], \nonumber
\end{align}
where we put
$$
H_m(v)=\square_{{\tilde R}_{j_0}^{m}(v)}A \setminus
       \square_{{\tilde R}_{j_0}^{m+1}(v)}A.
$$
For the first term in the right hand side of 
(\ref{eq:8-1 division}),
we also use the trivial estimate
\begin{equation}\label{eq:8-2 first term}
\left| E_{\mu_t^N} [ \{ \omega(v) - E_{\mu_t^N} [\omega(v)] \} : A\cap 
                               \Delta_vA ] \right| \leq
2\mu_t^N( A \cap  \Delta_vA).
\end{equation}
For the second term, we also have
$$
\left| E_{\mu_t^N} [ \{ \omega(v) - E_{\mu_t^N} [\omega(v)] \} :
      \square_vA \setminus \square_{{\tilde R}_{j_0}^2(v)}A ] \right|
\leq 2 \mu_t^N( A \setminus \square_{{\tilde R}_{j_0}^2(v)}A ).
$$
Note that if $\omega \in A \setminus \square_{{\tilde R}_{j_0}^2(v)}A$,
there exists a $\xi \in \Omega $ %and a $v_1\in {\tilde R}_{j_0}^2(v)$
such that $\xi_{{\tilde R}_{j_0}^2(v)}\omega \in A\cap\Delta_{v}A$
by changing configurations point by point.
Therefore by the finite energy property, we can find a constant
$C^\#_1>0$, which depends only on $j_0$, such that
\begin{equation}\label{eq:8-3 second term}
\left| E_{\mu_t^N} [ \{ \omega(v) - E_{\mu_t^N} [\omega(v)] \} :
      \square_vA \setminus \square_{{\tilde R}_{j_0}^2(v)}A ] \right|
\leq C_1^\# %\sum_{v_1\in {\tilde R}_{j_0}^2(v)} 
 \mu_t^N( A \cap \Delta_{v}A).
\end{equation}
As for the third term, we can assume that $m_0=m_{j_0}(v)\geq 3$.
By the mixing property, 
\begin{equation}\label{eq:8-4 third term-1}
|E_{\mu_t^N} [ \{ \omega(v) - E_{\mu_t^N} [\omega(v)] \} : H_m(v)] |
\leq Cde^{-\alpha d}\mu_t^N \bigl( H_m(v) \bigr),
\end{equation}
where we put $d=d_\infty (v, S(2^k)\setminus {\tilde R}_{j_0}^m(v))$.
Then by the condition (S), 
\begin{equation}\label{eq:8-4 third term-2}
\mu_t^m(H_m(v)) \leq b e^{a2^{j_0+m}}\mu_t^N(A\cap \Delta_{v}A).
\end{equation} 
Note that 
%%$
%%d=d_\infty (v , S(2^k)\setminus {\tilde R}_{j_0}^m(v))\leq 2^{j_0+m+1}$,
%%and that
%%$$
%%d\geq 2^{j_0+m}-d_\infty (v,v)-d_\infty (v,x_{j_0}(v))
%% \geq 2^{j_0+m}-2^{j_0}-2^{j_0}\geq 2^{j_0+m-1},
%%$$
$$ 
2^{j_0+m-1}<2^{j_0+m}-2^{j_0}\leq d < 2^{j_0+m+1}
$$
since $m\geq 2$.
This, together with (\ref{eq:8-4 third term-1}) and (\ref{eq:8-4 third term-2})
implies that
$$
|E_{\mu_t^N} [ \{ \omega(v) - E_{\mu_t^N} [\omega(v)] \} : H_m(v)] |
\leq 2Cb2^{j_0+m}e^{-(\alpha /2 -a)2^{j_0+m}}\mu_t^N(A\cap \Delta_{v}A).
$$
Therefore there exists a constant $C_2^\# >0$, depending only on
$j_0,a$ and $ b$, such that
\begin{equation}\label{eq:8-4 third term}
\sum_{m=2}^{m_0-1}
|E_{\mu_t^N} [ \{ \omega(v) - E_{\mu_t^N} [\omega(v)] \} : H_m(v)] | 
\leq  C_2^\# \mu_t^N( A \cap \Delta_{v}A).
\end{equation} 
Finally, as for the fourth term in (\ref{eq:8-1 division}), 
we can assume that $m_0\geq 2$ and therefore 
$$
2^{j_0+m_0+1}\geq d_\infty (v, S(2^k)\setminus {\tilde R}_{j_0}^{m_0}(v)) 
\geq 2^{j_0+m_0}-2^{j_0}\geq 2^{j_0+m_0-1}.
$$
Also by definition,
$$
2^{j_0+m_0}\geq 2^{-3}|v|_\infty \geq 2^{j_0+m_0-1}-2^{j_0}
\geq 2^{j_0+m_0-2},
$$
and hence we have 
%%$ 2^{-3}|v|_\infty \leq 2^{j_0+m_0} \leq 2^{-1}|v|_\infty $.
$ 2^{j_0+m_0+1}\leq |v|_\infty \leq 2^{j_0+m_0+4}$.
%%Since $|v|_\infty \geq 2^{j_0+5}$,
%%this implies that
%%$$
%%2^{j_0+m_0}\leq |v|_\infty -2^{j_0}\leq |v|_\infty \leq |v|_\infty +2^{j_0}
%% \leq 2^{j_0+m_0+4}.
%%$$
Therefore by the mixing property we have
\begin{align}\label{eq:8-4 fourth term}
|E_{\mu_t^N} [ \{ \omega(v) - E_{\mu_t^N} [\omega(v)] \} : \square_{{\tilde
R}_{j_0}^{m_0}(v)}A ]| &\leq 
C2^{j_0+m_0+1}e^{-\alpha 2^{j_0+m_0-1}}\mu_t^N(A) \\
&\leq C|v|_\infty e^{-\alpha 2^{-5}|v|_\infty }\mu_t^N(A). \nonumber
\end{align}
Combining (\ref{eq:8-0 crude})--(\ref{eq:8-4 fourth term}), we obtain
\begin{equation}\label{eq:8-4 sum inside}
\sum_{v\in \text{Supp}(A)}|z_v|
\leq C_3^\# \sum_{v\in \text{Supp}(A)} \mu_t^N\bigl( A \cap \Delta_vA\bigr)
+ C_4^\# \mu_t^N(A)
\end{equation}
for some positive constants $C_3^\#, C_4^\#$.
Here, $C_3^\# $ depends on $a,b$ and $j_0$, and $C_4^\#$ depends on $j_0$.

Now we turn to the sum of $z_v$'s for $v\not\in \text{Supp}(A)$.
We only treat the case where $0\leq v^2\leq v^1$, but other cases are treated in the
same way by symmetry.
%%Take arbitrarily a point $w\in \partial_{in}S(2^k)$ and let $w_0$ be the 
%%point in $S(2^k)\cap S(w,2^{j_0})$, specified by the condition (S) for
%%$w$.
%%If $w^2\not=2^k$, then we take $v\not\in S(2^k)$ with $v^2=w^2$.
%%If $w^2=2^k$, then we take $v\not\in S(2^k)$ with $v^1,v^2\geq 2^k$.
%%Then, as in (\ref{eq:8-1 division}) we have
Let $w_0=w_0(A,v)$ denote a point in $\text{Supp}(A)$, such that
$$
| w_0-v|_\infty = d_\infty (v, \text{Supp}(A)).
$$
If there are many of such points, we choose one of them in a specific way, 
say the youngest point in the lexicographic order.
Then we have
\begin{align}\label{eq:8-5 division}
z_v= & E_{\mu_t^N}[  \{ \omega(v) - E_{\mu_t^N} [\omega(v)] \} : 
           A \setminus \square_{{\tilde R}_{j_0}^2(w_0)}A ] \\
   &+ \sum_{m=2}^{m_{j_0}(w_0)-1}
E_{\mu_t^N} [ \{ \omega(v) - E_{\mu_t^N} [\omega(v)] \} : H_m(w_0)] \nonumber \\
   &+ E_{\mu_t^N}  [ \{ \omega(v) - E_{\mu_t^N} [\omega(v)] \} :
         \square_{{\tilde R}_{j_0}^{m_{j_0}(w_0)}(w_0)}A ]. \nonumber
\end{align}
By the mixing property
$$|E_{\mu_t^N}[  \{ \omega(v) - E_{\mu_t^N} [\omega(v)] \} : 
           A \setminus \square_{{\tilde R}_{j_0}^2(w_0)}A ]|
\leq Cd_ve^{-\alpha d_v}
\mu_t^N( A \setminus \square_{{\tilde R}_{j_0}^2(w_0)}A),
$$
where $d_v= d_\infty ( v, \text{Supp}(A))$.
By the finite energy property 
the right hand side of the above inequality can be replaced with
$$
C_5^\# d_v e^{-\alpha d_v}%\sum_{v_1 \in {\tilde R}_{j_0}^2(w_0)}
\mu_t^N( A\cap \Delta_{w_0}A),
$$
where $C_5^\#>0$ is %an absolute constant:
a constant depending only on $j_0$.
$$
C_5^\# =C e^{(4h_c+8)/\mathfrak{K}T\# ({\tilde R}_{j_0}^2(0))}.
$$ 
This estimates the first term in the right hand side of (\ref{eq:8-5 division}).
By the same argument as the one to obtain (\ref{eq:8-4 third term}), we have
$$
|E_{\mu_t^N} [ \{ \omega(v) - E_{\mu_t^N} [\omega(v)] \} : H_m(w_0)]|
\leq bC d_m(v)e^{-\alpha d_m(v)+a2^{j_0+m}}\mu_t^N(A\cap \Delta_{w_0}A),
$$
where we put 
\begin{align*}
d_m(v)&:= d_\infty (v, \text{Supp}(A)\setminus {\tilde R}_{j_0}^{m}(w_0) )\\
    &\geq 
 \max\{ d_v, d_\infty (w_0, \text{Supp}(A)\setminus {\tilde R}_{j_0}^m(w_0) ) -d_v\} .
\end{align*}

%%Since $m \geq 2$, we have
%%$d(w,S(2^k)\setminus {\tilde R}_{j_0}^m(w_0))\geq 2^{j_0+m}-2^{j_0+1}\geq 2^{j_0+m-1}$,
%%and
Note that $ d_m(v)\leq  d_v + 2^{j_0+m} $ and that
$$
-\alpha d_m(v) + a2^{j_0+m}\leq - \left( \frac{\alpha }{2}-a \right) \max\{ 2d_v, 2^{j_0+m}
\} + \alpha 2^{j_0-1}.
$$
Therefore we have
\begin{align*}
&|E_{\mu_t^N} [ \{ \omega(v) - E_{\mu_t^N} [\omega(v)] \} : H_m(w_0)]| \\
&\leq 2^{\alpha 2^{j_0-1}}bC ( d_v+ 2^{j_0+m}) e^{-(\frac{\alpha }{2}-a)
   \max\{ 2d_v, 2^{j_0+m} \} } \mu_t^N(A \cap \Delta_{w_0}A). 
\end{align*}  
Summing up this over $m$'s with $2\leq m \leq m_{j_0}(w_0)-1$, we can find a
constant $C_6^\# >0$ depending only on $j_0, a,b$, such that
\begin{align}\label{eq:8-5 second term}
&\sum_{m=1}^{ m_{j_0}(w_0)-1}
|E_{\mu_t^N} [ \{ \omega(v) - E_{\mu_t^N} [\omega(v)] \} : H_m(w_0)]| \\
&\leq C_6^\# d_v^2e^{-(\frac{\alpha }{2}-a)d_v}\mu_t^N( A \cap \Delta_{w_0}A).
\nonumber 
\end{align}
By the mixing property, 
\begin{equation}\label{eq:8-5 third term}
\left|E_{\mu_t^N} \left[ \{ \omega(v) - E_{\mu_t^N} [\omega(v)] \} : 
  \square_{{\tilde R}_{j_0}^{m_{j_0}(w_0)}(w_0)}A \right] \right|
\leq C \overline{d}e^{-\alpha \overline{d}}\mu_t^N(A),
\end{equation}
where $\overline{d} = d_{m_{j_0}(w_0)}(v)$.
Hence, if $|v|_\infty >2^{j_0+5}$, then we have
\begin{align*}
|z_v| &\leq  C_5^\# d_v e^{-\alpha d_v}\sum_{v_1 \in {\tilde R}_{j_0}^2(w_0)}
\mu_t^N( A\cap \Delta_{v_1}A) \\
    &\quad + C_6^\# d_v^2e^{-(\frac{\alpha }{2}-a)d_v}\mu_t^N( A \cap \Delta_{w_0}A) \\
    &\quad + C \overline{d}e^{-\alpha \overline{d}}\mu_t^N(A).
\end{align*}
Summing up this inequality over $v$'s  outside $\text{Supp}(A)$, we obtain
\begin{equation}\label{eq:8-6 sum outside}
\sum_{v \in S(N)\setminus \text{Supp}(A)} |z_v|
\leq C_7^\# \sum_{w\in \text{Supp}(A)}\mu_t^N( A \cap \Delta_{w}A)
    + C_8^\# \mu_t^N(A),
\end{equation}
where $C_7^\#>0$ depends only on $j_0, a, b$ and $C_8^\# >0$
depends only on $j_0$.

(\ref{eq:8-4 sum inside}) and (\ref{eq:8-6 sum outside}) proves the theorem.
\end{proof}

When the event $A$ is increasing we have much sharper lower bound.
\begin{Lemma} \label{RussoLower} Suppose that $2^k < N$. 
There exists a positive absolute constant $C_{17}$, 
such that for any increasing event $A \in \mathcal{F}_{S(2^k)}$,
\begin{equation}\label{eq:lower russo bound}
C_{17} \frac{|h-h_c|}{\mathfrak{K}T}
 \sum_{v\in \text{Supp}(A)}
\mu_t^N\bigl( A\cap \Delta_v A\bigr) \leq \frac{d}{dt} \mu_t^N (A). 
\end{equation}
\end{Lemma}

\begin{proof} 
Since $A$ and $\square_vA$ are increasing events, for every $v \in S(N)$,
\begin{align*}
E_{\mu_t^N} [ \{ \omega(v) -E_{\mu_t^N} [\omega(v)] \} : A ]&\geq 0, \quad
\mbox{ and }\\
E_{\mu_t^N}[ \{ \omega(v) -E_{\mu_t^N} [\omega(v)] \} : \square_vA ]&\geq 0.
\end{align*}
%%Thus we have
Noting that $A=(A\cap \Delta_vA)\cup \square_vA$ for each $v$, we have
\begin{align*}
\frac{d}{dt}\mu_t^N(A)
&\geq \frac{|h-h_c|}{\mathfrak{K}T} \sum_{v \in \text{Supp}(A)}  E_{\mu_t^N} [ \{ \omega(v)
-E_{\mu_t^N} [\omega(v)] \} : A \cap \Delta_v A]  \\
&= \frac{|h-h_c|}{\mathfrak{K}T} \{ 1 -E_{\mu_t^N} [\omega(\mathbf{O})] \} 
\sum_{v\in \text{Supp}(A)}{\mu_t^N} \bigl( A\cap \Delta_vA \bigr) .
\end{align*}
Since we assume that $0<h<2h_c(T)$, we remark that for every $t \in [0,1]$,
\begin{align*}
1- E_{\mu_t^N} [\omega(\mathbf{O})] &\geq  2\mu_t^N(\omega (\mathbf{O})=-1) \\
&\geq  2{\left( 1 + e^{(8+4h_c(T))/\mathfrak{K}T}\right) }^{-1}>0.
\end{align*}
This completes the proof of (\ref{eq:lower russo bound}).
\end{proof}
\begin{Remark}\label{RussoLower-2}
By the finite energy property, from (\ref{eq:lower russo bound})
there is an absolute constant $C_{17}'>0$
such that
$$
\frac{|h-h_c|}{\mathfrak{K}T}C_{17}'\sum_{v \in \text{Supp}(A)}
\mu_t^N\bigl( \Delta_vA \bigr) 
\leq \frac{d}{dt}\mu_t^N(A)
$$
for every increasing event $A\in{\mathcal F}_{S(2^k)}$.
This may be more useful.
\end{Remark}

\subsection{Russo's formula for crossing events and the one-arm event}\label{ss:Russo1arm}
Let us give some of examples of events for which 
Theorem \ref{russo formula-gen} is applicable.
We remark that the restriction that $2^k<L(h,\varepsilon_0)$
in this section hereafter is not serious as before.

The first example is the crossing event $A^+(2^k,2^k)$, such that
there exists a horizontal $(+)$-crossing in $S(2^k)$.  
Since this event is increasing, we have:
\begin{Corollary}\label{russo crossing}
Assume that $2^{j_1}<2^k<\min\{ L(h,\varepsilon_0), \frac{N}{2}\} $.
Then there exist positive constants $C_{18},C_{19}$, depending only
on $j_1, \varepsilon_0$ and $\eta $,
\begin{align*}
&\frac{|h-h_c|}{\mathfrak{K}T}C_{17}'\sum_{v \in S(2^k)}
\mu_t^N\bigl( \Delta_vA^+(2^k,2^k) \bigr) \\
& \leq \frac{d}{dt}\mu_t^N(A^+(2^k,2^k)) \\
&\leq \frac{|h-h_c|}{\mathfrak{K}T}
\biggl[ C_{18}\sum_{v\in S(2^k)}
\mu_t^N \bigl( A^+(2^k,2^k)\cap \Delta_vA^+(2^k,2^k) \bigr)  \\
&\quad  +C_{19}\mu_t^N \bigl(A^+(2^k,2^k) \bigr) \biggr] ,
\end{align*}
where $C_{17}'$ is the same constant as in Remark \ref{RussoLower-2}.
\end{Corollary}
\begin{proof}
We check the condition (S).
In this case, we take $j_0=j_1$. 
Let
$$
\omega \in \square_{{\tilde R}_{j_1}^m(v)}A^+(2^k,2^k)
     \setminus \square_{{\tilde R}_{j_1}^{m+1}(v)}A^+(2^k,2^k).
$$
Then there is a horizontal $(+)$-crossing $r$ of $S(2^k)$ such that 
$r\cap {\tilde R}_{j_1}^m(v)=\emptyset $, and there are $(-*)$-paths
$r_2^*, r_4^*$ in $S(2^k)\setminus {\tilde R}_{j_1}^{m+1}(v)$ such that 
$r_2^*$ connects $\partial {\tilde R}_{j_1}^{m+1}(v)$ with the top side
of $S(2^k)$, and $r_4^*$ connects $\partial {\tilde R}_{j_1}^{m+1}(v)$
with the bottom side of $S(2^k)$.
As a result $r$ intersects ${\tilde R}_{j_1}^{m+1}(v)$.
This means that 
$$
\omega \in {\tilde \Gamma}({\tilde R}_{j_1}^{m+1}(v), S(2^k)).
$$
Therefore we have to show that there are constants $a>0,b>0$ such that
\begin{equation}\label{eq:8-7}
\mu_t^N\bigl( {\tilde \Gamma}({\tilde R}_{j_1}^{m+1}(v), S(2^k))\bigr)
\leq be^{a 2^{j_1+m}}
  \mu_t^N\bigl( A^+(2^k,2^k)\cap \Delta_v A^+(2^k,2^k)\bigr) 
\end{equation}
for every $0\leq m \leq k-j_1-2$.
By Lemmas \ref{lem: relation of deltas-tilde} and \ref{lem:comparison of gamma and delta-tilde},
we have
\begin{align}\label{eq:8-8}
&\mu_t^N\bigl( {\tilde \Gamma}({\tilde R}_{j_1}^{m+1}(v), S(2^k))\bigr)\\
&\leq {\tilde K} \mu_t^N\bigl(
{\tilde \Delta} ({\tilde R}_{j_1}^{m+1}(v), S(2^k)) \bigr) \nonumber \\
&\leq {\tilde K}C_1^m\mu_t\bigl( 
  {\tilde \Delta} ({\tilde R}_{j_1}^1(v), S(2^k))  \nonumber \\  
&\leq  {\tilde K} C_1^m 
(1+ e^{(4h_c+8)/\mathfrak{K}T})^{\#{\tilde R}_{j_1}^1(v)}
\mu_t^N\bigl( \Delta_vA^+(2^k,2^k)\bigr) .\nonumber
\end{align}
The last inequality is by the finite energy property.
(\ref{eq:8-8}) proves (\ref{eq:8-7}).
\end{proof}

The next example is the one-arm event.
Let  
\[ A(2^k):=\{ \mbox{there exists a $(+)$-path from $\mathbf{O}$ to $\partial
S(2^k)$} \}. \]
and we put for $2^k<N$,
\[ \pi_t^N (2^k) := \mu_t^N\bigl( A(2^k) \bigr) . \]
\begin{Corollary} \label{RussoOneArm} There exist positive constants 
$C_{20}, C_{21}$, which depend on $j_1, \varepsilon_0$, and $\eta $
such that 
\begin{align}\label{eq:russo-one-arm}
C'_{17}\sum_{v\in S(2^k)}\mu_t^N\bigl( \Delta_v(A(2^k)\bigr) 
\leq & \frac{d}{dt} \pi_t^N (2^k)  \\
\leq & C_{20} \pi_t^N(2^k) + C_{21} \sum_{v \in S(2^k) } \mu_t^N
\bigl( \Delta_v A(2^k)\bigr) . \nonumber 
\end{align}
\end{Corollary}
\begin{proof}
Since $A(2^k)$ is increasing, the lower bound is a direct consequence of
Remark \ref{RussoLower-2}.
As for the  upper bound, we check the condition (S).
We put $j_0=j_1$, and $v=v$ except when $v=(\pm 2^k, \pm 2^k)$,
in which case we take $v$ as the corner point of $S(2^k-1)$ nearest to $v$.
Note that for each $v \in S(2^k)\setminus \{ (\pm 2^k, \pm 2^k)\} $,
$$
\square_{{\tilde R}_{j_1}^m(v)}A(2^k)\setminus 
   \square_{{\tilde R}_{j_1}^{m+1}(v)}A(2^k)\subset 
   {\tilde \Gamma}_0({\tilde R}_{j_1}^{m+1}(v), S(2^k)).
$$
Therefore by Lemmas \ref{lem:comparison of delta0s-tilde} and \ref{lem:comparison of gamma0 and delta0-tilde}, and the finite energy property,
\begin{align*}
&\mu_t^N\bigl( {\tilde \Gamma}_0({\tilde R}_{j_1}^{m+1}(v), S(2^k)) \bigr) \\
&  \leq {\tilde K}_0 C_1^{-m}
  { \bigl( 1+ e^{(4h_c+8)/\mathfrak{K}T}\bigr) }^{\# {\tilde R}_{j_1}^1(v)}
  \mu_t^N\bigl( \Delta_vA(2^k)\bigr) .
\end{align*}
Therefore the condition (S) is satisfied.
\end{proof}

\subsection{Russo's formula for four-arm paths} \label{ss:Russo4arm}
In this subsection we derive an estimate for the event 
$$
\Omega (2^k):= \Delta_0(A^+(2^k,2^k))
$$
with which we can check the condition (S).
We first note that $\Omega (2^k)$ can be written as the intersection of 
an increasing event $E_+$ and a decreasing event $E_{-*}$.
We first define $E_+$ to be the event such that
\begin{enumerate}
\item there exist $(+)$-paths $r_1, r_3$ in $S(2^k)\setminus \{ \mathbf{O}\} $,
$r_1$ connecting $\partial \{ \mathbf{O}\} $ with the left side of $S(2^k)$
and $r_3$ connecting $\partial \{ \mathbf{O}\} $ with the right side of $S(2^k)$,
and 
\item $r_1$ and $r_3$ are disjoint.
\end{enumerate}
Similarly, let $E_{-*}$ be the event such that
\begin{enumerate}
\item there exist $(-*)$-paths $r_2^*, r_4^*$ in $S(2^k)\setminus \{ \mathbf{O}\} $,
$r_2^*$ connecting  $\partial^* \{ \mathbf{O}\} $ with the top side of $S(2^k)$
and $r_4^*$ connecting $\partial^* \{ \mathbf{O}\} $ with the bottom side of $S(2^k)$,
and
\item $r_2^*$ and $r_4^*$ are disjoint.
\end{enumerate}
Let $v\in S(2^k)$ be fixed.
For $0\leq m\leq m_{j_1}(v)-1$, let 
${\tilde \Gamma}^{(0)}({\tilde R}_{j_1}^{m+1}(v), S(2^k)) $ be the event such that
all the following occur:
\begin{enumerate}
\item There exist paths $r_5,r_7$ in $S(2^k)\setminus {\tilde R}_{j_1}^{m+1}(v)$
such that $r_5$ connects $\partial {\tilde R}_{j_1}^{m+1}(v)$ with the left side 
of $S(2^k)$, and $r_7$ connects $\partial {\tilde R}_{j_1}^{m+1}(v)$ with 
right side of $S(2^k)$.
Further, 
$$
\omega (x) = +1 \quad \mbox{ for every $x\in (r_5\cup r_7)\setminus \{ \mathbf{O}\}$}.
$$
\item There exist $(*)$-paths $r_6^*, r_8^*$ in 
$S(2^k)\setminus {\tilde R}_{j_1}^{m+1}(v)$, such that $r_6^*$ connects 
$\partial^* {\tilde R}_{j_1}^{m+1}(v)$ with the top side of $S(2^k)$,
and $r_8^*$ connects $\partial^* {\tilde R}_{j_j}^{m+1}(v)$ with the bottom side
of $S(2^k)$.
Further, 
$$
\omega (x) = -1 \quad \mbox{ for every $x\in (r_6^*\cup r_8^*)\setminus \{ \mathbf{O}\} $}.
$$
\item At most one of $r_5,r_6^*,r_7,r_8^*$ passes through $\mathbf{O}$.
\end{enumerate}
Of course, if ${\tilde R}_{j_1}^{m+1}(v)\not=S_{j_1}^{m+1}(v)$, then
we allow $r_7$ to be an empty set, and if ${\tilde R}_{j_1}^{m+1}(v)\ni (2^k,2^k)$,
then we also allow $r_6^*$ to be an empty set.

Let us take ${\tilde {\mathcal B}}_i$'s for ${\tilde R}_{j_1}^{m+1}(v)$.
Then we define ${\tilde \Lambda }^{(0)}( {\tilde R}_{j_1}^{m+1}(v), S(2^k))$
as the subset of ${\tilde \Gamma}^{(0)}( {\tilde R}_{j_1}^{m+1}(v), S(2^k))$
such that any crossing $(+)$-cluster and any crossing $(-*)$-cluster in
one of ${\tilde {\mathcal B}}_i$'s has an $(\eta ,j_1+m+1)$-fence.
Let us take ${\tilde {\mathcal A}}_i$'s for ${\tilde R}_{j_1}^{m+1}(v)$,
too. 
Then let ${\tilde \Delta }^{(0)}({\tilde R}_{j_1}^{m+1}(v), S(2^k))$ be the 
subset of ${\tilde \Gamma}^{(0)}( {\tilde R}_{j_1}^{m+1}(v), S(2^k))$ such that
all the following occur:
\begin{enumerate}
\item $r_5 \cap \cup_j {\tilde {\mathcal B}}_j\subset {\tilde {\mathcal A}}_1$,
and $r_8^* \cap \cup_j {\tilde {\mathcal B}}_j\subset {\tilde {\mathcal A}}_4$.
\item $r_6^* \cap \cup_j {\tilde {\mathcal B}}_j\subset {\tilde {\mathcal A}}_2$ if
$r_6^*$ exists, and
$r_7 \cap \cup_j {\tilde {\mathcal B}}_j\subset {\tilde {\mathcal A}}_3$
if $r_7$ exists.
\item There exists a vertical $(+)$-crossing in ${\tilde {\mathcal A}}_1$,
and if we use ${\tilde {\mathcal A}}_3$, there exists a vertical
$(+)$-crossing in ${\tilde {\mathcal A}}_3$.
\item There exists a horizontal  $(-*)$-crossing in ${\tilde {\mathcal A}}_4$,
and if we use ${\tilde {\mathcal A}}_2$, there exists a horizontal 
$(-*)$-crossing in ${\tilde {\mathcal A}}_3$.
\end{enumerate}
With these definitions we can apply the outwards extension argument to
obtain the following lemma.

\begin{Lemma}\label{lem:delta-to-delta:four arm}
Let $2^{j_1}<2^k<\min\{ L(h,\varepsilon_0), \frac{N}{2}\} $.
Then we have for $t\in [0,1]$ and for every $1\leq m \leq m_{j_1}(v)-1$,
\begin{align}\label{eq:delta-to-delta:four arm-1}
&\mu_t^N\biggl(
\square_{{\tilde R}_{j_1}^{m+1}(v)}E_+\cap \Delta_{{\tilde R}_{j_1}^{m+1}(v)}E_{-*}
  \cap {\tilde \Delta}^{(0)}({\tilde R}_{j_1}^{m+1}(v), S(2^k))
    \biggr) \\
\leq &C_1 \mu_t^N\biggl(
 \square_{{\tilde R}_{j_1}^{m}(v)}E_+\cap \Delta_{{\tilde R}_{j_1}^{m}(v)}E_{-*}
  \cap {\tilde \Delta}^{(0)}({\tilde R}_{j_1}^{m}(v), S(2^k))
    \biggr) , \nonumber
\end{align}
and
\begin{align}\label{eq:delta-to-delta:four arm-2}
&\mu_t^N\biggl(
\Delta_{{\tilde R}_{j_1}^{m+1}(v)}E_+\cap \square_{{\tilde R}_{j_1}^{m+1}(v)}E_{-*}
  \cap {\tilde \Delta}^{(0)}({\tilde R}_{j_1}^{m+1}(v), S(2^k))
    \biggr) \\
\leq & C_1\mu_t^N\biggl(
 \Delta_{{\tilde R}_{j_1}^{m}(v)}E_+\cap \square_{{\tilde R}_{j_1}^{m}(v)}E_{-*}
  \cap {\tilde \Delta}^{(0)}({\tilde R}_{j_1}^{m}(v), S(2^k))
    \biggr) . \nonumber
\end{align}
Therefore if $A$ is either 
$$
\square_{{\tilde R}_{j_1}^{m}(v)}E_+\cap \Delta_{{\tilde R}_{j_1}^{m}(v)}E_{-*}
  \cap {\tilde \Delta}^{(0)}({\tilde R}_{j_1}^{m}(v), S(2^k))
$$
or
$$
\Delta_{{\tilde R}_{j_1}^{m}(v)}E_+\cap \square_{{\tilde R}_{j_1}^{m}(v)}E_{-*}
  \cap {\tilde \Delta}^{(0)}({\tilde R}_{j_1}^{m}(v), S(2^k)),
$$
then we have
\begin{equation}\label{eq:delta-to-delta:four arm-3}
\mu_t^N( A) \leq C_2^{-1}C_1^m,
\end{equation}
where $C_1$ and $C_2$ are the same constants as in %Lemma 16. 
Lemma \ref{relation of deltas}.
\end{Lemma}
\begin{proof}
By iteration we have from (\ref{eq:delta-to-delta:four arm-1}) that
\begin{align*}
&\mu_t^N\biggl(
\Delta_{{\tilde R}_{j_1}^{m}(v)}E_+\cap \square_{{\tilde R}_{j_1}^{m}(v)}E_{-*}
  \cap {\tilde \Delta}^{(0)}({\tilde R}_{j_1}^{m}(v), S(2^k))
\biggr) \\
\leq & C_1^{m-1}
 \mu_t^N\biggl(
\Delta_{{\tilde R}_{j_1}^{1}(v)}E_+\cap \square_{{\tilde R}_{j_1}^{1}(v)}E_{-*}
  \cap {\tilde \Delta}^{(0)}({\tilde R}_{j_1}^{1}(v), S(2^k))
\biggr) .
\end{align*}
Then by the finite energy property we can change the configuration inside the box
${\tilde R}_{j_1}^{1}(v)$ to obtain
$$
\mu_t^N\biggl(
\Delta_{{\tilde R}_{j_1}^{1}(v)}E_+\cap \square_{{\tilde R}_{j_1}^{1}(v)}E_{-*}
  \cap {\tilde \Delta}^{(0)}({\tilde R}_{j_1}^{1}(v), S(2^k))
\biggr)
\leq C_2^{-1}C_1^m,
$$
which proves (\ref{eq:delta-to-delta:four arm-3}).
Let us prove  (\ref{eq:delta-to-delta:four arm-1}).
Assume that 
$$
\omega \in 
 \square_{{\tilde R}_{j_1}^{m+1}(v)}E_+\cap \Delta_{{\tilde R}_{j_1}^{m+1}(v)}E_{-*}
  \cap {\tilde \Delta}^{(0)}({\tilde R}_{j_1}^{m+1}(v), S(2^k)).
$$
Then, since $\omega \in \square_{{\tilde R}_{j_1}^{m+1}(v)}E_+$, there exist 
disjoint $(+)$-paths $r_1,r_3$ in $S(2^k)$ outside ${\tilde R}_{j_1}^{m+1}(v)$,
such that $r_1$ connects $\partial \{ \mathbf{O}\}$ with the left side of $S(2^k)$
and $r_3$ connects  $\partial \{ \mathbf{O}\}$ with the right side of $S(2^k)$.
Then the path $r=r_1\cup \{ \mathbf{O}\} \cup r_3$ separates $S(2^k)$ into two parts,
above $r$ and below $r$.
Since ${\tilde R}_{j_1}^{m+1}(v)$ does not intersect $r$, it is located
above $r$ or below $r$.
Without loss of generality we can assume that  ${\tilde R}_{j_1}^{m+1}(v)$ 
is above $r$.

Since $\omega \in \Delta_{{\tilde R}_{j_1}^{m+1}(v)}E_{-*}$, there exist
disjoint $(*)$-paths $r_2^*, r_4^*$, such that $r_2^*$ connects $\partial^* \{ \mathbf{O}\}$
with the top side of $S(2^k)$, and $r_4^*$ connects $\partial^* \{ \mathbf{O}\}$ with
the bottom side of $S(2^k)$, and at any point
$x \in ( r_2^*\cup r_4^*)\setminus {\tilde R}_{j_1}^{m+1}(v)$, $\omega (x) = -1$.
This means that $r_4^*$ can not cross $r$ and it is a $(-*)$-path.
Also $r_2^*$ is forced to intersect ${\tilde R}_{j_1}^{m+1}(v)$ by our 
assumption that  $\omega \in \Delta_{{\tilde R}_{j_1}^{m+1}(v)}E_{-*}$.
Therefore there is a $(-*)$-path ${r_2^*}''$ connecting $\partial^* \{ \mathbf{O}\}$ with
$\partial^*{\tilde R}_{j_1}^{m+1}(v)$ and a $(-*)$-path ${r_2^*}'$ connecting 
$\partial^*{\tilde R}_{j_1}^{m+1}(v)$ with the top side of $S(2^k)$.

Since $\omega \in {\tilde \Delta }^{(0)}({\tilde R}_{j_1}^{m+1}(v), S(2^k))$,
there is a $(*)$-path $r_8^*$ starting from ${\tilde A}_4$ of
${\tilde R}_{j_1}^{m+1}(v)$
connecting the bottom side
of $\partial^*{\tilde R}_{j_1}^{m+1}(v)$ with the bottom side of $S(2^k)$ such that
$\omega (x) =-1$ for every $x\in r_8^*\setminus \{ \mathbf{O}\} $.
This means that $r_8^*$ goes through ${\mathbf O}$, and
we can take ${r_2^*}''$ as a part of $r_8^*$.
Then $r_5, r_7, r_6^*$ can not pass through ${\mathbf O}$ by definition of 
$ {\tilde \Delta }^{(0)}({\tilde R}_{j_1}^{m+1}(v), S(2^k))$,
therefore $r_5,r_7$ are $(+)$-paths and  $r_6^*$ is a $(-*)$-path, 
starting from ${\tilde {\mathcal A}}_1,
{\tilde {\mathcal A}}_3$ and $ {\tilde {\mathcal A}}_2$, respectively.
Thus, as in the proof of Lemma \ref{lem: relation of deltas-tilde} %Lemma 19 %
by the extension argument we can obtain (\ref{eq:delta-to-delta:four arm-1}) by 
extending $r_5, r_6^*, r_7, r_8^*$ to $\partial {\tilde R}_{j_1}^m(v)$.
The resulting configuration is surely in
$$
\square_{{\tilde R}_{j_1}^{m}(v)}E_+ \cap \Delta_{{\tilde R}_{j_1}^{m}(v)}E_{-*}
\cap {\tilde \Delta }^{(0)}({\tilde R}_{j_1}^{m}(v), S(2^k)).
$$
In the same way we can prove  (\ref{eq:delta-to-delta:four arm-2}). 
\end{proof}
By the same observation as in the above proof, we can obtain the following lemma.
\begin{Lemma}\label{lem:lambda-to delta:four arm}
There exists a constant ${\tilde K}^{(0)}(\eta )>0$ depending only on 
$j_1, \varepsilon_0$ and $\eta $, such that for every $0\leq m \leq m_{j_1}(v)-1$,
\begin{align}\label{eq:lambda-to-delta:four arm}
& \mu_t^N\biggl(
{\Delta}_{{\tilde R}_{j_1}^{m+1}(v)}\Omega (2^k)\cap
{\tilde \Lambda }^{(0)}({\tilde R}_{j_1}^{m+1}(v), S(2^k) )\biggr) \\
\leq &{\tilde K}^{(0)}(\eta )\left\{ 
\mu_t^N\biggl( \square_{{\tilde R}_{j_1}^m(v)}E_+\cap 
               \Delta_{{\tilde R}_{j_1}^m(v)}E_{-*}\cap
              {\tilde \Delta}^{(0)}({\tilde R}_{j_1}^m(v), S(2^k)) 
             \biggr)\right. \nonumber  \\
&\quad + \left. \mu_t^N\biggl( \Delta_{{\tilde R}_{j_1}^m(v)}E_+\cap 
       \square_{{\tilde R}_{j_1}^m(v)}E_{-*}\cap
         {\tilde \Delta }^{(0)}({\tilde R}_{j_1}^m(v), S(2^k)) \biggr)  
     \right\}. \nonumber
\end{align}   
\end{Lemma}
\begin{proof}
It is sufficient to show how to choose corridors in ${\tilde R}_{j_1}^{m+1}(v)\setminus
{\tilde R}_{j_1}^{m}(v)$.

On ${ \Delta }_{{\tilde R}_{j_1}^{m+1}(v)}\Omega (2^k)$, either of the following
events occurs:
\begin{align*}
&\square_{{\tilde R}_{j_1}^{m+1}(v)}E_+ \cap 
\Delta_{{\tilde R}_{j_1}^{m+1}(v)}E_{-*}, \\
&  \Delta_{{\tilde R}_{j_1}^{m+1}(v)}E_+\cap  
\square_{{\tilde R}_{j_1}^{m+1}(v)}E_{-*}, \quad \mbox{ or }\\
&  \Delta_{{\tilde R}_{j_1}^{m+1}(v)}E_+\cap 
 \Delta_{{\tilde R}_{j_1}^{m+1}(v)}E_{-*}.
\end{align*}
In the first two events, we can extend $r_5,r_6^*,r_7,r_8^*$ to
$\partial^* {\tilde R}_{j_1}^m(v)$ as in the above proof and in the proof of
Lemma \ref{lem:lambda-to-delta-tilde} %20 
to obtain 
\begin{align*}
&\square_{{\tilde R}_{j_1}^{m}(v)}E_+\cap \Delta_{{\tilde R}_{j_1}^{m}(v)}E_{-*}
\cap {\tilde \Delta }^{(0)}({\tilde R}_{j_1}^m(v), S(2^k)), \quad \mbox{ or }\\
&  \Delta_{{\tilde R}_{j_1}^{m}(v)}E_+\cap \square_{{\tilde R}_{j_1}^{m}(v)}E_{-*}
\cap {\tilde \Delta }^{(0)}({\tilde R}_{j_1}^m(v), S(2^k)).
\end{align*}
So, the remaining case is when the third event occurs.
Let $\omega $ be an element of 
$$
{\Delta }_{{\tilde R}_{j_1}^{m+1}(v)}\Omega (2^k)\cap 
\Delta_{{\tilde R}_{j_1}^{m+1}(v)}E_+\cap \Delta_{{\tilde R}_{j_1}^{m+1}(v)}E_{-*}.
$$
Then we can find a configuration $\zeta $ such that
$\omega':= \zeta_{{\tilde R}_{j_1}^{m+1}(v)}\omega \in \Omega (2^k)$.
Let $r_1',r_3'$ be $(+)$-paths and ${r_2^*}', {r_4^*}'$ be $(-*)$-paths
in $\omega'$ specified by the definition of $\Omega (2^k)$.
Then the choice of $\omega $ implies that both $r_1'\cup r_3'$ and 
${r_2^*}'\cup {r_4^*}'$ intersect ${\tilde R}_{j_1}^{m+1}(v)$.
We choose them such that the number $q(\omega')$ of elements in
$\{ r_1',{r_2^*}',r_3',{r_4^*}'\} $ which intersect  ${\tilde R}_{j_1}^{m+1}(v)$
is minimal among all the possible choice of $\{ r_1',{r_2^*}',r_3',{r_4^*}'\} $
in $\omega'$.
Since $\omega \in \Delta_{{\tilde R}_{j_1}^{m+1}(v)}E_+\cap 
 \Delta_{{\tilde R}_{j_1}^{m+1}(v)}E_{-*}$, $q(\omega')\in \{ 2,3,4\} $.

\medskip
\noindent
Case 1: \ $q(\omega')=2$.

In this case, only one of $\{ r_1', r_3' \} $ intersects 
${\tilde R}_{j_1}^{m+1}(v)$ and only 
one of $\{ {r_2^*}', {r_4^*}'\} $ intersects ${\tilde R}_{j_1}^{m+1}(v)$.
Without loss of generality we can assume that $r_1'$ and ${r_2^*}'$ intersect 
${\tilde R}_{j_1}^{m+1}(v)$.
Then $r_3'\cup {r_4^*}'$ does not intersect  ${\tilde R}_{j_1}^{m+1}(v)$, and they are 
determined by $\omega $ outside ${\tilde R}_{j_1}^{m+1}(v)$.

Let $r_1''$ be the part of $r_1'$ connecting $\partial \{ \mathbf{O}\} $ with 
$\partial {\tilde R}_{j_1}^{m+1}(v)$ and $r_5$ be the part of $r_1'$ connecting
$\partial {\tilde R}_{j_1}^{m+1}(v)$ with the left side of $S(2^k)$.
Also, let ${r_2^*}''$ be the part of ${r_2^*}'$ connecting $\partial^*\{ \mathbf{O}\} $
with $\partial^*{\tilde R}_{j_1}^{m+1}(v)$, and $r_6^*$ be the part of  ${r_2^*}'$
connecting $\partial^*{\tilde R}_{j_1}^{m+1}(v)$ with the top side of $S(2^k)$.
Then $r_1'',{r_2^*}'', r_5, r_6^*$ are determined by $\omega $ outside
${\tilde R}_{j_1}^{m+1}(v)$.

To obtain the event
$$   
\Delta_{{\tilde R}_{j_1}^{m}(v)}E_+\cap \square_{{\tilde R}_{j_1}^{m}(v)}E_{-*}
\cap {\tilde \Delta }^{(0)}( {\tilde R}_{j_1}^{m}(v), S(2^k))
$$
by extension argument, we need one more $(-*)$-path $t^*$ which connects 
$\partial^*{\tilde R}_{j_1}^{m+1}(v)$ with the bottom side of $S(2^k)$, forcing $r_1$
to intersect ${\tilde R}_{j_1}^{m+1}(v)$.
On ${\tilde \Lambda }^{(0)}({\tilde R}_{j_1}^{m+1}(v), S(2^k) )$, every
crossing $(+)$ and $(-*)$ cluster has an $(\eta , j_1+m+1)$-fence,
$r_1'',{r_2^*}'',r_3', {r_4^*}', r_5,r_6^*$ and $t^*$ can be 
extended to $\partial {\tilde R}_{j_1}^m(v)$.

Let $a_1, a_2, a_5, a_6$ and $a^*$ be endpoints of $r_1'', {r_2^*}'', r_5, 
r_6^*$ and $t^*$ on $\partial^* {\tilde R}_{j_1}^{m+1}(v)$, respectively.
Starting from $a_1$ they are located in the clockwise direction in order of
$a_1, a^*, a_5, a_6$ and $a_2$ on  $\partial^* {\tilde R}_{j_1}^{m+1}(v)$.
We choose corridors $U_1, U_2, U_3, U_4$ and $U_5$ in the following way.
\begin{enumerate}
\item $U_1$ connects $a_5$ with the left side of ${\tilde R}_{j_1}^m(v)$.
\item $U_2$ connects $a_6$ with the top side of ${\tilde R}_{j_1}^m(v)$.
\item $U_3$ connects $a_1$ with the right side of ${\tilde R}_{j_1}^m(v)$. 
\item $U_4$ connects $a^*$ with the bottom side of ${\tilde R}_{j_1}^m(v)$.
\item $U_5$ connects $a_2$ with $U_2$.
\item $U_1,U_2,U_3,U_4$ are disjoint and 
    $U_5\cap ( U_1\cup U_3\cup U_4 \cup \cup_i {\tilde {\mathcal B}}_i )
      =\emptyset $.
\item $U_j \cap \cup_i {\tilde {\mathcal B}}_i \subset {\tilde {\mathcal A}}_j$, for
$j=1,2,3,4$.
\end{enumerate} 
Here, ${\tilde {\mathcal A}}_i$'s and ${\tilde {\mathcal B}}_i$'s correspond
to ${\tilde R}_{j_1}^m(v)$.

If ${\tilde R}_{j_1}^m(v)$ touches the right side of $S(2^k)$, then we can
not find such a corridor $U_5$ since the above conditions require $U_5$
to go through right of  ${\tilde R}_{j_1}^m(v)$.
Otherwise it is possible to find such $U_i$'s.
So assume that  ${\tilde R}_{j_1}^m(v)$ touches the right side of $S(2^k)$.
Then we choose corridors $U_1, U_2, U_4$ and $U_5$ such that
\begin{enumerate}
\item $U_1$ connects $a_5$ with the left side of ${\tilde R}_{j_1}^m(v)$.
\item $U_2$ connects $a_6$ with the top side of ${\tilde R}_{j_1}^m(v)$.
\item $U_4$ connects $a_2$ with the bottom side of ${\tilde R}_{j_1}^m(v)$.
\item $U_5$ connects $a_1$ with $U_1$.
\item $U_1,U_2,U_4$ are disjoint and 
    $U_5\cap ( U_2\cup U_4 \cup \cup_i {\tilde {\mathcal B}}_i )
      =\emptyset $.
\item $U_j \cap \cup_i {\tilde {\mathcal B}}_i \subset {\tilde {\mathcal A}}_j$, for
$j=1,2,4$.
\end{enumerate}
Then the resulting configuration is in 
$$
\square_{{\tilde R}_{j_1}^m(v)}E_+\cap \Delta_{{\tilde R}_{j_1}^m(v)}E_{-*}
\cap {\tilde \Delta}^{(0)}({\tilde R}_{j_1}^m(v), S(2^k)).
$$
Note also that when ${\tilde R}_{j_1}^{m+1}(v) \ni (2^k,2^k)$, then we do not
have $r_6^*$. 
In this case, we simply connect $a_5$ with the left side of 
${\tilde R}_{j_1}^m(v)$ by a corridor $U_1$, $a_2$ with the bottom side of 
${\tilde R}_{j_1}^m(v)$ by a corridor $U_4$, and $a_1$ with $U_1$ by a
corridor $U_5$.

If in addition ${\tilde R}_{j_1}^m(v)$ does not touch the right side of 
$S(2^k)$, then we choose corridors $U_2$ and $U_3$ to connect 
$\partial {\tilde R}_{j_1}^m(v)$ right side of $S(2^k)$, respectively.
As before we can choose $U_1,\ldots, U_4$ to be disjoint, 
$U_j\cap \cup_i{\tilde {\mathcal B}}_i \subset {\tilde {\mathcal A}}_j$
for each $j$, and 
$$
U_5\cap ( U_2\cup U_3\cup U_4\cup \cup_i{\tilde {\mathcal B}}_i)=\emptyset 
$$

In the case where  ${\tilde R}_{j_1}^m(v)$ touches the right side of
$S(2^k)$, we do not need $U_3$, and if ${\tilde R}_{j_1}^m(v)\ni (2^k,2^k)$,
we do not need $U_2$, either.

\medskip
\noindent
Case 2: $q(\omega')=3$.

For simplicity, we consider the case where 
${\tilde R}_{j_1}^{m+1}(v)=S_{j_1}^{m+1}(v)$. 
The argument for other cases are easily modified as before.

Without loss of generality we can assume that 
$r_1'\cap {\tilde R}_{j_1}^{m+1}(v) = \emptyset $.
Then any of ${r_2^*}', r_3', {r_4^*}'$ intersects ${\tilde R}_{j_1}^{m+1}(v)$.
Let ${r_2^*}''$ and ${r_4^*}''$ are defined as parts of 
${r_2^*}'$ and ${r_4^*}'$ connecting $\partial^*\{ {\mathbf O}\} $ with
$\partial^*{\tilde R}_{j_1}^{m+1}(v)$, and let $r_3''$ be the part of 
$r_3'$ connecting $\partial \{ {\mathbf O}\} $ with 
$\partial {\tilde R}_{j_1}^{m+1}(v)$.
Also, we define $r_6^*$ as the part of ${r_2^*}'$ connecting 
$\partial^*{\tilde R}_{j_1}^{m+1}(v) $ with the top side of
$S(2^k)$, $r_7$ be the part of $r_3'$ connecting 
$\partial {\tilde R}_{j_1}^{m+1}(v)$ with the right side of $S(2^k)$,
and $r_8^*$ be the part of ${r_4^*}'$ connecting 
$\partial^*{\tilde R}_{j_1}^{m+1}(v) $ with the bottom side of $S(2^k)$.

Let $ a_2, a_3,  a_4,  a_6,  a_7,  a_8 $ be the endpoints on the boundary
$\partial^*{\tilde R}_{j_1}^{m+1}(v) $, of paths
$ {r_2^*}'', r_3'', {r_4^*}'', r_6^*, r_7, r_8^*$, respectively.  
Note that they are located on this boundary in the clockwise direction
in order of $a_4, a_3, a_2, a_6, a_7, a_8$.

To connect these point s with ${\tilde {\mathcal A}}_i$'s of 
${\tilde R}_{j_1}^m(v)$, we do in the following way.
We choose a corridor $U_1$ to connect $a_2$ with ${\tilde {\mathcal A}}_1$,
a corridor $U_2$ to connect ${\tilde {\mathcal A}}_2$ with $a_6$,
$U_3$ to connect ${\tilde {\mathcal A}}_3$ with $a_7$, and
$U_4$ to connect ${\tilde {\mathcal A}}_4$ with $a_8$.
Also, we choose a corridor $U_5$ to connect $a_3$ with $U_2$, and $U_6$
to connect $a_8$ with $U_4$.
Further, we choose $U_1,\ldots,U_4$ to be disjoint, and
$$
U_5 \cap ( U_1\cup U_3\cup U_4\cup \cup_i{\tilde {\mathcal B}}_i )=\emptyset,
\quad U_6 \cap ( U_1\cup U_2\cup U_3 \cup \cup_i {\tilde {\mathcal B}}_i )
= \emptyset
$$
as before.

With this extension argument, we obtain the event
$$
\Delta_{{\tilde R}_{j_1}^m(v)}E_+\cap \square_{{\tilde R}_{j_1}^m(v)}E_{-*}
\cap \Delta^{(0)}({\tilde R}_{j_1}^m(v), S(2^k)).
$$

\noindent
Case 3: \ $ q(\omega')=4$.

In this case, from  $\partial^*{\tilde R}_{j_1}^{m+1}(v) $ there are 
$(+)$-paths $r_5,r_7$ and $(-*)$-paths $r_6^*,r_8^*$ connecting
$\partial^*{\tilde R}_{j_1}^{m+1}(v) $ with the left, right, top and the
bottom sides of $S(2^k)$, respectively.
Let $a_5, a_6, a_7$ and $a_8$ be the endpoints of $r_5, r_6^*,r_7$ and
$r_8^*$ on the boundary of ${\tilde R}_{j_1}^{m+1}(v)$.
Further, there are two $(+)$-paths $r_1'', r_3''$ connecting 
$\partial \{ {\mathbf O}\} $ with $\partial {\tilde R}_{j_1}^{m+1}(v) $
and two $(-*)$-paths ${r_2^*}'', {r_4^*}''$ connecting
$\partial^* \{ {\mathbf O}\} $ with $\partial^*{\tilde R}_{j_1}^{m+1}(v) $.
Since $r_5,r_6^*, r_7,r_8^*$ separate the region 
$S(2^k)\setminus {\tilde R}_{j_1}^{m+1}(v)$ into four parts, 
these four paths $r_1'',{r_2^*}'',r_3'',{r_4^*}''$ are located in one of 
these four parts.
Without loss of generality we can assume that they are located in the region
between $r_5$ and $r_8^*$

Let $a_1,a_2,a_3,a_4$ be the endpoints of 
$\{ r_1'', {r_2^*}'', r_3'',{r_4^*}''\} $ on the boundary of 
$\tilde R_{j_1}^{m+1}(v)$, respectively.
Note that they are located on the boundary in a cyclic permutation of the order
$\{ a_4,a_3,a_2,a_1 \} $ when we go along the boundary from 
$a_8$ to $a_5$ in the clockwise direction.
There are two cases we have to consider. 
The first endpoint from $r_8^*$ in the clockwise direction belongs to a
$(-*)$-path, or to a $(+)$-path.

First, we consider the case where this first endpoint belongs to
a $(-*)$-path, ${r_2^*}''$ or ${r_4^*}''$.
Without loss of generality we can assume that $a_2$ is the first point
from $a_8$.

Then we choose the corridors in the following way to obtain a configuration in 
\begin{equation}\label{eq:8-9}
\Delta_{{\tilde R}_{j_1}^{m+1}(v)}E_+ \cap \square_{{\tilde R}_{j_1}^{m+1}(v)}E_{-*}
\cap {\tilde \Delta }^{(0)}({\tilde R}_{j_1}^m(v), S(2^k))
\end{equation}
We choose $U_1$ to connect $a_1$ with the left side of ${\tilde R}_{j_1}^{m+1}(v)$
ending in ${\tilde {\mathcal A}}_1$, 
$U_2$ to connect $a_6$ with the top side of ${\tilde R}_{j_1}^{m+1}(v)$
ending in ${\tilde {\mathcal A}}_2$,
$U_3$ to connect $a_7$ with the right side of ${\tilde R}_{j_1}^{m+1}(v)$
ending in ${\tilde {\mathcal A}}_3$, and
$U_4$ to connect $a_8$ with the bottom side of ${\tilde R}_{j_1}^{m+1}(v)$
ending in ${\tilde {\mathcal A}}_4$.
$U_1 ,\ldots,  U_4$ are disjoint.
Also, we choose $U_5$ to connect $a_3$ with $a_5$, outside 
$\cup_i {\tilde {\mathcal B}}_i$, $U_6$ to connect $a_2$ with $U_2$  outside
$\cup_i {\tilde {\mathcal B}}_i$, and $U_7$ to connect $a_4$ with $U_4$
outside $\cup_i {\tilde {\mathcal B}}_i$.
$U_5$ is disjoint from all other $U_i$'s, $U_6$ is disjoint from all other 
$U_i$'s except $U_2$, and $U_7$ disjoint from all other $U_i$'s except
$ U_4$.
The resulting configuration of the extension through these corridors 
is easily seen to be in (\ref{eq:8-9}).

Second, we assume that the first point of $\{ a_1 ,\ldots, a_4\} $
from $a_8$ in the clockwise direction belongs to a $(+)$-path.
Without loss of generality, we can assume that $a_1$ is the first point.
In this case, the endpoints are located in the clockwise order from $a_8$
as $a_1, a_4, a_3, a_2$.
Then we can easily see that there is either a $(+)$-path $s$ or a $(-*)$-path
$t^*$ connecting a boundary point of ${\tilde R}_{j_1}^{m+1}(v)$ between
$a_8$ and $a_1$, with another boundary point of ${\tilde R}_{j_1}^{m+1}(v)$
between $a_5$ and $a_2$.

This is because in the configuration $\omega'$, if $a_1$ is connected with $a_5$ by $r_1'$ and if $r_1'$ does not contain such a $(+)$-path, then ${r_4^*}'$ 
and ${r_2^*}'$ inevitably go outside ${\tilde R}_{j_1}^{m+1}(v)$, 
and go around the origin to enter ${\tilde R}_{j_1}^{m+1}(v)$ again.
As a consequence ${r_4^*}'$ contains such a $(-*)$-path $t^*$.
Also, if $a_1$ is connected with $a_7$, then $a_3$ is connected with $a_5$,
and then ${r_4^*}'$ inevitably contains such a $(-*)$-path $t^*$.

So, without loss of generality we can assume that there is a $(-*)$-path
$t^*$.
Let $b_1, b_2$ be the endpoints of $t^*$ such that $b_1$ is between $a_1$ and $a_8$, and $b_2$ is between $a_2$ and $a_5$.
In this case, we choose corridors $U_1 ,\ldots, U_8$ in the following way.
\begin{enumerate}
\item $U_1$ connects $a_1$ with the left side of ${\tilde R}_{j_1}^{m}(v)$,
ending in ${\tilde {\mathcal A}}_1$.
\item $U_2$ connects $a_6$ with the top side of ${\tilde R}_{j_1}^{m}(v)$,
ending in ${\tilde {\mathcal A}}_2$.
\item $U_3$ connects $a_7$ with  the right side of ${\tilde R}_{j_1}^{m}(v)$,
ending in ${\tilde {\mathcal A}}_3$.
\item $U_4$ connects $a_8$ with  the bottom side of ${\tilde R}_{j_1}^{m}(v)$,
ending in ${\tilde {\mathcal A}}_4$.
\item $U_5$ connects $a_3$ with $a_5$  outside 
$\cup_i {\tilde {\mathcal B}}_i$.
\item $U_6$ connects $a_4$ with $U_2$  outside 
$\cup_i {\tilde {\mathcal B}}_i$.
\item $U_7$ connects $a_2$ with $b_2$  outside 
$\cup_i {\tilde {\mathcal B}}_i$.
\item  $U_8$ connects $b_1$ with $U_4$  outside 
$\cup_i {\tilde {\mathcal B}}_i$.
\item $U_1, \ldots, U_7$ except $U_6$ are disjoint.
\item $U_6$ is disjoint from all other corridors except $U_2$.
\item $U_8$ is disjoint from all other corridors except $U_4$.
\end{enumerate}
The procedure of choosing corridors is the same as before and the resulting
configuration by extending $(+)$-paths and $(-*)$-paths through these
corridors belongs to the event (\ref{eq:8-9}).
\end{proof}

\begin{Lemma}\label{gamma-to-delta:four arm}
Let $2^{j_1+5}<2^k<\min\{ L(h, \varepsilon_0), \frac{N}{2}\} $.
There exists a constant $K^{(0)}(\eta )$ depending only on 
$j_1, \varepsilon_0$, and $\eta $ such that for $v\in S(2^k)$ and
$0\leq m \leq m_1^*=m_{j_1}(v)$, we have 
\begin{align}
&\mu_t^N\bigl( {\tilde \Gamma }^{(0)}({\tilde R}_{j_1}^m(v), S(2^k))\cap
    \Delta_{{\tilde R}_{j_1}^m(v)}\Omega (2^k) \bigr) \label{eq:8-10} \\
\leq & K^{(0)}(\eta ) \times \left\{
   \mu_t^N\bigl(  \Delta_{{\tilde R}_{j_1}^m(v)}E_+ \cap
              \square_{{\tilde R}_{j_1}^m(v)}E_{-*}\cap
       {\tilde \Delta }^{(0)}({\tilde R}_{j_1}^m(v), S(2^k)) \bigr) \right. 
\nonumber \\
& \left. + \mu_t^N\bigl(  \square_{{\tilde R}_{j_1}^m(v)}E_+ \cap
              \Delta_{{\tilde R}_{j_1}^m(v)}E_{-*}\cap
       {\tilde \Delta }^{(0)}({\tilde R}_{j_1}^m(v), S(2^k)) \bigr) \right\} 
\nonumber
\end{align}
\end{Lemma}
\begin{proof}
First, we prove (\ref{eq:8-10}) in the case where $m=m_1^*$. 
When $|v|_\infty\leq 2^{k-1}$, we can see that ${\tilde R}_{j_1}^{m_1^*}(v)
= S_{j_1}^{m_1^*}(v)$.
%% \begin{quote}
%% ---------------------------------------------
%% 
%% \noindent \framebox{\bf INTERMISSION}
Indeed, by definition, we have 
$$
2^{j_1+m_1^*}\geq 2^{-3}|v|_\infty \geq 2^{j_1+m_1^*-1}-2^{j_1}.
$$
Therefore if $|v|_\infty \leq 2^k$,
\begin{align*}
d_\infty (0,S_{j_1}^{m_1^*}(v)) &\leq  |v|_\infty + 2^{j_1+m_1^*}+2^{j_1} \\
&\leq  (1+ 2^{-2})|v|_\infty +2^{j_1+1}+2^{j_1}\\
&\leq  2^{k-1}(1+2^{-1}+ 2^{-3}+2^{-4})= \frac{7}{8}\cdot 2^k
\end{align*}
Also, from the above inequality
$$
2^{j_1+m_1^*}< 2^{-1}|v|_\infty \leq 2^{k-2}.
$$
This implies that ${\tilde {\mathcal B}}_i$'s for $S_{j_1}^{m_1^*}(v)$ is in
$$
S \left(\frac{7}{8}\cdot 2^k+ 2^{k-4} \right)\subset S(2^k).
$$
Hence,
${\tilde R}_{j_1}^{m_1^*}(v)
= S_{j_1}^{m_1^*}(v)$.

%% ---------------------------------------------------
%% \end{quote}
Further, if $|v|_\infty \geq 2^{j_1+5}$, then we have
\begin{equation}
d_\infty (0, S_{j_1}^{m_1^*}(v))\geq \frac{5}{8}|v|_\infty ,
\end{equation}
%% \begin{quote}
%% ---------------------------------------------
%% 
%% \noindent \framebox{\bf INTERMISSION}
%%
%%This is because
because 
\begin{align*}
d_\infty (0, S_{j_1}^{m_1^*}(v))&\geq |v|_\infty -2^{j_1+m_1^*}-2^{j_1}\\
     &\geq |v|_\infty -2^{-2}|v|_\infty -2^{j_1+1}-2^{j_1}   \\
&\geq \left(\frac{3}{4} -2^{-4}-2^{-5}\right)|v|_\infty =\frac{5}{8}|v|_\infty.
\end{align*}

%% ---------------------------------------------------
%% \end{quote}

In this case, we argue in the following way.
Since
$$
\Delta_{S_{j_1}^{m_1^*}(v)}\Omega (2^k) \cap {\tilde \Gamma }^{(0)}
(S_{j_1}^{m_1^*}(v), S(2^k) )
\subset \Gamma (0, S(2^{-3}|v|_\infty ))\cap {\tilde \Gamma }
( S(2|v|_\infty )S(2^k)),
$$
we first estimate the probability of the right hand side event.
By the mixing property,
\begin{align}
& \mu_t^N\biggl( 
 \Gamma (0, S(2^{-3}|v|_\infty ))\cap {\tilde \Gamma }
( S(2|v|_\infty )S(2^k)) \biggr) \label{eq:8-11} \\
\leq &\biggl\{ \mu_t^N\bigl(  \Gamma (0, S(2^{-3}|v|_\infty ))\cap
{\tilde \Gamma } ( S(2|v|_\infty ), S(2^k))
\bigr) + C(2^{-3}|v|_\infty )^2 \frac{|v|_\infty }{2}
e^{-\alpha |v|_\infty /2} \biggr\} \nonumber \\
&\times \mu_t^N\bigl( {\tilde \Gamma }(S(2|v|_\infty ), S(2^k)\bigr) .
\nonumber 
\end{align}
From Lemma \ref{lambda-to-delta},
we have
$$
\mu_t^N\bigl( \Gamma (0, S(2^{-3}|v|_\infty )) \bigr) 
\geq \mu_t^N \bigl( \Delta (0,2^{j_1+m_1^*})\bigr)\geq C_2 C_1^{-m_1^*}.
$$
If 
\begin{equation}\label{eq:8-12}
C{\left(\frac{|v|_\infty }{2} \right) }^3 e^{-\alpha |v|_\infty /2}
\leq C_2C_1^{-m_1^*},
\end{equation}
then this implies that
\begin{align*}
& \mu_t^N\bigl(  \Gamma (0, S(2^{-3}|v|_\infty )) \bigr) -
 C(2^{-3}|v|_\infty )^2 \frac{|v|_\infty }{2}
e^{-\alpha |v|_\infty /2}\\
\geq & (1-2^{-4})\mu_t^N\bigl(  \Gamma (0, S( 2^{-3}|v|_\infty )) \bigr) .
\end{align*}
From the definition of $m_1^*$, 
$
\frac{|v|_\infty }{2}\geq 2^{m_1^*+j_1}.
$
So if we put $j_3\geq j_1$ as the smallest number satisfying
\begin{equation}\label{eq:8-13}
Cn^3e^{-\alpha n}< C_2 C_1^{-\log_2 n } \qquad \mbox{for every } n \geq 2^{j_3},
\end{equation}
then
$$
C{\biggl( \frac{|v|_\infty }{2} \biggr) }^3 e^{-\alpha (|v|_\infty /2)}
\leq C_2C_1^{-\log_2 |v|_\infty /2 }\leq C_2C_1^{- m_1^*}
$$
if 
\begin{equation}\label{eq:8-14}
m_1^*+j_1 \geq j_3.
\end{equation} 
Note that
$j_3$ depends only on $C_1$ and $C_2$, which depend only on $\varepsilon_0$
and $j_1$.
So, assume that (\ref{eq:8-14}) is true. 
Then we have from (\ref{eq:8-11}) 
\begin{align}\label{eq:8-15}
& \mu_t^N\biggl( 
 \Gamma (0, S(2^{-3}|v|_\infty ))\cap {\tilde \Gamma }
( S(2|v|_\infty ), S(2^k)) \biggr) \\
\leq & (1-2^{-4})\mu_t^N\bigl( \Gamma (0, S(2^{-3}|v|_\infty )) \bigr) 
\mu_t^N\bigl( {\tilde \Gamma }
( S(2|v|_\infty ), S(2^k)) \bigr) . \nonumber
\end{align}
By Lemmas \ref{comparison of gamma and delta} and \ref{lem:comparison of gamma and delta-tilde}
the right hand side of (\ref{eq:8-15}) is bounded from above by
$$
K{\tilde K}(1-2^{-4})^{-1}\mu_t^N\bigl( \Delta (0, S(2^{-3}|v|_\infty ))\bigr)
\mu_t^N\bigl( {\tilde \Delta }(S(2|v|_\infty ), S(2^k))\bigr) .
$$
Again by the mixing property and (\ref{eq:8-13}), this is bounded from above by $$
K{\tilde K}(1-2^{-4})^{-2}\mu_t^N\bigl( \Delta (0, S(2^{-3}|v|_\infty ))\cap
{\tilde \Delta }(S(2|v|_\infty ), S(2^k))\bigr) .
$$
if (\ref{eq:8-14}) is satisfied.
But it is easy to find corridors in the event
$$
\Delta (0, S(2^{-3}|v|_\infty ))\cap
{\tilde \Delta }(S(2|v|_\infty ), S(2^k)),
$$
to connect ${\mathcal A}_i$'s of $S(2^{-3}|v|_\infty ) $ with 
${\tilde {\mathcal A}}_i$'s  of $S_{j_1}^{m_1^*}(v)$ or 
${\tilde {\mathcal A}}_i$'s  of $S(2|v|_\infty )$ to obtain 
$$
\Delta_{{\tilde R}_{j_1}^{m_1^*}(v)}E_+\cap \square_{S_{j_1}^{m_1^*}(v)}E_{-*}\cap {\tilde \Delta }^{(0)}(S_{j_1}^{m_1^*}(v), S(2^k)).
$$
The width of these corridors is $2^{-5}|v|_\infty $, and the length of them are
not larger than $8|v|_\infty $, and the ratio is bounded by an absolute
constant.
Therefore by the connection lemma, there exists a constant $C_1^\# >0$ depending only on $\varepsilon_0$ and $j_1$, such that
\begin{align*}
&\mu_t^N\bigl( \Gamma (0, S(2^{-3}|v|_\infty ))\cap {\tilde \Gamma }(
 S(2|v|_\infty ), S(2^k) )\bigr) \\
\leq & C_1^\# K {\tilde K}(1-2^{-4})^{-2}
\mu_t^N\bigl( \Delta_{S_{j_1}^{m_1^*}(v)}E_+\cap 
              {\tilde \Delta }^{(0)}( S_{j_1}^{m_1^*}(v), S(2^k)) \bigr)
\end{align*}
under the condition that   (\ref{eq:8-14}) is satisfied.
The outward extension argument then proves the desired inequality 
in this case.
%% \begin{quote}
%% ---------------------------------------------
%% 
%% \noindent \framebox{\bf INTERMISSION}
%% 
%% For $0\leq m \leq m_1^*-1$, we have as in Lemmas \ref{comparison of gamma and delta} and \ref{lem:comparison of gamma and delta-tilde},
%%
Indeed, for  $0\leq m \leq m_1^*-1$, we have as in Lemmas \ref{comparison of gamma and delta} and \ref{lem:comparison of gamma and delta-tilde},
\begin{align*}
&\mu_t^N\bigl( {\tilde \Gamma }^{(0)}(S_{j_1}^m(v), S(2^k)) \cap
  \Delta_{S_{j_1}^m(v)}\Omega ( 2^k)\bigr) \\
\leq & \mu_t^N\bigl( {\tilde \Lambda }^{(0)}(S_{j_1}^m(v), S(2^k)) \bigr)\\
& +\delta  \mu_t^N\bigl( {\tilde \Gamma }^{(0)}(S_{j_1}^{m+1}(v), S(2^k)) \cap
  \Delta_{S_{j_1}^{m+1}(v)}\Omega ( 2^k)\bigr)\\
\leq & \sum_{\ell = 0}^{m_1^*-m-1} \delta^\ell 
\mu_t^N\bigl( {\tilde \Gamma }^{(0)}(S_{j_1}^{m+\ell }(v), S(2^k)) \cap
  \Delta_{S_{j_1}^{m+\ell }(v)}\Omega ( 2^k)\bigr)\\
&+ \delta^{m_1^*-m}\mu_t^N\bigl(
 {\tilde \Gamma }^{(0)}(S_{j_1}^{m_1^*}(v), S(2^k)) \cap
  \Delta_{S_{j_1}^{m_1^* }(v)}\Omega ( 2^k)\bigr) .
\end{align*}
By Lemma \ref{lem:lambda-to delta:four arm}, the first summation is bounded by
\begin{align*}
&\sum_{\ell =0}^{m_1^*-m-1} \delta^\ell \left\{ 
\mu_t^N\bigl( \Delta_{S_{j_1}^{m+\ell +1}(v)}E_+\cap 
  \square_{S_{j_1}^{m+\ell +1}(v)}E_{-*} \cap
  {\tilde \Delta }^{(0)}(S_{j_1}^{m+\ell +1}(v), S(2^k)) \bigr) \right. \\
&+ \left. 
\mu_t^N\bigl(\square_{S_{j_1}^{m+\ell +1}(v)}E_+\cap 
          \Delta_{S_{j_1}^{m+\ell +1}(v)}E_{-*}\cap 
            {\tilde \Delta }^{(0)}(S_{j_1}^{m+\ell +1}(v), S(2^k)) \bigr)
  \right\} .
\end{align*}
Also, the last term is bounded from above by
$$
\delta^{m_1^*-m}\mu_t^N\bigl(
\Delta_{S_{j_1}^{m_1^*}(v)}E_+\cap \square_{S_{j_1}^{m_1^*}(v)}E_{-*} \cap
  {\tilde \Delta }^{(0)}(S_{j_1}^{m_1^*}(v), S(2^k)) \bigr)
$$
under the condition that $v$ satisfies (\ref{eq:8-14}).
Then by Lemma \ref{lem:delta-to-delta:four arm}, we obtain that there
is  some constant $C_2^\#(\eta ) >0$ depending only on 
$\varepsilon_0, j_1, \eta $,
\begin{align*}
&\mu_t^N\bigl( {\tilde \Gamma }^{(0)}(
(S_{j_1}^m(v), S(2^k)) \cap
  \Delta_{S_{j_1}^m(v)}\Omega ( 2^k)\bigr) \\
\leq & \sum_{\ell =0}^{m_1^*-m-1} (\delta C_1)^{\ell +1}\left\{ 
 \mu_t^N\bigl( \Delta_{S_{j_1}^m(v)}E_+ \cap \square_{S_{j_1}^m(v)}E_{-*}
  \cap {\tilde \Delta }^{(0)}(S_{j_1}^m(v), S(2^k))\bigr)\right.\\
&+\left. \mu_t^N\bigl( \square_{S_{j_1}^m(v)}E_+ \cap 
            \Delta_{S_{j_1}^m(v)}E_{-*}
 \cap {\tilde \Delta }^{(0)}(S_{j_1}^m(v), S(2^k))\bigr) \right\} \\
&+ (\delta C_1)^{m_1^*-m} \mu_t^N\bigl(
  \Delta_{S_{j_1}^m(v)}E_+ \cap \square_{S_{j_1}^m(v)}E_{-*}
  \cap {\tilde \Delta }^{(0)}(S_{j_1}^m(v), S(2^k))\bigr) \\
=& C_2^\# (\eta )  
\left[ \mu_t^N\bigl( 
   \Delta_{S_{j_1}^m(v)}E_+ \cap \square_{S_{j_1}^m(v)}E_{-*}
  \cap {\tilde \Delta }^{(0)}(S_{j_1}^m(v), S(2^k))\bigr) \right.\\
& \left. + \mu_t^N\bigl(  \square_{S_{j_1}^m(v)}E_+ \cap 
            \Delta_{S_{j_1}^m(v)}E_{-*}
 \cap {\tilde \Delta }^{(0)}(S_{j_1}^m(v), S(2^k))\bigr) \right] .
\end{align*}

%% ---------------------------------------------------
%% \end{quote}

If  (\ref{eq:8-14}) does not hold for $v$, then $|v|_\infty \leq 2^{j_1+m_1^*+3}\leq 2^{j_3+3}$.
So, putting $j_4= \max\{ j_3+3, j_1+5 \} $ we use the finite energy property
to obtain the desired inequality when $|v|_\infty \leq 2^{j_4}$.

It still remains to consider the case where $v\in S(2^k)\setminus S(2^{k-1})$.
In this case we do not have to use $S(2|v|_\infty )$, and we argue in the 
following way.
We start with the trivial inclusion
$$
\Delta_{{\tilde R}_{j_1}(v)}\Omega (2^k) \cap 
{\tilde \Gamma }^{(0)}( {\tilde R}_{j_1}(v), S(2^k)) \subset
\Gamma (0, 2^{-3}|v|_\infty ).
$$
Since we can see that 
$$
d_\infty ({\tilde R}_{j_1}^{m_1^*}(v), S(2^{-3}|v|_\infty )) 
 \geq 2^{-2}|v|_\infty ,
$$
and $|v_\infty |\geq 2^{k-1}$, we can arrange corridors to obtain
$$
\Delta_{{\tilde R}_{j_1}^m(v)}E_+ \cap \square_{{\tilde R}_{j_1}^m(v)}E_{-*}
  \cap {\tilde \Delta }^{(0)}({\tilde R}_{j_1}^m(v), S(2^k))
$$
from the event $\Delta (0, S(2^{-3}|v|_\infty ))$.
By Lemma \ref{comparison of gamma and delta}
and the Connection lemma, we obtain the desired inequality for $m=m_1^*$
in this case, too. The rest is usual outwards extension argument.
\end{proof}
\begin{Theorem}\label{Russo-four-arm}
Assume that $2^{j_3}+5 \leq 2^k< \min\{ L(h,\varepsilon_0),\frac{N}{2}\} $,
where $j_3\geq j_1$ is given by the smallest number satisfying (\ref{eq:8-13}).
Then there exist positive constants $C_{22}, C_{23}$, depending only on
$j_1, \varepsilon_0$ and $\eta $, such that
\begin{align}
&\bigg\vert \frac{d}{dt}\mu_t^N\bigl( 
   \Omega (2^k) \bigr) \bigg\vert \\
\leq  & \frac{|h-h_c|}{\mathfrak{K}T}\left[
C_{22}\mu_t^N\bigl( \Omega (2^k) \bigr) +
 C_{23} \sum_{v \in S(2^K)}
\mu_t^N\bigl( \Delta_v\Omega (2^k) \bigr) \right] \nonumber
\end{align}
\end{Theorem}
\begin{proof}
We check the condition (S).
Note first that 
$$
\square_{{\tilde R}_{j_1}^m(v)\Omega (2^k)}\setminus
\square_{{\tilde R}_{j_1}^{m+1}(v)}\Omega (2^k)
$$
for $v\in S(2^k)$ and $1\leq m \leq m_1^*-1$.
%% 
%% \begin{quote}
%% ---------------------------------------------
%% 
%% \noindent \framebox{\bf INTERMISSION}
%% 
To see this, first we note that if 
$$
\omega \in \square_{{\tilde R}_{j_1}^m(v)}\Omega (2^k)\setminus
\square_{{\tilde R}_{j_1}^{m+1}(v)}\Omega (2^k),
$$
then  $\Omega (2^k)$ occurs, but if we change the configuration in 
${\tilde R}_{j_1}^{m+1}(v)$ suitably, then $\Omega (2^k)$ does not occur
in the resulting configuration.
Therefore we can see that 
$\omega \in \Delta_{{\tilde R}_{j_1}^{m+1}(v)}\Omega (2^k)$.
this implies that at least one of $r_1, r_2^*, r_3, r_4^*$ is forced
to go through ${\tilde R}_{j_1}^{m+1}(v)$.
without loss of generality we can assume that $r_1$ is forced to go
through  ${\tilde R}_{j_1}^{m+1}(v)$.
Then there exist $(-*)$-paths $r_6^*$ and $r_8^*$ connecting 
$\partial {\tilde R}_{j_1}^{m+1}(v)$ with the top and the bottom sides
of $S(2^k)$, respectively, such that they force $r_1$ to go through
${\tilde R}_{j_1}^{m+1}(v)$.
Then, ${\tilde \Gamma }^{(0)}( {\tilde R}_{j_1}^{m+1}(v), S(2^k))$
occurs in $\omega $.

%% ---------------------------------------------------
%% \end{quote}

Next, by Lemmas \ref{gamma-to-delta:four arm}  and \ref{lem:delta-to-delta:four arm}, we have
\begin{align}
&\mu_t^N\bigl(
\Delta_{{\tilde R}_{j_1}^{m+1}(v)}\Omega (2^k) \cap
{\tilde \Gamma }^{(0)}({\tilde R}_{j_1}^{m+1}(v), S(2^k)) \bigr) 
\label{eq:8-16} \\
& \leq {\tilde K}^{(0)}C_1^m \times \left\{ 
\mu_t^N\bigl( 
\Delta_{{\tilde R}_{j_1}^{1}(v)}E_+ \cap 
 \square_{{\tilde R}_{j_1}^{1}(v)}E_{-*} \cap {\tilde \Delta }^{(0)}
 ( {\tilde R}_{j_1}^{1}(v), S(2^k)) \bigr)\right. \nonumber \\
&\quad + \left.
  \mu_t^N\bigl( \square_{{\tilde R}_{j_1}^{1}(v)}E_+ \cap
    \Delta_{{\tilde R}_{j_1}^{1}(v)}E_{-*} \cap
 {\tilde \Delta }^{(0)} ( {\tilde R}_{j_1}^{1}(v), S(2^k))\bigr)
 \right\} . \nonumber
\end{align}
by the finite energy property we obtain from this that the right hand side of
(\ref{eq:8-16}) is bounded by
$$
{\tilde K}^{(0)} C_1^m
{\bigl( 1+ e^{(4h_c+8)/{\mathfrak{K}T}}\bigr) }^{\# R_{j_1}^1(v)} 
\times 
\mu_t^N\bigl( \Delta_v\Omega (2^k) \cap \Omega (2^k) \bigr) .
$$
Thus, the condition (S) is satisfied.
\end{proof}

\newpage
\section{Branching argument}

Here we present a variant of Lemma 7 of \cite{K87scaling} whose proof is
essentially the same as the original one. Only we have to use a lower
bound in the extension of Russo's formula and the Connection lemma instead of independence.

\begin{Lemma}[cf. \cite{K87scaling} Lemma 7] \label{KLemma7}
Let $2^{j_1+5} < 2^k < \min \{ L, \frac{N}{2} \}$. 
Then there exist constants $C_{24}>0, \zeta >0$, which depend only on $j_1,\varepsilon_0$ and $\eta$, such that
for $j_1+5\leq j <  k$, we have
\begin{align}\label{eq:klem7-1}
&\int_0^1 dt \sum_{v \in S(2^{j-2})} \frac{|h-h_c|}{\mathfrak{K}T}
   \mu_t^N\left( \Gamma (v, S(2^j))\right) \\
 &\leq C_{24} \left(  2^{-\zeta (k-j)} +  \frac{|h-h_c|}{\mathfrak{K}T}2^{5(j+1)}e^{-\alpha 2^{j+1}}
     \right) . \nonumber 
\end{align}
\end{Lemma}

\begin{proof}
The strategy is just the same as in \cite{K87scaling}. 
It is sufficient to show for $j\leq k-11$, $v\in S(2^{j-2})$,
$0\leq t \leq 1$, and some constant $C_1^\#$ which depends only on $j_1,\varepsilon_0$ and $\eta$, 
\begin{align}\label{eq:klem7-2}
 &2^{\zeta (k-j)}\left\{ {\hat \mu}_t^N \bigl( \Delta (v, S(2^j)) \bigr)
     - C 2^{3(j+1)}e^{-\alpha 2^{j+1}}\right\} \\
   &\leq C_1^\# {\sum_{\mathbf n}}^*  \mu_t^N\left(
   v+ 2^{j+2}{\mathbf n} \mbox{ is pivotal for } A^+(2^k,2^k) \right), 
   \nonumber 
\end{align}
where $\displaystyle {\sum_{\mathbf n}}^*$ denotes the summation over all ${\mathbf n}$'s in $\mathbf{Z}^2$ such that
\begin{equation}\label{nCond}
S( 2^{j+2}{\mathbf n}, 2^{j+1}) \subset S(2^{k-3}).
\end{equation}
Indeed, from Remark \ref{RussoLower-2}, we have
\begin{align}\label{eq:klem7-3}
&\sum_{v\in S(2^{j-2})} \frac{|h-h_c|}{\mathfrak{K}T}
 {\sum_{\mathbf n}}^* \mu_t^N\left(
   v+ 2^{j+2}{\mathbf n} \mbox{ is pivotal for }  A^+(2^k,2^k)
    \right) \\
& \leq  C'_{17} \frac{d}{dt}\mu_t^N \bigl( A^+(2^k,2^k) \bigr). \nonumber 
\end{align}
We sum up (\ref{eq:klem7-2}) over $v$'s in $S(2^{j-2})$, and then use
(\ref{eq:klem7-3}) to integrate the resulting inequality in $t\in [0,1]$.
Then we obtain
\begin{align}\label{eq:klem7-4}
&\int_0^1 dt \sum_{v \in S(2^{j-2})}2^{\zeta (k-j)}  \frac{|h-h_c|}{\mathfrak{K}T}
  \left\{ \mu_t^N\left( \Delta (v, S(2^j)) \right)  
         - C2^{3(j+1)}e^{-\alpha 2^{j+1}}\right\} \\
 &\leq  C_1^\# C'_{17}. \nonumber  
\end{align}
From Lemma \ref{comparison of gamma and delta}, we have
\begin{equation}\label{eq:delta-gamma-boxchange}
\mu_t^N\left( \Delta (v, S(2^j))\right) \geq 
  K \mu_t^N\left( \Gamma (v, S(2^j))\right)
\end{equation}
This, together with (\ref{eq:klem7-4}) implies the inequality
(\ref{eq:klem7-1}).

For the case $k-11<j<k$, the argument in \cite{K87scaling} is valid.
%% So we have nothing to add.
%% 
%% \begin{quote} 
%% ----------------------------------------------------------------------- 
%% 
%% \noindent 
%% \framebox{\bf INTERMISSION} 
%%%%%%%%%%%%%%%%%%%%%%%%%%%%%% for the sake of completeness
%%%%%%%%%%%%%%%%%%%%%%%%%%%%%% we repeat the proof in this note
%%%%%%%%%%%%%%%%%%%%%%%%%%%%%% but in the final version the following
%%%%%%%%%%%%%%%%%%%%%%%%%%%%%% argument should disappear
%% 
But for the completeness we repeat the proof.
By Lemma \ref{comparison of gamma and delta} and Remark \ref{RussoLower-2}, we have
\begin{align*}
&\sum_{v\in S(2^{j-2})}  \frac{|h-h_c|}{\mathfrak{K}T} \mu_t^N\left(
          \Gamma (v, S(2^j))\right) \\
&\leq K \sum_{v\in S(2^{j-2})}  \frac{|h-h_c|}{\mathfrak{K}T}
        \mu_t^N\left( \Delta (v, S(2^j))\right) \\
&\leq K \sum_{v\in S(2^{j-2})} \frac{|h-h_c|}{\mathfrak{K}T}\mu_t^N\left(
        v \mbox{ is pivotal for }A^+(2^j,2^j) \right) \\
&\leq C_2^\# \frac{d}{dt} \mu_t^N\left( A^+(2^j,2^j)\right),
\end{align*}
where $C_2^\#$ depends only on $j_1,\varepsilon_0,$ and $\eta$. Therefore integrating this inequality in $t$ from $0$ to $1$,
we obtain 
\begin{equation}\label{eq:klem7-5}
\int_0^1 dt \sum_{v\in S(2^{j-2})} \frac{|h-h_c|}{\mathfrak{K}T} \mu_t^N\left(
 \Gamma (v,S(2^j)) \right)
   \leq C_3^\#.
\end{equation}
Thus, taking $C_{24}$ larger than $ 2^{10\zeta }C_3^{\#}$, we have 
(\ref{eq:klem7-1}) for $k-11<j<k$.

%%%%%%%%%%%%%%%%%%%%%%%%%%%%%%%%%%%%%%%%%%%%%%%%%%%%%%%%%%%%%%%%%%%%%%%%%
%\vspace{5mm}
%\noindent
%( Remark in this note )

% By this argument we realize that we have to prove 
% Kesten-Corollary 3 for both ${\hat \mu}_t$ and 
% ${\hat \mu}_t^N$ for the same constant $K_2$.

%\noindent
%(end of the remark in this note )
%\vspace{5mm}
%%%%%%%%%%%%%%%%%%%%%%%%%%%%%%%%%%%%%%%%%%%%%%%%%%%%%%%%%%%%%%%%%%%%%%%%
%%%%%%%%%%%%%%%%%%%%%%%%%%%%%%%%%%%%%%%%%%%%%%%%%%%%%%%%%%%%%%%%%%%%%%%%
%% ----------------------------------------------------------------------- 
%% \end{quote} 

So, hereafter we assume that $j \leq  k-11$.
If we take $\mathbf n $ satisfying (\ref{nCond}),
%the inclusion $S( 2^{j+2}{\mathbf n},2^{j+1}) \subset S(2^{k-3})$,
then we have for  $v \in S(2^{j-2})$,
$$
\mu_t^N\left( v+2^{j+2}{\mathbf n} \mbox{ is pivotal for }
                A^+(2^k,2^k)\right)
\geq \mu_t^N\left( 
    \Delta (v+2^{j+2}{\mathbf n}, S(2^k)) \right).
$$
By the mixing property, the Connection lemma, Lemmas \ref{comparison of gamma and delta} and \ref{lem:comparison of gamma and delta-tilde}, %(\ref{gamma-delta-out}), 
the right hand side of the above inequality is not less than
\begin{align*}
 &C_4^\# \mu_t^N\left(
   {\tilde \Delta }(S(2^{j+2}{\mathbf n}, 2^{j+2}) ,S(2^k))\right) \\
 & \quad \times \left\{
    \mu_t^N\left(
   \Delta (v+2^{j+2}{\mathbf n}, S(2^{j+2}{\mathbf n}, 2^{j+1}) )
     \right) 
   -C2^{3(j+1)}e^{-\alpha 2^{j+1}} \right\} \\
 &\geq C_5^\# \mu_t^N\left(
   {\tilde \Gamma }( S(2^{j+2}{\mathbf n}, 2^{j+2}),S(2^k) )\right) \\
 & \quad \times \left\{
    \mu_t^N\left(
   \Gamma (v+2^{j+2}{\mathbf n}, S(2^{j+2}{\mathbf n}, 2^{j+1}) )
     \right) 
   -\tilde{C}2^{3(j+1)}e^{-\alpha 2^{j+1}} \right\}
\\
 &= C_5^\# \mu_t^N\left(
   {\tilde \Gamma }( S(2^{j+2}{\mathbf n}, 2^{j+1}),S(2^k) )\right) \\
 &\quad \times \left\{
    \mu_t^N\left(
      \Gamma (v, S(2^{j+1})) \right)
   -\tilde{C}2^{3(j+1)}e^{-\alpha 2^{j+1}}
   \right\}.
\end{align*}
The last equality is because of the translation invariance of $\mu_t^N$.
The constants $\tilde{C}$ and $C_5^\#$ depend only on $j_1,\eta$ and $\varepsilon_0$.

Summing up this inequality with respect to ${\mathbf n}$ which satisfy (\ref{nCond}),
%$$
%S( 2^{j+2}{\mathbf n},2^{j+1}) \subset S(2^{k-3}),
%$$
we obtain
\begin{align*}
&{\sum_{\mathbf n}}^* \mu_t^N\left(
  v+2^{j+2}{\mathbf n} \mbox{ is pivotal for } A^+(2^k,2^k) \right) \\
&\geq  C_5^\# \left\{  \mu_t^N\left( \Gamma (v, S(2^{j+1})) \right)
 - \tilde{C}2^{3(j+1) }e^{-\alpha 2^{j+1}}\right\} \\
& \quad \times 
  {\sum_{\mathbf n}}^* \mu_t^N\left(
   {\tilde \Gamma }(S(2^{j+2}{\mathbf n}, 2^{j+1}),S(2^k)) \right).
\end{align*}
Thus, by Remark \ref{RussoLower-2}, we only have to show that
\begin{equation}\label{eq:klem7-6}
2^{\zeta (k-j)} \leq C_6^\# {\sum_{\mathbf n}}^*  \mu_t^N\left(
   {\tilde \Gamma }(S(2^{j+2}{\mathbf n},2^{j+1}),S(2^k)) \right).
\end{equation}
The constant $C_6^\#$ depends only on $j_1,\eta$ and $\varepsilon_0$.

Let ${\mathcal R}={\mathcal R}(\omega )$ be the lowest horizontal $(+)$-crossing of $S(2^k)$.
First we show that
$$
\mu_t^N\left( {\mathcal R}\mbox{ exists in }
                  [-2^k,2^k]\times [-2^{k-4},2^{k-4}] \right)
\geq \frac{1}{2} \delta_{48}^2\delta_1>0.
$$
As in \cite{K87scaling} we have to note that the event in the right hand side 
of the above inequality occurs when
\begin{enumerate}
\item there is a horizontal $(+)$-crossing in 
  $[-2^k,2^k]\times [\frac{1}{3}2^{k-4},2^{k-4}]$, and
\item there is a horizontal $(-*)$-crossing in 
  $[-2^k,2^k]\times [-2^{k-4},-\frac{1}{3}2^{k-4}]$
 and it is connected by a $(-*)$-path with the bottom of 
  $S(2^k)$. 
\end{enumerate}
The $\mu_t^N$-probability of the first event is not less than
$\delta_{48}$, and the probability of the second event is by the FKG
inequality not less than $\delta_{48}\delta_1$.
Since the first and the second events are occuring in distance not
less than $\frac{2}{3}2^{k-4}$, the $ \mu_t^N$-probability that the both events
occur is not less than 
$$
\delta_{48}^2\delta_1 - C2^{3k}e^{-\alpha \frac{2}{3}2^{k-4}}.
$$
The conditions (\ref{eq:the constant C1}), (\ref{eq:delta-1}) and (\ref{eq:cond-for-j1-2}) assures that this is not less than $\delta_{48}^2/2$.

Next, let us fix a horizontal crossing $r_0$ of 
  $[-2^k,2^k]\times [-2^{k-4}, 2^{k-4}]$.
By the Connection lemma, the $\mu_t^N(\, \cdot \,|\, {\mathcal R}=r_0 )$-probability of the event that there is a $(-*)$-path in
$[-2^{k-4},2^{k-4}]\times [-2^k,2^k]$ connecting the top side of $S(2^k)$
with $\partial^* r_0$ is not less than 
$\frac{1}{4}\delta_{128}$.
Further, if this event occurs, we write the leftmost one of such $(-*)$-paths
by ${\mathcal S}^*$.
We condition on the event 
$$
\bigl\{ {\mathcal R}=r_0, \, {\mathcal S}^*=s^* \bigr\} .
$$
and denote by ${\tilde \mu }$ the conditional probability;
$$
{\tilde \mu } := 
\mu_t^N \bigl(\, \cdot \,\,\bigl|\, {\mathcal R}=r_0,\, {\mathcal S}^*=s^* \bigr).
$$
Denote the point by $w$ such that $s^*$ connects $\partial^* \{w\}$ with the top side of $S(2^k)$. If there are more than one such points, then we take the rightmost one in $r_0$.
Let $S'$ be the region in $S(2^k)$
which is above $r_0$ and to the right of $s^*$.
The lower boundary of $S'$ is a part of $r_0$ that connects 
$w$ with the right side of $S(2^k)$.
We write this $(+)$-path by $r$.
Let us take a $v \in S(2^{j-2})$ and a vector ${\mathbf n}$ which satisfies (\ref{nCond}) and 
$$ 
%S(2^{j+2}{\mathbf n}, 2^{j+1})\subset S(2^{k-3}) \mbox{ and }
   S(2^{j+2}{\mathbf n}, 2^{j+1})\cap r \not= \emptyset .
$$
If there exists a $(-*)$-path in $S'$ that connects $s^*$ with
$ S(2^{j+2}{\mathbf n}, 2^{j+1})$, then 
${\tilde \Gamma }(S(2^{j+2}{\mathbf n}, 2^{j+1}),S(2^k))$
occurs, since we are conditioning on the event
$$
\bigl\{ {\mathcal R}=r_0, \, {\mathcal S}^*=s^* \bigr\}.
$$
In particular, if the box $ S(2^{j+2}{\mathbf n}, 2^{j+1})$ intersects 
both $s^*$ and $r$, then the event
${\tilde \Gamma }( S(2^{j+2}{\mathbf n}, 2^{j+1}),S(2^k))$
occurs.

Let $\ell $ be an integer such that $j+5\leq \ell \leq k-6$.
For this $\ell $, note that the condition
$$
S(2^{j+2}{\mathbf n}, 2^{j+1})\cap S(w, 3\cdot 2^\ell )\not=\emptyset
$$
implies that $S(2^{j+2}{\mathbf n},2^{j+1})\subset S(2^{k-3})$.

For fixed $j\leq k-11, \ell$ with $j+5\leq \ell \leq k-6$, %$ v\in S(2^{j-2})$,
${\mathbf n}=(n_1,n_2) \in {\mathbf Z}^2$, 
and a horizontal crossing $r_0$ of $[-2^k,2^k]\times [-2^{k-4},2^{k-4}]$,
a point $w$ in $r_0$, and a $(*)$-path $s^*$ connecting $\partial^* \{ w \}$ with the
top of $S(2^k)$ above $r_0$, we define
\begin{align}\label{eq:klem7-7}
&Y({\mathbf n}, w, \ell ,r_0, s^*) \\
&= 
 \begin{cases}
    1 & \mbox{if } S(2^{j+2}{\mathbf n},2^{j+1})\mbox{ intersects both } r_0
        \mbox{ and } S(w, 3\cdot 2^\ell ),  \\
   &  \mbox{and there exists a $(-*)$-path $t^*$ connecting $S(2^{j+2}{\mathbf n},2^{j+1})$}\\
   &  \mbox{with $s^*$ in $S(w, 3\cdot 2^\ell )$ ($t^*$ can be part of $s^*$)},\\
    0 & \mbox{otherwise}.\\
\end{cases} \nonumber
\end{align}
Note that if $Y({\mathbf n},w,\ell ,r_0,s^*)=1$, then the event
$$ {\tilde \Gamma }(S(2^{j+2}{\mathbf n},2^{j+1}),S(2^k) )$$
occurs, provided that ${\mathcal R}=r_0$ and that $s^*$ is a 
$(-*)$-path.

Finally, we put
\begin{equation}\label{eq:klem7-8}
Z(\ell ) = \min E_{ \mu_t^N}\left[
 \sum_{\mathbf n} Y({\mathbf n},w,\ell , r_0, s^*) 
 \,\bigg|\, {\mathcal F}_{S(2^k)\setminus S'} \right] (\omega ) ,
\end{equation}
where  $S'$ is the region in $S(2^k)$ above $r_0$ and right to $s^*$.
Also, the minimum is taken over all of the following;
\begin{enumerate}
\item $t$ with $0\leq t \leq 1$,
\item horizontal crossing $r_0$ of $[-2^k,2^k]\times [-2^{k-4},2^{k-4}]$, 
\item $w \in r_0 $ such that $S(w, 3\cdot 2^\ell ) \subset S(2^k)$,
\item $(*)$-path $s^*$ that connects $\partial^* \{ w \}$ with the top of $S(2^k)$ above $r_0$,
%\item $v\in S(2^{j-2})$,
\item $\omega $ satisfying 
\begin{enumerate}
\item $\omega (u)=-1 $ for every $u\in s^*$, and
\item ${\mathcal R}(\omega )=r_0$.
\end{enumerate}
\end{enumerate}
%Note that 
%$
%{\tilde \mu } = {\hat \mu }_t^N( \cdot \mid {\mathcal R}=r_0,{\mathcal S}^* =s^* )
%$.
Consider an annulus $A_\ell (w)= S(w,2^{\ell +1})\setminus S(w,2^{\ell })$ 
and take 
a point $w_1\in A_\ell (w)$ that is on $r_0$ between $w$ and the right side of 
$S(2^k)$.
Then $S(w_1,3\cdot 2^{\ell -4}) \subset S(w, 3\cdot 2^\ell )$, and
$$ 
S(w,3\cdot 2^{\ell -4})\cap S(w_1,3\cdot 2^{\ell -4})=\emptyset 
$$
So, if there is a $(-*)$-path $t^*$ in $A_\ell (w)$ that connects $s^*$
with a $(*)$-adjacent point of $w_1$ above $r_0$, then we can find a 
$(-*)$-path
$s_1^*$ connecting the top of $S(2^k)$ with $w_1$ above $r_0$ by first
going down along $s^*$ until we get to the point where $t^*$ starts and then 
switching to $t^*$ until we reach $w_1$.  
Then in this case we have
\begin{align*}
&\sum_{\mathbf n}Y({\mathbf n},w,\ell ,r_0,s^*)\\
&\geq \sum_{\mathbf n}Y({\mathbf n},w,\ell -4 ,r_0,s^*) +
      \sum_{\mathbf n}Y({\mathbf n},w_1,\ell -4 ,r_0,s_1^*).
\end{align*}
Recall that $r$ denotes the part of $r_0$ that connects $w$ with the right side
of $S(2^k)$, and define the following event;
$$
E_\ell (r, s^*)=
     \left\{ 
 \begin{array}{@{\,}c@{\,}} 
  \mbox{there exists a $(-*)$-path in $S'\cap A_\ell (w) $}\\
  \mbox{connecting $s^*$ with $r$}
 \end{array}
 \right\}.
$$
Then by the Connection lemma, we have
$$
{\tilde \mu } \bigl( E_\ell (r,s^* ) \bigr) \geq \frac{1}{4} \delta_{128}=:C_7^\#.
$$
On this event, we can take $t^*$ the minimal $(-*)$-path connecting
$s^*$ with a $(*)$-neighbour of $r$ in $S'\cap A_\ell (w)$, and $w_1$ 
as a point in $r$ that is $(*)$-neighbour to $t^*$.
Thus, we have
\begin{align*}
&E_{\tilde \mu }\left[ \sum_{\mathbf n}Y({\mathbf n}, w, \ell, r_0, s^*)
           \right] \\
&\geq  E_{\tilde \mu }\left[ 
\sum_{\mathbf n}Y({\mathbf n}, w, \ell -4, r_0, s^*) \right]
 +  E_{\tilde \mu }\left[ 
\sum_{\mathbf n}Y({\mathbf n}, w_1, \ell -4, r_0, s_1^*) ; E_\ell(r,s^*)
\right] \\
&\geq \left( 1 + C_7^\# \right) Z(\ell -4 ).
\end{align*}
The above argument depends on the definition of ${\tilde \mu }$, but it is
easy to notice that we can repeat this argument for
$$ Y({\mathbf n},w_1,\ell-4, r_0,s_1^*)$$
with respect to the conditional probability
$$\mu_t^N \bigl( \, \cdot \, \bigl| \, {\mathcal F}_{S(2^k)\setminus S'_1} \bigr) (\omega ),$$
where $S'_1$ is the region in $S(2^k)$ above $r_0$ and right to $s_1^*$,
and that $\omega $ satisfies ${\mathcal R}=r_0, \, \omega (u)=-1 , u\in s_1^*$.
Then recursively we obtain for some $p\in \{ 1,2,3,4\} $,
$$
Z(\ell ) \geq (1+ C_7^\# )^{\frac{1}{4} (\ell - j-p )}Z(j+p),
$$
where $\ell -j \equiv p\, (\mbox{mod }  4)$, and by definition $Z(j+p)\geq 1$.

Thus, taking $\zeta = \frac{1}{4}\log _2(1+C_7^\#)$ and
$\ell = k-6,\, C_6^\#=2(\delta_{48}^2\delta_1)^{-1}(1+C_7^\# )^{5/2}$,
we have (\ref{eq:klem7-6}).
\end{proof}
\newpage
\section{Proof of Kesten-Theorem 1}

\begin{Theorem}[cf. \cite{K87scaling} Theorem 1] \label{KTheorem1} There exist constants $0<C_{25},C_{26}<\infty$, depending only on $j_1$ and $\varepsilon_0$, such that
\begin{equation}\label{eq:thm1-1}
  C_{25}\leq  \frac{\pi_h(n)}{\pi_{\text{cr}} (n)} \leq C_{26} \quad \mbox{for all }
   n< L(h,\varepsilon_0).
\end{equation}
\end{Theorem}

\begin{proof} It suffices to prove (\ref{eq:thm1-1}) for $\pi_h^N (n)$ and $\pi_{\text{cr}}^N (n)$ with sufficiently large $N$. We define
$$
\pi_t^N(n) := \mu_t^N\left(
      \mathbf{O}\stackrel{+}{\leftrightarrow}  \partial_{in} S(n)
                      \right).
$$
Let $n < L(h,\varepsilon_0)$ and $k= \lfloor \log_2 n\rfloor $.
By monotonicity we have
\begin{equation}\label{eq:thm1-2}
\pi_t^N(2^{k+1})   \leq \pi_t^N(n)
 \leq  \pi_t^N(2^{k}).
\end{equation}
Also, by Lemma \ref{RSWbound} and the FKG inequality, there exists a constant $C_1^\#>0$
such that 
\begin{equation}\label{eq:thm1-3}
\pi_t^N(2^{k+1})  \geq C_1^\# \pi_t^N(2^{k}).
\end{equation}
This constant $C_1^\#$ depends only on $\varepsilon_0$.
Therefore we only have to show (\ref{eq:thm1-1}) for $2^k$ instead of $n$.

By Theorem \ref{RussoOneArm}, we have
\begin{align}\label{eq:thm1-4}
&\frac{d}{dt}\pi_t^N(2^k) \\
&\leq \frac{|h-h_c|}{\mathfrak{K}T}\biggl\{ C_2^\# \pi_t^N(2^k) \nonumber  \\
&\quad +   C_3^\# \sum_{v \in S(2^{k})\setminus S(2^{j_1+5})}
              \mu_t^N\bigl( v \mbox{ is pivotal for }
                \{ \mathbf{O}\stackrel{+}{\leftrightarrow} \partial_{in} S(2^k) \}
                      \bigr) \biggr\}. \nonumber
\end{align}
%The summation in (\ref{eq:thm1-4}) is taken over $v$'s with 
%$2^{j_1+M}<|v|_\infty \leq 2^k$.

1) First, we consider $v$'s with $ 2^{j_1+5}<|v|_\infty \leq 2^{k-1}$.
Let $j_1+5\leq j\leq k-2$, and take $v$ with $2^j<|v|_\infty \leq 2^{j+1}$.
Write 
$$
{\overline Q}_{j-5}(\ell_1,\ell_2) := [0,2^{j-5}]^2+2^{j-5}(\ell_1,\ell_2).
$$
Then $v$ belongs to at least one of ${\overline Q}_{j-5}(\ell_1,\ell_2)$
with $-2^6 \leq \ell_1,\ell_2 \leq 2^6$.
Let us assume that 
${\overline Q}_{j-5}(\ell_1,\ell_2)\supset Q_{j-5}(v)$.
Then, since $S_{j-5}^3(v)\subset S(2^k)$, we have
\begin{align}\label{eq:thm1-5}
&\mu_t^N \bigl( v \mbox{ is pivotal for }
         \{ \mathbf{O}\stackrel{+}{\leftrightarrow} \partial_{in} S(2^k) \} \bigr) \\
&\leq  \mu_t^N\bigl( \Gamma (v, S_{j-5}^2(v)) \cap
                 {\tilde \Gamma }_0( S_{j-5}^3(v),S(2^k)) \bigr). \nonumber
\end{align}
By the mixing property the right hand side of (\ref{eq:thm1-5})
is bounded from above by
\begin{equation}\label{eq:thm1-6}
\mu_t^N \bigl( {\tilde \Gamma }_0(S_{j-5}^3(v),S(2^k)) \bigr) \left\{
 \mu_t^N \bigl( \Gamma (v,S_{j-5}^2(v)) \bigr) +C2^{3(j-2)}e^{-\alpha 2^{j-3}}
  \right\}.
\end{equation}
Let $B(S_{j-5}^3(v),S(2^k))$ be the event that 
\begin{enumerate}
\item there is a $(+)$-path $r$ in $S(2^k)$ connecting the origin with 
$S_{j-5}^3(v)$, and
\item there is a $(+)$-path $r'$ in $S(2^k)$ connecting $S_{j-5}^3(v)$
with $\partial_{in}S(2^k)$.
\end{enumerate}
Note that ${\tilde \Gamma }_0(S_{j-5}^3(v),S(2^k))$ is a subset of 
$B(S_{j-5}^3(v),S(2^k))$, and by the FKG inequality, we have
\begin{equation}\label{eq:thm1-7}
\mu_t^N \bigl( B(S_{j-5}^3(v),S(2^k)) \bigr)\leq \delta_4^{-4}\pi_t^N(2^k).
\end{equation}
Thus, by (\ref{eq:thm1-5})--(\ref{eq:thm1-7}), we have
\begin{align}\label{eq:thm1-8}
&\mu_t^N \bigl( v \mbox{ is pivotal for }\{ 
             \mathbf{O}\stackrel{+}{\leftrightarrow} \partial_{in} S(2^k) \} \bigr) \\
&\leq \delta_4^{-4}\pi_t^N(2^k)\left\{
  \mu_t^N \bigl( \Gamma (v,S_{j-5}^2(v)) \bigr) + C2^{3(j-2)}e^{-\alpha 2^{j-3}}
      \right\}. \nonumber
\end{align}
Summing up (\ref{eq:thm1-8}) over $v$'s in a  
${\overline Q}_{j-5}(\ell_1,\ell_2)$, we obtain
\begin{align}\label{eq:thm1-9}
&\sum_{v\in {\overline Q}_{j-5}(\ell_1,\ell_2)}
 \mu_t^N\bigl( v \mbox{ is pivotal for }
       \{ \mathbf{O}\stackrel{+}{\leftrightarrow} \partial_{in} S(2^k) \}
        \bigr) \\
&\leq \delta_4^{-4}\pi_t^N(2^k)\left\{
   \sum_{v\in {\overline Q}_{j-5}(\ell_1,\ell_2)}
      \mu_t^N\bigl( \Gamma (v,S_{j-5}^2(v)) \bigr)
       +C2^{5(j-2)}e^{-\alpha 2^{j-3} } \right\} \nonumber
\end{align}
Since $\mu_t^N$ is translation invariant,
$$
\sum_{v\in {\overline Q}_{j-5}(\ell_1,\ell_2)}
      \mu_t^N\bigl( \Gamma (v,S_{j-5}^2(v)) \bigr)
 \leq \sum_{v\in S(2^{j-5})}
      \mu_t^N\bigl( \Gamma (v, S(2^{j-3})) \bigr).
$$
Therefore summing up (\ref{eq:thm1-9}) over $\ell_1,\ell_2$'s with
$-2^6\leq \ell_1,\ell_2\leq 2^6$, we obtain
\begin{align*}
&\sum_{2^j<|v|_\infty\leq 2^{j+1}}
     \mu_t^N\bigl( v \mbox{ is pivotal for }
        \{ \mathbf{O}\stackrel{+}{\leftrightarrow} \partial_{in} S(2^k) \}
         \bigr) \\
&\leq C_4^\# \pi_t^N(2^k) \left\{
   \sum_{v\in S(2^{j-5})}
         \mu_t^N \bigl(\Gamma (v,S(2^{j-3}) ) \bigr)+C2^{5(j-2)}e^{-\alpha 2^{j-3}}
         \right\} 
\end{align*}
for some positive constant $C_4^\#$ depending only on $j_1,\varepsilon_0$ and $\eta$.
Finally performing the summation over $j$'s we obtain
\begin{align} \label{eq:thm1-10}
&\sum_{2^{j_1+5}<|v|_\infty\leq 2^{k-1}}
  \mu_t^N\bigl( v \mbox{ is pivotal for }
          \{ \mathbf{O}\stackrel{+}{\leftrightarrow} \partial_{in} S(2^k) \}
            \bigr) \\
&\leq C_5^\# \pi_t^N(2^k) \left\{
       \sum_{j=j_1+5}^{k-1}\sum_{v\in S(2^{j-5})}
          \mu_t^N \bigl(\Gamma (v,S(2^{j-3}))\bigr) + 1
          \right\} \nonumber
\end{align}
for some constant $C_5^\#$, which depends only on $j_1,\varepsilon_0$ and $\eta$.

2) Second, let us consider the case where $2^{k-1} < |v|_\infty \leq 2^{k}$.
This time we consider $S_{k-6}^2(v)$.
If  $S_{k-6}^2(v)\subset S(2^k)$, then we also obtain
\begin{align}\label{eq:thm1-8'}
&\mu_t^N \bigl( v \mbox{ is pivotal for }\{ 
             \mathbf{O}\stackrel{+}{\leftrightarrow} \partial_{in} S(2^k) \} \bigr) \\
&\leq \delta_4^{-4}\pi_t^N(2^k)\left\{
  \mu_t^N \bigl( \Gamma (v,S_{k-6}^2(v)) \bigr) + C2^{3(k-3)}e^{-\alpha 2^{k-4}}
      \right\}. \nonumber
\end{align}
in the same way as (\ref{eq:thm1-8}).
By applying Lemma \ref{comparison of gamma and delta} similar to (\ref{eq:delta-gamma-boxchange}), we obtain
$$
\mu_t^N\bigl( \Gamma (v,S_{k-6}^2(v))\bigr) \leq
 C_6^\# \mu_t^N\bigl( \Delta (v, S_{k-6}^2(v))\bigr),
$$
and then we apply Lemma \ref{relation of deltas}, four times to obtain
\begin{align}\label{eq:thm1-11}
& \mu_t^N\bigl( \Delta (v, S_{k-6}^2(v))\bigr) \\
&\leq C_7^\# \mu_t^N\bigl( \Delta(v,S(2^k))\bigr) \nonumber \\
&\leq C_7^\# \mu_t^N\bigl( v \mbox{ is pivotal for }
                         A^+(2^k,2^k) \bigr) \nonumber
\end{align}
for some positive constant $C_7^\#$ depending only on $j_1,\varepsilon_0$ and $\eta$.
If $S_{k-6}^2(v)\not\subset S(2^k)$, then we consider
$T_{k-6}^2(v)$, and we have
\begin{align}\label{eq:tem1-12}
&\mu_t^N \bigl( v \mbox{ is pivotal for }\{ 
             \mathbf{O}\stackrel{+}{\leftrightarrow} \partial_{in} S(2^k) \} \bigr) \\
&\leq \delta_4^{-3}\pi_t^N(2^k)\left\{
  \mu_t^N \bigl( \Gamma (v,T_{k-6}^2(v)) \bigr) + C2^{3(k-3)}e^{-\alpha 2^{k-4}}
      \right\}. \nonumber
\end{align}
As above by Lemma \ref{comparison of gamma and delta} and Lemma \ref{relation of deltas}, we have
\begin{align}\label{eq:thm1-13}
& \mu_t^N\bigl( \Gamma (v, T_{k-6}^2(v))\bigr) \\
&\leq C_8^\# \mu_t^N\bigl( \Delta(v,S(2^k))\bigr) \nonumber \\
&\leq C_8^\# \mu_t^N\bigl( v \mbox{ is pivotal for }
                         A^+(2^k,2^k) \bigr). \nonumber
\end{align}
Thus, from (\ref{eq:thm1-8'})--(\ref{eq:thm1-13}),
we obtain
\begin{align}\label{eq:thm1-14}
&\sum_{2^{k-1}<|v|_\infty \leq 2^k}
 \mu_t^N \bigl( v \mbox{ is pivotal for }\{
      \mathbf{O}\stackrel{+}{\leftrightarrow} \partial_{in} S(2^k) \} \bigr) \\
&\leq  C_9^\# \pi_t^N(2^k) \left\{
  \sum_{2^{k-1}<|v|_\infty \leq 2^k}
  \mu_t^N\bigl(
     v \mbox{ is pivotal for } A^+(2^k,2^k) \bigr)
    + 1 \right\} \nonumber
\end{align}
for some positive constant $C_9^\#$ depending only on $j_1,\varepsilon_0$ and $\eta$.

Combining (\ref{eq:thm1-4}), (\ref{eq:thm1-10}) and (\ref{eq:thm1-14}), we obtain
\begin{align}\label{eq:thm1-15}
&\frac{d}{dt}\log \pi_t^N(2^k) \\
&\leq C_{10}^\# \frac{|h-h_c|}{\mathfrak{K}T} 
      \left\{
       \sum_{j=j_1+5}^{k-1}\sum_{v\in S(2^{j-5})}
          \mu_t^N\bigl(\Gamma (v,S(2^{j-3}))\bigr) + 1 \right.  \nonumber\\
&\quad \left. + \sum_{2^{k-1}<|v|_\infty \leq 2^k}
     \mu_t^N\bigl(
     v \mbox{ is pivotal for } A^+(2^k,2^k) \bigr) \right\} \nonumber
\end{align}
for some positive constant $C_{10}^\#$ depending only on $j_1,\varepsilon_0$ and $\eta$.
Integrating (\ref{eq:thm1-15}) with respect to $t\in [0,1]$, we obtain
\begin{align} \label{eq:thm1-16}
&\frac{\pi_1^N(2^k)}{\pi_0^N(2^k)} \\
&\leq  C_{11}^\# \left\{
          \sum_{j=j_1+5}^{k-1}\frac{|h-h_c|}{\mathfrak{K}T}\int_0^1dt\sum_{v\in S(2^{j-5})}
          \mu_t^N\bigl(\Gamma (v,S(2^{j-3}))\bigr) \right. \nonumber \\
&\quad +   \frac{|h-h_c|}{\mathfrak{K}T} + 1 \Biggr\}. \nonumber
\end{align}
By Lemma \ref{KLemma7}, the right hand side of (\ref{eq:thm1-16}) is bounded by
$$
C_{11}^\# \left\{ C_{24} \sum_{j=j_1+5}^{k-1}\left( 
           2^{\zeta (k-j)} + \frac{|h-h_c|}{\mathfrak{K}T} 2^{5(j+1)}e^{-\alpha 2^{j+1}}
            \right) +  \frac{|h-h_c|}{\mathfrak{K}T} + 1 \right\},
$$
which is a constant depending only on $j_1,\eta$ and $\varepsilon_0$.
\end{proof}
\newpage
\section{Proof of Kesten-Theorem 2}
\begin{Theorem}[cf. \cite{K87scaling} Theorem 2] \label{KTheorem2} For $h > h_c$, \
\begin{align*}
\pi_{\text{cr}} (L(h,\varepsilon_0)) &\leq C_{27}  \pi_h (L(h,\varepsilon_0)) \\
&\leq C_{28} \theta(h) \\
&\leq C_{28} \pi_h (L(h,\varepsilon_0)) \\
&\leq C_{29}  \pi_{\text{cr}} (L(h,\varepsilon_0)).
\end{align*}
\end{Theorem}

\begin{proof}
Essentially, the only thing to be proved is the second inequality. 
But in order to use Theorem \ref{KTheorem1}, we have to remark
that (\ref{eq:thm1-1}) is valid for $n\leq 2L(h,\varepsilon_0)-2$
when we replace the constants $C_{25}$ and $C_{26}$ with
\begin{equation}\label{eq:thm1-1-ext}
C_{25}'=\delta_4^4\delta_2C_{25}, \quad C_{26}'=(\delta_4^4\delta_2)^{-1}C_{26}.
\end{equation}
In particular, (\ref{eq:thm1-1}) is valid for $n=L(h,\varepsilon_0)$,
too, with the constants $C_{25}'$ and $ C_{26}'$ above.
Let $L=L(h,\varepsilon_0)$. 
By a standard construction together with Lemma \ref{lem:accfr} and 
Lemma \ref{RSW-gen}, we can see that
\begin{align*}
\theta(h)  &\geq \mu_{T,h} 
\left( 
\begin{array}{@{\,}c@{\,}} 
\mbox{$ \mathbf{O}   \stackrel{+}{\leftrightarrow} \partial_{in} S(L)$, and there exists a $(+)$-circuit} \\
\mbox{surrounding $S(L/2)$ in $S(L) \setminus S(L/2)$} \\
\mbox{which is $(+)$-connected to the infinite $(+)$-cluster} \\
\end{array} 
\right) \\
&\geq \pi_{h}(L) \cdot \delta_4^4\delta_2 \cdot \prod_{n=1}^{\infty} (1-\lambda \theta^{2^n}).
\end{align*}

\end{proof}

\newpage
\section{Proof of Kesten-Corollary 1 and Kesten-Theorem 3}

The argument in this section is valid both for $(+)$- and $(-*)$-connection.
As usual we state results only for $(+)$-connection.

\subsection{One-arm event}

%We use a special case of (\ref{K(2.20)}).
%\begin{Lemma}[cf. \cite{K87scaling} (2.20)]\label{lem:12-2}
%$\pi_{\text{cr}}(n)$ is decreasing in $n$, but
%\begin{equation}  \delta_4^4 \delta_2 \pi_{\text{cr}}(n) \leq \pi_{\text{cr}} (2n) \leq \pi_{\text{cr}} (n) %\label{K(2.20)}
%\end{equation}
%holds for $n\geq 4n_2$.
%\end{Lemma}

\begin{Lemma}\label{inc-7}
Let $L_0$ be a constant such that $L_0\geq 6n_3(8,\varepsilon_0)$.
Then there exists a constant $ C_{30}>0$ depending only on 
$L_0,\varepsilon_0$,
such that for $L_0\leq n\leq  L(h,\varepsilon_0)$,
\begin{equation}\label{eq:inc-7}
( n-L_0 )\pi_h(n) \leq \sum_{k=L_0+1}^n \pi_h(k) \leq C_{30} (2n-L_0)\pi_h(n)
\end{equation}
\end{Lemma}
As mentioned in \cite{K86i}, this implies that $n \pi_h(n) $ is 
essentially increasing, and therefore 
$$
\pi_h(n) \geq  \frac{C_{31}'}{n}.
$$
for some constant $C_{31}'>0$ depending only on $L_0$ and $\varepsilon_0$.
Then inserting this into (\ref{eq:inc-7}) repeatedly, we obtain for 
$m\geq 1$,
$$
\pi_h(n) \geq \frac{ C_{31}''(m) (\log n)^m }{n} 
$$
for some constant $C_{31}''(m)>0$, depending on 
$L_0, m$, and $\varepsilon_0$ for every $m\geq 1$.
The proof of the Lemma \ref{inc-7} is similar to the original one
in \cite{K86i}. Only we have to use the connection lemma instead of
independence.

%% \begin{quote} 
%% ----------------------------------------------------------------------- 
%% 
%% \noindent 
%% \framebox{\bf INTERMISSION} 
%% 
%Here in this note, for the sake of completeness, we prove 
%Lemma \ref{inc-7}. 

\begin{proof} The first inequality immediately follows from the fact that $\pi_{h} (k)$ is decreasing in $k$.

Let
\[ S^{\infty} (n) := \{ (x^1,x^2) : - n \leq x_1 \leq n \} \]
and
\[ V_n := \# \left\{ (0,k) : 
\begin{array}{@{\,}c@{\,}} 
L_0 \leq k \leq 2n, \, \mbox{ such that }\\
\mbox{$(0,k) \stackrel{+}{\leftrightarrow} \{ x^1 = n \}$ in $S^{\infty} (n)$} \\
\end{array}
\right\}. \]
By the translation invariance,
\begin{align} \label{K86i(7)-1}
E_{T,h} [V_n] &= (2n-L_0) \mu_{T,h} \bigl( \mbox{$\mathbf{O} \stackrel{+}{\leftrightarrow} \{ x^1 = n \}$ in $S^{\infty} (n)$}\bigr) \\
&\leq (2n-L_0) \pi_{h} (n), \notag
\end{align}
where $E_{T,h}$ denotes the expectation with respect to $\mu_{T,h}$.
We define
\[ E_n^+ := \{ \mbox{there exists a horizontal $(+)$-crossing in $[-n,n] \times [0,n]$} \}. \]
For $\omega \in E_n^+$, let $\mathcal{R}=\mathcal{R} (\omega)$ denote the lowest horizontal $(+)$-crossing in $[-n,n] \times [0,n]$. Put
\[
\tilde{V}_n (\omega) := 
\begin{cases}
\# \left\{ (0,k) : 
\begin{array}{@{\,}c@{\,}} 
\mbox{$L_0 \leq k \leq 2n$, $(0,k)$ is above $\mathcal{R}$, and} \\
\mbox{$(0,k)\stackrel{+}{\leftrightarrow}\mathcal{R}$ in $S^{\infty} (n)$} \\
\end{array} \right\} & \mbox{on $E_n^+$}, \\
0 & \mbox{on $(E_n^+)^c$}.
\end{cases} 
\]
Note that $V_n \geq \tilde{V}_n$, and
\begin{align} \label{K86i(7)-2}
E_{T,h} [V_n] &\geq E_{T,h} \bigl[ \tilde{V}_n \bigr] =   E_{T,h} \bigl[ \tilde{V}_n : E_n^+ \bigr] =\sum_{R}  E_{T,h} \bigl[ \tilde{V}_n : \mathcal{R}=R \bigr], %\notag
\end{align}
where $\displaystyle \sum_{R}$ denotes the summation over all horizontal crossings $R$ in $[-n,n] \times [0,n]$. 
Now we fix such a crossing $R$ arbitrarily, and define
\[ v^* = v^* (R) = (0,k_0^*) \]
as the highest point of $\{ x^1=0\} \cap R$.
Then we have
\begin{align} \label{K86i(7)-3}
&E_{T,h} \bigl[ \tilde{V}_n : \mathcal{R}=R \bigr] \\
&\geq \mu_{T,h} \bigl( \mathcal{R}=R \bigr) \sum_{\ell = L_0+1}^{2n-k_0^*} \mu_{T,h} \bigl( \mbox{$(0,k_0^*+\ell) \stackrel{+}{\leftrightarrow} R$ in $S^{\infty} (n)$} \,\bigr|\,  \mathcal{R}=R \bigr), \notag
\end{align}
Since $k_0^*\leq n$, the right hand side of (\ref{K86i(7)-3}) is not
less than
$$ \mu_{T,h} \bigl( \mathcal{R}=R \bigr) \sum_{\ell = L_0+1}^{n} \mu_{T,h} \bigl( \mbox{$(0,k_0^*+\ell) \stackrel{+}{\leftrightarrow} R$ in $S^{\infty} (n)$} \,\bigr|\,  \mathcal{R}=R \bigr). $$

Now for each $\ell \leq n$, we define an annulus
\[ 
\textstyle
A(\ell) := \left\{ \left( \left[ -\frac{2\ell}{3},\frac{2\ell}{3} \right] \times [-\frac{4}{3}\ell,\frac{4}{3}\ell] \right) \setminus \left( \left[ -\frac{\ell}{3},\frac{\ell}{3} \right] \times [-\ell,\ell] \right) \right\} + (0,k_0^*+\ell). 
\]
%(Strictly speaking, $A(\ell)$ depends on $R$, but we fix $R$ now.) 
Since 
\[ d_\infty \bigl( A(\ell), S^{\infty} (n)^c \bigr) \geq \dfrac{n}{3}, \]
we can see that $A (\ell) \subset S^{\infty} (n)$. The highest point $v^*$ of $R$ divides $R$ into two parts, which are denoted by $R_1$ and $R_2$, respectively. Let $A^+(\ell)$ be a portion of $A(\ell)$ above $R$, in which $R_1$ can be connected to $R_2$; $A^+(\ell) \cup R$ can contain a circuit surrounding $(0,k_0^*+\ell)$. We consider the event
\[
F^+(\ell) := \left\{
\begin{array}{@{\,}c@{\,}} 
\mbox{there is a $(+)$-path $r$ in $A^+(\ell)$} \\
\mbox{such that $r$ connects $R_1$ with $R_2$, and} \\
\mbox{$A^+(\ell) \cup R$ contains a $(+)$-circuit surrounding $(0,k_0^*+\ell)$} \\
\end{array}
\right\}.
\] 
By the condition that $L_0\geq 6n_3(8,\varepsilon_0)$, if $\ell \geq L_0$, then
we can use the Connection lemma, and obtain
\[ \mu_{T,h} \bigl( F^+(\ell)\, \bigl| \, \mathcal{R}=R \bigr) \geq \left( \dfrac{\delta_{64} \delta_{32}}{16} \right)^2 =: C_1^\# >0.
\]
For $\omega \in F^+(\ell)$, we denote the maximal $(+)$-circuit surrounding $(0,k_0^*+\ell)$ in $A^+(\ell) \cup R$ by $\mathcal{C}$. By the Markov property, the FKG inequality and (\ref{K(2.20)}),
\begin{align*}
&\mu_{T,h} \bigl( \mbox{$(0,k_0^*+\ell) \stackrel{+}{\leftrightarrow} R$ in $S^{\infty}(n)$} \, \bigl| \, \mathcal{R}=R,\,\mathcal{C}=C \bigr)\\
&=q_{\Theta_C,T,h}^{\omega^+} \bigl( (0,k_0^*+\ell) \stackrel{+}{\leftrightarrow} C \bigr) \\
&\geq \mu_{T,h} \bigl( (0,k_0^*+\ell) \stackrel{+}{\leftrightarrow} C \bigr) \geq \pi_{h} (2\ell),
\end{align*}
where $\Theta_C$ is the interior of the circuit $C$.
Since $\ell \leq n\leq L(h,\varepsilon_0)$, the right hand side is bounded from below by
$ \delta_4^4\delta_2\pi_h(\ell )$. Thus we have
\begin{align*}
&\mu_{T,h}\bigl( \mbox{$(0,k_0^*+\ell) \stackrel{+}{\leftrightarrow} R$ in $S^{\infty}(n)$} \, \bigl| \, \mathcal{R}=R \bigr)\\
&\geq\mu_{T,h} \bigl( \{ \mbox{$(0,k_0^*+\ell) \stackrel{+}{\leftrightarrow} R$ in $S^{\infty}(n)$}\} \cap F^+(\ell)\, \bigl| \, \mathcal{R}=R \bigr) \\
&= \sum_{C}  \mu_{T,h} \bigl( \mbox{$(0,k_0^*+\ell) \stackrel{+}{\leftrightarrow} R$ in $S^{\infty}(n)$} \, \bigl| \, \mathcal{R}=R,\,\mathcal{C}=C \bigr)
\mu_{T,h} \bigl( \mathcal{C}=C \,\bigl|\,\mathcal{R}=R \bigr) \\
&\geq \delta_4^4\delta_2\pi_{h}(\ell)\mu_{T,h} \bigl( F^+(\ell) \,\bigl|\,\mathcal{R}=R \bigr) \geq C_1^\# C_2^\# \pi_{h}(\ell),
\end{align*}
where $C_2^\# := \delta_4^4 \delta_2$.
Together with (\ref{K86i(7)-3}), this implies that
\[
E_{T,h} \bigl[ \tilde{V}_n : \mathcal{R}=R \bigr] \geq \mu_{T,h} \bigl( \mathcal{R}=R \bigr) C_1^\#C_2^\# \sum_{\ell=L_0+1}^n \pi_{h} (\ell).
\]
By (\ref{K86i(7)-2}),
\[ E_{T,h} \bigl[ V_n \bigr] 
\geq \mu_{T,h} \bigl( E_n^+ \bigr) C_1^\#C_2^\# \sum_{\ell=L_0+1}^n \pi_{h} (\ell).
\]
Since $\mu_{T,h} (E_n^+) \geq \delta_2$ by (\ref{eq:crossing}), (\ref{K86i(7)-1}) implies that
\begin{align*}
(2n-L_0) \pi_{h}(n) \geq C_1^\#C_2^\#\delta_2 \sum_{\ell=L_0+1}^n \pi_{h} (\ell),
\end{align*}
which proves the desired upper bound.
\end{proof}
%% ----------------------------------------------------------------------- 
%% \end{quote}

\subsection{Connectivity function}

For $x \in \mathbf{Z}^2$, let
\[\tau_{\text{cr}}(\mathbf{O},x) := \mu_{\text{cr}} \bigl( \mathbf{O}  \stackrel{+}{\leftrightarrow} x \bigr). \]

\begin{Lemma}[cf. \cite{K87IMA} Lemma] \label{lem:K87IMA}There exist 
$C_{32}$ and $C_{33}$, depending only on $\varepsilon_0$, such that
\[ C_{32} \pi_{\text{cr}}(|x|_{\infty})^2 \leq \tau_{\text{cr}}(\mathbf{O} ,x) 
 \leq C_{33} \pi_{\text{cr}}(|x|_{\infty})^2 \]
for all $x \in \mathbf{Z}^2$. 
\end{Lemma}

\begin{proof} For the upper bound, note that 
\begin{align*}
\mu_{\text{cr}} \bigl( \mathbf{O}  \stackrel{+}{\leftrightarrow} x \bigr) \leq \mu_{\text{cr}} \bigl( \mathbf{O}  \stackrel{+}{\leftrightarrow} \partial_{in} S(|x|_{\infty}/4),\, x \stackrel{+}{\leftrightarrow} \partial_{in} S(x,|x|_{\infty}/4) \bigr).
\end{align*}
Using the mixing property, the translation-invariance, (\ref{K(5.1)}), and (\ref{K(2.20)}), we have
\begin{align*}
\tau_{\text{cr}} (\mathbf{O} ,x)
\leq \pi_{\text{cr}}(|x|_{\infty}/4)^2 + C \left( \frac{|x|_{\infty}}{4} \right)^2 \cdot \frac{|x|_{\infty}}{2} \cdot e^{-\alpha |x|_{\infty}/2}
\leq C_{33} \pi_{\text{cr}}(|x|_{\infty})^2.
\end{align*}

For the lower bound, note that
\begin{align*}
\mu_{\text{cr}} \bigl( \mathbf{O} \stackrel{+}{\leftrightarrow} x \bigr) 
\geq \mu_{\text{cr}} \left(
\begin{array}{@{\,}c@{\,}} 
\mathbf{O} \stackrel{+}{\leftrightarrow} \partial_{in} S(2|x|_{\infty}),\, x \stackrel{+}{\leftrightarrow} \partial_{in} S(2|x|_{\infty}), \\
\mbox{there exists a $(+)$-circuit in $S(2|x|_{\infty}) \setminus S(|x|_{\infty})$} \\
\end{array} 
\right),
\end{align*}
and 
\[ \mu_{\text{cr}} \bigl( x \stackrel{+}{\leftrightarrow} \partial_{in} S(2|x|_{\infty}) \bigr) \geq  \mu_{\text{cr}} \bigl( x \stackrel{+}{\leftrightarrow} \partial_{in} S(x,3|x|_{\infty}) \bigr) = \pi_{\text{cr}}(3|x|_{\infty}). \]
Using the FKG inequality, (\ref{eq:crossing}) and (\ref{K(2.20)}), we have
\begin{align*}
\tau_{\text{cr}} (\mathbf{O} ,x) 
\geq \pi_{\text{cr}}(2|x|_{\infty}) \cdot \pi_{\text{cr}}(3|x|_{\infty}) 
   \cdot \delta_4^4
\geq C_{32} \pi_{\text{cr}} (|x|_{\infty})^2.
\end{align*}
\end{proof}

\subsection{Radius and volume of a percolation cluster}

We introduce the radius of the $(+)$-cluster ${\mathbf C}_0^+$
containing the origin:
\[ R=R({\bf C}_0^+):= \max \{ |v|_{\infty} : v \in {\bf C}_0^+ \}.  \]
\begin{Lemma}[cf. \cite{K87IMA} (10)]\label{lem:radi-1}
Let $n_5\geq n_2$ be given by 
\begin{equation}\label{eq:bound n5}
2C n^3 e^{-\alpha n/2}< \frac{\delta_8^4}{2} \quad \mbox{ for every }
n \geq n_5.
\end{equation}
Then we have
\begin{equation}\label{Kbg(10)}
\mu_{T,h} \bigl( n \leq R < 2n \bigr) \geq 
 \frac{\delta_8^4}{2} \pi_{h}(n)
\end{equation}
for every $n_5\leq n\leq L(h,\varepsilon_0)$.
\end{Lemma}
Note that $2^{j_1+1}\geq n_5$ by (\ref{eq:cond-for-j1-2}).
\begin{proof} Note that
\begin{align*}
\mu_{T,h} \bigl( n \leq R < 2n \bigr)
\geq \mu_{T,h} \left( 
\begin{array}{@{\,}c@{\,}} 
\mathbf{O}  \stackrel{+}{\leftrightarrow} \partial S(n), \\
\mbox{there exists a $(-*)$-circuit in $S(2n) \setminus S(3n/2)$} \\
\end{array} 
\right).
\end{align*}
By the mixing property, (\ref{K(2.19)}), and (\ref{Kbg(10)}), we have
\begin{align*}
\mu_{T,h} \bigl( n \leq R < 2n \bigr)
\geq \pi_h(n) \cdot \biggl( 
 \delta_8^4 - 2Cn^3 e^{-\alpha n/2}
         \biggr)
\geq \frac{\delta_8^4}{2} \pi_h(n).
\end{align*}
\end{proof}

The next proposition is a variant of Theorem 8 of \cite{K86i}.

\begin{Proposition}[cf. \cite{K87scaling} (3.3)]\label{inc-8}
For every $t\geq 1$, there exist constants $C_{34}(t),C_{35}(t)>0$,
depending only on $t, j_1$ and $ \varepsilon_0$, such that
\begin{align*}
C_{34}(t){\bigl( n^2\pi_h(n) \bigr) }^t
& \leq  E_{T,h}\left[ \left. {[\# ({\bf C}_0^+\cap S(n))] }^t \,\right|\, 
   \mathbf{O}  \stackrel{+}{\leftrightarrow} \partial_{in} S(n) \right] \\
&\leq C_{35}(t) {\bigl( n^2 \pi_h(n) +1 \bigr) }^t ,
\end{align*}
for every $2^{j_1+5}< n \leq L(h,\varepsilon_0)$.
\end{Proposition}
%% The proof of this proposition is similar to the one in \cite{K86i} and
%% its simpler version in \cite{Nguyen88}.
%%%%%%%%%%%%%%%%%%%%%%%%%%%%%%%%%%%%%%%%%%%%%%%%%%%%%%%
%% \begin{quote} 
%% ----------------------------------------------------------------------- 
%% 
%% \noindent 
%% \framebox{\bf INTERMISSION} 
\begin{proof}
The proof of this proposition is similar to the one in \cite{K86i} and
its simpler version in \cite{Nguyen88}.
We first show the lower bound for $t=1$.
Observe that
\begin{align*}
&E_{T,h}\left[
  \# ({\bf C}_0^+\cap S(n))  \,\left|\, 
      \mathbf{O}  \stackrel{+}{\leftrightarrow} \partial_{in} S(n)
      \right. \right] \\
&= \sum_{v\in S(n)}\mu_{T,h}\left(\left.
     v \stackrel{+}{\leftrightarrow} \mathbf{O} \,\right|\, 
    \mathbf{O}  \stackrel{+}{\leftrightarrow} \partial_{in} S(n)
       \right) \\
&\geq \sum_{v \in S(n/2)} 
     \mu_{T,h}\left(\left.
      \{ v \stackrel{+}{\leftrightarrow} \mathbf{O} \}
     \cap E_n \,\right|\, 
   \mathbf{O}  \stackrel{+}{\leftrightarrow} \partial_{in} S(n)
       \right) \\
&=  \sum_{v \in S(n/2)} 
     \mu_{T,h}\left(\left.
      \{ v \stackrel{+}{\leftrightarrow} \partial_{in} S(n) \}
     \cap E_n\,\right|\, 
   \mathbf{O}  \stackrel{+}{\leftrightarrow} \partial_{in} S(n)
       \right),
\end{align*}
where $E_n$ is the event that there is a $(+)$-circuit 
in $S(n)$ surrounding $ S(n/2)$.
By the FKG inequality and Lemma \ref{RSWbound}, we have for $n\geq 4n_2$,
\begin{align*}
& \sum_{v \in S(n/2)} 
     \mu_{T,h}\left(\left.
      \{ v \stackrel{+}{\leftrightarrow} \partial_{in} S(n) \}
     \cap E_n\,\right|\, 
   \mathbf{O}  \stackrel{+}{\leftrightarrow} \partial_{in} S(n)
       \right) \\
&\geq \sum_{v \in S(n/2)}
    \mu_{T,h}\left(
       v \stackrel{+}{\leftrightarrow} \partial_{in} S(n) 
              \right)
   \mu_{T,h}( E_n ) \\
&\geq \delta_4^4 \frac{n^2}{4} \pi_h(2n) \geq \delta_4^8\delta_2 \frac{n^2}{4}\pi_h(n)\\
&=: C_1^\# n^2 \pi_h(n).
\end{align*}
This is true if $ n \leq L(h,\varepsilon_0)$.
The constant $C_1^\# >0 $ depends only on $\varepsilon_0$.
Thus, we obtain a desired lower bound for $t=1$.
The lower bound for general $t>1$ follows from this and Jensen's
inequality.% and (\ref{eq:12-1}).

Next, we show the upper bound.
We show it for an integer $t\geq 1$; 
the bound for general $t\geq 1$ follows by H\"older's inequality.
First we show it for $t=1$ as before.
$$ %\begin{eqnarray*} &&
E_{T,h}\left[
     \# ({\bf C}_0^+\cap S(n))  \,\left|\, 
      \mathbf{O} \stackrel{+}{\leftrightarrow} \partial_{in} S(n) 
       \right. \right] %\\
%&=& 
=    \sum_{v \in S(n)}\mu_{T,h}\left( \left.
      v \stackrel{+}{\leftrightarrow} \mathbf{O}\,\right|\, 
      \mathbf{O} \stackrel{+}{\leftrightarrow} \partial_{in} S(n) 
           \right) .
%\end{eqnarray*}
$$
Take a $v\in S(n)\setminus S(2^{j_1})$, and let $|v|_\infty =k$.
Let $E_k(v)$ be the event given by
$$
E_k(v)%E(v,k)
=\left\{
 \begin{array}{@{\,}c@{\,}}
 \mbox{there is a $(+)$-circuit in the annulus}\\
  \mbox{$S(v,k)\setminus S(v, k/2 )$ surrounding $v$}
 \end{array}
 \right\} .
$$
Then we have by the FKG inequality
\begin{align*}
&\mu_{T,h}\left(
     v \stackrel{+}{\leftrightarrow} \mathbf{O}, \mbox { and }
      \mathbf{O} \stackrel{+}{\leftrightarrow} \partial_{in} S(n) 
           \right) \\
& \leq  \delta_4^{-4} 
  \mu_{T,h}\left(
     \{ v \stackrel{+}{\leftrightarrow} \mathbf{O} \} \cap
     \{  \mathbf{O} \stackrel{+}{\leftrightarrow} \partial_{in} S(n) \}
     \cap
     E_k(v)%E(v,k)
     \right) \\
&\leq   \delta_4^{-4}
     \mu_{T,h}\left(
    \{ v \stackrel{+}{\leftrightarrow} 
              \partial_{in} S(v, k/4 ) \}
   \cap E_k(v)%E(v,k) 
     \cap
    \{  \mathbf{O} \stackrel{+}{\leftrightarrow} \partial_{in} S(n) \}
    \right) .
\end{align*}
Since 
\begin{align*}
E_k(v)%E(v,k) 
 \cap \{
  \mathbf{O} \stackrel{+}{\leftrightarrow} \partial_{in} S(n) \} 
&= E_k(v)%E(v,k) 
  \cap \left\{   
  \mathbf{O} \stackrel{+}{\leftrightarrow} \partial_{in} S(n)
  \mbox{ outside } S(v, k/2 )
   \right\} \\
&\subset  \left\{   
  \mathbf{O} \stackrel{+}{\leftrightarrow} \partial_{in} S(n)
  \mbox{ outside } S(v, k/2 )
   \right\} ,
\end{align*} 
by the mixing property we have
$$
\mu_{T,h}\left(
     v \stackrel{+}{\leftrightarrow} \mathbf{O}, \mbox { and }
      \mathbf{O} \stackrel{+}{\leftrightarrow} \partial_{in} S(n) 
           \right) \\
\leq \delta_4^{-4}\left( \pi_h( k/4 )
                   + C \frac{k^3}{4^2} e^{-\alpha k/4} \right) \pi_h(n).
$$       
By Lemma \ref{RSWbound} and the FKG inequality, we have
$$
\pi_h( k/4 )\leq (\delta_4^4 \delta_2)^{-2}\pi_h(k)  \quad \mbox{for }
         k\geq 16n_2.
$$
This is true when $2^{j_1+5}\leq k$,
since $n_4\geq 4n_3(8,\varepsilon_0)$ by definition, and from 
(\ref{eq:bound for j_1}), $2^{j_1+5}\geq 6n_3(8,\varepsilon_0)$.
Also, $n_3(\lceil 8\eta^{-1}\rceil ,\varepsilon_0)$ is set not less
than $n_2$, we also have $2^{j_1+5}\geq 16n_2$, and
we can use Lemma \ref{inc-7} for $k$'s with $k\geq 2^{j_1+5}$ to
obtain
\begin{align*}
&\sum_{v\in S(n)}\mu_{T,h}\left( \left.
  v \stackrel{+}{\leftrightarrow} \mathbf{O} \, \right| \,
      \mathbf{O} \stackrel{+}{\leftrightarrow} \partial_{in} S(n) 
           \right) \\
&\leq 2^{2(j_1+6)} + \delta_4^{-4}(\delta_4^4\delta_2)^{-2}
\sum_{k=2^{j_1+5}}^n 8k\left( \pi_h(k) +
            C\frac{k^3}{4^2}e^{-\alpha k/4}\right) \\
&\leq C_1^\# (  n^2\pi_h(n) + 1 )   
\end{align*}
for some constant $C_1^\# >0$, which depends on 
$j_1$ and $\varepsilon_0$.

For an integer $t\geq 1$, we prove the bound by induction
following \cite{Nguyen88}.
The argument is similar to the one for $t=1$.
Since
\begin{align*}
&E_{T,h}\left[\left. 
{\bigl[ \# ({\bf C}_0^+\cap S(n))\bigr] }^t \, \right| \,
      \mathbf{O} \stackrel{+}{\leftrightarrow} \partial_{in} S(n) 
        \right] \\
&= \sum_{v_1, \ldots , v_t\in S(n)}
  \mu_{T,h}\left( \left. v_1 , \ldots, v_t 
       \stackrel{+}{\leftrightarrow} \mathbf{O} \, \right| \,
      \mathbf{O} \stackrel{+}{\leftrightarrow} \partial_{in} S(n) 
           \right) ,
\end{align*}
we fix $v_1, \ldots , v_{t-1}\in S(n)$ and estimate the sum
$$
\sum_{v_t\in S(n)}
    \mu_{T,h}\left( \left. v_1, \ldots , v_t 
       \stackrel{+}{\leftrightarrow} \mathbf{O} \, \right| \,
      \mathbf{O} \stackrel{+}{\leftrightarrow} \partial_{in} S(n) 
           \right)
$$
from above.
This time we put 
$k= d_\infty ( v_t, \{ \mathbf{O}, v_1, \ldots , v_{t-1}\} )$.
Then we use $E_k(v_t)$ as before, and  noting that the
$$
\# \{ v \in S(n) :
   d_\infty ( v_t, \{ \mathbf{O}, v_1, \ldots , v_{t-1}\} )=k \}
 \leq 8tk,
$$
by the same argument as before, we obtain
\begin{align*}
&\sum_{v_t\in S(n)}
    \mu_{T,h}\left(\left.  v_1, \ldots , v_t 
       \stackrel{+}{\leftrightarrow} \mathbf{O} \, \right| \,
      \mathbf{O} \stackrel{+}{\leftrightarrow} \partial_{in} S(n) 
           \right) \\
&\leq t C_1^\# ( n^2\pi_h(n) +1)
    \mu_{T,h}\left( \left. v_1, \ldots , v_{t-1}
       \stackrel{+}{\leftrightarrow} \mathbf{O} \, \right| \,
      \mathbf{O} \stackrel{+}{\leftrightarrow} \partial_{in} S(n) 
           \right).
\end{align*}
This gives the desired upper bound.

\end{proof}
%% ----------------------------------------------------------------------- 
%% \end{quote}

Following Kesten \cite{K87IMA}, we can prove the next theorem.

\begin{Theorem}\label{K87IMA-th}
(i) For every $t\geq 1$, 
there exists a constant $C_{36}(t)$, depending only on $t$ and $\varepsilon_0$,
such that for every $\lambda \geq 1$, we have
\begin{equation}\label{eq:K871IMA-1}
\mu_{T,h}\left( \frac{\# {\bf C}_0^+}{R^2\pi_h(R)}\geq 
       \lambda \ \bigg\vert \ 
      n \leq R< 2n \right) \leq C_{36}(t)\lambda^{-t}
   \quad \mbox{for }\  2^{j_1+5}\leq n\leq L(h,\varepsilon_0)/2,
\end{equation}
Here, $R=R({\mathbf C}_0^+)$ is the radius of the cluster ${\mathbf C}_0^+$.

\noindent
(ii) For every $\varepsilon >0$, there exists a $\lambda >0$ and
an integer $N_1 \geq 2^{j_1+5}$ 
such that if $N_1\leq n \leq L(h,\varepsilon_0) $,
then we have
\begin{equation}\label{eq:K87IMA-2}
\mu_{T,h}\left(
\frac{\# {\bf C}_0^+ }{R^2\pi_h(R) } < \frac{1}{\lambda }
            \ \bigg\vert \ 
      n\leq R < 2n \right) \leq \varepsilon.
\end{equation}
\end{Theorem}
\begin{proof}
(i) By Lemma \ref{lem:radi-1}, 
%\begin{equation}\label{eq:K87IMA-3}
%\mu_{T,h}\bigl( n\leq R < 2n\bigr) \geq \frac{\delta_8^4}{2}\pi_h(n)
%\end{equation}
(\ref{Kbg(10)}) holds 
if $n_5\leq n \leq L(h,\varepsilon_0)/2$.
We put $C_1^\# := \delta_8^4/2$ for simplicity. 
Then we have
\begin{align*}
& \mu_{T,h}\left(
  \frac{\# {\bf C}_0^+}{R^2\pi_h(R)} \geq \lambda \  \bigg\vert\ 
                n\leq R < 2n 
          \right) \\
&\leq \mu_{T,h}\left(
    \frac{\# ({\bf C}_0^+\cap S(2n) )}{n^2\pi_h(2n)} \geq \lambda 
      \ \bigg\vert \ n \leq R < 2n 
            \right) \\
&\leq \frac{1}{C_1^\# \pi_h(n)} \mu_{T,h}\left(
    \frac{\# ({\bf C}_0^+ \cap S(2n) )}{n^2\pi_h(2n) }\geq \lambda,\,
       n\leq R 
         \right).
\end{align*}
By a similar argument as the proof of (\ref{K(2.20)}), we can see that
there exists a constant $C_2^\# >0$ 
which depends only on $\varepsilon_0$, such that
\begin{align}\label{eq:K87IMA-4}
&\frac{1}{C_1^\#\pi_h(n)}\mu_{T,h}\left( 
\frac{ \# ( {\bf C}_0^+\cap S(2n) )}{n^2\pi_h(2n)} \geq \lambda
    ,\,n \leq R 
       \right) \\
&\leq   \frac{C_2^\# }{\pi_h(2n)} \mu_{T,h}\left(
 \frac{ \# ({\bf C}_0^+ \cap S(2n) )}{n^2\pi_h(2n)} \geq \lambda 
         ,\, 2n\leq R 
         \right) . \nonumber
\end{align}
By Proposition \ref{inc-8} and Chebyshev's inequality we have
\begin{align*}
&\mu_{T,h}\left(
 \frac{ \# ({\bf C}_0^+ \cap S(2n) )}{n^2\pi_h(2n)} \geq \lambda 
         \,\bigg\vert\, 2n\leq R 
         \right) \\
&\leq \frac{E_{T,h}\left[ 
           [\# ( {\bf C}_0^+\cap S(2n) ) ]^t 
           \,  \big\vert\, 2n\leq R 
             \right]}{ \lambda^t {\bigl( n^2 \pi_h(2n) \bigr)}^t } \\
&\leq  C_{35}(t) \frac{{\bigl( 4n^2\pi_h(2n) + 1\bigr) }^t}{ \lambda^t 
               { \bigl(n^2 \pi_h(2n) \bigr) }^t},
\end{align*}
if $2^{j_1+5}\leq n \leq L(h,\varepsilon_0)/2$.
By Lemma \ref{inc-7} and the comment after it,
$$
\pi_h(2n) \geq C_{31}'n^{-1},
$$
which implies that
$$
\frac{(4n^2\pi_h(2n)+1)^t}{(n^2\pi_h(2n))^t}
$$
is bounded in $n$.
Thus, we have 
$$
\mu_{T,h}\left( \left. \frac{\# ( {\mathbf C}^+_0\cap S(2n))}{n^2\pi_h(2n)}
   \geq \lambda \,\right|\, 2n \leq R \right)
\leq C_{36}(t)\lambda^{-t}
$$
for $2^{j_1+5}\leq n \leq L(h,\varepsilon_0)$ and some constant
$C_{36}(t)$, which depends on $t $ and $\varepsilon_0$.

\noindent (ii) We take $5m<n$, and define the annuli $A_i(m), i=1,2,3$ by
$$
A_i(m)= S((i+1)m)\setminus S(im).
$$
Also, we define a random variable $Y(m)$ by
$$
Y(m) = \# \left\{  v \in A_2(m) :
         v  \stackrel{+}{\leftrightarrow} \partial_{in} S(4m) \mbox{ or }
         v  \stackrel{+}{\leftrightarrow} \partial S(m)
         \right\} .
$$
Let $F(m)$ be the event given by
$$
F(m) := \left\{
    \begin{array}{@{\,}c@{\,}}
     \mbox{for $i=1,3$, there exists a $(+)$-circuit ${\mathcal C}_i$}\\
     \mbox{in the annulus $A_i(m)$, surrounding the origin}
    \end{array}
    \right\} .
$$
Then on the event
$F(m) \cap \{ {\mathbf O}\stackrel{+}{\leftrightarrow}\partial_{in} S(n) \} $,
every point $v \in A_2(m)$ contributing to the sum $Y(m)$ belongs to 
${\bf C}_0^+$.
By the definition of $Y(m)$ and (\ref{K(2.20)}), we have
for $m\geq 4n_2$,
\begin{align}\label{eq:K87IMA-5}
E_{T,h}\bigl[ Y(m)\bigr] &= \sum_{v\in A_2(m)}
           \mu_{T,h}\left( v  \stackrel{+}{\leftrightarrow} \partial_{in} S(4m),
            \mbox{ or } v \stackrel{+}{\leftrightarrow}\partial S(m)
                   \right) \\
&\geq 20m^2\pi_h(7m) \geq C_3^\# m^2\pi_h(m), \nonumber 
\end{align}
where $C_3^\# >0$ is a constant depending only on $T$ and $\varepsilon_0$.

Let $k$ be the largest integer such that $5^k\leq n$, and
for a given $\varepsilon_1>0$, let $j$ satisfy 
$\varepsilon_1 \leq \frac{1}{2} 5^{-2j-2}$.
Then for any $m$ with $5^{k-j}\leq m < 5^{k-1}$, by (\ref{eq:K87IMA-5}) we have
\begin{equation}\label{eq:K87IMA-6}
\frac{1}{2}E_{T,h}\bigl[ Y(m) \bigr]
\geq  \frac{1}{2} C_3^\# m^2\pi_h(m)
\geq  \frac{1}{2}C_3^\# 5^{2k-2j}\pi_h(n)
\geq  \varepsilon_1 C_3^\# n^2\pi_h(n). 
\end{equation}
By monotonicity we have
\begin{align*}
&\mu_{T,h}\left(
\frac{\#{\bf C}_0^+ }{R^2\pi_h(R)} \geq \frac{1}{\lambda }
 \ \bigg\vert \  n\leq R < 2n
        \right) \\
&\geq  \mu_{T,h}\left(
     \frac{\# ( {\bf C}_0^+ \cap S(n) )}{4n^2\pi_h(n)} \geq
         \frac{1}{\lambda } \ \bigg\vert \ n\leq R < 2n
                                  \right) . 
\end{align*}
For $1\leq \ell \leq j$, let
$$
G_\ell=G_\ell^{(k)}:= F(5^{k^\ell }) \cap \left\{
        Y(5^{k-\ell }) \geq \frac{1}{2}E_{T,h}\bigl[ Y(5^{k-\ell })\bigr]
       \right\} .
$$
Choose $\lambda >0$ sufficiently large so that 
$4\lambda^{-1}< \varepsilon_1C_3^\# $, and then from (\ref{eq:K87IMA-5}), 
we have 
\begin{align}\label{eq:K87IMA-7}
&\mu_{T,h}\left(
\frac{\# {\bf C}_0^+}{R^2\pi_h(R)} \geq \frac{1}{\lambda }
         \ \bigg\vert \ 
        n\leq R <2n 
        \right) \\
&\geq \mu_{T,h}\left( \left.
       \bigcup_{1\leq \ell \leq j}
                   G_\ell \, \right| \, n \leq R <2n 
         \right) . \nonumber
\end{align}
By the FKG inequality we have 
\begin{align*}
&\mu_{T,h}\left( \bigcap_{1\leq \ell \leq j}G_\ell^c \cap \{ n\leq R<2n\}
\right) \\
&\leq \mu_{T,h}\left( \bigcap_{1\leq \ell \leq j}G_\ell^c \cap \{ n\leq R \}
\right) \\
&\leq \mu_{T,h}\left( \bigcap_{1\leq \ell \leq j}G_\ell^c \right)\pi_h(n).
\end{align*}
%since $G_\ell $ is a positive event.
%On the other hand, 
%$$
%\mu_{T,h}( n\leq R < 2n) \geq \mu_{T,h}(\{ n\leq R \} \cap {\tilde E}_n^{-*}),
%$$
%where ${\tilde E}_n^{-*}$ is teh event that there exists a $(-*)$-circuit
%in $S(2n)$ surrounding $S(3n/2)$.
%Then by the mixing property
%$$
%\mu_{T,h}(\{ n\leq R \} \cap {\tilde E}_n^{-*})\geq \pi_h(n) (
%  \delta_8^4-2Cn^3e^{-\alpha n/2}).
%$$
%Assuming that 
%\begin{equation}\label{eq:bound for N1}
%\frac{\delta_8^4}{2}> 2Cn^3e^{-\alpha n/2},
%\end{equation}
%we have
%$$ 
%\mu_{T,h}( n\leq R<2n) \geq \pi_h(n)\frac{\delta_8^4}{2}.
%$$
Combining this with (\ref{Kbg(10)}), we have 
\begin{equation}\label{eq:K87IMA-9c}
\mu_{T,h}\left( \left. \bigcap_{1\leq \ell \leq j}G_\ell^c \, \right| \,
 n\leq R <2n \right) \leq 
    2\delta_8^{-4}\mu_{T,h}\left( \bigcap_{1\leq \ell \leq j}G_\ell^c \right) ,
\end{equation}
provided that (\ref{eq:bound n5}) is satisfied.
Again, for each $1\leq \ell \leq j$, by the mixing property
we have
\begin{align}\label{eq:K87IMA-9a}
\mu_{T,h}\biggl(
  G_\ell \,\bigg\vert\, {\mathcal F}_{S(4\cdot 5^{k-\ell -1})}
       \biggr) 
\geq \mu_{T,h}(G_\ell ) - C5^{3(k-\ell )}e^{-\alpha 5^{k-\ell -1}}.
%         \nonumber 
\end{align}
If $5^{k-\ell }\geq n_2$, then by Lemma \ref{RSWbound} and the FKG inequality,
\begin{align}\label{eq:K87IMA-9d}
\mu_{t,h}\left( G_\ell \right) &=\mu_{T,h}\left(
        F(5^{k-\ell })\cap \left\{ Y(5^{k-\ell })\geq 
             \frac{1}{2}E_{T,h}\bigl[ Y(5^{k-\ell })\bigr] \right\}
            \right) \\
&\geq  (\delta_8\delta_4)^4 \mu_{T,h}\biggl(
       Y(5^{k-\ell })\geq 
             \frac{1}{2}E_{T,h}\bigl[ Y(5^{k-\ell })\bigr]
            \biggr) \nonumber
\end{align}
By the one-sided analogue of Chebyshev's inequality, 
\begin{align}\label{eq:K87IMA-9b}
&\mu_{T,h}\biggl(
       Y(5^{k-\ell })\geq 
             \frac{1}{2}E_{T,h}\bigl[ Y(5^{k-\ell })\bigr]
            \biggr) \\ 
&\geq   \frac{
      \frac{1}{4}{\left( E_{T,h}[ Y(5^{k-\ell })]\right) }^2
      }{
       \frac{1}{4}{\left( E_{T,h}[ Y(5^{k-\ell })]\right) }^2
      + \text{Var}_{T,h}[Y(5^{k-\ell })]
         }. \nonumber
\end{align}
We shall prove that the right hand side is bounded by a positive constant.
By the same reason as in the proof of Proposition \ref{inc-8}, 
we can see that 
\begin{equation}\label{eq:K87IMA-10}
E_{T,h}\left[ Y(m)^2\right] \leq C_4^\# (m^2\pi_h(m) +1 )^2
\end{equation}
for some constant $C_4^\# >0$ depending on $\varepsilon_0$, if $m\geq 2^{j_1+5}$.
Inserting (\ref{eq:K87IMA-10}) and (\ref{eq:K87IMA-5}) in (\ref{eq:K87IMA-9b}),
we obtain: 
\begin{align}
&\mu_{T,h}\biggl( Y(5^{k-\ell }) \geq 
    \frac{1}{2}E_{T,h}\left[ Y(5^{k-\ell })\right]
\biggr) \label{eq:K87IMA-9e}\\
&\geq  \frac{ (C_3^\# m^2 \pi_h(m) )^2}
{(C_3^\# m^2 \pi_h(m) )^2+ C_4^\# ( m^2 \pi_h(m)+1)^2} \nonumber \\
&= \frac{( C_3^\# )^2}{( C_3^\# )^2+ C_4^\# ( 1 +( m^2 \pi_h(m) )^{-1} )^2}.
\nonumber
\end{align}
Since by Lemma \ref{inc-7}, this is bounded from below by a positive
constant $C_5^\# $ depending only on $L_0$ and $\varepsilon_0$.
(\ref{eq:K87IMA-9a}) and (\ref{eq:K87IMA-9d}) implies that
\begin{align}\label{eq:K87IMA-9f}
&\mu_{T,h}\biggl(
  G_\ell \, \bigg\vert \,  {\mathcal F}_{S(4\cdot 5^{k-\ell -1})}
       \biggr) \\
&\geq (\delta_8\delta_4)^4\cdot C_5^\# - C 5^{3(k-\ell )}e^{-\alpha 5^{k-\ell -1}}.        \nonumber 
\end{align}
If we put $n^*$ as the smallest integer such that 
$$
Cn^3 e^{-n/5}< \frac{1}{2}(\delta_8\delta_4)^4 C_5^\#
$$
for every $n\geq n^*$, then
$$
\mu_{T,h}\biggl(
  G_\ell \, \bigg\vert\,  {\mathcal F}_{S(4\cdot 5^{k-\ell -1})}
       \biggr)
\geq \frac{1}{2}(\delta_8\delta_4)^4 C_5^\#
$$
as long as $5^{k-\ell } > \max\{ 4n_2, n^*\} $.
This implies that
$$
\mu_{T,h}\biggl( \bigcap_{1\leq \ell \leq j} G_\ell^c
\, \bigg\vert \,  n\leq R < 2n \biggr) \leq 2\delta_8^{-4}
{ \left( 1 - \frac{1}{2}(\delta_8\delta_4)^4C_5^\# \right) }^j
$$
if $5^{k-j}\geq \max\{ 4n_2, n^*\} $.
Therefore for every $\varepsilon >0$, we choose $j\geq 1$ to satisfy
$$
2\delta_8^{-4}
{ \left( 1 - \frac{1}{2}(\delta_8\delta_4)^4C_5^\# \right) }^j
<\varepsilon ,
$$
and then we choose $\varepsilon_1 >0$ to satisfy
$0< \varepsilon_1 < 5^{-2j-2}$.
Then we take $n$ sufficiently large so that
$$
n\geq N_1:= 5^{j+1}\max\{ 4n_2, n^*\} .
$$
If $N_1\leq n\leq L(h,\varepsilon_0)$ and $4\lambda^{-1}<\varepsilon_1C_3^\#$,
then the inequality (\ref{eq:K87IMA-2}) holds.
\end{proof}

\begin{Theorem}[cf. \cite{K87scaling} (1.26) in Theorem 3] \label{KTheorem3a} For $t \geq 0$, we have as $h\rightarrow h_c(T)$,
\begin{equation} \label{K(1.26)}
\sum_{v \in \mathbf{Z}^2} |v|^t 
  \mu_{T,h}\{ \mathbf{O} \stackrel{+}{\leftrightarrow} v,\, 
   \# {\bf C}_0^+ < \infty \} 
 \asymp L(h,\varepsilon_0)^{t+2}\pi_{\text{cr}} (L(h,\varepsilon_0))^{2}.
\end{equation}
\end{Theorem}

\begin{proof} 
First, we prove the lower bound.
Let $L=L(h,\varepsilon_0)$.% and $j_1 \geq \log_2 n_2$, where
%$n_2$ is given in Lemma \ref{lem:radi-1}.
We start with the following inequality.
\begin{align*}
&\sum_{y\in{\mathbf Z}^2} |y|^t \mu_{T,h}( 
   \mathbf{O} \stackrel{+}{\leftrightarrow } y ,\, 
        \#{\bf C}_0^+ <\infty ) \\
&\geq  \sum_{k=j_1+5}^{\log_2L}
   \sum_{|y|\leq 2^{k-1}}|y|^t
    \mu_{T,h}( 
     \mathbf{O} \stackrel{+}{\leftrightarrow } y ,\,
    2^k\leq R <2^{k+1} ).
\end{align*}
Let
$$
A_k :=  \left\{  \begin{array}{@{\,}c@{\,}} 
\mbox{ there exists a $(+)$-circuit} \\
\mbox{in $S(2^k)\setminus S(2^{k-1})$ surrounding the origin} \\
\end{array} \right\} .
$$
Then for $|y|\leq 2^{k-1}$, we have
\begin{align*}
& \mu_{T,h}( 
     \mathbf{O} \stackrel{+}{\leftrightarrow } y ,\,
    2^k\leq R <2^{k+1} ) \\
&\geq  \mu_{T,h} \bigl( 
   \{  \mathbf{O} \stackrel{+}{\leftrightarrow } y \} \cap
   A_k \cap \{ 2^k\leq R <2^{k+1} \} \bigr) \\
&= \mu_{T,h}\bigl( 
  \{  y \stackrel{+}{\leftrightarrow } \partial S(2^k) \}
  \cap A_k \cap \{ 2^k \leq R < 2^{k+1} \} \bigr).
\end{align*}
Let
$$
B_{k+1}:= \left\{
   \begin{array}{@{\,}c@{\,}} 
     \mbox{there exists a $(-*)$-circuit}  \\
    \mbox{in $S(2^{k+1})\setminus S(2^k+2^{k-1})$ surrounding the origin}
   \end{array}
     \right\} .
$$
Then since
$$
\{ 2^k\leq R<2^{k+1}\} \supset \{ R\geq 2^k \} \cap B_{k+1}
\ \mbox{ and } \ \mu_{T,h}(B_{k+1})\geq \delta_8^4,
$$
by the mixing property we have
\begin{align*}
& \mu_{T,h}\bigl(  
  \{ y \stackrel{+}{\leftrightarrow } \partial S(2^k) \}
  \cap A_k \cap \{ 2^k \leq R < 2^{k+1} \} \bigr) \\
&\geq \mu_{T,h}\bigl( 
  \{ y \stackrel{+}{\leftrightarrow } \partial S(2^k) \}
  \cap A_k \cap \{ R \geq 2^k \} \bigr) 
( \delta_8^4 - C2^{3k+1}e^{-\alpha 2^{k-1}}\bigr).
\end{align*}
By Lemma \ref{lem:radi-1}, the last term in the right hand side of the
above inequality is not less than
$\frac{1}{2}\delta_8^4$, since $2^k\geq 2^{j_1+5}$.
Thus, we have
$$
\mu_{T,h}( 
     \mathbf{O} \stackrel{+}{\leftrightarrow } y ,\,
    2^k\leq R <2^{k+1} )
  \geq \frac{1}{2}\delta_8^4\delta_4^4 \pi_h(2^{k+1})\pi_h(2^k).
$$
By Lemma \ref{RSWbound}, there is a constant $C_1^\#$ depending only on 
$\varepsilon_0$, such that
$$
\mu_{T,h}( 
     \mathbf{O} \stackrel{+}{\leftrightarrow } y ,\,
           2^k\leq R <2^{k+1} ) \geq C_1^\# \pi_h(2^k)^2.
$$
Therefore we have 
\begin{align}\label{eq:12-2}
&\sum_{j_1+5\leq k\leq \log_2 L} \sum_{|y|\leq 2^{k-1}}
    |y|^t \mu_{T,h}(
   \mathbf{O} \stackrel{+}{\leftrightarrow } y ,\,
    2^k\leq R <2^{k+1} )\\
&\geq  C_2^\# (t) \sum_{j_1+5\leq k\leq \log_2 L}2^{k(t+2)}\pi_h(2^k)^2 \nonumber\\
&\geq  C_2^\# (t)L^{t+2}\pi_h(L)^2, \nonumber
\end{align}
where $C_2^\# (t)$ depends only on $ t,  \varepsilon_0$.
%This estimate is valid if $L=L(h,\varepsilon_0)\geq 2^{j_1+5}$.
Finally, by Theorem \ref{KTheorem1}, we have
$$
\pi_h(L) \geq C_{25} \pi_{\text{cr}}(L).
$$

Next, we prove the upper bound.
We break up the summation into two parts; 
\[ {\sum}^1=\sum_{|y|_\infty \leq L} \quad \mbox{and} \quad
{\sum}^2=\sum_{|y|_{\infty} >L}. \] 
By the mixing property, we have
$$
\mu_{T,h}\bigl( {\mathbf O} \stackrel{+}{\leftrightarrow }y \bigr)
\leq \pi_h( |y|_\infty/3)\left( 
   \pi_h(|y|_\infty /3) 
      + 4C{\left( \frac{|y|_\infty}{3}\right) }^3e^{-\alpha |y|_\infty /3}  
        \right) .
$$
First we estimate $\sum^1$.
Let $L_0$ be the same number as given in Lemma \ref{inc-7}.
Then for $|y|_\infty >2L_0$, we can replace $\pi_h(|y|_\infty /3)$ with 
constant multiple of
$\pi_h(|y|_\infty )$ by Lemma \ref{RSWbound}. For $ |y|_\infty \leq 2L_0$, we simply replace
$
\mu_{T,h}\bigl( {\mathbf O} \stackrel{+}{\leftrightarrow }y \bigr)
$
with $1$.
As a result, we obtain
\begin{align*}
{\sum }^1 &\leq  \sum_{|y|_\infty \leq L} |y|^t \mu_{T,h} \bigl(
          {\mathbf O}\stackrel{+}{\leftrightarrow } y \bigr) \\
&\leq  8\sum_{k=1}^{2L_0} k^{t+1}
 + C_4^\# \sum_{2L_0< |y|_\infty \leq  L}|y|^t\pi_h(|y|_\infty )
      \left( C_4^\# \pi_h(|y|_\infty )
               + 4C\frac{|y|_\infty^3}{3^3}e^{-\alpha |y|_\infty /3}
      \right) \\
&\leq  C_3^\# (t)(L_0)^{t+2} +C_4^\#\sum_{k=2L_0+1}^L 8k^{1+t} \pi_h(k)
      \left( 
   C_4^\# \pi_h(k) + 4C\frac{k^3}{3^3}e^{-\alpha k/3}
      \right) .
\end{align*}
The constant $C_4^\# $ depends only on $\varepsilon_0$, and
$C_3^\# (t)$ depends on $t$ and $\varepsilon_0$.
Then by Lemma \ref{inc-7}, we can see that
$$
k\pi_h(k) \leq (2k-2L_0)\pi_h(k)\leq 2C_{30}(2L)\pi_h(L)
$$
holds for $k\geq 2L_0$.
Thus, we can find a constant $C_5^\# $ depending only on $\varepsilon_0$,
such that
$$
k^{t+1}\pi_h(k)\leq C_5^\# L^{t+1}\pi_h(L).
$$
Therefore we have
$$
{\sum }^1 \leq C_2(t)(L_0)^{t+2} + C_4^\# C_5^\# L^{t+1}\pi_h(L)
    \sum_{k=2L_0+1}^L\left( \pi_h(k) + 
             4C\frac{k^3}{3^3}e^{-\alpha k/3} \right) .
$$
Again by Lemma \ref{inc-7}, 
$$
\sum_{k=2L_0+1}^L \pi_h(k) \leq 2C_{30}L\pi_h(L),
$$ 
and we finally obtain the following inequality.
\begin{equation}\label{eq:12-3}
{\sum }^1 \leq C_6^\# (t) \{ L^{t+2}\pi_h(L)^2 +1\} ,
\end{equation}
where $C_6^\# (t)$ depends on $L_0, t$ and $ \varepsilon_0$.

To estimate $\sum^2$, let $\sum^{2a}$ and $\sum^{2b}$ denote summations 
over $y$'s with $L<|y|_\infty \leq 6L$ and $|y|_\infty >6L$, respectively.

For $\sum^{2a}$, we use the same argument:
$$
\mu_{T,h}( {\mathbf O}\stackrel{+}{\leftrightarrow }y )
\leq \left( \pi_h( |y|_\infty /3) +
   4C  \frac{|y|_\infty^3 }{3^3} e^{-\alpha |y|_\infty /3}
  \right) \pi_h( |y|_\infty /3) .
$$
Let $n_6\geq 2^{j_1}$ be the smallest integer such that
\begin{equation}\label{eq:bound for n5}
4Cn^3e^{-\alpha n}< (\delta_4^4\delta_2)^{\lceil \log_2n\rceil -j_1}
    \pi_h(2^{j_1})
\end{equation}
for every $n\geq n_6$.
If $L/3\geq n_6$, then we have
$$
4C\frac{|y|_\infty^3}{3^3}e^{-\alpha |y|_\infty /3}<
 (\delta_4^4\delta_2)^{\lceil \log_2 (|y|_\infty /3) \rceil -j_1}
 \pi_h( 2^{j_1})
$$
for every $y$ with $|y|_\infty >L$.
This is possible since we are considering the case where $h\rightarrow h_c(T)$,
and hence $L(h,\varepsilon_0)\rightarrow \infty $.
The right hand side of this inequality is not larger than
$\pi_h (|y|_\infty /3)$ as long as $|y|_\infty \leq 3L$ 
by Lemma \ref{lem:circuits}.
Therefore we have
$$
\mu_{T,h}( {\mathbf O}\stackrel{+}{\leftrightarrow }y )\leq
 2\pi_h(|y|_\infty /3)^2
$$
for every $y$ with $L\leq |y|_\infty \leq 3L$.
Further, by Lemma \ref{lem:circuits}, we have
$$
\mu_{T,h}(|y|_\infty /3 )\leq \delta_4^{-4}\delta_3^{-1} \mu_{T,h}(|y|_\infty )
     \leq \delta_4^{-4}\delta_3^{-1}\pi_h(L).
$$
On the other hand, if $|y|_\infty >3L$, then we still have
$$
4C \frac{|y|_\infty^3}{3^3}e^{-\alpha |y|_\infty /3}
\leq 4CL^3e^{-\alpha L}\leq \pi_h(L),
$$
and we have 
$$
\pi_h(|y|_\infty /3)\leq \pi_h(L).
$$
Since $\delta_4^{-4}\delta_3^{-1} >1$, we have in any case 
$$
\mu_{T,h}({\mathbf O}\stackrel{+}{\leftrightarrow }y ) \leq 
2 (\delta_4^{-4}\delta_3^{-1} \pi_h(L) )^2.
$$
Therefore 
\begin{equation}\label{eq:12-4a}
{\sum}^{2a} |y|^t \mu_{T,h}( {\mathbf O}\stackrel{+}\leftrightarrow y )
\leq C_7^\# (t) L^{t+2}\pi_h(L)^2,
\end{equation}
where $C_7^\# (t) >0$ is a constant depending only on $t$ and $\varepsilon_0 $.

As for $\sum^{2b}$, observe that
\begin{align}
&\mu_{T,h}( {\mathbf O}\stackrel{+}{\leftrightarrow }y )\label{eq:12-5a}\\
&\leq  {\bigl( \pi_h(L)+4CL^3e^{-\alpha L} \bigr) }^2
\mu_{T,h}\bigl( \partial S(2L) \stackrel{+}{\leftrightarrow }
\partial_{in}S((k-1)L) \bigr) \nonumber
\end{align}
for every $y$ with $kL \leq |y|_\infty < (k+1)L$.
By the same reason as before, we have
$$
4CL^3 e^{-\alpha L} < \pi_h(L).
$$
Therefore we have
\begin{align*}
&{\sum}^{2b}|y|_\infty^t 
\mu_{T,h}( {\mathbf O}\stackrel{+}{\leftrightarrow } y )\\
&\leq \sum_{k=5}^\infty 4^2((k+1)L)^{t+2}\pi_h(L)^2
\mu_{T,h}\bigl( \partial S(2L)\stackrel{+}{\leftrightarrow }
    \partial_{in}S((k-1)L) \bigr) .
\end{align*}
By the FKG inequality, it is easy to see that for $h<h_c(T)$,
\begin{align}
&\mu_{T,h}\bigl( \partial S(2L)\stackrel{+}{\leftrightarrow }
    \partial_{in}S((k-1)L) \bigr) \label{eq:12-6} \\
&\leq (\delta_4^4\delta_2)^{-1}\mu_{T,h}
 \bigl( \bigl( \partial S(L)\stackrel{+}{\leftrightarrow }
    \partial_{in}S((k-1)L) \bigr) \nonumber \\
& \leq (\delta_4^4\delta_2)^{-1}K_6\exp\{ - K_7(k-1) \} . \nonumber
\end{align}
The last inequality is by Lemma \ref{Kp121}.
Thus, combining (\ref{eq:12-3}) and (\ref{eq:12-4a})--(\ref{eq:12-6}),
we obtain 
\begin{align*}
\sum_{y\in {\mathbf Z}^2} |y|_\infty^t 
\mu_{T,h}\bigl( {\mathbf O}\stackrel{+}{\leftrightarrow } y, \# 
 {\mathbf C}_0^+ <\infty ) \leq C_8^\# (t) \{ L^{t+2}\pi_h(L)^2+1\} ,
\end{align*}
Where the constant $C_8^\# (t)$ depends only on $t, L_0$ and $\varepsilon_0$. 
When $L\rightarrow \infty $, this is bounded from above by
$ 2C_8^\# (t) L^{t+2}\pi_h(L)^2$. 

If $h>h_c(T)$, only the estimate of $\sum^{2b}$ changes.
In this case, the dual connectivity function decays exponentially.
By the FKG inequality we have
\begin{align*}
&\mu_{T,h}\bigl( {\mathbf O} \stackrel{+}{\leftrightarrow } y, 
\# {\mathbf C}_0<\infty )\\
&\leq \mu_{T,h}( {\mathbf O}\stackrel{+}{\leftrightarrow } y )
\mu_{T,h}\left( 
\begin{array}{@{\,}c@{\,}}
\hbox{there exists a $(-*)$-circuit which}\\
\hbox{surrounds both ${\mathbf O}$ and $y$}
\end{array}
\right)
\end{align*}
By the mixing property the first term is bounded from above by 
$2\pi_h(L)^2 $ as before, and
$$
{\textstyle \sum^{2b}}\, \,   |y|_\infty^t \mu_{T,h}\left(
 \begin{array}{@{\,}c@{\,}}
  \hbox{there exists a $(-*)$-circuit which}\\
\hbox{surrounds both ${\mathbf O}$ and $y$}
\end{array}
\right) \leq C_9^\# L^{t+2}
$$
for some constant $C_9^\# >0$ depending only on $t$ and $\varepsilon_0$.
\end{proof}

\subsection{Proof of Kesten-IMA-Corollary and Kesten-Corollary 1}

Let us begin with the following lemma.

\begin{Lemma}[cf. \cite{K87scaling} (3.6)] \label{lem:K(3.6)}
Let $L_0$ be the same as Lemma \ref{inc-7}. Then there exists an constant 
$C_{37}>0$ depending only on $L_0$ and $\varepsilon_0$, such that
for any $k_0>\log_2 L_0+1$,
\begin{equation} \label{K(3.6)}
\frac{\pi_{\text{cr}}(2^k)}{\pi_{\text{cr}}(2^{k_0})} \leq C_{37} 2^{k_0-k}
\end{equation}
for $2L_0\leq 2^k \leq 2^{k_0}$.
\end{Lemma}

\begin{proof}
The following argument has already appeared in the proof of Theorem \ref{KTheorem3a}. By Lemma \ref{inc-7}, we have
\begin{align*}
2^k\pi_{\text{cr}}(2^k)&\leq 2(2^k-L_0)\pi_{\text{cr}}(2^k)\\
  &\leq 2 \sum_{\ell = L_0+1}^{2^k}\pi_{\text{cr}}(\ell) \\
  &\leq 2\sum_{\ell = L_0+1}^{2^{k_0}}\pi_{\text{cr}}(\ell) \\
  &\leq 4C_{30}2^{k_0}\pi_{\text{cr}}(2^{k_0}).
\end{align*}
\end{proof}

In the following we mean by $f(n) \approx n^{\zeta}$ that
\[\lim_{n \to \infty} \frac{\log f(n)}{\log n} = \zeta. \]

\begin{Corollary}[cf. \cite{K87IMA} Corollary,  \cite{K87scaling} Corollary 1] \label{K87IMACor} If one of
\begin{align}
\pi_{\text{cr}}(n)  \approx n^{-1/\delta_{\text r}} \label{K(1.18)} %=\pi_{h_c(T)} (n)
\intertext{or}
\tau_{\text{cr}}(\mathbf{O} ,(n,0)) \approx n^{-\eta} \label{K(1.7)b}
\end{align}
holds, then both statements as well as
\begin{equation}
\mu_{\text{cr}} \{ \# {\bf C}_0^+ \geq n \} \approx n^{-1/\delta} \label{K(1.6)b}
\end{equation}
hold, and
\begin{align*}
\theta (T,h) &\asymp L(h,\varepsilon_0)^{-1/\delta_{\text r}} = L(h,\varepsilon_0)^{-2/(\delta+1)}, \\
\eta &= \frac{2}{\delta_{\text r}}, \\
\delta &= 2\delta_{\text r} -  1 = \frac{4}{\eta} -1.
\end{align*}
If in addition for some $\nu >0$
\begin{align} \label{K(1.1)} 
\xi(T,h) &:=\left[ \dfrac{1}{\chi(T,h)} \sum_{v \in \mathbf{Z}^2} |v|^2 \mu_{T,h} \bigl( \mbox{$\mathbf{O} \stackrel{+}{\leftrightarrow} v$, $\# {\bf C}_0^+ <\infty$} \bigr) \right]^{1/2} \\
&\approx  \left| h-h_c(T) \right|^{-\nu} \notag
\end{align}
holds, then
\[ \beta = \frac{2\nu}{\delta+1}. \]
\end{Corollary}

\begin{proof}
By Lemma \ref{lem:K87IMA}, the existence of $\eta $ and $\delta_{\text r}$ is
equivalent, and we have $\eta = 2/\delta_{\text r}$.
Further, by Theorem \ref{KTheorem3a}, 
$$
\xi (T,h)\asymp L(h,\varepsilon_0)
$$
as $h\rightarrow h_c(T)$.
By Theorem \ref{KTheorem2}, we have
$$
(C_1C_2)^{-1}\pi_{\text{cr}}(L(h,\varepsilon_0)) \leq
\theta (h) \leq \pi_h(L(h,\varepsilon_0)).
$$ 
This, together with (\ref{K(1.18)}), (\ref{K(1.1)}) implies that %{K(1.6)b}
$\beta = \nu /\delta_{\text r}$. %assuming the existence of $\nu $.
So, the remaining thing is to show the equality $\delta = 2\delta_{\text r}-1$.
The argument is quite parallel to \cite{K87IMA}.
By Theorem \ref{K87IMA-th}  we can choose
$\lambda (\frac{1}{2})$  such that
$$
\mu_{\text{cr}}\biggl( 
 \frac{\#{\bf C}_0^+}{R^2\pi_{\text{cr}}(R)}
   \leq \frac{1}{\lambda(\frac{1}{2})}
   \ \bigg\vert \  n \leq R<2n \biggr) \leq \frac{1}{2}.
$$
For an arbitrarily small $\varepsilon $, 
we put $m=n^{\delta_{\text r}/((2-\varepsilon )\delta_{\text r}-1)}$.
By Lemma \ref{inc-7}, we know that $\pi_{\text{cr}}(n)\geq C_{31}'n^{-1}$.
This means that $\delta_{\text r}\geq 1$, and $(2-\varepsilon )\delta_{\text r}-1>0$
if $\varepsilon $ is small.
Since
$$
n= m^{2-\varepsilon -1/\delta_{\text r}},
$$
we can assume that $n\leq (\lambda(1/2) )^{-1}m^2\pi_{\text{cr}}(2m)$.
Therefore we have
\begin{align*}
\mu_{\text{cr}}( \#{\bf C}_0^+\geq n ) &\geq  
 \mu_{\text{cr}}\left( \# {\bf C}_0^+ \geq 
         \frac{1}{\lambda (\frac{1}{2} )}m^2\pi_{\text{cr}}(2m) \right) \\
&\geq \mu_{\text{cr}}\left(
   \frac{\# {\bf C}_0^+}{R^2\pi_{\text{cr}}(R)}
         \geq \frac{1}{\lambda ( \frac{1}{2} )}
        , \ m \leq R < 2m \right) \\
&\geq \frac{1}{2} C_1^\# \cdot \pi_{\text{cr}}(m).
\end{align*}
for some constant $C_1^\# >0$ which depends only on $\varepsilon_0$.
Dividing both sides of the above inequality by $\log n$, and 
letting $n\rightarrow \infty $, we obtain
$$
\frac{\mu_{\text{cr}}( \# {\bf C}_0^+ \geq n )}{\log n}
\geq \frac{\log \frac{C_1^\# }{2} + \log \pi_{\text{cr}}(m)}
        {(2-\varepsilon - 1 /\delta_{\text r} )\log m }
 \rightarrow -\frac{1}{(2-\varepsilon )\delta_{\text r} -1}.
$$
Finally we let $\varepsilon \rightarrow 0$, and obtain the
inequality
$$
\liminf_{n\rightarrow \infty }
\frac{\log \mu_{\text{cr}}( \# {\bf C}_0^+ \geq n )}{\log n}
   \geq - \frac{1}{2\delta_{\text r}-1}.
$$
For the converse inequality, we fix $k_0$ arbitrarily.
From Theorem \ref{K87IMA-th}, we have
\begin{align*}
\mu_{\text{cr}}( \# {\bf C}_0^+ \geq n )
&\leq  
 \sum_{k\leq k_0} \pi_{\text{cr}}(2^k)
  \mu_{\text{cr}}( \# {\bf C}_0^+ \geq n \, | \, 
           2^k \leq R < 2^{k+1} ) + \pi_{\text{cr}}(2^{k_0}) \\
&\leq \sum_{k\leq k_0} \pi_{\text{cr}} (2^k)C_{36}(1) n^{-1}4\cdot 2^{2k}\pi_{\text{cr}}(2^k)
       + \pi_{\text{cr}}(2^{k_0}) \\
&\leq \pi_{\text{cr}}(2^{k_0})\left( \frac{C_{36}(1)}{n}
           \sum_{k\leq k_0} 
         \frac{\pi_{\text{cr}}(2^k)^2 2^{2(k+1)}}{\pi_{\text{cr}}(2^{k_0})}
            +1 \right).
\end{align*}
Note that $R<2^k$ implies that $\# {\bf C}_0^+ \leq 2^{2(k+1)}$. 
Since $n$ is supposed to be large,
we can assume that the above summation is taken over $k$'s with
$2^k\geq 2L_0$.
Then by Lemma \ref{lem:K(3.6)},
$$
\frac{C_{36}(1)}{n}
           \sum_{k\leq k_0} 
         \frac{\pi_{\text{cr}}(2^k)^2 2^{2(k+1)}}{\pi_{\text{cr}}(2^{k_0})}
            +1
\leq \frac{4C_{36}(1)C_{37}\pi_{\text{cr}}(2^{k_0})2^{2k_0}k_0}{n} +1.
$$
Take $\varepsilon >0$ arbitrarily, and put
$$
k_0= \frac{(1-\varepsilon )\log_2 n}{2 -1/\delta_{\text r}}.
$$
Then the above value is bounded and we have
$$\limsup_{n\rightarrow \infty }
 \frac{\log \mu_{\text{cr}}( \# {\bf C}_0^+ \geq n )}{\log n} \leq 
  -(1-\varepsilon )\frac{1}{\delta_{\text r}}\frac{1}{2-1/\delta_{\text r}}.
$$
Since $\varepsilon >0$ is arbitrary, we obtain
$$
\limsup_{n\rightarrow \infty }
\frac{\log \mu_{\text{cr}}( \# {\bf C}_0^+ \geq n )}{\log n} 
   \leq - \frac{1}{2\delta_{\text r}-1}.
$$
\end{proof}

\subsection{Proof of Kesten-Theorem 3}

%From Theorem \ref{K87IMA-th},
%we can choose $\varepsilon_1>0$ so that
%\[ \mu_{T,h} \bigl( \# {\bf C}_0^+ \geq \varepsilon_1 L^2 \pi_{T,h}(L) \,\big\vert \, \mathbf{O}  \stackrel{+}{\leftrightarrow} \partial S(L) \bigr) \geq \frac{1}{2}. \]
%(As in \cite{K87scaling} p.141.)

\begin{Theorem}[cf. \cite{K87scaling} (1.25) in Theorem 3]\label{KTheorem3b}

\noindent
(i) For $t>1$, we have
\begin{equation} \label{K(1.25)}
E_{T,h} [ (\# {\bf C}_0^+ )^t : \# {\bf C}_0^+ < \infty ] \asymp L(h,\varepsilon_0)^{2t} \pi_{\text{cr}} (L(h,\varepsilon_0))^{t+1}
\end{equation}
as $h\rightarrow h_c(T)$.

\noindent
(ii) For $t=1$, we have the same order as above for the lower bound,
but we need an extra logarithmic factor for the upper bound.
Namely, there exist constants $C_{38},C_{39}>0$ depending on $j_1$ and 
$ \varepsilon_0$, such that
\begin{align}\label{K(1.25a)}
&C_{38}L(h,\varepsilon_0)^{2}\pi_{\text{cr}}(L(h,\varepsilon_0))^{2}\\
&\leq  E_{T,h}\left[ \# {\bf C}_0^+ :
      \# {\bf C}_0^+ <\infty \right] \nonumber \\
&\leq  C_{39} L(h,\varepsilon_0)^{2}\pi_{\text{cr}}(L(h,\varepsilon_0))^{2}
       \log_2L(h,\varepsilon_0).  \nonumber 
\end{align}
\end{Theorem}

\begin{proof} Let $L=L(h,\varepsilon_0)$. 
For the lower bound, we estimate as follows.
As in the proof of Theorem \ref{K87IMA-th} %49 
(ii), we can choose $\varepsilon_1$
for $\varepsilon =1/2$, such that 
$$
\mu_{T,h} \left( \frac{\# {\bf C}_0^+ \cap S(L)}{L^2\pi_h(L)}\geq \varepsilon_1 \,
\bigg\vert \, L\leq R <2L \right)
\geq \frac{1}{2}.
$$
Then we have
\begin{align*}
&E_{T,h}[ (\# {\mathbf C}_0^+ )^t : \# {\mathbf C}_0< \infty ]\\
&\geq (\varepsilon_1 L^2\pi_h(L) )^t 
\mu_{T,h}\left( 
  L\leq R <2L
 \right) \\
&\times \mu_{T,h}\left( \infty >
\# {\mathbf C}_0^+\geq \varepsilon_1 L^2 \pi_h(L)
 \, \big\vert \,
\, L\leq R < 2L \right) \\
&\geq \frac{1}{2} \frac{\delta_8^4}{2}\varepsilon_1^t L^{2t}\pi_h(L)^{t+1}.
\end{align*}
The last inequality is from (\ref{Kbg(10)}) and the choice of $\varepsilon_1$.
But Theorem \ref{KTheorem1} enables us to replace $\pi_h(L)$ with
a constant multiple of $\pi_{\text{cr}}(L)$.

Now we turn to the upper bound. 
We will show the inequality for an integer $t\geq 1$.
The general case can be obtained by H\"older's inequality.
Similarly to the proof of Theorem \ref{KTheorem3a}, we divide the 
expectation according to the size of the radius $R$.
\begin{align}\label{eq:moment-1}
&E_{T,h} [ (\# {\bf C}_0^+ )^t : \# {\bf C}_0^+ < \infty ] \\
&\leq C_1^\# (t)+ \sum_{2L_0 \leq 2^{k+1}\leq L}
       E_{T,h}\left[
     {\bigl[ \# ({\bf C}_0^+ \cap S(2^{k+1})) \bigr] }^t
         : 2^k\leq R < 2^{k+1}
              \right] \nonumber \\
&\quad + E_{T,h}\left[ { \biggl\{ \# ( {\bf C}_0^+ \cap S(L)) +
              \sum_{{\mathbf m}\not= {\mathbf O}}
            \# ( {\bf C}_0^+ \cap B({\mathbf m})) \biggr\} }^t
             : \frac{L}{2} \leq R < \infty 
             \right] , \nonumber
\end{align}
where $B({\mathbf m}) = S(2L{\mathbf m}, L)=2L{\mathbf m}+S(L)$,
and $C_1^\# (t)>0$ is a constant depending on $t,L_0$ as before.

Note that for $2^k\leq L$, by the FKG inequality
\begin{align*}
&E_{T,h}\left[
     {\bigl[ \# ({\bf C}_0^+ \cap S(2^{k+1})) \bigr] }^t
         : 2^k\leq R < 2^{k+1}
              \right] \\
&\leq E_{T,h}\left[
     {\bigl[ \# ({\bf C}_0^+ \cap S(2^{k+1})) \bigr] }^t
         : 2^k\leq R 
              \right]  \\
&\leq \delta_4^{-4}\delta_2^{-1}
     E_{T,h}\left[
     {\bigl[ \# ({\bf C}_0^+ \cap S(2^{k+1})) \bigr] }^t
         : 2^{k+1} \leq R 
              \right]   .
\end{align*}
By Proposition \ref{inc-8}, the right hand side of the above inequality is
bounded from above by
\begin{align*}
&\delta_4^{-4}\delta_2^{-1}C_{35}(t)( 2^{2(k+1)}\pi_h(2^k)+1)^t\pi_h(2^{k+1}) \\
& \leq \delta_4^{-4}\delta_2^{-1}C_{35}(t)2^t 
 ( 2^{2t(k+1)}\pi_h(2^{k+1})^{t+1} +1 ).
\end{align*}
Since the inequality $2^{k+1}\geq L \geq 2^k$ may occur,
these constants may change a little to $C_{35}'(t)$, but it depends only on 
$t$ and $\varepsilon_0$, too.
Thus, thanks to Theorem \ref{KTheorem1}, we have
\begin{align*}
&E_{T,h}\left[
     {\bigl[ \# ({\bf C}_0^+ \cap S(2^{k+1})) \bigr] }^t
         : 2^k\leq R < 2^{k+1}
              \right]    \\
&\leq  C_2^\# (t) \biggl( 2^{2t(k+1)}\pi_{\text{cr}}(2^k)^{t+1} +1 \biggr)
\end{align*}
for some constant $C_2^\# (t)$ which depends only on $t, j_1$ and 
$\varepsilon_0$.

Let $k_0$ be the largest integer such that $2^{k_0}\leq L$.
By Lemma \ref{lem:K(3.6)}, we have
\begin{align*}
&C_2^\# (t)\sum_{2L_0\leq 2^{k+1}\leq L}\left(
 2^{2t(k+1)}\pi_{\text{cr}}(2^{k+1})^{t+1}+1 \right) \\
&\leq  C_2^\# (t)k_0 + 
      C_2^\# (t) \sum_{2L_0 \leq 2^{k+1}\leq 2^{k_0}}
      2^{(k+1)(t-1)}
      {  \bigl\{ C_{37}2^{k_0}\pi_{\text{cr}}(2^{k_0}) \bigr\} }^{t+1}.
\end{align*}
The second term in the right hand side of the above inequality is 
bounded from above by
$$
C_2^\# (t)C_{37}2^{2t(k_0+1)}\pi_{\text{cr}}(2^{k_0})^{t+1} \quad \mbox{if } t>1,
$$
and by
$$
C_2^\# (t)C_{37}k_02^{2t(k_0+1)}\pi_{\text{cr}}(2^{k_0})^{t+1} \quad \mbox{if }
t=1.
$$
Therefore the second term in the right hand side of  (\ref{eq:moment-1})
has the following upper bound as $h\rightarrow h_c(T)$.
\begin{align}\label{eq:moment-2}
& \sum_{2L_0\leq 2^{k+1}\leq L} E_{T,h}\left[
     {\bigl[ \# ({\bf C}_0^+ \cap S(2^{k+1})) \bigr] }^t
         : 2^k\leq R < 2^{k+1}
              \right] \\
&\leq 
\begin{cases}
 C_3^\# (t)L^{2t}\pi_{\text{cr}}(L)^{t+1},  & \text{if } t>1,\\
 C_3^\# (t)L^2\pi_{\text{cr}}(L)^2 \log L,      & \text{if } t=1.
\end{cases} \nonumber
\end{align}
Here, $C_3^\# (t)$ is a constant depending only on $t, j_1$ and 
$\varepsilon_0$.

Next, we estimate the third term in the right hand side of (\ref{eq:moment-1}).
Since $(a+b)^t\leq 2^t(a^t+b^t)$,
\begin{align}\label{eq:moment-3}
&E_{T,h}\left[ \biggl\{ \# \bigl( {\bf C}_0^+ \cap S(L) \bigr) +
              \sum_{{\mathbf m}\not= {\mathbf O}}
            \# \bigl( {\bf C}_0^+ \cap B({\mathbf m}) \bigr) \biggr\}^t
             : \frac{L}{2} \leq R < \infty 
             \right] \\
&\leq 2^t E_{T,h}\left[ \# \bigl( {\bf C}_0^+ \cap S(L)
                          \bigr)^t
         + \biggl\{\sum_{{\mathbf m}\not= {\mathbf O}}
              \# \bigl( {\bf C}_0^+ \cap B({\mathbf m})
                          \bigr)
              \biggr\}^t 
        : \frac{L}{2}<R<\infty 
         \right] .\nonumber
\end{align}
From Proposition \ref{inc-8} and Theorem \ref{KTheorem1}, the first term in 
the right hand side of (\ref{eq:moment-3}) is bounded from above by
$$
C_3^\# (t) \pi_{\text{cr}}(L)\left( (L^2\pi_{\text{cr}}(L) )^t +1 \right) 
$$
for some constant $C_3^\# (t)$ depending on $t$ and $ \varepsilon_0$.
As for the second term in the right hand side of (\ref{eq:moment-3}),
we first use Minkowski's inequality to obtain
\begin{align}\label{eq:moment-4}
& \left( E_{T,h} \left[ \biggl\{ \sum_{{\mathbf m}\not= {\mathbf O}}
              \# \bigl( {\bf C}_0^+ \cap B({\mathbf m}) \bigr) \biggr\}^t
        : \frac{L}{2}<R<\infty \right] \right)^{1/t}\\
&\leq  \sum_{{\mathbf m}\not= {\mathbf O}}
          \left( 
              E_{T,h}\left[ 
           \bigl[ \# \bigl({\bf C}_0^+ \cap B({\mathbf m}) \bigr) \bigr]^t
          : \frac{L}{2}<R<\infty \right] \right)^{1/t}. \nonumber
\end{align}

As in the proof of Theorem \ref{KTheorem3a}, we consider two cases separately; whether $h<h_c(T)$ or $h>h_c(T)$.

\noindent
Case $1^\circ )$ $h<h_c(T)$.
We will show that
\begin{align}\label{eq:moment-5}
& \sum_{ {\mathbf m}\not= {\mathbf O}}
 \left(  E_{T,h}\left[ 
           \bigl[ \# \bigl({\bf C}_0^+ \cap B({\mathbf m}) \bigr) \bigr]^t
          : \frac{L}{2}<R<\infty \right] \right)^{1/t} \\
&\leq  C_4^\# (t)(L^2\pi_{\text{cr}}(L)+1) \left( \pi_{\text{cr}}(L)
                 + C\frac{L^3}{3}e^{-\alpha L/3} \right)^{1/t}.\nonumber
\end{align}
Here, $C_4^\# (t)$ depends only on $j_1$ and $\varepsilon_0$.
If this is true, 
combining (\ref{eq:moment-1})--(\ref{eq:moment-5}), we obtain the 
correct order in Theorem \ref{KTheorem3b}, (\ref{K(1.25)}) and 
(\ref{K(1.25a)}).

First fix ${\mathbf m}\not= {\mathbf O}$ arbitrarily.
Then since
\begin{align*}
& E_{T,h}\left[ 
           \bigl[ \# \bigl({\bf C}_0^+ \cap B({\mathbf m}) \bigr) \bigr]^t
          : \frac{L}{2}<R<\infty \right] \\
&= E_{T,h}\left[ 
           \bigl[ \# \bigl({\bf C}_0^+ \cap B({\mathbf m}) \bigr) \bigr]^t
          : L<R<\infty \right]\\
&= \sum_{v_1,\ldots ,v_t\in B({\mathbf m})}
  \mu_{T,h}\bigl(
  {\mathbf O}\stackrel{+}{\leftrightarrow } v_1, \ldots ,v_t
    ,\, L<R<\infty \bigr) ,
\end{align*}
as in the proof of Proposition \ref{inc-8}, by induction we can show that the 
right hand side of the above equality is bounded from above by
\begin{equation}\label{eq:moment-6}
C_5^\# (t)\sum_{v\in B({\mathbf m})}
  \mu_{T,h}\bigl( 
   {\mathbf O}\stackrel{+}{\leftrightarrow }v \bigr)
   ( L^2\pi_h(L)+1)^{t-1},
\end{equation}
where $C_5^\# (t)$ depends only on $t, j_1$ and $\varepsilon_0$.

When $|{\mathbf m}|_\infty \leq 3$, by the mixing property
$$
\mu_{T,h}\bigl( 
   {\mathbf O}\stackrel{+}{\leftrightarrow }v \bigr)
  \leq 
\left( \pi_h(L/3) + C \frac{L^3}{3}e^{-\alpha L/3} \right)
 \pi_h(L/3).
$$
Hence, by Lemma \ref{RSWbound} and Theorem \ref{KTheorem1},
\begin{align*}
\sum_{v\in B({\mathbf m})}
  \mu_{T,h}\bigl( 
   {\mathbf O}\stackrel{+}{\leftrightarrow }v \bigr)
& \leq  4L^2\pi_h(L/3)\left(
    \pi_h(L/3) + C\frac{L^3}{3}e^{-\alpha L/3} \right) \\
& \leq C_6^\#  L^2 \pi_{\text{cr}}(L)
         \left( \pi_{\text{cr}}(L) + C\frac{L^3}{3}e^{-\alpha L/3}
         \right) .
\end{align*}
Here, the constant $C_6^\# \geq 1$ depends only on $\varepsilon_0$.
Inserting the above inequality into (\ref{eq:moment-6}), and summing it up
over ${\mathbf m}$'s with
$|{\mathbf m}|_\infty \leq 3$, we obtain the correct order in the right hand
side of (\ref{eq:moment-5}).
When $|{\mathbf m}|_\infty >3$, by the mixing property 
\begin{align*}
& \mu_{T,h}\bigl( 
   {\mathbf O}\stackrel{+}{\leftrightarrow }v \bigr)\\
& \leq  \mu_{T,h}\left(
   {\mathbf O}\stackrel{+}{\leftrightarrow }\partial S(L/2),
 S(L)\stackrel{+}{\leftrightarrow } \partial S(2L|{\mathbf m}|_\infty -L),
   v \stackrel{+}{\leftrightarrow } \partial S(v,L/2)
     \right) \\
& \leq  \left(\pi_h(L/2)+ C \frac{L^3}{2}e^{-\alpha L/2} \right)^2
\mu_{T,h}\bigl(
    S(L)\stackrel{+}{\leftrightarrow } \partial S(2L|{\mathbf m}|_\infty -L)
 \bigr) .
\end{align*}
Therefore by Lemma \ref{Kp121} and (\ref{K(2.24)}), 
$$
\sum_{v\in B({\mathbf m})}
  \mu_{T,h}\bigl( 
   {\mathbf O}\stackrel{+}{\leftrightarrow }v \bigr)
\leq 4L^2 \left( \pi_h(L/2)+C\frac{L^3}{2}e^{-\alpha L/2} \right)^2
  \times K_6e^{-K_7(2|{\mathbf m}|_\infty -1)}.
$$
Summing up this inequality over ${\mathbf m}$'s with 
$|{\mathbf m}|_\infty >3$, we obtain by Lemma \ref{RSWbound} and
Theorem \ref{KTheorem1},
$$
\sum_{|{\mathbf m}|_\infty >3}\sum_{v \in B({\mathbf m})}
\mu_{T,h}\bigl( 
   {\mathbf O}\stackrel{+}{\leftrightarrow }v \bigr)
\leq \mbox{Const.} \times L^2 \left( \pi_{\text{cr}}(L) 
+ C\frac{L^3}{2} e^{-\alpha L/2} \right)^2,
$$
where the constant above depends only on $\varepsilon_0$.
Assuming that $|h-h_c(T)|$ is small so that $L=L(h,\varepsilon_0)$
is so large that $ CL^3 e^{-\alpha L/3}<1$, we obtain the desired
inequality (\ref{eq:moment-5}).

\noindent
Case $2^\circ )$ $h>h_c(T)$.
In this case, the only thing we have to check is that
the contribution from ${\mathbf m}$'s with $|{\mathbf m}|_\infty >3$
to the third term in the right hand side of (\ref{eq:moment-1})
can be controlled, as well.
To this end, we start with the following inequality.
\begin{align} \label{eq:moment-7}
&\mu_{T,h}\bigl( 
{\mathbf O}\stackrel{+}{\leftrightarrow } v_1, \ldots , v_t \,|\,L<R<\infty \bigr) \\
&\leq \mu_{T,h}\bigl( 
{\mathbf O}\stackrel{+}{\leftrightarrow } v_1, \ldots , v_t \bigr) \nonumber \\
&\quad  \times 
 \mu_{T,h}\left(
  \begin{array}{@{\,}c@{\,}} 
   \hbox{there exists a $(-*)$-circuit $\sigma^*$ such that}\\
   \hbox{$\sigma^*$ surrounds $v_1, \ldots, v_t$ and ${\mathbf O}$} \\
  \end{array}
      \right) \nonumber
\end{align}
For simplicity we assume that ${\mathbf m}=(m_1,m_2)$ and 
$| {\mathbf m} |_\infty = m_1>3$.
Then the above $(-*)$-circuit $\sigma^*$ must intersect 
$\cup_{k\geq 0}B((-k,0))$ and also $\cup_{\ell \geq 0} B((m_1+\ell ,m_2))$.
Thus,
%\begin{align} \label{eq:moment-8}
%& \mu_{T,h}\left(
%    \begin{array}{@{\,}c@{\,}} 
%     \hbox{there exists a $(-*)$-circuit $\sigma^*$ such that}\\
%     \hbox{$\sigma^*$ surrounds $v_1, \ldots, v_t$ and ${\mathbf O}$} \\
%    \end{array}
%        \right) \\
%&\leq  \sum_{k=0}^\infty \mu_{T,h}\left(
%         B((-k,0))\stackrel{-*}{\leftrightarrow} B({\mathbf m})
%            \right) \nonumber \\
%&\leq  \sum_{k=0}^\infty \mu_{T,h}\left(
%           S(L)\stackrel{-*}{\leftrightarrow }
%        \partial S(2(|{\mathbf m}|_\infty +k-1))
%          \right) . \nonumber
%\end{align}
applying Lemma \ref{Kp121} for $(-*)$-connection, we have
%\begin{align} \label{eq:moment-9}
%&  \sum_{k=0}^\infty \mu_{T,h}\left(
%           S(L)\stackrel{-*}{\leftrightarrow }
%        \partial S(2(|{\mathbf m}|_\infty +k-1))
%          \right) \\ %\nonumber 
%&\leq  K_6\sum_{k=0}^\infty \exp\bigl( -2K_7(|{\mathbf m}|_\infty +k-1)
%                              \bigr) \nonumber \\ 
%&\leq  \mbox{Const.} \times \exp\bigl( -2K_7(|{\mathbf m}|_\infty -1) \bigr). \nonumber
%\end{align}
%Therefore,
$$
\sum_{|{\mathbf m}|_\infty >3} \mu_{T,h}\left(
  \begin{array}{@{\,}c@{\,}} 
   \hbox{there exists a $(-*)$-circuit $\sigma^*$ such that}\\
   \hbox{$\sigma^*$ surrounds $v_1, \ldots, v_t$ and ${\mathbf O}$} \\
  \end{array}
      \right) \leq \mbox{Const.},
$$
and the constant in the right hand side depends only on $\varepsilon_0$.

On the other hand, we have by induction
$$%\begin{eqnarray*}
\sum_{v_1, \ldots , v_t\in B({\mathbf m})} \mu_{T,h}\bigl(
  {\mathbf O}\stackrel{+}{\leftrightarrow }v_1, \ldots , v_t
       \bigr) 
\leq \sum_{v \in B({\mathbf m})} 
   \mu_{T,h}\bigl(
   {\mathbf O}\stackrel{+}{\leftrightarrow }v \bigr)
   { \bigl( L^2 \pi_h(L) +1 \bigr) }^{t-1}
$$ 
and by the mixing property,
$$
\mu_{T,h}\bigl(
   {\mathbf O}\stackrel{+}{\leftrightarrow }v \bigr)
\leq \pi_h(L)\bigl( \pi_h(L) + 4CL^3 e^{-2\alpha L } \bigr) .
$$
Thus, we have
\begin{align} \label{eq:moment-10}
& \sum_{v_1, \ldots , v_t \in B({\mathbf m})}  \mu_{T,h}\bigl(
  {\mathbf O}\stackrel{+}{\leftrightarrow }v_1, \ldots , v_t
       \bigr) 
\\
&\leq (L^2\pi_{\text{cr}} (L)+1)^t(\pi_{\text{cr}} (L)+4CL^3e^{-2\alpha L}). \nonumber
\end{align}
Note that the above bound does not depend on ${\mathbf m}$. Combining (\ref{eq:moment-7})--(\ref{eq:moment-10}), we can see that the contribution of ${\mathbf m}$'s with
$|{\mathbf m}|_\infty >3$ to  the third term in the right hand side 
of (\ref{eq:moment-1})
is under a good control. 
This completes the proof. 
\end{proof}

%\subsection{Proof of Kesten-Corollary 2}

\begin{Corollary}[cf. \cite{K87scaling} Corollary 2] We have
\begin{align*}
\xi(h) &\asymp L(h,\varepsilon_0), \\
\dfrac{E_{h} [(\#{\bf C}_0^+)^{t} : \#{\bf C}_0^+<\infty]}{E_{h} [(\#{\bf C}_0^+)^{t-1} : \#{\bf C}_0^+<\infty]}
&\approx \xi(h)^2 \pi_{h_c} (\xi(h)) \\
&\quad \mbox{ for $t \geq 2$,} \\
\left[\dfrac{1}{\chi(h)} \sum_{v \in \mathbf{Z}^2} |v|^t \mu_{h} \bigl( \mbox{$\mathbf{O} \stackrel{+}{\leftrightarrow} v$, $\# {\bf C}_0^+ <\infty$} \bigr) \right]^{1/t}
&\asymp \xi(h) \\
&\quad \mbox{ for $t > 0$.}
\end{align*}
If (\ref{K(1.1)}) and (\ref{K(1.18)}) hold, then
\begin{align*}
\dfrac{E_{ h} [ (\#{\bf C}_0^+)^{k} : \# {\bf C}_0^+<\infty]}{E_{ h} [(\# {\bf C}_0^+)^{k-1} : \# {\bf C}_0^+ <\infty]} &\approx 
|h-h_c|^{-\Delta_k} \quad \mbox{for $k \geq 2$}, \\
\left[\dfrac{1}{\chi(h)} \sum_{v \in \mathbf{Z}^2} |v|^k \mu_{h} \bigl( \mbox{$\mathbf{O} \stackrel{+}{\leftrightarrow} v$, $\# {\bf C}_0^+ <\infty$} \bigr) \right]^{1/k} &\approx |h-h_c|^{-\nu_k}\quad \mbox{for $k \geq 1$,}
\end{align*}
and
\begin{align*} 
\gamma &= 2\nu \frac{\delta-1 }{\delta +1}, \\
\Delta_k &= 2\nu \frac{\delta }{\delta+1} \quad \mbox{for $k \geq 2$,} \\
%\nu &\geq \dfrac{\delta+1}{\delta}, \\
\nu_k &= \nu \quad \mbox{for $k \geq 1$.}
\end{align*}
%\lim_{h \to h_c(T)} \dfrac{-1}{\log|h-h_c(T)|} \log \left[\dfrac{1}{\chi(T,h)} \sum_{v \in \mathbf{Z}^2} |v|^k \mu_{T,h} \bigl( \mbox{$\mathbf{O} \stackrel{+}{\leftrightarrow} v$, $\# {\bf C}_0^+ <\infty$} \bigr) \right]^{1/k}
\end{Corollary}

%\begin{proof} Corollary \ref{K87IMACor} gives the results except the inequality
%\[ \nu \geq \dfrac{\delta+1}{\delta}. \]
%\framebox{\bf How to prove this inequality?}
%\end{proof}

\newpage
\section{Symmetry of the critical exponents}
\subsection{Kesten-Lemma 8}
In Section 8, we introduced the event 
$\Omega (2^k)=\Delta_0(A^+(2^k,2^k))$.
Here we introduce a similar event $\Omega (v,S(n))$ for 
%$n\leq L(h,\varepsilon_0)$ and 
$v \in S(n)$, by
$$
\Omega (v, S(n))= \Delta_v(A^+(n,n)).
$$

\begin{Lemma}[cf. \cite{K87scaling} Lemma 8] \label{KLemma8} %%Suppose that 
%%$T>T_c$ and 
%%$0<h<2h_c(T)$. 
We fix a number $\kappa \in (0,1)$ arbitrarily. There exist positive
constants $C_{40}(\kappa )$ and $C_{41}(\kappa) $  depending on $\kappa , \eta, j_1$ and $\varepsilon_0$, such that
\begin{align}\label{eq:klemma8-1}
  &C_{40}( \kappa )\leq  
\frac{\mu_{T,h}^N \bigl( \Omega(v,S(n)) \bigr)}
{\mu_{\text{cr}}^N \bigl( \Omega(\mathbf{O},S(n)) \bigr)} 
\leq C_{41}(\kappa ) \\
&\mbox{for $ 2^{j_1+3}\leq (1-\kappa )n < n \leq \min\{ 2L(h,\varepsilon_0),\frac{N}{2}\}$ and $v \in S(\kappa n)$.} \nonumber
\end{align}
\end{Lemma}

The proof is divided into several lemmas.

\begin{Lemma} \label{lem1ofKLemma8}
We have only to prove (\ref{eq:klemma8-1}) for $v=\mathbf{O}$ and $n=2^k$, under the condition of Lemma \ref{KLemma8}.
\end{Lemma}

\begin{proof}
%We write $\mu^N$ for $\mu_{T,h}^N$.
Let
$$
\ell = \max \{ p \geq 1 : 2^p \leq (1-\kappa )n \} .
$$
Then $S(v,2^\ell )\subset S(n)$ for every $v\in S(\kappa n)$.
%\begin{align*}
% \ell &= \max \{ p \geq 1 : 2^p \leq d_{\infty} (v,S(n)^c) \}, \\
% k &= \max \{ m \geq 1 : 2^m \leq n\}.
%\end{align*}
%Since
%\[ 2^{\ell + 1} > d_{\infty} (v,S(n)^c) \geq n - \kappa n.
%  %\geq (1-\kappa) 2^k,
%  \]
%we have $k-\ell \leq 1 - \log_2 (1-\kappa)$.
Clearly,
$$    %%\begin{align*}
\mu_{T,h}^N \bigl( \Omega(v,S(n)) \bigr) %%&
\leq \mu_{T,h}^N \bigl( \Gamma(v,S(v,2^{\ell})) \bigr) %%\\
%%&
= \mu_{T,h}^N \bigl( \Gamma(\mathbf{O},S(2^{\ell})) \bigr).
$$  %%\end{align*}
%\framebox{By an extension argument,} 
By Lemma \ref{comparison of gamma and delta} (i), 
$$ %%\begin{align*}
\mu_{T,h}^N \bigl( \Gamma(\mathbf{O},S(2^{\ell})) \bigr) %%&
\leq K \mu_{T,h}^N \bigl( \Delta(\mathbf{O},S(2^{\ell})) \bigr) %% \\
%%&
\leq K \mu_{T,h}^N \bigl( \Omega(\mathbf{O},S(2^{\ell})) \bigr), 
$$ %%\end{align*}
where $K$ depends only on $j_1,\eta $ and $\varepsilon_0$. Thus we have
\[ \mu_{T,h}^N \bigl( \Omega(v,S(n)) \bigr) \leq K \mu_{T,h}^N \bigl( \Omega(\mathbf{O},S(2^{\ell})) \bigr).\]
For the %other direction, %\framebox{again by an extension argument,} 
converse inequality, let ${\mathcal A}_1,\ldots, {\mathcal A}_4$ be the
rectangles corresponding to $S(v,2^\ell )$, as before.
The longer side length of ${\mathcal A}_i$ is equal to $2^{\ell -1}$,
and the shorter side length is $2^{\ell -2}$.
Let $U_1$ be a rectangle connecting ${\mathcal A}_1$ with the left
side of $S(n)$. To be more precise, $U_1$ has width $2^{\ell -1}$,
its right side coincides with the right side of ${\mathcal A}_1$, and
its left side is a part of the left side of $S(n)$.
We define rectangles $U_2 ,U_3, U_4$ similarly corresponding to 
${\mathcal A}_2,{\mathcal A}_3,{\mathcal A}_4$.
Then using these rectangles $U_1,\ldots, U_4$, by the Connection lemma
we have  
%\begin{align*}
%\mu_{T,h}^N \bigl( \Omega(v,S(2^k)) \bigr) &\geq 
% C_1^\# (\kappa ) \mu_{T,h}^N \bigl( \Delta(v,S(v, 2^\ell )) \bigr) \\
%&\geq C_1^\# ( \kappa ) \mu_{T,h}^N \bigl( \Gamma(v,S(v,2^{\ell})) \bigr),\\
%&=  C_1^\# ( \kappa ) \mu_{T,h}^N \bigl( \Gamma ({\mathbf O}, S(2^\ell ))
%  \bigr)
%\end{align*}
$$
\mu_{T,h}^N \bigl( \Omega(v,S(n)) \bigr) \geq 
C_1^\# (\kappa ) \mu_{T,h}^N \bigl( \Delta(v,S(v, 2^\ell )) \bigr),
$$
where $C_1^\# (\kappa )$ depends only on $\kappa $ and $\varepsilon_0$. 
Then again by Lemma \ref{comparison of gamma and delta} (i),
we have
\begin{align*}
\mu_{T,h}^N \bigl( \Delta(v,S(v, 2^\ell )) \bigr) &\geq
K^{-1} \mu_{T,h}^N\bigl(  \Gamma(v,S(v,2^{\ell})) \bigr) \\
&=  K^{-1} \mu_{T,h}^N \bigl( \Gamma ({\mathbf O}, S(2^\ell )) \\
&\geq  K^{-1} \mu_{T,h}^N\bigl( \Omega ( \mathbf{O}, S(2^\ell )) \bigr) .
\end{align*}
Hence we have 
$$
\mu_{T,h}^N\bigl( \Omega (v, S(n)) \bigr) \geq C_1^\#(\kappa )K^{-1}
 \mu_{T,h}^N\bigl( \Omega ( \mathbf{O}, S(2^\ell )) \bigr) .
$$
This completes the proof of Lemma \ref{lem1ofKLemma8}. %(Note that $C_1$ and $C_2$ in (\ref{eq:klemma8-1}) depend on $\kappa,\eta,T$ and $\varepsilon_0$.)
\end{proof}

Next, by Theorem \ref{Russo-four-arm},
\begin{align*}
&\left| \frac{d}{dt}\mu_t^N \bigl( \Omega (\mathbf{O},S(2^k)) \bigr) \right| \\
&\leq
  \frac{|h-h_c|}{\mathfrak{K}T} \left\{
    C_{22}\sum_{v \in S(2^k)}\mu_t^N \bigl( \Delta_v( \Omega (\mathbf{O},S(2^k)) ) \bigr)
+ C_{23} \mu_t^N \bigl(\Omega (\mathbf{O},S(2^k)) \bigr)
   \right\}
\end{align*}
for $2^{j_3+5}\leq 2^k\leq \min\{ L(h,\varepsilon_0), \frac{N}{2}\} $,
where $j_3$ is given by (\ref{eq:8-13}). 
%The constants $C_6,C_7>0$ depend on $\eta,T,\varepsilon_0$ and $M$. 
As we remarked in the beginning of section \ref{ss:Russo1arm}, we have the same estimate for
$2^k\leq 2L(h,\varepsilon_0)$ with the constants $C_{22}, C_{23}$
slightly changed.
The new constants depend on $j_1, \eta $ and $\varepsilon_0$.
Recall that
\[ \Delta_v \Omega (\mathbf{O},S(2^k))
= \bigl( \square_v E_+ \cap \Delta_v E_{-*} \bigr) \cup 
\bigl( \Delta_v E_+ \cap \square_v E_{-*} \bigr), \]
where $E_+$ and $E_{-*}$ are the events defined in section \ref{ss:Russo4arm}. 
In the following we summarize what happens when 
$\omega \in \Delta_v\Omega (S(2^k))$. 
\begin{Lemma} \label{lem2ofKLemma8} 1) In $\omega \in \square_v E_+ \cap \Delta_v E_{-*}$, the following occurs:
\begin{itemize}
\item There exists a $(+)$-path $r_1$ connecting $\partial \{\mathbf{O}\}$ and the left side of $S(2^k)$ with $r_1 \not\ni v$.
\item There exists a $(+)$-path $r_3$ connecting $\partial \{\mathbf{O}\}$ and the right side of $S(2^k)$ with $r_3 \not\ni v$ and $r_1 \cap r_3 = \emptyset$.
\item $r:= r_1 \cup \{ \mathbf{O} \} \cup r_3$ is a horizontal crossing of $S(2^k)$, and $v$ is either above or below $r$.
\item There exists a $(+)$-path $r_5$ connecting $\partial \{v\}$ and the left side of $S(2^k)$ with $r_5 \not\ni \mathbf{O}$.
\item There exists a $(+)$-path $r_7$ connecting $\partial \{v\}$ and the right side of $S(2^k)$ with $r_7 \not\ni v$ and $r_5 \cap r_7 = \emptyset$. 
\end{itemize}
When $v$ is above $r$,
\begin{itemize}
\item There exists a $(-*)$-path $r_2^{*\prime \prime}$ connecting $\partial^* \{\mathbf{O}\}$ and $\partial^* \{v\}$.
\item There exists a $(-*)$-path $r_2^{*\prime}$ connecting $\partial^* \{v\}$ and the upper side of $S(2^k)$.
\item There exists a $(-*)$-path $r_4^{*}$ connecting $\partial^* \{\mathbf{O}\}$ and the lower side of $S(2^k)$.
\end{itemize}
When $v$ is below $r$,
\begin{itemize}
\item There exists a $(-*)$-path $r_2^{*}$ connecting $\partial^* \{\mathbf{O}\}$ and the upper side of $S(2^k)$.
\item There exists a $(-*)$-path $r_4^{*\prime \prime}$ connecting $\partial^* \{\mathbf{O}\}$ and $\partial^* \{v\}$.
\item There exists a $(-*)$-path $r_4^{*\prime}$ connecting $\partial^* \{v\}$ and the lower side of $S(2^k)$.
\end{itemize}

\noindent 2) In $\omega \in \Delta_v E_+ \cap \square_v E_{-*}$, the following occurs:
\begin{itemize}
\item There exists a $(-*)$-path $r_2^*$ connecting $\partial^* \{\mathbf{O}\}$ and the upper side of $S(2^k)$ with $r_2^* \not\ni v$.
\item There exists a $(-*)$-path $r_4^*$ connecting $\partial^* \{\mathbf{O}\}$ and the lower side of $S(2^k)$ with $r_4^* \not\ni v$ and $r_2^* \cap r_4^* = \emptyset$.
\item $r^*:= r_2^* \cup \{ \mathbf{O} \} \cup r_4^*$ is a vertical $(*)$-crossing of $S(2^k)$, and $v$ is either on the left or right of $r^*$.
\item There exists a $(-*)$-path $r_6^*$ connecting $\partial^* \{v\}$ and the upper side of $S(2^k)$ with $r_6^* \not\ni \mathbf{O}$.
\item There exists a $(-*)$-path $r_8^*$ connecting $\partial^* \{v\}$ and the lower side of $S(2^k)$ with $r_8^* \not\ni v$ and $r_6^* \cap r_8^* = \emptyset$. 
\end{itemize}
When $v$ is on the right of $r^*$,
\begin{itemize}
\item There exists a $(+)$-path $r_1$ connecting $\partial \{v\}$ and the left side of $S(2^k)$.
\item There exists a $(+)$-path $r_3^{\prime \prime}$ connecting $\partial \{\mathbf{O}\}$ and $\partial \{v\}$.
\item There exists a $(+)$-path $r_3^{\prime}$ connecting $\partial \{v\}$ and the right side of $S(2^k)$.
\end{itemize}
When $v$ is on the left of $r^*$,
\begin{itemize}
\item There exists a $(+)$-path $r_1^{\prime \prime}$ connecting $\partial \{\mathbf{O}\}$ and $\partial \{v\}$.
\item There exists a $(+)$-path $r_1^{\prime}$ connecting $\partial \{v\}$ and the left side of $S(2^k)$.
 \item There exists a $(+)$-path $r_3$ connecting $\partial \{\mathbf{O}\}$ and the right side of $S(2^k)$.
\end{itemize}
\end{Lemma}

The proof is clear so we omit it.

Now fix a $v \in \mathbf{Z}^2$ satisfying that $2^{j_1+6} < |v|_{\infty} \leq 2^{k-1}$. We can find a number $j$ with $2^j < |v|_{\infty} \leq 2^{j+1}$; note that $j \leq k-2$. 
%%Let
%%\[ \Bar{Q}_{j-5} (\ell_1,\ell_2) := [0,2^{j-5}] + (\ell_1,\ell_2) \cdot 2^{j-5}. \]

\begin{Lemma} 1) $S_{j-5}^2 (v) \subset S(2^{k-1}+2^{k-5})$.

\noindent 2)  $\displaystyle \Delta_v \Omega(\mathbf{O},S(2^k))$ is a
subset of
$$
\Gamma(\mathbf{O},S(2^{j-3})) \cap \Gamma(v,S_{j-5}^2(v)) \cap \tilde{\Gamma} (S(2^{j+1}+2^{j-1}),S(2^k)).
$$
%\begin{align*}
%&\Delta_v \Omega(\mathbf{O},S(2^k)) \\
%&\subset \Gamma(\mathbf{O},S(2^{j-3})) \cap \Gamma(v,S_{j-5}^2(v)) \cap \tilde{\Gamma} (S(2^{j+1}+2^{j-1}),S(2^k)).
%\end{align*}

\noindent 3) $d_{\infty} \bigl( S(2^{j-3}) \cup S(2^{j+1}+2^{j-1})^c,S_{j-5}^2(v) \bigr) \geq 2^{j-2}$.
\end{Lemma}

\begin{proof} 1) If $w \in S_{j-5}^2 (v)$, then 
\[ |w|_{\infty} \leq 2^{j+1}+2^{j-3} \leq 2^{k-1} + 2^{k-5}. \]

\noindent 2) follows from 1).

\noindent 3) Note that
\begin{align*}
d_{\infty} \bigl( S(2^{j-3}) ,S_{j-5}^2(v) \bigr) &\geq 2^j - 2^{j-3} - 2^{j-3} \\
&=2^j-2^{j-2} = 3 \cdot 2^{j-2}, \\
\intertext{and}
d_{\infty} \bigl( S_{j-5}^2(v), S(2^{j+1}+2^{j-1})^c\bigr) &\geq 2^{j+1} + 2^{j-1} - (2^{j+1}+2^{j-3}) \\
&=2^{j-1}-2^{j-3} = 2^{j-2} + 2^{j-3} > 2^{j-2}.
\end{align*}
\end{proof}
By the mixing property, from the above lemma we have
\begin{align*}
%%&
\mu_t^N \bigl( \Delta_v \Omega(\mathbf{O},S(2^k)) \bigr) %%\\
&\leq \mu_t^N \bigl(  \Gamma(\mathbf{O},S(2^{j-3})) \cap \tilde{\Gamma} (S(2^{j+1}+2^{j-1}),S(2^k)) \bigr) \\
&\quad \times \left\{ \mu_t^N \bigl( \Gamma(v,S_{j-5}^2(v)) \bigr) + C(2^{j-2})^3 e^{-\alpha 2^{j-2}} \right\} \\
&\leq \left\{ \mu_t^N \bigl(  \Gamma(\mathbf{O},S(2^{j-3})) \bigr) + C(2^{j+1})^3 e^{-\alpha 2^{j+1}} \right\} \\
&\quad \times \mu_t^N \bigr( \tilde{\Gamma} (S(2^{j+1}+2^{j-1}),S(2^k)) \bigr) \\
&\quad \times \left\{ \mu_t^N \bigl( \Gamma(v,S_{j-5}^2(v)) \bigr) + C(2^{j-2})^3 e^{-\alpha 2^{j-2}} \right\}.
\end{align*}
%\framebox{By an extension argument,}
By Lemma \ref{comparison of gamma and delta} (i),
$$
\mu_t^N \bigl( \Gamma(\mathbf{O},S(2^{j-3})) \bigr)
\leq K \mu_t^N \bigl( \Delta({\mathbf O},S(2^{j-3})) \bigr).
$$
Also, by Lemma \ref{lem:comparison of gamma and delta-tilde} and 
the Connection lemma,
we can find a constant $C_2^\# >0$, depending on $j_1, \eta $ and 
$\varepsilon_0$, such that $C_2^\# \geq K$ and
$$
\mu_t^N\bigl( {\tilde \Gamma }(S(2^{j+1}+2^{j-1}), S(2^k)) \bigr)
\leq C_2^\# \mu_t \bigl( \tilde{\Delta} (S(2^{j+1}+2^{j-1}),S(2^k)) \bigr) .
$$
%and Lemma \ref{lem:comparison of gamma and delta-tilde}, 
%\begin{align*}
%\mu_t^N\bigl( {\tilde \Gamma }(S(2^{j+1}+2^{j-1}), S(2^k)) \bigr) \\
%&\leq {\tilde K}
% \mu_t \bigl( \tilde{\Delta} (S(2^{j+1}+2^{j-1}),S(2^k)) \bigr),
%\intertext{and}
%\mu_t^N \bigl( \Gamma(\mathbf{O},S(2^{j-3})) \bigr)
%&
%\end{align*}
Substituting these into the above inequality, we have
%\framebox{\bf To be continued...}
\begin{align*}
&\mu_t^N \bigl( \Delta_v\Omega (\mathbf{O}, S(2^k))\bigr)\\
&\leq (C_2^\# )^2\left\{ \mu_t^N\bigl( 
 \Delta ({\mathbf O},S(2^{j-3}))\bigr) + C2^{3(j+1)}e^{-\alpha 2^{j+1}} 
\right\} \\
& \quad \times \mu_t^N\bigl( {\tilde \Delta }( S(2^{j+1}+2^{j-1}), S(2^k))\bigr)\\
& \quad \times \left\{
  \mu_t^N \bigl( \Gamma (v, S_{j-5}^2(v))\bigr) +
    C2^{3(j-2)}e^{-\alpha 2^{j-2}} \right\}  \\
& \leq (C_2^\# )^2 \left[ \mu_t^N\bigl( 
 \Delta ({\mathbf O},S(2^{j-3}))\bigr) 
\mu_t^N\bigl( {\tilde \Delta }( S(2^{j+1}+2^{j-1}), S(2^k))\bigr)
\mu_t^N \bigl( \Gamma (v, S_{j-5}^2(v))\bigr) \right. \\
&\quad + (2^9e^{-\alpha 2^j}+2)C 2^{3(j-2)}e^{-\alpha 2^{j-2}}\left.
 \mu_t^N\bigl( {\tilde \Delta }( S(2^{j+1}+2^{j-1}), S(2^k))\bigr) \right]  .
\end{align*}
By the mixing property,
\begin{align*}
\mu_t^N\bigl( 
 \Delta ({\mathbf O},S(2^{j-3}))\bigr) \leq &
 \mu_t^N\bigl( 
 \Delta ({\mathbf O},S(2^{j-3}))\,  \big\vert \,
   {\tilde \Delta }( S(2^{j+1}+2^{j-1}), S(2^k))\bigr) \\
&  + C2^{3(j+1)}e^{-\alpha 2^{j+1}},
\end{align*}
and hence we have
\begin{align*}
\mu_t^N \bigl( \Delta_v\Omega (\mathbf{O}, S(2^k))\bigr)
&\leq  (C_2^\# )^2 \mu_t^N \bigl(  \Delta ({\mathbf O},S(2^{j-3}))\cap
  {\tilde \Delta }( S(2^{j+1}+2^{j-1}), S(2^k)) \bigr) \\
  &\quad \times \mu_t^N\bigl(
 {\Gamma }( v, S_{j-5}^2(v)))\bigr) \\
& \quad + C_3^\# C2^{3(j-2)}e^{-\alpha 2^{j-2}}
 \mu_t^N\bigl( {\tilde \Delta }( S(2^{j+1}+2^{j-1}), S(2^k))\bigr) 
\end{align*}
for some constant $C_3^\# >0$ depending on $j_1, \eta $, and $\varepsilon_0$.
By the Connection lemma, we have
\begin{align*}
&\mu_t^N \bigl(  \Delta ({\mathbf O},S(2^{j-3}))\cap
  {\tilde \Delta }( S(2^{j+1}+2^{j-1}), S(2^k)) \bigr) \\
&\leq C_4^\# \mu_t^N\bigl( \Omega ( {\mathbf O}, S(2^k)) \bigr) ,
\end{align*}
where the constant $C_4^\# $ depends only on $\varepsilon_0$.
On the other hand, by Lemma \ref{lem: relation of deltas-tilde} and
the finite energy property
$$
\mu_t^N \bigl( {\tilde \Delta }(S(2^{j+1}+2^{j-1}), S(2^k)) \bigr)
\leq C_5^\# C_1^{j-j_1-5} \mu_t^N\bigl(
 \Omega ({\mathbf O}, S(2^k)) \bigr)
$$ 
where $C_5^\# >0$ is a constant depending only on $j_1$ and $\varepsilon_0$.
Thus, for $j_1+6\leq j\leq k-2$ and for $2^j<|v|_\infty \leq 2^{j+1}$ we have
\begin{align*}
\mu_t^N\bigl( \Delta_v\Omega ({\mathbf O}, S(2^k))\bigr)
 &\leq (C_2^\# )^2 C_4^\# \mu_t^N\bigl( \Omega ({\mathbf O}, S(2^k))\bigr)
   \mu_t^N\bigl( \Gamma (v, S_{j-5}^2(v))\bigr) \\
 &\quad + C_6^\# C_1^{j-j_1-5}2^{3(j-2)}e^{-\alpha 2^{j-2}}
     \mu_t^N\bigl( \Omega ({\mathbf O}, S(2^k)) \bigr) ,
\end{align*}
for some constant $C_6^\# >0$ depending on $j_1, \eta $ and $\varepsilon_0$.
This implies that we can find some constants $C_7^\#, C_8^\# >0$ depending
on $j_1, \eta $ and $\varepsilon_0$, such that
\begin{align*}
& \left| \frac{d}{dt}\mu_t^N\bigl( \Omega ({\mathbf O}, S(2^k))\bigr) \right|\\
&\leq \frac{|h-h_c(T)|}{\mathfrak{K}T}\biggl( C_7^\# \sum_{j=j_1+6}^{k-2}\sum_{2^j<|v|_\infty \leq 2^{j+1}}
  \mu_t^N\bigl( \Omega ({\mathbf O}, S(2^k))\bigr)
  \mu_t^N\bigl( \Gamma (v, S_{j-5}^2(v))\bigr) \\
&\quad + C_8^\#  
  \mu_t^N\bigl( \Omega ({\mathbf O}, S(2^k)) \bigr) \biggr).
\end{align*}
Note that we used a trivial estimate
$$ \mu_t^N\bigl( \Delta_v\Omega ({\mathbf O}, S(2^k))\bigr)
 \leq Const. \mu_t^N\bigl( \Omega ({\mathbf O}, S(2^k))\bigr),
$$
where the constant is an absolute constant.
Dividing both sides by $\displaystyle \mu_t^N\bigl( \Omega ({\mathbf O}, S(2^k))\bigr)$, and integrating in $t$ over the interval $[0,1]$,
we obtain from Lemma \ref{KLemma7} that
\begin{align*}
&\left| \log \frac{ \mu_{T,h}\bigl(\Omega ({\mathbf O}, S(2^k))\bigr) }{
  \mu_{\text{cr}}\bigl(\Omega ({\mathbf O}, S(2^k))\bigr) } \right| \\
&\leq C_9^\# \left( 1+ \sum_{j=j_1+5}^{k-2}\biggl( 2^{\zeta (k-j)} 
 + \frac{|h-h_c(T)|}{\mathfrak{K}T}2^{5(j+1)}e^{-\alpha 2^{j+1}} \biggr)
\right)
\end{align*}
for some constant $C_9^\# $ depending on $j_1, \eta $ and $\varepsilon_0$.
This proves Lemma \ref{KLemma8}.

\subsection{Kesten-Theorem 4}

\begin{Lemma}\label{Lem1KThm4} For any $\varepsilon > 0$, there exist $\delta_1>0$ and $\delta_2>0$ and $k_0=k_0(\kappa , \varepsilon )$ such that if $0<1 - \kappa<\delta_1$ and $|h-h_c|\leq \delta_2$, then for $2^{k_0}\leq 2^k \leq L(h,\varepsilon_0)$, we have
\[ \sum_{v \in S(2^k) \setminus S(\kappa 2^k)} \int_0^1 \dfrac{|h-h_c|}{\mathfrak{K}T} \mu_t^N \bigl( \Omega(v,S(2^k)) \bigr) dt \leq \varepsilon. \]
\end{Lemma}

\begin{proof}
We use the estimates of arm exponents in a half plane in section \ref{s:powerlaw}.
So, again we remark that the restriction that $2^k<L(h,\varepsilon_0)$ is not
serious; we can extend  results there for $2^k\leq 2L(h,\varepsilon_0)$, too.

Let $v=(v^1,v^2)\in S(2^k)\setminus S(\kappa 2^k)$ satisfy 
$$
2^p < d_\infty (v, S(2^k)^c)\leq 2^{p+1}
$$
for some $p\geq j_3+5$, where $j_3$ is the constant given in Theorem \ref{3 arm bound}. 
For simplicity we consider the case where $0\leq v^1\leq v^2$.
Let $Q_{p-1}(v)= (\ell_12^{p-1},(\ell_1+1)2^{p-1}]\times (\ell_22^{p-1},
(\ell_2+1)2^{p-1}] $ containing $v$, and $x_{p-1}(v)$ be its lower left
corner $(\ell_12^{p-1}, \ell_22^{p-1})$.
Let 
$$ {\tilde S}_{p-1}(v):= S(x_{p-1}(v), 3\cdot 2^{p-1}).$$
Note that ${\tilde S}_{p-1}(v)\subset S(2^k)$, but $S_{p-1}^2(v)$ may not
be a subset of $S(2^k)$.
We put $m(v)$ to be the largest integer $m$ such that
${\tilde R}_{p-1}^m(v)\not\ni (2^k,2^k)$.
Then $T_{p-1}^{m(v)+1}(v)\ni (2^k,2^k)$, and we have
\begin{align}\label{eq:13-1}
&\Omega (v,S(2^k))\\
&\subset \Gamma'(\{ v\} , {\tilde S}_{p-1}(v))\cap
 \Gamma'(T_{p-1}^2(v), T_{p-1}^{m(v)}(v) )  \cap  {\tilde \Gamma }( T_{p-1}^{m(v)+1}(v),S(2^k)). \nonumber
\end{align}
Here we used a new notation:
\begin{enumerate}
\item  $\Gamma'(\{ v\} , {\tilde S}_{p-1}(v))$ is the event that there are 
$(+)$-paths $r_1,r_3$ and $(-*)$-paths $r_2^*, r_4^*$ in 
${\tilde S}_{p-1}(v)\setminus \{ v\} $ such that
\begin{enumerate}
\item $r_1$ and $r_3$ connect $\partial \{ v\} $ with 
$\partial_{in}{\tilde S}_{p-1}(v)$, $r_1\cap r_3=\emptyset $,
\item $r_2^*$ and $r_4^*$ connect $\partial^* \{ v\} $ with 
$\partial_{in}{\tilde S}_{p-1}(v)$, $r_2^*\cap r_4^*=\emptyset $,
\item $r_1\cup \{ v\} \cup r_3$ separates $r_2^*$ and $r_4^*$ in 
${\tilde S}_{p-1}(v)$.
\end{enumerate}
\item $\Gamma'(T_{p-1}^2(v), T_{p-1}^{m(v)}(v) )$ is the event that
there exist $(+)$-paths $s_1,s_3$ and a $(-*)$-path $s_4^*$ such that
\begin{enumerate}
\item $s_1,s_3$ connect $\partial T_{p-1}^2(v)$ with 
$\partial_{in}T_{p-1}^{m(v)}(v)\setminus \{ x^2=2^k\} $ in 
$T_{p-1}^{m(v)}(v)\setminus T_{p-1}^2(v)$, and
\item $s_4^*$ connects $\partial^*T_{p-1}^2(v)$ with 
$\partial_{in}T_{p-1}^{m(v)}(v)\setminus \{ x^2=2^k\} $
in $T_{p-1}^{m(v)}(v)\setminus T_{p-1}^2(v)$.
\item $s_4^*$ separates $s_1$ and $s_3$ in $T_{p-1}^{m(v)}(v)\setminus T_{p-1}^2(v)$.
\end{enumerate}
\end{enumerate}
Note that $\Gamma'(T_{p-1}^2(v), T_{p-1}^{m(v)}(v) )$ is  a subset of 
the shifted event of the 3-arm event ${\tilde {\mathcal B}}^+(2,1,2^{p+2},2^{m(v)+p-1})$
by the vector $(\ell_12^{p-1}, 2^k-2^{p-1+m(v)})$.
If $m(v)\leq 2$, then we understand that 
$$
\Gamma'(T_{p-1}^2(v), T_{p-1}^{m(v)}(v) )=\Omega .
$$
%This is the case where $x_{p-1}(v)$ is on the diagonal line $\{ x^1=x^2\} $.
Note also that ${\tilde \Gamma }(T_{p-1}^{m(v)+1}(v),S(2^k))$ is a subset of the shifted event of 2-arm event 
${\mathcal B}^+(1,1,2^{p+m(v)+1},2^{k+1})$.

By the mixing property we have from (\ref{eq:13-1}) that
\begin{align}
&\mu_t^N\bigl( \Omega (v,S(2^k)\bigr)\label{eq:13-2}\\
&\leq \mu_t^N\bigl( 
 {\tilde \Gamma }( T_{p-1}^{m(v)+1}(v),S(2^k)) \bigr) \nonumber \\
&\quad \times \left\{ \mu_t^N\bigl(  \Gamma'(T_{p-1}^2(v), T_{p-1}^{m(v)}(v) )
  \bigr) + 4C2^{3(p+m(v)-1)}e^{-\alpha 2^{p+m(v)-1}} \right\} \nonumber \\
&\quad \times\left\{ 
   \mu_t^N\bigl( \Gamma'(\{ v\} , {\tilde S}_{p-1}(v)) \bigr)
    + C( 3\cdot 2^p)^2 \cdot 2^{p-1} e^{-\alpha 2^{p-1}} \right\} . \nonumber
\end{align}
By the above remark, translation invariance and Theorem \ref{2 arm bound},
we have
\begin{align*}
\mu_t^N\bigl(  {\tilde \Gamma }( T_{p-1}^{m(v)+1}(v),S(2^k)) \bigr)
& \leq \mu_t^N\bigl( {\mathcal B}^+( 1,1,2^{p+m(v)+1}, 2^{k+1})\bigr) \\
%%&\leq C_{10}( 2^{k-p-m(v)}-1)^{-1}
&\leq C_{10}2^{-k+p+m(v)}.
\end{align*}
%%provided that $k-p-m(v)\geq 1$.
%%When $m(v)\geq k-p-1$, then we use the trivial bound
%%$$
%%\mu_t^N\bigl(  {\tilde \Gamma }( T_{p-1}^{m(v)+1}(v),S(2^k)) \bigr) \leq 1.
%%$$
Also, by translation invariance and Theorem \ref{3 arm bound},
we have
\begin{align*}
\mu_t^N\bigl( \Gamma'(T_{p-1}^2(v), T_{p-1}^{m(v)}(v) )\bigr) 
&\leq  \mu_t^N\bigl( {\tilde {\mathcal B}}^+( 2,1,2^{p+2},2^{m(v)+p-1})\bigr) \\
&\leq 2^6C_{14}2^{-2m(v)}.
\end{align*}
Therefore from (\ref{eq:13-2}), we have %for $m(v)\leq k-p-1$,
\begin{align}\label{eq:13-3}
\mu_t^N\bigl( \Omega (v,S(2^k)\bigr) \leq & 
 \left\{ 2^6C_{14}C_{10}2^{-k-m(v)+p}+ 4C2^{3(p+m(v))}e^{-\alpha 2^{p+m(v)}}
   \right\} \\
&\times \left\{ 
 \mu_t^N\bigl(  \Gamma'(\{ v\} , {\tilde S}_{p-1}(v)) \bigr)
  +  C( 3\cdot 2^p)^2 \cdot 2^{p-1} e^{-\alpha 2^{p-1}} \right\} \nonumber
\end{align}
%%If $m(v)=k-p$ or $k-p+1$, then we have
%%\begin{align}
%%&\mu_t^N\bigl( \Omega (v,S(2^k))\bigr) \label{eq:13-2-a}\\
%%&\leq 4C_{14}2^{-2m(v)}\left\{
%% \mu_t^N\bigl( \Gamma'(\{ v\} , {\tilde S}_{p-1}(v)) \bigr)
%%     + C( 6\cdot 2^p)^2 \cdot 2^p e^{-\alpha 2^p} \right\} .\nonumber
%% \end{align} 
By the finite energy property there is an absolute constant $C_1^\# >0$ such that
$$
\mu_t^N\bigl( \Gamma'(\{ v\} , {\tilde S}_{p-1}(v)) \bigr)
\leq C_1^\# 
 \mu_t^N\bigl( \Gamma ( v , {\tilde S}_{p-1}(v)) \bigr) .
$$
When $m(v)\leq 2$, we use
\begin{equation}\label{eq:13-2-a}
\mu_t^N\bigl( \Omega (v,S(2^k))\bigr) \leq \mu_t^N\bigl( \Gamma'(\{ v\} , {\tilde S}_{p-1}(v))\bigr) .
\end{equation}
Note that there are only four distinct $Q_{p-1}(v)$'s with $m(v)\leq 2$.
Since $p\geq j_3+5$, and since 
$$
S_{j_1}^{p-j_1}(v)\subset {\tilde S}_{p-1}(v) \subset S_{p-1}^2(v)
\subset S_{j_1}^{p-j_1+3}(v),
$$
by Lemma \ref{comparison of gamma and delta}(i), then by Lemma \ref{relation of deltas}, we have
%%Let $j^*$ be the maximum of integers $j\geq 1$ with
%%$$ S_{j_1}^j(v) \subset {\tilde S}_{p-1}(v).$$
%%Then since $p\geq j_3+5$, we have $j^*\geq 5$, and by Lemma \ref{comparison of%% gamma and delta} (i), then by Lemma \ref{relation of deltas}, we have
\begin{align*}
\mu_t^N\bigl( \Gamma (v, {\tilde S}_{p-1}(v) ) \bigr)
&\leq  \mu_t^N\bigl( \Gamma (v, S_{j_1}^{p-j_1}(v) )\bigr) \\
&\leq  KC_1^3 \mu_t^N\bigl( \Delta (v, S_{j_1}^{p-j_1+3}(v)) \bigr) \\
&\leq  KC_1^3 \mu_t^N\bigl( \Gamma (v, S_{p-1}^2(v)) \bigr) .
\end{align*}
%%The last inequality follows from the fact that 
%%$ S_{j_1}^{j^*}(v) \supset S_{p-1}^2(v)$.

For a fixed $3\leq m \leq k-p$, we sum up this inequality over $v$'s such that
$m(v)=m, 2^p< d_\infty (v, S(2^k)^c)\leq 2^{p+1}$, and $0\leq v^1\leq v^2$,
and then we obtain 
\begin{align*}
\sum_{ \genfrac{}{}{0pt}{2}{v: m(v)=m,\, 
   2^p<d_\infty (v,S(2^k)^c )\leq 2^{p+1}}
       {0\leq v^1\leq v^2} 
     } &\mu_t^N\bigl( 
       \Gamma (v, S_{p-1}^2(v)) \bigr) \\
&\leq 2^{m+1}\sum_{v\in S(2^{p-1})}\mu_t^N\bigl( 
  \Gamma (v, S(2^{p+1})) \bigr) .
\end{align*}
By Lemma \ref{KLemma7}, we have
\begin{align*}
\sum_{v\in S(2^{p-1})}&\int_0^1 \frac{|h-h_c(T)|}{\mathfrak{K}T}
        \mu_t^N\bigl( \Gamma (v,S(2^{p+1}))\bigr) \, dt \\
 \leq & C_{24}\left( 2^{-\zeta (k-p-1)} +
  \frac{|h-h_c(T)|}{\mathfrak{K}T} 2^{5(p+2)}e^{-\alpha 2^{p+1}}\right).
\end{align*}
Therefore for some constant $C_2^\# >0$ depending on $j_1, \eta ,\varepsilon_0$,we have
\begin{align*}
\sum_{\genfrac{}{}{0pt}{2}{2^p<d_\infty (v,S(2^k)^c)\leq 2^{p+1}}{0\leq v^1\leq v^2}} & \int_0^1 \frac{|h-h_c(T)|}{\mathfrak{K}T}
 \mu_t^N\bigl( \Omega (v, S(2^k))\bigr) \, dt\\
&\leq C_2^\#
\left( 2^{-\zeta (k-p-1)}+ \frac{|h-h_c(T)|}{\mathfrak{K}T}2^{5(p+2)}
 e^{-\alpha 2^{p-1}}\right) .
\end{align*}
Finally, we sum up this inequality over $p$'s with $2^p\leq (1-\kappa )2^k$ to
obtain
\begin{align}\label{eq:13-4}
&\sum_{ \genfrac{}{}{0pt}{2}{2^{j_3+5}<d_\infty ( v , (S(2^k))^c)}{ v \in S(2^k)\setminus S(\kappa 2^k)}}
 \int_0^1  \frac{|h-h_c(T)|}{\mathfrak{K}T}
 \mu_t^N\bigl( \Omega (v, S(2^k))\bigr) \, dt \\
&\leq C_1^\# (1+C_2^\# )
 \sum_{p=j_3+5}^{k+ \log_2(1-\kappa )}\left(
 2^{-\zeta (k-p-1)}+ \frac{|h-h_c(T)|}{\mathfrak{K}T}2^{5(p+2)}
  e^{-\alpha 2^{p-1}} \right) \nonumber \\
&\leq  C_3^\# \left( \frac{|h-h_c(T)|}{\mathfrak{K}T} + \frac{2^\zeta (1-\kappa )^\zeta }{1-2^{-\zeta }} \right) ,\nonumber 
\end{align}
where $C_3^\# >0$ depends only on $j_1, \eta $ and $\varepsilon_0$.
This goes to zero as $h\rightarrow h_c(T)$ and $\kappa \rightarrow 1$.

Now, we consider the remaining case: $p< j_3+5$.
Let 
$$
x_{j_3+5}(v)= (2^k-\ell2^{j_3+5}, 2^k-2^{j_3+5}),
$$
and for $\ell \geq 2$, let $m(v)$ be
$$
m(v):= \max\{ m\geq 1 : T_{j_3+5}^m(v)\not\ni (2^k,2^k)\} .
$$
Then as in (\ref{eq:13-2}), we have for $\ell \geq 2$,
$$
\Omega (v,S(2^k)) \subset \Gamma'(\{ v\} , T_{j_3+5}^{m(v)}(v)) \cap
{\tilde \Gamma }(T_{j_3+5}^{m(v)+1}(v), S(2^k)).
$$
From this and the mixing property, we have
\begin{align*}
\mu_t^N\bigl( \Omega (v,S(2^k))\bigr) \leq & \left\{
 \mu_t^N\bigl( \Gamma'(\{ v\} , T_{j_3+5}^{m(v)}(v)) \bigr)
   +4C 2^{3(j_3+m(v)+5)}e^{-\alpha 2^{j_3+m(v)+5}} \right\} \\
&\times \mu_t^N\bigl(   {\tilde \Gamma }(T_{j_3+5}^{m(v)+1}(v), S(2^k))
 \bigr) .
\end{align*}
As before  by ( \ref{eq:2 arm bound for boxes}) %if $m(v)\leq k-j_3-8$,  
we have
\begin{align*}
\mu_t^N\bigl(   {\tilde \Gamma }(T_{j_3+5}^{m(v)+1}(v), S(2^k)) \bigr)
&\leq \mu_t^N \bigl( {\mathcal B}^+(1,1,2^{j_3+m(v)+7},2^{k+1}) \bigr)\\
&\quad + C_{10}2^{-k+j_3+m(v)+6}.
\end{align*}
Also, by the finite energy property, changing the configuration in 
$S_{j_3+5}(v)$, we can obtain 
\begin{align*}
\mu_t^N\bigl( \Gamma'(\{ v\} , T_{j_3+5}^{m(v)}(v)) \bigr)
&\leq C_4^\# \mu_t^N\bigl({\tilde {\mathcal B}}^+(2,1,1,2^{m(v)+j_3+5})\bigr) \\
&\leq C_4^\# C_{14}2^{-2(m(v)+j_3+5)},
\end{align*}
where $C_4^\# >0$ depends only on $j_3$.
Thus, we have for $\ell \geq 2$,
\begin{align}\label{eq:13-5}
\mu_t^N\bigl(\Omega (v,S(2^k))\bigr) \leq &
C_{10}2^{-k+j_3+m(v)+6}\\
&\times \left\{ C_4^\# C_{14}2^{-2(m(v)+j_3+5)} +
 4C 2^{3(j_3+m(v)+5)}e^{-\alpha 2^{j_3+m(v)+5}}\right\} . \nonumber
\end{align}
For $\ell =1$, we just use (\ref{eq:2 arm bound}) to obtain
\begin{equation}\label{eq:13-6}
\mu_t^N\bigl( \Omega (v,S(2^k)) \bigr) \leq C_4^\# \mu_t^N\bigl( {\mathcal E}_2
(2^k) \bigr) \leq C_4^\# C_8 2^{-k}
\end{equation}
Thus, from (\ref{eq:13-5}) and (\ref{eq:13-6}) we have
%\begin{align*}
$$
\sum_{\genfrac{}{}{0pt}{2}{d_\infty (v,S(2^k)^c)<2^{j_3+5}}{0\leq v^1\leq v^2}}
 \int_0^1  \frac{|h-h_c(T)|}{\mathfrak{K}T} \mu_t^N\bigl(
 \Omega (v, S(2^k)) \bigr) \, dt \leq C_5^\# k2^{-k}
\frac{|h-h_c(T)|}{\mathfrak{K}T}
$$
for some constant $C_5^\#>0$ depending only on $j_3,\eta $ and $\varepsilon_0$.
The right hand side of the above inequality converges to zero uniformly in $k$
as $h\rightarrow h_c(T)$.
\end{proof}

\begin{Theorem}[cf. \cite{K87scaling} Theorem 4] \label{KTheorem4} 
For $\delta >0$, let
\begin{equation}\label{eq:L_0(delta )}
L_0(\delta ):= \min \left\{ n\geq 1  :  
 n^2\mu_{\text{cr}}\bigl( \Omega (\mathbf{O},S(n) )\bigr) \geq \frac{1}{\delta } \right\} .
\end{equation}
Then there exist positive constants $C_{42}, C_{43}$, depending on $C(1),j_1,
\eta $ and $\varepsilon_0$, such that for every small $\delta >0$, we have
\begin{align}\label{eq: L equiv L_0below}
C_{42}\leq \frac{L(h_c-\delta , \varepsilon_0)}{L_0(\delta )} \leq C_{43}
%\end{equation}
\intertext{and}
%\begin{equation}
\label{eq: L equiv L_0above}
C_{42}\leq \frac{L(h_c+\delta , \varepsilon_0)}{L_0(\delta )} \leq C_{43}.
\end{align}
Further, for every pair $0<\varepsilon_1, \varepsilon_2\leq \varepsilon_0$,
we have
\begin{equation}\label{L-varying epsilon}
L(h,\varepsilon_1)\asymp L(h,\varepsilon_2) \quad \hbox{ as }\,  h \rightarrow
h_c(T).
\end{equation}
\end{Theorem}

%% The proof goes parallel to the proof of Theorem 4 in \cite{K87scaling}.
%%
%% ------------------------------------------------------------------
%%\begin{quote}
%%\framebox{\bf INTERMISSION}
%%
\begin{proof} 
The proof goes parallel to the proof of Theorem 4 in \cite{K87scaling}.
Since the argument is the same we prove only (\ref{eq: L equiv L_0below}).
By the definition of $L(h,\varepsilon_0)$, and the fact that $0<\varepsilon_0 < \frac{1}{2}C(1)$, we can see that
\[ \left| \mu_h^N \bigl( A^+(L(h,\varepsilon_0),L(h,\varepsilon_0)) \bigr) - \mu_{\text{cr}}^N \bigl( A^+(L(h,\varepsilon_0),L(h,\varepsilon_0)) \bigr)  \right| \geq \frac{1}{2} C(1). \]
if we take $N$ sufficiently large. By Corollary \ref{RussoOneArm}, 
\begin{align*}
\frac{1}{2} C(1) &\leq \int_0^1dt  \dfrac{|h-h_c|}{\mathfrak{K}T} \biggl\{ C_{18} \sum_{v \in S(L(h,\varepsilon_0))} \mu_t^N \bigl( \Omega(v,S(L(h,\varepsilon_0))) \bigr) \\
&\quad  + C_{19} \mu_t^N \bigl( A^+(L(h,\varepsilon_0),L(h,\varepsilon_0)) \bigr)\biggr\} .
\end{align*}
If 
\[ \dfrac{|h-h_c|}{\mathfrak{K}T} C_{19}\leq \dfrac{1}{8} C(1), \]
then we have
\[ \frac{3}{8} C(1) \leq \int_0^1 dt 
\dfrac{|h-h_c|}{\mathfrak{K}T} \biggl\{ 
 C_{18} \sum_{v \in S(L(h,\varepsilon_0))} 
   \mu_t^N \bigl( \Omega(v,S(L(h,\varepsilon_0))) \bigr)\biggr\} . \]
%%\framebox{By an extension argument,}
Let $k$ be the integer such that $2^k<L(h, \varepsilon_0)\leq 2^{k+1}$
and $a=L(h,\varepsilon_0)-2^k$. Then we have
$$
S(L(h,\varepsilon_0)) \subset \bigcup_{x \in {\{ -1, +1\} }^2} S(ax, 2^k).
$$
If $ v \in S(ax, 2^k)$ for some $x\in {\{ -1, +1\} }^2$, then 
$$
\Omega (v, S(L(h,\varepsilon_0)) \subset \Gamma (v, S(ax, 2^k)),
$$
and by translation invariance together with Lemma \ref{comparison of gamma and delta} we have
$$
\mu_t^N \bigl( \Omega(v,S(L(h,\varepsilon_0))) \bigr) \leq 
K\mu_t^N \bigl(  \Omega(v-ax,S(2^k)) \bigr).
$$
Therefore we have
$$
\sum_{v \in S(L(h,\varepsilon_0))}\mu_t^N 
  \bigl( \Omega(v,S(L(h,\varepsilon_0))) \bigr) \leq
4K \sum_{v\in S(2^k)}\mu_t^N\bigl( \Omega (v,S(2^k))\bigr) .
$$
%%
%%\begin{align*}
%%\sum_{v \in S(L(h,\varepsilon_0))} \mu_t^N \bigl( \Omega(v,S(L(h,\varepsilon_0))) \bigr)
%%\leq C_3 \sum_{v \in S(2^k)} \mu_t^N \bigl( \Omega(v,S(2^k)) \bigr).
%%\end{align*}
Choose an $\varepsilon>0$ such that $4C_{18}K \varepsilon \leq C(1)/8$. 
By Lemma \ref{Lem1KThm4}, we can find $\delta_1,\delta_2>0$ such that if $1-\kappa < \delta_1$ and $|h-h_c| < \delta_2$, then
\[ \int_0^1 dt \dfrac{|h-h_c|}{\mathfrak{K}T} \sum_{v \in S(2^k) \setminus S(\kappa 2^k)}  \mu_t^N \bigl( \Omega( v,S(2^k)) \bigr) \leq \varepsilon. \]
Thus we have
$$
\frac{1}{4} C(1) \leq \int_0^1 dt \dfrac{|h-h_c|}{\mathfrak{K}T} 
 \biggl\{ 4C_{18}K \sum_{v \in S(\kappa 2^k)} \mu_t^N \bigl( \Omega(v,S(2^k)) \bigr)\biggr\} dt.
$$

%% \\
%% &\quad + C_1C_3 \varepsilon,
%%\end{align*}
%%which implies that
%%\[  \frac{1}{4} C(1) \leq \int_0^1 \dfrac{|h-h_c|}{\mathfrak{K}T} \biggl\{ C_1C_3 \sum_{v \in S(\kappa 2^k)} \mu_t^N \bigl( \Omega(v,S(2^k)) \bigr)\biggr\} dt. \]
By Lemma \ref{KLemma8}, the right hand side is not more than 
\begin{equation}
C_1^\# \dfrac{|h-h_c|}{\mathfrak{K}T}L(h,\varepsilon_0)^2 \mu_{\text{cr}}^N \bigl( \Omega(\mathbf{O},S(L(h,\varepsilon_0))) \bigr) ,
\end{equation}
where $C_1^\# >0$ depends on $j_1, \eta $ and $\varepsilon_0$.
Therefore 
\begin{align*}
C_2^\# \left( \dfrac{|h-h_c|}{\mathfrak{K}T} \right)^{-1} &\leq L(h,\varepsilon_0)^2 \mu_{\text{cr}}^N \bigl( \Omega (\mathbf{O},S(L(h,\varepsilon_0))) \bigr),\\
& \rightarrow L(h,\varepsilon_0)^2 
 \mu_{\text{cr}}\bigl( \Omega (\mathbf{O},S(L(h,\varepsilon_0))) \bigr)
\end{align*}
as $N\rightarrow \infty $, where $C_2^\# >0$ depends only on $j_1, \eta $ and
$\varepsilon_0$. 
This means that
\begin{equation} \label{Thm4Eq****}
L_0 \left( \dfrac{|h-h_c|}{C_2^\# \mathfrak{K}T} \right) 
\leq L(h,\varepsilon_0).
\end{equation}
On the other hand by Lemma \ref{KLemma7}, for 
$2^{j_1+5} \leq 2^j \leq 2^\ell \leq L(h,\varepsilon_0)$,  we have
(\ref{eq:klem7-1}):
\begin{align*}
&\int_0^1 dt \sum_{v \in S(2^{j-2})} \dfrac{|h-h_c|}{\mathfrak{K}T} \mu_t^N \bigl( \Omega(v,S(2^j)) \bigr) \\
&\leq C_{24} \left( 2^{-\zeta (\ell -j)} +  \dfrac{|h-h_c|}{\mathfrak{K}T} C2^{5(j+1)}e^{-\alpha 2^{j+1}} \right).
\end{align*}
Take $L_1 \leq L_2 \leq L(h,\varepsilon_0)$ and $j, \ell $ with
\[ 2^{j-1} < L_1 \leq 2^{j}, \quad 2^{\ell -1} < L_2 \leq 2^\ell . \]
%%For $0 \leq t \leq 1$ and $v \in S(2^{k-1}+2^{k-3})$, by Lemmas \ref{comparison of gamma and delta} and \ref{KLemma8},
Then by Lemmas \ref{relation of deltas}, \ref{comparison of gamma and delta} and 
\ref{KLemma8}, for $0\leq t\leq 1$ and for $v \in S(2^{j-2})$, we have
\begin{align*}
\mu_t^N\bigl( \Omega (v, S(2^j))\bigr) 
& \geq C_{40}\left(\frac{1}{4}\right)\mu_{\text{cr}}^N\bigl(
 \Omega (\mathbf{O}, S(2^j))\bigr) \\
& \geq C_{40}\left(\frac{1}{4}\right)\mu_{\text{cr}}^N\bigl(
 \Delta (\mathbf{O}, S(2^j))\bigr) \\
& \geq C_{40}\left(\frac{1}{4}\right)C_1^{-1}\mu_{\text{cr}}^N\bigl(
 \Delta (\mathbf{O}, S(2^{j-1}))\bigr)  \\
& \geq C_{40}\left(\frac{1}{4}\right)C_1^{-1}K^{-1}\mu_{\text{cr}}^N\bigl(
 \Gamma (\mathbf{O}, S(2^{j-1}))\bigr) \\  
& \geq C_{40}\left(\frac{1}{4}\right)C_1^{-1}K^{-1}\mu_{\text{cr}}^N\bigl(
 \Omega (\mathbf{O}, S(L_1)) \bigr) .
\end{align*}
%%
%%
%%
%%
%%\begin{align*}
%%\mu_t^N \bigl( \Gamma (v,S(2^j)) \bigr) &\leq \mu_t^N \bigl( \Gamma (v,S(L_1)) \bigr) \\
%% &\leq K \mu_t^N \bigl( \Delta (v,S(L_1)) \bigr) \\
%% &\leq K \mu_t^N \bigl( \Omega (v,S(L_1)) \bigr) \\
%% &\leq KC_{41}(\kappa ) \mu_{\text{cr}} \bigl( \Omega (\mathbf{O},S(L_1)) \bigr).
%%\end{align*}
Thus we have 
\begin{align}\label{eq:13-7}
&C_3^\# \dfrac{|h-h_c|}{\mathfrak{K}T} (L_1)^2 \mu_{\text{cr}}^N \bigl( \Omega (\mathbf{O},S(L_1)) \bigr) \\
&\leq C_{24} \left( 2^{-\zeta (\ell -j)} +  \dfrac{|h-h_c|}{\mathfrak{K}T} C2^{5(j+1)}e^{-\alpha 2^{j+1}} \right) , \nonumber
\end{align}
where we put $C_3^\#= 2^{-2}C_{40}(\frac{1}{4})C_1^{-1}K^{-1}$.
Letting $N\rightarrow \infty $, we obtain that
\begin{align*}
&C_3^\# \frac{|h-h_c|}{\mathfrak{K}T}L_1^2 \mu_{\text{cr}} \bigl( 
 \Omega (\mathbf{O},S(L_1))\bigr) \\
&\leq C_{24}\left(
 2^{\zeta (\ell -j)}+ \frac{|h-h_c|}{\mathfrak{K}T}2^{5(j+1)}e^{-\alpha 2^{j+1}}
 \right) .
\end{align*}
We show that the second term in the right hand side of the above inequality is small
compared with the left hand side.
Indeed, by Lemma \ref{relation of deltas}, and 
letting $N \rightarrow \infty $, we have
\begin{align*}
\mu_{\text{cr}} \bigl( \Omega (\mathbf{O},S(n)) \bigr) &\geq \mu_{\text{cr}} \bigl( \Delta (\mathbf{O},S(n)) \bigr) \\
&\geq C_2 C_1^{-\lfloor \log_2 n \rfloor}
\end{align*}
for every $n\geq 2^{j_1}$.
Hence we can find an integer $n_0$ such that for all $n \geq n_0$,
\[ \mu_{\text{cr}} \bigl( \Omega (\mathbf{O},S(n)) \bigr)
  \geq 2^{11}\frac{C_{24}}{C_3^\# } n^3 e^{-\alpha n}. \]
If $n_0 \leq L_1 \leq L_2 \leq L(h,\varepsilon_0)$, then from (\ref{eq:13-7}),
\begin{equation}\label{Thm4Eq**}
C_4^\# \dfrac{|h-h_c|}{\mathfrak{K}T} (L_1)^2 \mu_{\text{cr}} \bigl( \Omega (\mathbf{O},S(L_1)) \bigr) \leq \left( \dfrac{L_2}{L_1} \right)^{-\zeta},
\end{equation}
where we put $C_4^\# = C_3^\# / (2^{1+\zeta }C_{24})$.
%%By Lemma \ref{Lem1KThm4},
%%\begin{align*}
%%\dfrac{1}{4} C(1) C_1 C^{-1} \left( \dfrac{|h-h_c|}{\mathfrak{K}T} \right)^{-1} &\leq L(h,\varepsilon_0)^2 \mu_{\text{cr}}^N \bigl( \Omega (\mathbf{O},S(L(h,\varepsilon_0))) \bigr),
%%\end{align*}
%%provided that $|h-h_c|(<\delta_2)$ is sufficiently small. Letting $N \to \infty$,
%%\begin{align*}
%%C_5 \left( \dfrac{|h-h_c|}{\mathfrak{K}T} \right)^{-1} &\leq L(h,\varepsilon_0)^2 \mu_{\text{cr}} \bigl( \Omega (\mathbf{O},S(L(h,\varepsilon_0))) \bigr)
%%\end{align*}
%%This means that
%%\begin{equation} \label{Thm4Eq****}
%% L_0 \left( C_5^{-1}\dfrac{|h-h_c|}{\mathfrak{K}T} \right) \leq L(h,\varepsilon_0).
%%\end{equation}
We set 
\[ L_1 =  L_0 \left( \dfrac{|h-h_c|}{C_2^\# \mathfrak{K}T} \right) \quad \mbox{and} \quad L_2 = L(h,\varepsilon_0)  \]
in (\ref{Thm4Eq**}).
%%\[ C_4 \dfrac{|h-h_c|}{\mathfrak{K}T} (L_1)^2 \mu_{\text{cr}}^N \bigl( \Omega (\mathbf{O},S(L_1)) \bigr) \leq \left( \dfrac{L_2}{L_1} \right)^{-\zeta}.
%%\]
By definition of $L_0(\delta )$, the left hand side of (\ref{Thm4Eq**})
is not less than $C_4^\# C_2^\# $. Hence  we have
\[ C_4^\# C_2^\# \left( \dfrac{L_2}{L_1} \right)^{\zeta} \leq 1, \]
which is equivalent to
\begin{equation}\label{Thm4Eq*****}
L(h,\varepsilon_0) \leq (C_4^\# C_2^\# )^{-1/\zeta} L_0 
\left( \dfrac{|h-h_c|}{C_2^\# \mathfrak{K}T} \right).
\end{equation}
By (\ref{Thm4Eq****}) and (\ref{Thm4Eq*****}), 
\[ L(h,\varepsilon_0) \asymp L_0 \left( \dfrac{|h-h_c|}{C_2^\# \mathfrak{K}T} \right). \]

Finally, we shall show that $L_0(\delta) \asymp L_0((C_2^\# )^{-1}\delta)$ as $\delta \to 0$. We can assume that $C_2^\# >1$; otherwise we replace $\delta$ with 
$(C_2^\#)^{-1} \delta$. 
Note that $L_0 (\delta)$ is non-increasing in $\delta$.
%% Indeed, if $\delta_1 \geq \delta_2$, then 
%%\[ n^2  \mu_{\text{cr}} \bigl( \Omega (\mathbf{O},S(n)) \bigr) \geq \delta_2^{-1}
%%\quad\mbox{implies that}\quad
%%n^2  \mu_{\text{cr}} \bigl( \Omega (\mathbf{O},S(n)) \bigr) \geq \delta_1^{-1}. \]
%%So $n \geq L_0(\delta_2)$ implies that $n \geq L_0(\delta_1)$. This means $L_0(\delta_2) \geq L_0(\delta_1)$. 
So, the inequality 
$L_0(\delta) \leq L_0((C_2^\# )^{-1}\delta)$ follows. 
For the other direction, note that
\[ L_0 ((C_2^\# )^{-1} \delta) \leq L(h,\varepsilon_0) \quad \mbox{if } \delta \geq \dfrac{|h-h_c|}{\mathfrak{K}T}. \]
We choose an $h$ such that 
\[ \delta \geq \dfrac{|h-h_c|}{\mathfrak{K}T} \geq \dfrac{\delta}{2}. \]
If we set 
\[ L_1 =  L_0 (\delta) (\leq L_0 ( (C_2^\# )^{-1}\delta )) \quad \mbox{and} \quad L_2 = L(h,\varepsilon_0)  \]
in (\ref{Thm4Eq**}), then we have
\[ \dfrac{C_4^\# }{2} \leq 
C_4^\# \dfrac{|h-h_c|}{\mathfrak{K}T} (L_0(\delta))^2 \mu_{\text{cr}}^N \bigl( \Omega (\mathbf{O},S(L_0(\delta ))) \bigr) \leq \left( \dfrac{L(h,\varepsilon_0) }{L_0(\delta)} \right)^{-\zeta}.
\]
This shows that
\[  \dfrac{L(h,\varepsilon_0) }{L_0(\delta)} \leq \left( \dfrac{2}{C_4^\# } \right)^{1/\zeta}, \]
which is equivalent to 
\[ L(h,\varepsilon_0) \leq L_0(\delta) \left( \dfrac{2}{C_4^\# } \right)^{1/\zeta}. \]
Combining the above inequalities, we have
\[ \left( \dfrac{C_4^\#}{2} \right)^{1/\zeta}L_0 ( (C_2^\# )^{-1}\delta ) \leq L_0(\delta) \leq L_0 ((C_2^\# )^{-1} \delta). \]
This completes the proof of (\ref{eq: L equiv L_0below}) and (\ref{eq: L equiv L_0above}).
As for (\ref{L-varying epsilon}), all the constants change depending on $\varepsilon_i$ in place of $\varepsilon_0$. 
But we have
$$
C_{42}(\varepsilon_i)\leq \frac{L(h_c-\delta , \varepsilon_i)}{L_0(\delta )} \leq 
C_{43}(\varepsilon_i) .
$$
as above for each $i=1,2$. 
Since $L_0$ does not depend on the choice of $\varepsilon_i$, this proves
(\ref{L-varying epsilon}).
\end{proof}
%% \end{quote}

\newpage
\noindent {\bf Acknowledgement.} This work was supported by JSPS KAKENHI Nos. 17340029, 22244007 and 21740087, and OECU Foundation for Global Partnership.

%\newpage
\begin{flushleft}
Yasunari Higuchi

Department of Mathematics,\\
Kobe University,\\
1-1 Rokko, Kobe 657-8501, Japan.

E-mail: {\tt higuchi@math.kobe-u.ac.jp}
\vspace{2mm}

Masato Takei

Department of Engineering Science, Faculty of Engineering, \\
Osaka Electro-Communication University,\\
Neyagawa, Osaka 572-8530, Japan.

E-mail: {\tt takei@isc.osakac.ac.jp}
\vspace{2mm}

Yu Zhang

Department of Mathematics,\\
University of Colorado,\\
Colorado Springs, CO80933, USA.

E-mail: {\tt yzhang3@uccs.edu}
\end{flushleft}

\end{document}